\documentclass[12pt]{amsart}
\usepackage{amssymb,latexsym,amscd,verbatim,color}
\usepackage[T1]{fontenc}
\usepackage[all]{xy}
\usepackage[colorlinks,linkcolor=blue,citecolor=blue,urlcolor=red]{hyperref}

\swapnumbers
\newtheorem{thm}{Theorem}[subsection]
\newtheorem{propose}[thm]{Proposition}
\newtheorem{lemma}[thm]{Lemma}
\newtheorem{cor}[thm]{Corollary}
\newtheorem{hyp}[thm]{Hypothesis}
\theoremstyle{definition}
\newtheorem{defn}[thm]{Definition}
\newtheorem{notation}[thm]{Notation}
\newtheorem{remark}[thm]{Remark}
\newtheorem{remarks}[thm]{Remarks}
\newtheorem{example}[thm]{Example}
\newtheorem{examples}[thm]{Examples}
\newcounter{spec}
\newenvironment{thlist}{\begin{list}{\rm{(\roman{spec})}}%
{\usecounter{spec}\labelwidth=20pt\itemindent=0pt\labelsep=10pt}}%
{\end{list}}

\renewcommand{\d}{{\text{\LARGE $\cdot $}}}
\newcommand{\Xs}{X_{\d}}            
\renewcommand{\hat}{\widehat}
\newcommand{\Spec}{\operatorname{Spec}} 
\newcommand{\Hom}{\operatorname{Hom}}      
\newcommand{\Aut}{\operatorname{Aut}}      
\newcommand{\Map}{\operatorname{Map}}      
\newcommand{\Ext}{\operatorname{Ext}}      
\newcommand{\Biext} {\operatorname{Biext}}  
\newcommand{\RHom}{\operatorname{RHom}}    
\newcommand{\DM}{\operatorname{DM}}          
\newcommand{\DA}{\operatorname{DA}}          
\newcommand{\M}{\mathcal{M}_1}   
\newcommand{\MM}{\mathcal{MM}}  
\newcommand{\MHS}{{\operatorname{MHS}}}  
\renewcommand{\1}{{}_{\leq 1}}
\newcommand{\Shv}{\operatorname{Shv}}
\newcommand{\Sm}{\operatorname{Sm}}
\newcommand{\Sch}{\operatorname{Sch}}
\newcommand{\Hodge}{{\operatorname{Hodge}}}
\newcommand{\Mod}{{\operatorname{Mod\text{--}}}}

\newcommand{\op}{{\operatorname{op}}}
\newcommand{\PST}{\operatorname{PST}}
\newcommand{\NST}{{\operatorname{NST}}}
\newcommand{\EST}{{\operatorname{EST}}}

\newcommand{\ES}{{\operatorname{ES}}}
\newcommand{\EH}{{\operatorname{EH}}}
\newcommand{\HI}{\operatorname{HI}}
\newcommand{\CH}{\operatorname{CH}}

\newcommand{\Chow}{\operatorname{Chow}}

\newcommand{\End}{\operatorname{End}}      
\newcommand{\car}{\operatorname{char}}
\newcommand{\Tr}{\operatorname{Tr}}
\newcommand{\Div}{\operatorname{Div}}

\newcommand{\sext}{\text{${\mathcal E}xt\,$}}  
\newcommand{\shom}{\text{${\mathcal H}om\,$}}  
\newcommand{\srhom}{\operatorname{\mathcal{RH}om}}
\newcommand{\ihom}{{\rm\underline{Hom}}}  
\newcommand{\oo}{\operatornamewithlimits{\otimes}\limits}

\renewcommand{\P}{\mathbb{P}}   
\newcommand{\Aff}{\mathbb{A}}   
\newcommand{\A}{{\rm\underline A}}      
\newcommand{\E}{{\rm\underline E}}      
\newcommand{\sH}{\mathcal{H}}
\newcommand{\sO}{\mathcal{O}}

\newcommand{\C}{\mathbb{C}}     
\newcommand{\F}{\mathbb{F}}     
\newcommand{\Q}{\mathbb{Q}}     
\newcommand{\Z}{\mathbb{Z}}     
\renewcommand{\L}{\mathbb{L}}
\newcommand{\N}{\mathbb{N}}
\newcommand{\G}{\mathbb{G}}     
\newcommand{\HH}{\mathbb{H}}    
\newcommand{\EExt}{{\rm \mathbb{E}xt}} 
\newcommand{\R}{\mathbb{R}}     
  
\newcommand{\bT}{\mathbf{T}}       
\newcommand{\uG}{\underline{G}}
\newcommand{\uH}{\underline{H}}
\newcommand{\uF}{\underline{F}}
\newcommand{\uR}{\underline{R}}
\newcommand{\uT}{\underline{\rm T}}

\newcommand{\im}{\operatorname{Im}}        
\renewcommand{\ker}{\operatorname{Ker}}  
\newcommand{\coker}{\operatorname{Coker}} 
\newcommand{\gr}{\operatorname{gr}}        
\newcommand{\Pic}{\operatorname{Pic}}     
\newcommand{\RPic}{\operatorname{RPic}}     
\newcommand{\Alb}{\operatorname{Alb}}     
\newcommand{\LAlb}{\operatorname{LAlb}}     
\newcommand{\SAb}{\operatorname{SAb}}
\newcommand{\AbS}{\operatorname{AbS}}

\newcommand{\LA}[1]{\mbox{${\rm L}_{#1}{\rm Alb}$}}
\newcommand{\RA}[1]{\mbox{${\rm R}^{#1}{\rm Pic}$}}

\newcommand{\DR}{{\operatorname{DR}}}
\newcommand{\Mix}{\operatorname{Mix}}
\newcommand{\Perv}{\operatorname{Perv}}
\newcommand{\Cor}{\operatorname{Cor}}

\newcommand{\Tot}{\operatorname{Tot}}     
\newcommand{\tot}{\operatorname{tot}}     
\newcommand{\rk}{\operatorname{rk}}    
\newcommand{\codim}{\operatorname{codim}}
\newcommand{\NS}  {\operatorname{NS}}      
\newcommand{\rank}{\operatorname{rank}}    
\newcommand{\tors}{{\operatorname{tors}}}
\newcommand{\sing}{{\operatorname{sing}}}
\newcommand{\cone}{{\operatorname{cone}}}
\newcommand{\tr}{{\operatorname{tr}}}        
\newcommand{\qi}{{\rm q.i.}\,}      

\newcommand{\by}[1]{\stackrel{#1}{\rightarrow}}
\newcommand{\longby}[1]{\stackrel{#1}{\longrightarrow}}
\newcommand{\vlongby}[1]{\stackrel{#1}{\mbox{\large{$\longrightarrow$}}}}
\newcommand{\iso}{\longby{\sim}}
\newcommand{\longyb}[1]{\stackrel{#1}{\longleftarrow}}
\newcommand{\osi}{\longyb{\sim}}

\renewcommand{\tilde}{\widetilde}
\newcommand{\df}{\mbox{\,${:=}$}\,}
\newcommand{\ie}{{\it i.e. }}
\newcommand{\cf}{{\it cf. }}
\newcommand{\eg}{{\it e.g. }}
\newcommand{\ibid}{{\it ibid. }}
\newcommand{\resp}{{\it resp. }}
\newcommand{\loccit}{{\it loc. cit. }}
\newcommand{\et} {{\operatorname{\acute{e}t}}}
\newcommand{\eh} {{\operatorname{\acute{e}h}}}
\newcommand{\Zar}{{\operatorname{Zar}}}
\newcommand{\Nis}{{\operatorname{Nis}}}
\newcommand{\cdh}{{\operatorname{cdh}}}
\newcommand{\an}{{\operatorname{an}}}
\newcommand{\fr}{{\operatorname{fr}}}
\newcommand{\tor}{{\operatorname{tor}}}

\newcommand{\red}{{\operatorname{red}}}
\newcommand{\tf}{{\operatorname{tf}}}
\newcommand{\eff}{{\operatorname{eff}}}

\newcommand{\gm}{{\operatorname{gm}}}
\newcommand{\sn}{{\operatorname{sn}}}
\newcommand{\sab}{{\operatorname{sab}}}
\newcommand{\Fr}{\operatorname{Fr}}
\newcommand{\sm}{{\operatorname{sm}}}
\newcommand{\un}{\mathbf{1}}

\renewcommand{\bar}{\overline}
\newcommand{\into}{\hookrightarrow}
\renewcommand{\implies}{\mbox{$\Rightarrow$}}
\newcommand{\veq}{\mbox{\large $\parallel$}}  
\newcommand{\sZ}{\mbox{\scriptsize{$\Z$}}}   
\newcommand{\sC}{\mbox{\scriptsize{$\C$}}}   
\newcommand{\sQ}{\mbox{\scriptsize{$\Q$}}}   
\renewcommand{\c}{\mbox{\scriptsize{$\circ$}}} 

\newcommand{\limdir}[1]{\mathop{\rm
lim}_{\buildrel\longrightarrow\over{#1}}}
\newcommand{\liminv}[1]{\mathop{\rm
lim}_{\buildrel\longleftarrow\over{#1}}}
\renewcommand{\lim}{\varprojlim}
\newcommand{\colim}{\varinjlim}
\newcommand{\onto}{\mbox{$\to\!\!\!\!\to$}}

\newcommand{\boxtensor}{\def\boxtimesten{\Box\kern-7.59pt\raise1.2pt
\hbox{$\times$} }}                                  

\newcounter{elno}                   

\newcommand{\cA}{\mathcal{A}}
\newcommand{\cB}{\mathcal{B}}

\newcommand{\cD}{\mathcal{D}}
\newcommand{\cE}{\mathcal{E}}
\newcommand{\cF}{\mathcal{F}}
\newcommand{\cG}{\mathcal{G}}
\newcommand{\cH}{\mathcal{H}}
\newcommand{\cI}{\mathcal{I}}

\newcommand{\cM}{\mathcal{M}}

\newcommand{\cO}{\mathcal{O}}
\newcommand{\cP}{\mathcal{P}}

\newcommand{\cS}{\mathcal{S}}
\newcommand{\cT}{\mathcal{T}}
\newcommand{\cU}{\mathcal{U}}

\newcommand{\cX}{\mathcal{X}}

\newcommand{\cMR}{\mathcal{MR}}

\renewcommand{\phi}{\varphi}
\renewcommand{\epsilon}{\varepsilon}

\makeindex

\begin{document}

\title{On the derived category of 1-motives}
\begin{center}
Sur la cat\'egorie d\'eriv\'ee des $1$-motifs
\end{center}
\author{Luca Barbieri-Viale}
\address{Dipartimento di Matematica ``F. Enriques", Universit{\`a} degli Studi di Milano\\ Via C. Saldini, 50\\ I-20133 Milano\\ Italy}
\email{luca.barbieri-viale@unimi.it}
\author{Bruno Kahn}
\address{IMJ-PRG\\Case 247 \\ 4
Place Jussieu\\75252 Paris Cedex 05\\ France}
\email{bruno.kahn@imj-prg.fr}

\begin{abstract}
We embed the derived category of  Deligne 1-motives over a perfect field into the \'etale version of Voevodsky's triangulated category of
geometric motives, after inverting the
exponential characteristic.  We then show that this full embedding ``almost"
has a left adjoint $\LAlb$.  Applying $\LAlb$ to the motive of a variety we 
get a bounded complex of 1-motives, that we compute fully for smooth
varieties and partly for singular varieties. Among applications, we give motivic
proofs of Ro\v\i tman type theorems and new cases of 
Deligne's conjectures on $1$-motives.\\

\noindent {\sc R\'esum\'e.} Nous plongeons la cat\'egorie d\'eriv\'ee des $1$-motifs de Deligne sur un corps parfait dans la version \'etale de la cat\'egorie triangul\'ee des motifs g\'eom\'etriques de Voevodsky, apr\`es avoir invers\'e l'exposant caract\'eristique. Nous montrons ensuite que ce plongement a ``presque'' un adjoint \`a gauche $\LAlb$. En appliquant $\LAlb$ au motif d'une vari\'et\'e, on obtient un complexe de $1$-motifs, que nous calculons enti\`erement dans le cas des vari\'et\'es lisses et partiellement dans le cas des vari\'et\'es singuli\`eres. Parmi les applications, nous donnons des preuves motiviques de th\'eor\`emes de type Ro\v\i tman, et \'etablissons de nouveaux cas des conjectures de Deligne sur les $1$-motifs.
\end{abstract}
\thanks{The first author acknowledges the support of {\it Agence Nationale de la Recherche} (ANR) under reference ANR-12-BL01-0005.
The second author acknowledges the support of the {\it Ministero dell'Istruzione,  dell'Universit\`a e della Ricerca} (MIUR)  through the Research Project  (PRIN 2010-11) ``Arithmetic Algebraic Geometry and Number Theory''.}
\maketitle

\newpage

\tableofcontents

\newpage
\section*{Introduction}

In this book, we compare two categories of motivic nature: the derived category of Deligne's $1$-motives and Voevodsky's triangulated category of motives, and draw several geometric applications.
\bigskip

\enlargethispage*{30pt}

Let us recall the players of this story. Inspired by his theory of mixed Hodge structures, Deligne  introduced $1$-motives in \cite{D} as an algebraic version of ``Hodge theory in level $\le 1$'': they were the first nontrivial examples of mixed motives (as opposed to Grothendieck's pure motives \cite{MA,demazure,kleiman,scholl,andre}). We shall denote the category of $1$-motives over a field $k$  by $\M(k)$ or $\M$.

A different step towards mixed motives was taken much later by
Voevodsky, who defined in \cite{V} a \emph{triangulated category} of
motives $\DM^{\eff}_{\gm}(k)$\index{$\DM^{\eff}_{\gm}$}. Taken with rational coefficients, this
category is conjectured to have a ``motivic" $t$-structure whose heart
should be the searched-for abelian category of mixed motives.

Since $\M(k)$ is expected to be contained in such a heart, it is natural to try and
relate it with $\DM^{\eff}_{\gm}(k)$. This can be done rationally.  Denote by $\M(k)\otimes\Q$ the abelian category  of
1-motives up to isogeny over $k$. When $k$ is perfect, Voevodsky asserts in \cite[p. 218]{V} (see
also \cite[Pretheorem 0.0.18]{V0}) that there exists a
\emph{fully faithful functor}
\begin{equation}\label{eqVO}
\Tot^\Q:D^b(\M(k)\otimes\Q)\into \DM^{\eff}_{\gm}(k)\otimes\Q
\end{equation}
whose essential
image is the thick subcategory generated by motives of smooth curves. This is  justified by 
F. Orgogozo in  \cite{OR}.
\bigskip

Our main result is that the functor $\Tot^\Q$ has a \emph{left adjoint} $\LAlb^\Q$ (Theorem \ref{ladj}). Much of this can in fact be done integrally, or more accurately $\Z[1/p]$-integrally where $p$ is the exponential characteristic of $k$. First, the embedding $\Tot^\Q$ has a $\Z[1/p]$-integral version $\Tot$ provided we replace $\DM_\gm^\eff(k)$ by its \'etale variant (Theorem \ref{t1.2.1}). Then $\LAlb^\Q$ has an integral version $\LAlb:\DM_\gm^\eff(k)\to D^b(\M(k)[1/p])$ (Definition \ref{LAlb}). It is convenient to denote by $\RPic$ the composition of $\LAlb$ with the Cartier duality of $D^b(\M[1/p])$. These notations want to suggest: derived Albanese, derived Picard functor.
\bigskip

Applying $\LAlb$  to an object  $M\in \DM_\gm^\eff(k)$ and taking $1$-motivic homology, we get a series of $1$-motives $\LA{i}(M)$, $i\in\Z$. Taking for example $M=M(X)$ for $X$ a $k$-variety, we get new invariants $\LA{i}(X)$ of $X$ with values in the category of $1$-motives, as well as their duals $\RA{i}(X)$. 

We then compute the $\LA{i}(X)$ when $X$ is a smooth variety, justifying the notation (Theorem \ref{trunc}).
We partly extend this computation to singular schemes in Section \ref{comps}; in many cases, we recover for $i=1$ the  homological and cohomological Albanese and Picard 1-motives $\Alb^-(X)$, $\Alb^+(X)$, $\Pic^-(X)$ and $\Pic^+(X)$ constructed in \cite{BSAP} by the first author and Srinivas (Section \ref{12}).
\bigskip

This in turn sheds a new light on Ro\v\i tman's theorem on torsion in groups of $0$-cycles modulo rational equivalence and yields many generalisations of it, as is done in Section \ref{rovi}.\bigskip

Let us come back to Deligne's  $1$-motives. In \cite[(10.4.1)]{D}, he conjectured that some Hodge-theoretic constructions could be described as realisations of \emph{a priori} constructed $1$-motives; he also conjectured a compatibility with $l$-adic and de Rham cohomology.  In \cite{BRS}, part of the first conjecture was proven rationally (see also \cite{ram2}).

What started us on this project was the desire to use the full embedding $\Tot^\Q$ of \eqref{eqVO} to give a natural proof of Deligne's conjecture. In the spirit of the present work, we tackle this problem by first considering an axiomatically defined realisation functor and get an abstract result, Theorem \ref{abstractdelconj}, by comparing $\LAlb^\Q$ with a corresponding functor on the level of categories of realisations. 

To get to Deligne's conjecture, we use Huber's mixed realisation \cite{HuberH}. Using its Hodge component, we get the first part of Deligne's conjecture in Corollary \ref{isodelcor}. Using its Hodge and $l$-adic (resp. de Rham) components, we then get the second part of Deligne's conjecture in Theorem \ref{isodelMR}.\footnote{This description is not completely correct, see Remark \ref{notall}. Deligne's conjecture in \cite{D} concerns three cases $I, II_n$ and $II_N$. We prove $I$ and $II_N$ but not $II_n$. On the other hand, we get many cases not considered by Deligne where the obvious analogue of his conjecture is true.} The latter is new relatively to the existing literature.

Finally, we examine what happens in characteristic $p$, using the $l$-adic realisation functor of Ivorra \cite{ivorra} and Ayoub \cite{real.etale}. Not surprisingly, we get a version of ``Deligne's conjecture'', whose truth depends this time on the Tate conjecture in codimension $1$: Theorem \ref{telladic}.
\bigskip

In the first version of this book, the proof of Proposition \ref{ext=ext} contained a gap, which was kindly pointed out by J. Riou and J. Ayoub. One solution to save this proof was to tensor coefficients with $\Q$ and use the comparison theorem\footnote{Established by F. Morel over a field, this comparison theorem was extended by D.-C. Cisinski and F. D\'eglise \cite{cis-deg2} to a general base  and  J. Ayoub gave a simplified proof in \cite[Ann. B]{ayoubHopf}. See also \cite{AV}.} between motives with and without transfers.  Rather than doing this, we looked at Ext groups in more detail, which added some length to Section \ref{homotopy} but saved Proposition \ref{ext=ext} integrally.

\bigskip

Let us stress that we have to invert the exponential characteristic $p$ of the base field $k$ everywhere. This is due to several reasons:

\begin{itemize}
\item Since the category $\DM_{-,\et}^\eff(k)$ of Definition \ref{dmet} is $\Z[1/p]$-linear by \cite[Prop. 3.3.3 2)]{V}, we cannot expect better comparison results.
\item To be in the spirit of Voevodsky, we want to use only the \'etale topology and not the fppf topology which would be more natural from the viewpoint of $1$-motives. Trying to prove anything meaningful without inverting $p$ in this context seems doomed to failure.
\end{itemize}

The basic reason why $p$ is inverted in $\DM_{-,\et}^\eff(k)$ is homotopy invariance (the
Artin-Schreier exact sequence). In parallel, if one wants to deal with non homotopy invariant phenomena,
Deligne $1$-motives are not sufficient and one should enlarge them to include $\G_a$ factors as
in Laumon's $1$-motives (\cf  \cite{lau}, \cite{BM}). See \cite{alessandra} and \cite{alessandragen} for work in this
direction.

\enlargethispage*{30pt}

\subsubsection*{Acknowledgements}
This work has been done during periods of stay of the first author
at the IH\'ES and the ``Institut de Math\'ematiques de Jussieu" University of Paris 7 and of
the second author at the Universities of Milano, Padova and Roma ``La Sapienza''. Both authors jointly visited the Centro di Ricerca Matematica ``Ennio De Giorgi" Scuola Normale Superiore of Pisa under the program {\it Research in Pairs}. We like to thank these
institutions for their hospitality and financial support\footnote{As well as the Eyjafjallaj\"okull  volcano for delayed touches to this work.}. 

The second  author also thanks the first for his warm hospitality in Pietrafredda, where much of this work was conceived. We also thank J. Ayoub, A.
Bertapelle, C. Bertolin, D. Bertrand, D. Ferrand, S. Pepin Lehalleur, M. Hindry, L. Illusie, P. Jossen, C. Mazza, M. Ojanguren, J. Riou, M. Saito, P. Schapira, T. Szamuely, B. Totaro, C. Voisin, V. Vologodsky, J. Wildeshaus, H. Zhao for helpful comments or suggestions regarding this work. The first author is grateful to
all his coauthors in the subject of $1$-motives; the second author would like to acknowledge the inspiration provided by the paper
of Spie\ss-Szamuely \cite{spsz}. 

We thank the referee for an outstanding report and his/her many suggestions for improvement.

Finally, our intellectual debt to A. Grothendieck, P. Deligne and V. Voevodsky is
evident.

\newpage
\subsection*{General assumptions and notations.}
\emph{Throughout, $k$ is a perfect field of exponential characteristic $p$. We write
$\Sm(k)$ \index{$\Sch(k)$, $\Sm(k)$} for the category of smooth separated $k$-schemes of finite type
and $\Sch(k)$  for the category of all separated $k$-schemes of finite type. If $\cA$ is an additive category, we let $\cA[1/p]$ denote 
the corresponding $\Z[1/p]$-linearised category (Hom groups tensored with $\Z[1/p]$, compare Definitions \ref{1mot} and \ref{1p}).}

\section*{Outline}
We now give a detailed overview of the contents of this work.

\subsection{The derived category of $1$-motives, integrally}\label{0.1} While the $\Z[1/p]$-linear category $\M[1/p]$ is not an abelian category, it fully embeds into the abelian category ${}^t\M[1/p]$ of \emph{$1$-motives with torsion} introduced in \cite{BRS}, which makes it an exact category in the sense of Quillen (see \S \ref{1.2d}). Its derived category $D^b(\M[1/p])$ with respect to this exact structure makes
sense, and moreover the functor $D^b(\M[1/p])\to D^b({}^t\M[1/p])$ turns out to be an equivalence (Theorem \ref{ptors}).

\subsection{$\Z[1/p]$-integral equivalence} \label{0.2}Let
$\DM_{\gm,\et}^\eff\df\DM_{\gm,\et}^\eff(k)$ be the thick subcategory of
$\DM_{-,\et}^\eff(k)$ generated by the image of $\DM^{\eff}_{\gm}(k)$ under the ``change of topology'' functor $\alpha^s:\DM_{-}^\eff\to \DM_{-,\et}^\eff$ (see Corollary \ref{cD.1} and Definition \ref{d2.1.1}). Let $d\1\DM_{\gm,\et}^\eff$ be the thick
subcategory of $\DM_{\gm,\et}^\eff$ generated by motives of smooth
curves.  In Theorem \ref{t1.2.1}, we refine the full embedding $\Tot^{\Q}$ of \eqref{eqVO} to an equivalence of categories
\begin{equation}\label{eqeq}
D^b(\M[1/p])\iso d\1\DM_{\gm,\et}^\eff.
\end{equation}

\subsection{Duality}\label{0.3} Deligne's extension of Cartier duality to 1-motives \cite{D}
provides the category of 1-motives with a natural involution $M
\mapsto M^*$ which extends to $D^b(\M[1/p])$: see Proposition \ref{pcd}. This duality exchanges
the category ${}^t\M[1/p]$ of  \S \ref{0.1} with a new abelian category ${}_t\M[1/p]$ of \emph{$1$-motives
with cotorsion} (see \S \ref{1.8}). Rationally, the two $t$-structures give back the standard one, corresponding to the abelian category $\M\otimes\Q$.

We show in Theorem \ref{teq} that, under $\Tot$, Deligne's Cartier duality is transformed into
the involution $M\mapsto  \ihom_\et (M,\Z_\et (1))$ on $d\1\DM^{\eff}_{\gm,\et}$ given by
the internal (effective) Hom (see \S \ref{2.5} and Proposition \ref{cd}). Of course, this result involves biextensions. 

\subsection{Left adjoint}\label{leftadj} Composing \eqref{eqeq} with the
inclusion into $\DM_{\gm,\et}^\eff$, we obtain a ``universal
realisation functor"
\[\Tot:D^b(\M[1/p])\to \DM_{\gm,\et}^\eff.\] 

It was conjectured by Voevodsky (\cite{V1}; this is also implicit in \cite[Preth. 0.0.18]{V0})
that, rationally, $\Tot$ has a left adjoint. We prove this in Section \ref{qlalb}.

It is shown in Remark \ref{noleft} that $\Tot$ does not have a left adjoint integrally.
There is nevertheless an integral statement, which involves an interplay between the \'etale and
the
\emph{Nisnevich} topology. Considering  the functor $\alpha^s:\DM_\gm^\eff\to \DM_{\gm,\et}^\eff$ from Definition \ref{d2.1.1}, we find a functor
\[\LAlb: \DM^{\eff}_{\gm}\to D^b(\M[1/p])\]
together with a \emph{motivic  Albanese map}
\begin{equation}\label{eq0.2}
a_M:\alpha^sM \to \Tot \LAlb (M)
\end{equation} 
natural in $M\in \DM_\gm^\eff$. 
 Thus if $(M,N)\in \DM_\gm^\eff\times D^b(\M[1/p])$,
there is a functorial homomorphism
$$\Hom (\LAlb M, N)\to \Hom (\alpha^sM, \Tot(N))$$
which is an isomorphism rationally (but not integrally in general, see Remark \ref{noleft}).

We give the construction of $\LAlb$ and  the motivic Albanese map in Section \ref{lalb}. This is the central result of the present work.

Experimentally, $\LAlb$ is best adapted to the $t$-structure with heart ${}_t\M[1/p]$ (see above): thus we define for any $M\in \DM_\gm^\eff$ and $i\in\Z$
\[\LA{i}(M)={}_tH_i(\LAlb(M))\]
the homology relative to this $t$-structure.

\subsection{Smooth schemes} \label{0.5}
In Theorem \ref{trunc}, we compute  $\LAlb (X):=\LAlb(M(X))$ for any smooth scheme $X$: in principle this determines $\LAlb$ on the whole of $\DM_\gm^\eff$ since this category
is generated by the $M(X)$. It is related to the ``Albanese scheme" $\cA_{X/k}$ of \cite{ram}
(extending the Serre Albanese variety of \cite{serrealb}) in the following way: $\LAlb(X)$ is a
``$3$-extension" of $\cA_{X/k}$ by the Cartier dual of the N\'eron-Severi group of $X$, that we
define as the \'etale sheaf given by cycles of codimension $1$ on $X$ modulo algebraic
equivalence.   We deduce that
$\LA{1}(X)$ is isomorphic to the $1$-motive $\Alb^- (X)$ of \cite{BSAP} (corresponding to the Serre Albanese).

\subsection{$\LAlb$ and $\RPic$} Composing
$\LAlb$ with duality, we obtain a
contravariant functor
$$\RPic: \DM^{\eff}_{\gm}\to D^b(\M[1/p])$$
such that 
\[\RA{i}(M)\df{}^tH^i(\RPic (M)) \simeq {}_tH_i(\LAlb (M))^*\]
for any $M\in \DM^{\eff}_{\gm}$. Here, 
${}^t H^i$ is defined with respect to the $t$-structure with heart ${}^t\M[1/p]$. We call
$\RPic$ the {\it motivic Picard}\, functor. We define the \emph{cohomological Picard
complex} by  $\RPic (X):=\RPic (M(X))$.

\subsection{Singular schemes}\label{sing}
When $k$ is of characteristic $0$ (but see  \S  \ref{s8.1} for the issue of positive characteristic), the motive and motive with compact
support $M(X)$ and $M^c(X)$ are defined for any $k$-scheme of finite type $X$ as objects
of $\DM_\gm^\eff$, so that
$\LAlb(X)$ and the 
\emph{Borel-Moore Albanese complex} $\LAlb^c (X) \df \LAlb (M^c
(X))$ make sense. We further
define, for an equidimensional scheme $X$ of dimension
$n$, the \emph{cohomological Albanese complex}  $\LAlb^* (X)
\df\allowbreak \LAlb (M (X^{(n)})^*(n)[2n])$ where $X^{(n)}$ is the union of the $n$-dimensional components of $X$ and $(-)^*$ is the dual in $\DM_\gm$. (Note that $M (X^{(n)})^*(n)[2n]$ is effective by Lemma \ref{leffe}). Similarly with  
$\RPic$.  We give general properties of these complexes in Section \ref{6}.

We then give some qualitative estimates for $\LA{i}(X)$ in Proposition \ref{c3.1} (see also Proposition \ref{c3.1bis}) as well as $\LA{i}^c (X)\allowbreak\df {}_tH_i(\LAlb^c (X))$ in
Proposition \ref{c3.1c}. We prove that $\LA{1}(X)$ is canonically isomorphic to the $1$-motive $\Alb^-(X)$ of \cite{BSAP} if $X$ is normal (Proposition \ref{c12.2.2})  or proper (Corollary \ref{c12.3}). Here, the interplay between $\LAlb$ and $\RPic$ (duality between Picard and Albanese) plays an essential role. We also prove  in Theorem~\ref{*=-} that $\RA{1}^*(X)\simeq \Pic^-(X)$, hence $\LA{1}^*(X)\simeq \Alb^+(X)$, for any $X$.

It is striking that $\LA{i}(X), \LA{i}^c(X)$ and $\LA{i}^*(X)$ are actually Deligne $1$-motives for $i\le 1$, but  have cotorsion in general for $i\ge 2$ (already for $X$ smooth projective).

\subsection{Curves} We completely compute $\LA{i}(X)$ for any curve $X$, showing that $M (X)$
has a ``Chow-K\"unneth decomposition" in 
$\DM^{\eff}_{\gm,\et}\otimes\Q$ and that the
$\LA{i} (X)$ coincide with Deligne-Licht\-en\-baum motivic homology
of the curve $X$ (see Theorem~\ref{lasc}, \cf \cite{LI}, \cite{LI1} and \cite{BH}).
Dually, we recover Deligne's 1-motivic $H^1$ of any curve: namely,  
$\RA {1} (X) = H^1_m(X)(1)$ as denoted in \cite{D}. We also compute $\LA{i}^c (X)$
of a smooth curve $X$ (see Theorem~\ref{bmlac}).

\enlargethispage*{20pt}

\subsection{Ro\v\i tman's theorem}\label{0.8} If $X$ is smooth projective, the motivic Albanese map \eqref{eq0.2}
applied to $M=M(X)$ gives back the Albanese map from the $0$-th Chow group to the rational points of the
Albanese variety. This translates very classical mathematics to a motivic setting.  When $X$ is only smooth, we
recover a generalised Albanese map from Suslin homology
\[a_X^\sing:H_0^\sing(X;\Z)[1/p]\to \cA_{X/k}(k)[1/p]\] 
which was first constructed by Ramachandran \cite{ra1} and
Spie\ss-Szamuely \cite{spsz}.\footnote{The
observation that Suslin homology is related to $1$-motives is initially
due to Lichtenbaum \protect\cite{LI}.} The map $a_X^\sing$ is an isomorphism if
$\dim (X)\leq 1$ (see Proposition~\ref{sus}). 

We then get a natural proof of the theorem of Ro\v\i
tman on torsion $0$-cycles and its generalisation to open smooth varieties by Spie\ss-Sza\-mu\-e\-ly \cite[Th.
1.1]{spsz} (removing their hypothesis on the existence of a smooth compactification): see Theorem
\ref{troitclas}. 

We also deal with singular schemes when $\car k=0$ (but see \S \ref{s8.1}) in Proposition
\ref{p12.3} and its corollaries. We get a Borel-Moore version of
Ro\v\i tman's theorem  as well, see Proposition \ref{p13.4} and its corollary.
Finally, we obtain a ``cohomological" Ro\v\i tman theorem, involving torsion in a motivic
cohomology group: see Corollary \ref{c14.4}. 

\subsection{The homotopy $t$-structure and $1$-motivic sheaves} We already have two dual $t$-structures  on $D^b(\M[1/p])$, see \ref{0.3}. The homotopy $t$-structure on $\DM_{-,\et}^\eff$ and the equivalence of categories \eqref{eqeq} induce a \emph{third}
$t$-structure (Theorem \ref{t3.2.3}; see also Corollary \ref{c3.3.2}). Its heart is formed of so-called
\emph{$1$-motivic sheaves}: their consideration is central both for the duality theorem \ref{teq} and for the computation of
$\LAlb(X)$ for smooth $X$ (Theorem \ref{trunc}). This idea was pursued by Ayoub--Barbieri-Viale in \cite{ABV} and by
Bertapelle in \cite{alessandra}.

\subsection{Internal Hom and tensor structure} In Corollary \ref{c3.13}, we show that the internal Hom of $\DM_{-,\et}^\eff$ yields an internal Hom, $\ihom_1$, on $D^b(\M\otimes\Q)$ via the equivalence \eqref{eqeq}; for an integral refinement see Remark \ref{rkc3.13}. This internal Hom has cohomological dimension $1$ with respect to the homotopy $t$-structure (Corollary \ref{l3inv}) and is exact with respect to the standard $t$-structure (Theorem \ref{t3.14}).  In Proposition \ref{tens1}, we use $\LAlb$ to
show that $\ihom_1$ has a left adjoint $\otimes_1$. This tensor structure on $D^b(\M\otimes\Q)$ is exact (for the standard $t$-structure), respects the weight filtration and may be computed explicitly. 

\subsection{A conceptual proof of Deligne's conjectures} \label{0.11}
We introduce an axiomatic framework to formulate a version of  De\-ligne's conjectures \cite[(10.4.1)]{D} for \emph{any} suitable realisation functor. This involves an abstract notion of weight filtration, which is given in Appendix \ref{AppendixD}. 

Assume given a triangulated
category
$\cT$ provided with a 
$t$-structure, with heart $\cB$ with a weight filtration $\cB_{\le n}$ (see \ref{dE.6} for
the definition) and a $t$-exact functor $D^b(\cB)\to \cT$ which is the identity on $\cB$. Considering  ``Lefschetz''  objects in $\cB$, we define a full subcategory
of level $\leq 1$ objects $\cB_{(1)} \subset \cB$ along with a left adjoint $\Alb^\cB$.
Define $\cT_{(1)}\subset \cT$ by objects with $H^*$ in $\cB_{(1)}$: assuming that the $t$-structure is  bounded, we get a left
adjoint $\LAlb^\cT$ (see \ref{pAlbT}).  

Suppose now given a triangulated ``realisation"
functor $R: \DM_{\gm}^\eff\otimes \Q\allowbreak\to \cT$  which behaves as a usual homology theory in
degrees $\leq 1$, as  explained in Hypotheses \ref{h15.1} and \ref{h16.1}. Let $R_1=R\Tot$. We then get a ``base change" natural transformation 
\begin{equation}\label{bc1}
\LAlb^\cT R\Rightarrow R_1\LAlb^\Q
\end{equation}
and, for all $M\in \DM_\gm^\eff\otimes\Q$ and $i\in\Z$, a map
\begin{equation}\label{bc2}
\Alb^\cB H^R_i(M)\to R_1\LA{i}^\Q(M)
\end{equation}
where $H^R_i(M)\df H_i(R(M))$. Theorem
\ref{abstractdelconj} now gives an abstract version of Deligne's conjectures: \eqref{bc1} and
\eqref{bc2} are isomorphisms if, in addition, the ``homological" equivalence induced by $R$ 
coincides with algebraic equivalence in codimension $1$ and the corresponding geometric cycle
class map satisfies a Lefschetz (1,1)-type theorem (or Tate conjecture in codimension $1$) for
smooth projective varieties. If \eqref{bc2} is an isomorphism for $M=M(X)$, $X$
smooth projective, then it is an isomorphism for any $M$ and \eqref{bc1} is an isomorphism of
functors.

\subsection{Hodge structures} For $X$ smooth over $\C$, Corollary \ref{HRPic} shows that the $1$-motive $\RA{i} (X)$ has a Hodge realisation abstractly isomorphic to
$H^i(X_{\an},\Z )^{\leq 1}$, the largest 1-motivic part of the mixed Hodge structure on $H^i(X_{\an},\Z)$ Tate-twisted by 1. The above abstract framework provides such an isomorphism in a functorial way.

Namely,  if $\cB= {\rm MHS}$ is the category of (graded polarizable, $\Q$-linear)
Deligne's mixed Hodge structures and  $\cT=D^b({\rm MHS})$, the weight filtration provides
$\cB$ with a weight structure. Then $\cB_{(1)} ={\rm MHS}_{(1)}$ is the full subcategory of
${\rm MHS}^\eff$ given by level $\leq 1$ mixed Hodge structures.  By Theorem \ref{t16.1}, Huber's Hodge realisation  \cite{HuberH} restricts to an equivalence on 1-motives.\footnote{Moreover, Huber's and Deligne's Hodge realisation are naturally isomorphic: see \cite{vadim} and \cite{ABN}.}
The conditions of Theorem \ref{abstractdelconj} are satisfied thanks to Lefschetz's theorem on $(1,1)$-classes. In
the  isomorphism \eqref{bc2} for $M=M(X)$, $X$ a complex variety, the mixed Hodge structure
$\Alb^\cB H^R_i(X) = H_i(X,\Q)_{\leq 1}$ is the largest quotient of level $\leq 1$ so that 
\eqref{bc2} yields Deligne's conjecture on a ``purely algebraic'' definition of this mixed Hodge
structure. Similarly,
$\LA{i}^*(X)$ has Hodge realisation $\Alb^\cB H_i(R(M (X)^*(n)[2n]))=H^{2n-i}(X,\Q(n))_{\leq
1}$: this provides new cases where Deligne's conjecture holds true (up to isogeny) not included in
\cite{BSAP}, \cite{BRS} or \cite{ram2}. All this is Theorem \ref{isodel} and its Corollary
\ref{isodelcor}, see also Remark \ref{notall}.

\subsection{Mixed realisations} Huber's Hodge realisation functor is only one component of her much richer mixed realisation functor \cite{HuberLN, HuberH}. In Section \ref{shuber}, we show that it fits with our axiomatic approach: this yields a reasonable interpretation of the second part of Deligne's conjecture on comparison isomorphisms, see Corollaries  \ref{ceqhodgeMR} and \ref{ceqhodgeMR2}.

\subsection{$\ell$-adic} Here we use Ayoub's $\ell$-adic realisation functor \cite{real.etale}, which is compatible with Deligne's realisation functor from \cite[(10.1.5)]{D}. This provides, in characteristic $p$, an $\ell$-adic version of Deligne's conjectures, which depends on the Tate conjecture in codimension $1$ (Theorem \ref{telladic}).

\subsection{Going further}
One could explore situations like the one around the $p$-adic
period isomorphism. We leave these developments to the motivated reader. See also \cite{vadim2} and \cite{bondarko4}
for further developments concerning weights.

\subsection{Caveat} While one might hope that these results are a partial template for a
future theory of mixed motives (see \eg \cite{ayoub2mot}), we should stress that some of them are definitely special to level $\le 1$. Namely:

\begin{itemize}
\item It is succintly pointed out in \cite[\S 3.4 p. 215]{V} (see \cite[\S 2.5]{ABV} for a proof) that the non finite generation of
the Griffiths group prevents higher-dimensional analogues of $\LAlb$ to exist. (This goes
against \cite[Conj. 0.0.19]{V0}.)
\item Contrary to Theorem \ref{t3.2.3}, the homotopy $t$-structure on $\DM_{-,\et}^\eff$ does not induce a
$t$-structure on $d_{\le n}\DM_{\gm,\et}^\eff$ for $n\ge 2$. This can already be seen on
$\Z(2)$, although its homology sheaves are conjecturally ind-objects of $d_{\le
2}\DM_{\gm,\et}^\eff$ (see \cite[\S 6]{V0}). 
\end{itemize}

These two issues seem related in a mysterious way! However, see \cite{ABV} for a possible approach to $n$-motivic sheaves and a conjectural picture linking the subject to the Bloch-Beilinson motivic filtration.

\subsection*{A small reading guide} We offer some suggestions to the reader, hoping that they will be helpful.

One might start by quickly brushing through \S \ref{1.1} to review the definition of Deligne's
$1$-motives, look up \S \ref{1.2d} to read the definition of $D^b(\M[1/p])$ and then proceed
directly to Theorem \ref{t1.2.1} (full embedding), referring to Section \ref{sect1} ad libitum
to read the proof of this theorem. The lengths of Sections \ref{homotopy} and \ref{dual} are
necessary evils; they may be skipped at first reading with a look at their main
results (Theorem \ref{t3.2.3}: the homotopy $t$-structure, Corollary \ref{c3.13} and Theorem \ref{t3.14}: internal Hom, and Theorem \ref{teq}: agreement of the two Cartier dualities). 

One may then read Section \ref{lalb} on the construction of $\LAlb$ and $\RPic$
(which hopefully will be pleasant enough), glance through Section \ref{qlalb} (their rational
versions) and have a look in passing at Section \ref{tens} for the tensor structure on
$D^b(\M\otimes\Q)$. After this, the reader might fly over the mostly formal Section 
\ref{6}, jump to Theorem \ref{trunc} which computes
$\LAlb(X)$ for a smooth scheme $X$, read Sections \ref{comps} and \ref{12} on $\LAlb$ of
singular schemes where he or she will have a few surprises, read Section \ref{rovi} on Ro\v\i
tman's theorem and its generalisations,  have a well-earned rest in recovering familiar
objects in Section \ref{scurves} (the case of curves). Then jump to  Section \ref{Hodge} and Corollary \ref{isodelcor} which gives the Hodge realisations of $\RPic(X)$ and $\LAlb^*(X)$ for $X$ a complex algebraic variety, look at the main results in Section \ref{shuber} and consult Section \ref{elladic}. After which, one can backtrack to Section \ref{axDel} to see the technical details.

And never look at the appendices.

The reader will also find an index of notations at the end.

\newpage 
\numberwithin{equation}{subsection}
\part{The universal realisation functor}

\section{The derived category of $1$-motives} \label{sect1}

The main reference for (integral, free) $1$-motives is \cite[\S
10]{D}, see also \cite[\S 1]{BSAP}. We also provide an Appendix~\ref{AppendixB} 
on 1-motives with torsion which were introduced in \cite[\S 1]{BRS}. For the derived category of
$1$-motives up to isogeny we refer to \cite[Sect. 3.4]{V} and \cite{OR}: here we are interested
in the integral version.

\enlargethispage*{30pt}

\subsection{Commutative group schemes} In this subsection, we fix our notation and recall some well-known and not so well-known facts on commutative group schemes over a field.

We shall work within the category $\cG^*$ \index{$\cG^*$, $\cG$, $\cG_\sm$, $\cG_\sab$} of commutative $k$-group sche\-mes locally of finite type. Let $\cG^*_\sm$ be the full subcategory of smooth $k$-group schemes. Recall that $\cG^*_\sm=\cG^*$ if $\car k=0$ (Cartier, \cite[Exp. VI, Th. 1.6.1]{sga3}). If $\car k >0$, $G\in \cG^*$ is smooth if and only if $\Fr^n_{G/k} : G \to G^{p^n}$ is faithfully flat for some (or any) $n\ge 1$ where $\Fr^n_{G/k}$ is the relative Frobenius morphism of $G$ by \cite[Exp. VIIA, Cor. 8.3.1 (iii)]{sga3}. 

We shall have to juggle a little with nonreduced group schemes in positive characteristic (see \eg footnote \ref{fn6}), and invert the exponential characteristic of $k$ in much (but not all) of this book; the reader who is only interested in characteristic $0$ can completely discard these issues.

Let $\cG\subset \cG^*$ be the category of commutative algebraic $k$-groups (\ie  commutative group schemes which are of finite type over $k$).  Recall that $G\in \cG$ is \emph{semi-abelian} if $G$ can be represented by an extension 
\[0\to T \to G \to A \to 0\] 
where $T$ is a torus and $A$ is an abelian variety; $A$ and $T$ are then unique.
We have:

\begin{lemma}\label{l1.1.4}  The categories $\cG$ and $\cG^*$ are abelian. The full subcategory 
\begin{multline*}\cG_\sab=\{G\in \cG \mid \text{ the reduction of the connected component $G'$}\\
 \text{of the identity in $G$ is semi-abelian}\}
\end{multline*}
is a Serre subcategory of $\cG$, \ie is closed under subobjects, quotients and extensions.\footnote{\label{fn6} Note that we use here the perfectness of $k$ to ensure that the reduced part of a $k$-group scheme is a $k$-group scheme, see \cite[Exp. 6A, 0.2]{sga3}. Also, $G'\in \cG$ for any $G\in\cG$ by ibid., 2.2.} In particular, $\cG_\sab$ is abelian.
\end{lemma}

\begin{proof} The first statement is a well-known theorem of Gro\-then\-dieck the proof given in  \cite[Exp. 6A, Th. 5.4.2]{sga3} works for $\cG$ and $\cG^*$ as well. Note that any morphism $f:G\to H$ in $\cG$ restricts to a map $f':G'\to H'$. If $f$ is mono, $f'$ is mono, hence $G'$ is semi-abelian if $H'$ is. If $f$ is epi, then $f'$ is epi (because $H'$ is reduced and connected, and $\coker f'$ is finite by a diagram chase); therefore, $H'$ is semi-abelian if $G'$ is. So $\cG_\sab$ is closed under subobjects and quotients in $\cG$. In the case of an extension $0\to G\to H\to K\to 0$ with $G,K\in \cG_\sab$, a diagram chase shows that the homology of $G'\to H'\to K'$ is finite, hence $H'$ is semi-abelian.
\end{proof}

\enlargethispage*{30pt}

We shall also work with the following full subcategories of $\cG^*$:

\begin{defn}\label{d1.1.1} a) A $k$-group scheme $L$ is 
\emph{discrete} if it is reduced, the connected component of the identity is trivial and the abelian group $L(\bar k)$ is finitely generated. This category is denoted by ${}^t\cM_0(k)={}^t\cM_0$. \\   
b) A \emph{lattice} is a
discrete $k$-group scheme $L$ such that $L(\bar k)$ is torsion-free. The full subcategory of lattices is denoted by
$\cM_0(k)=\cM_0$. \index{$\cM_0$, ${}^t\cM_0$}
\end{defn}

\begin{lemma}\label{discab} ${}^t\cM_0$ is a Serre subcategory of $\cG^*$, hence is abelian.\qed
\end{lemma}

\begin{remark}[Serre subcategories and thick subcategories]\label{Serrethick} In Lem\-mas \ref{l1.1.4} and \ref{discab} we used the notion of a Serre subcategory of an abelian category. There is a weaker notion; a full subcategory $\cB$ of an abelian category $\cA$ is \emph{thick} if the following condition holds:  for any short exact sequence $0\to A'\to A\to A''\to 0$ in $\cA$, if two among $A',A,A''$ belong to $\cB$, so does the third. This still implies that $\cB$ is abelian. The two notions are going to appear in this book.
\end{remark}

\begin{defn}\label{dsabt} We denote by ${}^t\AbS(k)={}^t\AbS$ 
\index{${}^t\AbS$, $\AbS$} \index{$\SAb$} the 
category of commutative
$k$-group schemes $G$ such that the connected component $G^0$ of the identity is (reduced and)  semi-abelian and $G/G^0$ is discrete.  An object of ${}^t\AbS$ is called a \emph{semi-abelian scheme with torsion}. We denote by $\AbS$ the  full subcategory of ${}^t\AbS$ formed by those $G$ such that $G/G^0$ is a lattice. Finally, we write $\SAb$ for the full subcategory of semi-abelian varieties.
\end{defn}

We summarise this array of subcategories in the following diagram:

\[\begin{CD}
&&&&\cM_0&\subset&{}^t\cM_0\\
&&&&\cap&&\cap\\
&&\SAb\ &\subset\ &\AbS\ &\subset\ &{}^t\AbS\\
&&\cap&&&&\cap\\
&&\cG_\sm&&\subset&&\cG^*_\sm\\
&&\cap&&&&\cap\\
\cG_\sab\ &\subset&\cG&&\subset&&\cG^*
\end{CD}\]


\subsection{Deligne 1-motives}\label{1.1} 

A \emph{Deligne $1$-motive} over $k$ is a complex
of $k$-group  schemes 
$$M= [ L \by{u} G]$$ 
where $L\in \cM_0$ is a lattice and $G\in \SAb$ is  semi-abelian.

As a complex, \emph{we shall place $L$ in degree $0$ and $G$ in degree $1$}.
Note that this convention is only partially shared by the existing
literature.

A map from $M = [ L \by{u} G]$ to $M'= [ L' \by{u'} G']$ is a
commutative square
\begin{equation}\label{csq}
\begin{CD}
 L @>u>>  G\\
@V{f}VV  @V{g}VV  \\
L'@>u'>> G'
\end{CD}
\end{equation}
in the category of group schemes. Denote by $(f, g):M\to M'$ such a
map. The natural composition of squares makes up the category
of Deligne's 1-motives. We shall denote this category by
$\M(k)$. We shall usually write $\M$ instead
of $\M(k)$, unless it is necessary to specify $k$. The following lemma is
immediate:

\begin{lemma}\label{lidco} $\M$ is an idempotent complete additive category.\qed
\end{lemma}

In Proposition \ref{isoab} below, we recall that $\M$ becomes abelian if we tensor morphisms with $\Q$. For this we introduce the following categories:

\begin{defn}\label{Meff}  Let $\cM$ \index{$\cM$, $\cM_\sab$, $\M$} 
denote the category given by complexes of group schemes $[L\to G]$ where $L\in {}^t\cM_0$ is discrete and $G\in \cG$ is a commutative algebraic group. Morphisms are still commutative squares \eqref{csq} in the category of group schemes. Write $\cM_\sab=\{[L\to G] \mid G\in\cG_\sab\}$: the category $\cM_\sab$ contains $\M$ as a full subcategory.
\end{defn}

\begin{propose} \label{nca}
The category $\cM$ is abelian, and $\cM_\sab$ is a Serre subcategory of $\cM$ (hence is abelian). 
\end{propose} 

\begin{proof}    For a map  $\phi = (f,g):N\to N'$, consider $\ker (\phi) = [\ker (f)
\to \ker (g)]$ and $\coker (\phi) = [\coker (f) \to \coker (g)]$ where kernels and cokernels of $f$ and $g$ are taken respectively in the abelian categories  ${}^t\cM_0$  and $\cG$ (see Lemmas \ref{discab} and \ref{l1.1.4}). It is immediately checked that they verify the universal properties of kernel and cokernel in $\cM$. Finally, the canonical map from image to coimage is an isomorphism, since the obvious functor $\cM\to {}^t\cM_0\times \cG$ is conservative. We have shown that $\cM$ is abelian. The fact that $\cM_\sab$ is a Serre subcategory of $\cM$ follows trivially from Lemma \ref{l1.1.4}.
\end{proof}

\begin{defn}\label{1mot} Let $R$ be a commutative ring. For any additive
category
$\cA$, we denote by $\cA\otimes R$ \index{$\cA\otimes R$} the $R$-linear category obtained from $\cA$ by tensoring morphisms by $R$, and by $\cA\boxtimes R$
\index{$\cA\boxtimes R$} the
pseudo-abelian hull (idempotent completion) of $\cA\otimes R$.
\end{defn}

This distinction is useful as $\cA\otimes R$ may not be idempotent
complete even if $\cA$ is. We shall also use the following notation:

\begin{defn}\label{1p} Let $p$ be a prime number. If $A$ is an abelian group, we abbreviate $A\otimes \Z[1/p]$ into $A[1/p]$. Same convention for presheaves of abelian groups, complexes of abelian groups, etc. If $\cA$ is an additive category, we abbreviate $\cA\otimes \Z[1/p]$ into $\cA[1/p]$. 
\index{$\cA[1/p]$}
\end{defn}

\begin{propose}[\cf \protect{\cite[3.2.2]{OR}}] \label{isoab} The
inclusion $\M\to\cM_\sab$ induces an equivalence of categories 
\[e:\M\otimes\Q\iso\cM_\sab\otimes \Q.\]
In particular, the category $\M\otimes\Q$ is abelian, hence
$\M\otimes\Q=\M\boxtimes\Q$.
\end{propose}

\begin{proof} (See also Lemma \ref{lB.1.3}.)  It is enough to show that $e$ is essentially
surjective. Let  $[L\by{u} G]\in\cM_\sab$. We then have a diagram
\[\begin{CD}
[L'\to G']@>>> [L'_{\fr}\to G'/u(L'_{\tor})]\\
@VVV\\
[L\to G]
\end{CD}\]
where $L'\df L\times_G G'$ and $L'_{\fr}\df L'/L'_{\tor}$,  $L'_{\tor}$ being the
torsion subgroup of $L'$ (note that $L'\into L$ is discrete). Both maps are isomorphisms in
$\cM_\sab\otimes\Q$. The last assertion follows from the fact that
$\cM_\sab\otimes\Q$ is abelian (see Proposition \ref{nca} and Corollary \ref{cB.1.3}).
\end{proof}

\begin{lemma}\label{lisodel} Let $\phi=(f,g):M = [ L \by{u}G] \to N= [K\by{v} H]$ be a morphism in $\M$. The following conditions are equivalent:
\begin{thlist}
\item $\phi$ becomes invertible in $\M\otimes\Q$.
\item $f$ and $g$ fit in a diagram
$$
\begin{CD}
&& 0 \\
&& @V{}VV \\
0&& F\\
@V{}VV @V{}VV  \\
L @>{u}>>  G\\
@V{f}VV @V{g}VV  \\
K @>{v}>> H\\
@V{}VV @V{}VV  \\
E&&0\\
@V{}VV&&\\
0&&
\end{CD}
$$
where $F$ and $E$ are finite.
\item There exists an integer $n>0$ such that $n1_M$ and $n1_N$ both factor through $\phi$.
\end{thlist}
Such a morphism is called an \emph{isogeny} (\cf Def. \ref{dA.2} c)).
\end{lemma}

\begin{proof} This is essentially a special case of Proposition \ref{pB.1.2} b):

 (i) $\Rightarrow$ (ii): if $\phi\otimes \Q$ is invertible, its kernel and cokernel are isomorphic to $0$ in $\M\otimes\Q$. As in the proof of Proposition \ref{isoab}, they are computed as $[\ker(f)\to\ker(g)]$ and $[\coker(f)\to\coker(g)]$ respectively. Thus $\ker(f),\ker(g),\coker(f)$ and $\coker(g)$
must be torsion in $\cG$, hence finite. This forces $\ker(f)=0$ and $\coker(g)=0$.

(ii) $\Rightarrow$ (iii): take for $n$ some integer killing both $E$ and $F$.

(iii)$\Rightarrow$ (i): if $\phi\otimes\Q$ is left and right invertible, it is invertible.
\end{proof}

\begin{remark} The category $\M$ has kernels and cokernels (see Proposition~\ref{lim})
but is not abelian. This easily follows  from the diagram in Lemma \ref{lisodel}: $(f,g)$
 has vanishing kernel and cokernel but is not an isomorphism in
$\M$.
\end{remark}

\subsection{Weights and cohomological dimension} \label{s.weights} Recall that $M =[L\to G]
\in \M$ has an increasing filtration by sub-$1$-motives as follows:\index{$W_i(M)$}
\[W_i(M) =\begin {cases} M & i\geq 0\\ G[-1]
& i= -1\\ T[-1] & i= -2\\ 0 & i \leq -3 \end{cases} \]
where $G$ is an extension of an abelian variety $A$ by a torus $T$. 
We then have $\gr_{-2}^W(M) = T[-1]$, $\gr_{-1}^W(M) = A[-1]$ and
$\gr_{0}^W(M) = L$ (according to our convention of placing $L$ in
degree zero). We say that $M$ is \emph{pure of weight $i$} if
$\gr_j^WM=0$ for all $j\ne i$. Note that for two  pure 1-motives
$M, M'$, $\Hom (M, M')\neq 0$ only if they have the
same weight.

\begin{propose}[\protect{\cite[3.2.4]{OR}}] \label{iso1} The
category $\M\otimes\Q$ is of cohomological dimension $\leq 1$, \ie if
$\Ext^i (M,M') \neq 0$, for $M, M'\in \M\otimes\Q$, then $i =0$ or $1$.
\end{propose}

Recall a sketch of the proof in \cite{OR}.  One first checks that
$\Ext^1(M,M')\allowbreak=0$ if $M,M'$ are pure of weights $i,i'$ and $i\le i'$. This
formally reduces the issue to checking that if $M,M',M''$ are pure
respectively of weights $0,-1,-2$, then the Yoneda product of two classes
$(e_1,e_2)\in \Ext^1(M,M')\times \Ext^1(M',M'')$ is $0$. Of course we may
assume $e_1$ and $e_2$ integral. By a transfer argument, one may further
reduce to $k$ algebraically closed. Then the point is that $e_1$ and
$e_2$ ``glue" into a $1$-motive, so are induced by a 3 step filtration on
a complex of length $1$; after that, it is formal to deduce that $e_2\cdot
e_1=0$ (\cf \cite[IX, Prop. 9.3.8 c)]{sga7}).

\begin{remark} We observe that Proposition \ref{iso1} can be regarded as
an algebraic version of a well-known property of $\M (\C)\otimes \Q$.
Namely, $\M (\C)\otimes \Q$ can be realised as a  Serre abelian
sub-category of $\Q$-mixed Hodge structures, see \cite{D}. Since the
latter is of cohomological dimension $\le 1$, so is $\M(\C)\otimes\Q$
(use \cite[Ch. III, Th. 9.1]{mcl}).
\end{remark}

\subsection{Group schemes and sheaves with transfers}\label{ur}

\begin{defn} We denote as in \cite[Def. 2.1]{VL} by $\PST(k) =\PST$ \index{$\PST$} the abelian category of presheaves
with transfers on smooth $k$-varieties: these are additive contravariant functors on the category $\Cor_k$ \index{$\Cor_k$} of finite correspondences from \loccit These categories were denoted by $SmCor(k)$ and  $PreShv(SmCor(k))$ in \cite{V}.  We denote by $\EST(k)$, or simply $\EST\subset \PST$ (\resp $\NST$), the 
full abelian subcategory of \'etale (\resp Nisnevich) sheaves with transfers: it was denoted by $Shv_{et}(SmCor(k))$ (\resp $Shv_{Nis}(SmCor(k))$) in \cite{V}. \index{$\EST$} \index{$\NST$}

We shall also denote by $\ES=\Shv_\et(\Sm(k))$ \index{$\ES$} the category of abelian \'etale sheaves on $\Sm(k)$.
\end{defn}

We write $F\mapsto F_\Nis$ or $F\mapsto F_\et$ for the ``associated sheaf'' functors: they preserve the transfer structure by \cite[Th. 6.17 and 13.1]{VL}.
\index{$F_\et$, $F_\Nis$}

We keep the old notation $L(X)$ from \cite{V} for the presheaf with transfers represented by a $k$-scheme $X$: it was replaced by $\Z_\tr(X)$ in \cite{VL}. 
\index{$L(X)$} This is actually an \'etale sheaf by \cite[Lemma 6.2]{VL}.

\begin{defn}\label{uG} Let $G\in \cG^*$. We denote by $\uG\in \ES$ the corresponding \'etale sheaf
of abelian groups on $\Sm(k)$. \index{$\uG$} 
\end{defn}

As is well-known, the restriction of this functor to ${}^t\cM_0$ induces a equivalence of categories between ${}^t\cM_0$ and the category of locally constant $\Z$-cons\-truct\-ible \'etale sheaves over $\Sm(k)$ (\ie with
finitely generated geometric fibres), and also with the category of finitely generated $Gal(k_s/k)$-modules. More generally:

\begin{lemma}\label{l1.4} The restriction of $G\mapsto \uG$ to $\cG^*_\sm$ is fully faithful.
\end{lemma}

\begin{proof} The presheaf defined by $G$ on $\Sm(k)$ is an \'etale sheaf since $G$ is smooth; the claim then follows from a small extension of Yoneda's lemma (to allow for smooth group schemes locally of finite type).
\end{proof}

In fact, under a minor assumption, $\uG$ is an
\'etale sheaf with transfers, as explained by Spie\ss-Szamuely
\cite[Proof of Lemma 3.2]{spsz}, \cf also Orgogozo \cite[3.1.2]{OR}.
Both references use symmetric powers, hence deal only with smooth
quasi-projective varieties. Here is a cheap way to extend their
construction to arbitrary smooth varieties.

\begin{lemma} \label{l1.3} For $G\in \cG^*$, the \'etale sheaf $\uG$ is provided with a canonical structure of presheaf with transfers.
Moreover, if $G'$ is a semi-abelian variety, then $\uG$ is homotopy invariant (see Definition \ref{dD.1} below).
\end{lemma}

\begin{proof} Without loss of generality, we may assume $G$ reduced (note that $\underline{G_\red}=\uG$); by \cite[Cor. 1.2]{conrad}, the connected component $G'=G^0$ is
quasi-pro\-jec\-tive.  For two smooth $k$-varieties $X,Y$, we have to provide a
pairing
\[c(X,Y)\otimes \uG(X)\to \uG(Y)\]
with the obvious compatibilities. As in \cite[Ex. 2.4]{VL}, it is enough
to construct a good transfer $f_*:\uG(W)\to \uG(X)$ for any finite
surjective map $f:W\to X$ with $X$ a normal $k$-variety. For $X$ and $W$
quasi-projective, this is done in \cite{spsz} or \cite{OR} using symmetric powers\footnote{Note that the symmetric powers
of $G$ exist as schemes since $G^0$ is quasi-projective.}. In general, cover $X$
by affine opens $U_i$ and let $V_i=f^{-1}(U_i)$. Since $f$ is finite, $V_i$ is also affine,
hence transfers $\uG(V_i)\to \uG(U_i)$ and
$\uG(V_i\cap V_j)\to \uG(U_i\cap U_j)$ are defined; the commutative
diagram
\[\begin{CD}
0\to \uG(W)@>>> \prod \uG(V_i)@>>> \prod \uG(V_i\cap V_j)\\
&&@V{f_*}VV @V{f_*}VV\\
0\to \uG(X)@>>> \prod \uG(U_i)@>>> \prod \uG(U_i\cap U_j)
\end{CD}\]
uniquely defines the desired $f_*$.

The second statement of the lemma is well-known (\eg  \cite[3.3.1]{OR}).
\end{proof}

Actually, the proof of \cite[Lemma 3.2]{spsz} defines a homomorphism of \'etale sheaves with transfers 
\begin{equation}\label{sigma}
\sigma_G:L (G)\to \uG
\end{equation}
using the sum maps on symmetric powers of $G$. It is split by the obvious morphism of sheaves
\[\gamma_G:\uG\to L (G)\]
given by the graph of a morphism.  Therefore $\sigma_G$ is an epimorphism of
sheaves. One should be careful, however, that $\gamma_G$ is not additive; note, however, that both $\sigma_G$ and $\gamma_G$ are natural in $G$.

We shall need a little more. For $X\in \Sm(k)$, write 
\[\tilde L(X)=\ker(L(X)\by{p_X}\Z)\] 
where $p_X$ is induced by the structural morphism.

\begin{lemma}\label{lBG} Let $m,p_1,p_2:G\times G\to G$ denote respectively the multiplication and the  first and second projection. Then the sequences
\[\begin{CD}
L(G\times G)@>p_1+p_2-m>>L(G)@>\sigma_G>> \uG@>>> 0\\
\tilde L(G\times G)@>p_1+p_2-m>>\tilde L(G)@>\tilde\sigma_G>> \uG@>>> 0
\end{CD}\]
are exact in $\PST$, hence in $\EST$.
\end{lemma}

\begin{proof} We shall give a completely formal and general argument. The second exactness reduces to the first one thanks to the commutative diagram
\[\begin{CD}
L(G\times G)@>p_1+p_2-m>>L(G)\\
@V{p_{G\times G}}VV @Vp_GVV\\
\Z@= \Z.
\end{CD}\]

The first sequence is a complex by naturality of $G\mapsto \sigma_G$. Write $\bar L(G)$ for the cokernel of $p_1+p_2-m$: the induced morphism $\bar \gamma_G:\uG\to \bar L(G)$ is still a section of $\bar \sigma_G:\bar L(G)\to \uG$; moreover $\bar\gamma_G$ is now additive, by naturality. It remains to prove that $\bar \gamma_G\bar\sigma_G$ is the identity. Using the additivity of $\bar \gamma_G$, a Yoneda argument in the category $\Cor_k$ reduces us  to test it on the universal section $1_G\in \bar L(G)(G)$, for which it is obvious.
\end{proof}

\begin{remark} This is a variation``with transfers'' on \cite[Exp. VII, 3.5]{sga7}.
\end{remark}

\subsection{Representable sheaves and exactness}

Recall that the ``representable sheaf'' functor $\cG\to \Shv_{fppf}(\Sch(k))$ is exact \cite[Exp. 6A, Th. 3.3.2 (i)]{sga3}. We shall see that this remains as true as it can when working with the \'etale topology. 

More precisely, consider the functor induced by Lemma \ref{l1.3}
\begin{align*}
\rho:\cG &\to \EST\\
G&\mapsto\uG[1/p]
\end{align*}
where $p$ is the exponential characteristic of $k$. The aim of this subsection is to prove the following

\begin{thm}\label{greta} The functor $\rho$ is exact.
\end{thm}

Inverting $p$ cannot be avoided here: consider multiplication by $p$ on $\G_m$.

We start with two lemmas.

\begin{lemma}\label{l1.4.2} Let $0\to F\to G\longby{f} H\to 0$ be an exact sequence in $\cG$, with $F$ smooth. Then the sequence in $\EST$
\[0\to \uF\to \uG\to \uH\to 0\]
is exact.
\end{lemma}

\begin{proof} It suffices to prove the exactness in the category $\ES$ of \'etale sheaves on $\Sm(k)$. The smoothness of $F$ implies the smoothness of $f$; then it is classical that if $R$ is a strictly henselian local ring, $G(R)\to H(R)$ is surjective  (\cf \cite[Th. 18.5.17]{ega4}).
\end{proof}

\begin{lemma}\label{l1.4.3} Let $0\to F\to G\to H\to 0$ be an exact sequence in $\cG$, with $F$ infinitesimal. Then $\uF=0$ and the quotient sheaf $\uH/\uG$ is killed by some power of $p$.
\end{lemma}

\begin{proof} The first claim is obvious (recall that we evaluate sheaves only on smooth schemes). For the second one, let $n$ be such that  $p^nF=0$. There exists a morphism $\phi$ in $\cG$ such that the diagram
\[\xymatrix{
G\ar[d]_{p^n}\ar[r] & H\ar[d]^{p^n}\ar[dl]_\phi\\
G\ar[r] & H
}\]
commutes. This implies that multiplication by  $p^n$ is $0$ on $\uH/\uG$.
\end{proof}

\begin{proof}[Proof of Theorem \ref{greta}] Let $0\to F\to G\to H\to 0$ be an exact sequence in $\cG$. Write $F^\dag=F/F_\red$ and $G^\dag=G/F_\red$, so that we have a commutative diagram of exact sequences in $\cG$
\[\begin{CD}
&& 0 && 0\\
&& @VVV @VVV\\
&& F_\red @= F_\red\\
&& @VVV @VVV\\
0@>>> F@>>> G@>>> H@>>> 0\\
&& @VVV @VVV ||\\
0@>>> F^\dag@>>> G^\dag@>>> H@>>> 0\\
&& @VVV @VVV\\
&& 0 && 0 
\end{CD}\]

Note that $F_\red$ is smooth and $F^\dag$ is infinitesimal. Sheafifying with transfers and applying Lemma \ref{l1.4.2}, we get a commutative diagram of exact sequences in $\EST$:
\[\begin{CD}
\uG@>>> \uH@>>> \uH/\uG@>>> 0\\
@VVV  ||&& @VfVV\\
 \underline{G^\dag}@>>> \uH@>>>\uH/\underline{G^\dag}@>>> 0\\
@VVV \\
0 
\end{CD}\]
which shows that $f$ is an isomorphism. By Lemma \ref{l1.4.3}, $\uH/\underline{G^\dag}$ is killed by $p^n$ for some integer $n$, which concludes the proof.
\end{proof}

\subsection{Local Exts and global Exts}\label{1.5} The Ext groups of the abelian categories $\PST, \NST$ and $\EST$ have the following properties. We start with a basic result of Voevodsky:

\begin{propose}[\protect{\cite[Lemma 6.23]{VL}}]\label{exth} For any $X\in \Sm(k)$ and any $F\in \NST$ (\resp $\EST$), there are canonical isomorphisms
\[H^i_\Nis(X,F) \simeq \Ext^i_{\NST}(L(X),F) \quad { (\resp H^i_\et(X,F) \simeq \Ext^i_{\EST}(L(X),F)).}\]
\end{propose}

Recall now that $\PST$ has a canonical tensor structure extending the one of $\Cor$, see \cite[\S 3.2]{V} or \cite[Def. 8.2]{VL}. This tensor product has a right adjoint $\shom_\tr$ given  by the formula in  \loccit 
\[\shom_\tr(F,G)(X)=\Hom(F\otimes L(X),G)\]
for  $F,G\in \PST$, see \cite[Lemma 8.3]{VL}. If $G\in \NST$ (\resp $\EST$), so does $\shom_\tr(F,G)$: this is the internal Hom of $ \NST$ (\resp $\EST$) as in \cite[Def. 8.2 and Lemma 8.3]{VL}.\index{$\shom_\tr$}

The categories $\NST$ and $\EST$ have enough injectives  \cite[Lemma 3.1.7]{V} and  \cite[Prop. 6.19]{VL}, which yields derived functors of $\shom_\tr$ in these categories. We shall only consider the case of $\EST$ and write $\sext^i_\tr$ for these derived functors (the case of $\NST$ is identical).

\begin{propose}\label{localglobal} Let $F,G\in \EST$ and $i\ge 0$. Then the \'etale sheaf with transfers $E^i(F,G)$ associated to the presheaf with transfers
\[X\mapsto \Ext^i_\EST(F\otimes L(X),G)\]
is canonically isomorphic to $\sext^i_\tr(F,G)$.
\end{propose}

\begin{proof} The claim is clear for $i=0$.  Since $G\mapsto E^*(F,G)$ is a $\delta$-functor, it suffices to observe that $E^i(F,I)\allowbreak =0$ for $i>0$ if $I$ is injective in $\EST$. 
\end{proof}

\subsection{Homotopy invariance and strict homotopy invariance}\label{rv} 

\begin{defn}[\protect{\cite[Def. 2.15 and 9.22]{VL}}]\label{dD.1} a) An object $F$ of $\PST$, $\NST$ or $\EST$ is \emph{homotopy invariant} if
$F(X)\iso F(X\times \Aff^1)$ for any smooth $k$-variety $X$.\\
b) Let $F\in \NST$. Then $F$ is \emph{strictly homotopy invariant} if $H^i_\Nis(X,F)\allowbreak\iso
H^i_\Nis(X\times \Aff^1,F)$ for any smooth $k$-variety $X$ and any $i\ge 0$.\\
c) Similarly, let $F\in \EST$. Then $F$ is \emph{strictly homotopy invariant} if $H^i_\et(X,F)\iso
H^i_\et(X\times \Aff^1,F)$ for any smooth $k$-variety $X$ and any $i\ge 0$.

We write $\HI_\Nis$ (\resp $\HI_\et$) for the full subcategory of $\NST$ (\resp $\EST$) formed of homotopy invariant sheaves. 
\end{defn}

One of the main results of Voevodsky concerning presheaves with transfers is: 

\begin{thm}[\protect{\cite[Th. 3.1.12]{V}, \cite[Th. 24.1]{VL}}]\label{ves} Over a perfect field, any object of $\HI_\Nis$ is strictly homotopy invariant.
\end{thm}

 In the \'etale topology, the same turns out to be true in characteristic $0$, but the situation is slightly different in positive characteristic.

\begin{defn}\label{his} We denote  by $\HI_\et^s(k)=\HI_\et^s$ the full subcategory of $\HI_\et$ consisting
of strictly homotopy invariant sheaves.\index{$\HI_\et^s$, $\HI_\Nis$}
\end{defn}

The main example of a sheaf $F$ which is in $\HI_\et$ but not in $\HI_\et^s$ is $F=\Z/p$ in
characteristic $p$: because of the Artin-Schreier exact sequence we have
\[k[t]/\cP(k[t])\iso H^1_\et(\Aff^1_k,\Z/p)\]
where $\cP(x) = x^p-x$.

We are going to show that this captures entirely the obstruction for a sheaf in $\HI_\et$ not
to be in $\HI_\et^s$  (\cf \cite{V} and \cite{VL}).

\begin{lemma}\label{l1.3.4} $\HI_\et^s$ is a thick $\Z[1/p]$-linear subcategory of $\EST$ (see Remark \ref{Serrethick}). In particular, it is abelian.
\end{lemma}

\begin{proof} The thickness follows immediately from the 5 lemma. For the second statement, it suffices to show that if $F\in \HI_\et^s$ verifies $pF=0$, then $F=0$. From the Artin-Schreier exact sequence, we get an exact sequence in $\EST$:
\begin{multline*}
0\to \shom_\tr(\G_a,F)\longby{\cP^*}\shom_\tr(\G_a,F)\to F\\
\to \sext_\tr^1(\G_a,F)\longby{\cP^*}\sext_\tr^1(\G_a,F)\to 0
\end{multline*}
where $\shom_\tr$ and $\sext^1_\tr$  are as in \S \ref{1.5} above. So it is sufficient to show that $\shom_\tr(\G_a,F)=\sext_\tr^1(\G_a,F)=0$.

Let $M$ be the kernel of the morphism $\tilde\sigma_{\G_a}$ from Lemma \ref{lBG}. By this lemma, we have an exact sequence
\begin{multline*}
0\to \shom_\tr(\G_a,F)\to \shom_\tr(\tilde L(\G_a),F)\to \shom_\tr(M,F)\\
\to \sext_\tr^1(\G_a,F)\to \sext_\tr^1(\tilde L(\G_a),F)
\end{multline*}
and a monomorphism
\[\shom_\tr(M,F)\into \shom_\tr(\tilde L(\G_a\times \G_a),F).\]

Since $F\in\HI_\et^s$, $\shom_\tr(\tilde L(\G_a),F)=\shom_\tr(\tilde L(\G_a\times \G_a),F)=\sext_\tr^1(\tilde L(\G_a),F)=0$ using Propositions \ref{exth} and \ref{localglobal}, and we are done.
\end{proof}

In fact:

\begin{propose}[\cf \protect{\cite[Th. 13.8]{VL}}]\label{lD.1.3} A sheaf $F\in\HI_\et$ belongs to $\HI_\et^s$ if and only if it is a sheaf of $\Z[1/p]$-modules, where $p$ is the exponential
characteristic of $k$.
\end{propose}

\begin{proof} Necessity was proven in Lemma \ref{l1.3.4}. For the converse, the following method is classical: let $F\in\HI_\et$. Using the exact sequence
\[0\to F_\tors\to F\to F\otimes\Q\to F\otimes\Q/\Z\to 0\]
and the fact that $\HI_\et^s$ is thick in $\EST$ (Lemma \ref{l1.3.4}), we are reduced to the following cases:
\begin{itemize}
\item $F$ is a sheaf of $\Q$-vector spaces. Then the result is true by \cite[Lemma
14.25]{VL} (reduction to \cite[Th. 13.8]{VL} by the comparison theorem \cite[Prop. 14.23]{VL}).
\item $F$ is a sheaf of torsion abelian groups. Since, by assumption, this torsion is prime
to $p$, $F$ is locally constant by Suslin-Voevodsky rigidity \cite[Th. 7.20]{VL}. Then the
result follows from \cite[XV 2.2]{sga4} (compare \cite[Lemma 9.23]{VL}).
\end{itemize}
\end{proof}

\begin{remark}\label{r1.3.6} In positive characteristic, it is not clear  whether $\HI_\et$ is thick in $\EST$, and actually whether it is abelian. The inclusion $\Z\subset\Z[1/p]$ shows that $\HI_\et^s$ is not a Serre subcategory of $\EST$ (see Remark \ref{Serrethick}). Also, Proposition \ref{lD.1.3} shows that the inclusion functor $\HI_\et^s\into \HI_\et$ has a left adjoint, given by $F\mapsto F[1/p]$ (see Definition \ref{1p}).
\end{remark}

\subsection{\'Etale motivic complexes}\label{ts} Let $D^-(\NST)$ be the bounded above derived category of $\NST$. Recall from \cite{V} the full subcategory
\begin{equation}\label{dmnis}
\DM_{-}^\eff(k)=\DM_{-}^\eff\subset D^-(\NST)
\end{equation}
of complexes with  homotopy invariant cohomology sheaves. Theorem \ref{ves} implies that $\DM_{-}^\eff$ is a triangulated subcategory of $D^-(\NST)$, and that the canonical $t$-structure on the latter induces a $t$-structure on $\DM_{-}^\eff$ with heart $\HI_\Nis$, called the homotopy $t$-structure. Voevodsky also constructs a left adjoint to the inclusion in \eqref{dmnis}, with an explicit formula.

Let now $D^-(\EST)$ be the bounded above derived category of $\EST$. In \cite[\S 3.3]{V}, Voevodsky considered the full subcategory of $D^-(\EST)$ given by complexes with homotopy invariant cohomology sheaves, and said that it is a triangulated subcategory. This is not clear in positive characteristic  (see Remark \ref{r1.3.6}); in order to have a smoothly working theory, we have to slightly change this definition.  (See also \cite[Lect. 9]{VL}.)

\begin{defn}\label{dmet}   We write $\DM_{-,\et}^\eff(k)=\DM_{-,\et}^\eff\subset D^-(\EST)$ for the full subcategory of complexes with \emph{strictly} homotopy invariant cohomology sheaves. \index{$\DM_{-}^\eff$, $\DM_{-,\et}^\eff$}
\end{defn}

With this definition, the following fact becomes trivial:

\begin{propose}\label{cD.2} The category $\DM_{-,\et}^\eff$ is $\Z[1/p]$-linear triangulated subcategory of $D^-(\EST)$. The canonical $t$-structure of $D^-(\EST)$ induces a $t$-structure on 
$\DM_{-,\et}^\eff$, with heart $\HI_\et^s$.\qed
\end{propose}

\begin{defn}\label{d1.7} We call the above $t$-structure  the \emph{homotopy $t$-struct\-ure}.  For $C\in \DM_{-,\et}^\eff$ we denote by $\sH^n(C)\allowbreak\in \HI_\et^s$ the cohomology objects  relative to the homotopy $t$-structure. \index{$\sH^n$, $\sH_n$}
\end{defn}

We now follow Voevodsky's construction in \cite{V} of a left adjoint
\[RC:D^-(\NST)\to \DM_-^\eff\index{$RC$, $RC_\et$}\]
 to the inclusion $\DM_-^\eff\subset D^-(\NST)$, working with the \'etale topology. Only minimal changes are needed. 

Recall the Suslin complex $C_*(F)$ \index{$C_*(F)$}
 associated to a presheaf $F\in\PST$, whose $n$-th term is given by
 \[C_n(F)(X)=F(X\times\Delta^n), \quad \Delta^n =\Spec \left(k[t_0,\dots,t_n]/\textstyle\sum t_i=1\right).\]

If $F\in \NST$ (\resp $\EST$), so does $C_n(F)$.

This can be extended to complexes of sheaves by taking total complexes, as in the following theorem:

\begin{thm}\label{pD.2} The inclusion functor $\DM_{-,\et}^\eff(k)\into D^-(\EST)$ has a left adjoint 
\[RC_\et:D^-(\EST)\to \DM_{-,\et}^\eff\index{$RC$, $RC_\et$}\]
which is given by the formula 
\[RC_\et(K)=  C_*(K)[1/p]\quad \text{ (see Definition \ref{1p})}\]
for $K\in D^-(\EST)$.  Here $C_*(K)$ is the total complex associated to the double complex $C_p(K_q)$.
\end{thm}

\begin{proof} We follow  the template of \cite[Proof of Prop. 3.2.3]{V}, where Voevodsky does this for the Nisnevich topology: let $\cA$ be the localising subcategory of $D^-(\EST)$ generated by complexes of the form $L(X\times \Aff^1)\to L(X)$ for $X\in\Sm(k)$. It is sufficient to prove:
\begin{enumerate}
\item For any  $F\in\EST$, the canonical morphism $F\to C_*(F)$ is an isomorphism in $D^-(\EST)/\cA$.
\item For any object $T\in \DM_{-,\et}^\eff$ and any object $B\in \cA$, one has $\Hom(B,T)=0$.
\item For any $F\in\EST$, $C_*(F)[1/p]\in \DM_{-,\et}^\eff$.
\end{enumerate}

(Concerning (1) and (3), see \cite[Lemma 9.12]{VL} for the reduction from complexes to sheaves.) 

The proof of (1) is identical to that in \cite{V}. For (2), Voevodsky reduces to the case where $B$ is of the form $L(X\times \Aff^1)\to L(X)$ and uses a hypercohomology spectral sequence argument. To avoid the issue raised by infinite cohomological dimension when working with the \'etale topology, we modify this as follows (\cf \cite[C.5]{Azumaya}): one needs to show that the morphism
\[T\to R\pi_*\pi^*T\]
is an isomorphism in $D^-(\EST)$, where the morphism of \'etale sites $\pi:\Sm(\Aff^1)_\et\to \Sm(k)_\et$ is induced by the structural morphism of $\Aff^1$. Since this statement is local for the \'etale topology, we may reduce to $k$ separably closed, and then the issue disappears.

For (3), the homology presheaves $h_i(F)$ of $C_*(F)$ are homotopy invariant by \cite[Prop. 3.6]{Vp}. By \cite[Cor. 5.29]{Vp} (see also \cite[Prop. 1.1.2 (3)]{ABV})  and Remark \ref{r1.3.6}, the \'etale sheaves with transfers $h_i(F)_\et[1/p]=h_i(F[1/p])_\et$ are then in $\HI_\et^s$.
\end{proof}

\begin{cor}\label{cD.1} Let $\alpha^s$ denote the
composition
\[\DM_-^\eff\into
D^-(\NST)\longby{\alpha^*}D^-(\EST)\longby{RC_\et}\DM_{-,\et}^\eff\] 
where $\alpha^*$ is induced by the inverse image functor $F\mapsto F_\et$ (change of topology) on sheaves. Then,  $\alpha^s F\simeq  F_\et[1/p]$ for any $F\in \HI_\Nis$.
\end{cor}
\index{$\alpha^s$, $\alpha^*$}
\begin{proof} Indeed,
\[\alpha^s F = RC_\et\alpha^*F \simeq C_*(\alpha^*F)[1/p] = C_*(\alpha^*F[1/p])\osi \alpha^*F[1/p]\]
 by Theorem \ref{pD.2} and its proof.
\end{proof}

Let $X$ be a $k$-scheme. Recall that its \emph{motive} in $\DM_-^\eff$ is defined by
\[M(X)=RCL(X).\]

Similarly:

\begin{defn}\label{met}  We write
\begin{align*}
M_\et(X)&=RC_\et L(X)\in \DM_{-,\et}^\eff.\index{$M(X)$, $M_\et (X)$}\\
\Z_\et(0)&=M_\et(\Spec k) =\Z[1/p]\\
\Z_\et(1)&=RC \tilde L(\P^1)[-2].\index{$\Z (n)$, $\Z_\et(n)$}
\end{align*}
\end{defn}

\begin{lemma}\label{Niset} We have $M_\et(X)= \alpha^s M(X)$ and $\Z_\et(1)\simeq \G_m[1/p][-1]$.
\end{lemma}

\begin{proof} This follows immediately from Corollary \ref{cD.1} and the analogous result for $\Z (1)$ in $\DM_-^\eff$ (\cite[Cor. 3.4.2]{V}, \cite[Th. 4.1]{VL}).
\end{proof}

(There are also objects $\Z_\et(n)=\alpha^s\Z(n)$ for $n>1$; we won't use them in this book, see \cite[Lect. 3]{VL}.)\\

Let us come back to a commutative group scheme $G$, as in \S \ref{ur}. As in
\cite[Remark 3.3]{spsz}, when $\uG$ is homotopy invariant, the homomorphism $\sigma_G$ of \eqref{sigma} extends to a morphism  in $C^-(\EST)$
\[
C_*(L (G))\to \uG
\]
whence a morphism in $\DM_{-,\et}^\eff$
\begin{equation}\label{eq1.4}
M_\et (G)\to \uG[1/p]
\end{equation}
by using   Theorem \ref{pD.2}.

\subsection{$1$-motives with torsion and an exact structure on $\M[1/p]$}\label{1.2}
Let ${\HI_\et^s}^{[0,1]}$ be the category
of complexes of length $1$ (concentrated in
degrees $0$ and $1$) of objects of $\HI_\et^s$.
From Proposition \ref{lD.1.3} we get a functor
\begin{align}
\rho:\cG^*_{\rm hi}[1/p]&\to \HI_\et^s\label{rho}\\
G&\mapsto\underline{G}[1/p]\notag
\end{align}
where $\cG^*_{\rm hi} \df \{G \in \cG^* \mid \uG \in \HI_\et\}$. From Lemma \ref{l1.3} the category $\cG^*_{\rm hi}$ contains
$\cG_\sab$ and the induced functor $\rho:\cG_\sab[1/p]\to \HI_\et^s$  is exact  by  Theorem \ref{greta}.
Hence a functor
\begin{align*}
\rho:\cM_\sab[1/p] &\to {\HI_\et^s}^{[0,1]}\\
M&\mapsto\underline{M}[1/p]
\end{align*}
and, by composing with the embedding $\M\into \cM_\sab$, another functor
\begin{equation}\label{eq2.2}
\M [1/p]\to {\HI_\et^s}^{[0,1]}.
\end{equation}

\begin{propose}\label{pexact} Let $M^\cdot$ be a complex of
objects of
$\cM_\sab[1/p]$. The following conditions are equivalent:
\begin{thlist}
\item The total complex $\Tot(\rho M^\cdot)$ in $C(\HI_\et^s)$ is a\-cycl\-ic.
\item For any $q\in\Z$, $H^q(M^\cdot)$ is of the form $[F^q =F^q]$, where $F^q$ is finite.
\end{thlist}
\end{propose}

\begin{proof} (ii) $\Rightarrow$ (i) is obvious. For the converse, let
$M^q=[L^q\to G^q]$ for all $q$. Let $L^\cdot$ and $\uG^\cdot$ be the two
corresponding ``column" complexes of sheaves. By the exactness of $\rho$, we have a long exact
sequence in
$\HI_\et^s$:
\[\dots\to\rho  H^q(L^\cdot)\to \rho H^q(G^\cdot)\to H^q(\Tot(\rho M^\cdot))\to
\rho H^{q+1}(L^\cdot)\to\dots\]

The assumption implies that $\rho H^q(L^\cdot)\iso \rho H^q(G^\cdot)$ for all $q$.
Since $H^q( L^\cdot)$ is discrete and $H^q(G^\cdot)$ is 
a commutative algebraic group, both must be finite.
\end{proof}

We now restrict to complexes of $\M[1/p]$.

\begin{defn}\label{dexact} A complex of $\M[1/p]$ is \emph{acyclic} if it
satisfies the equivalent conditions of Proposition \ref{pexact}. An
acyclic complex of the form $0\to N'\to N\to N''\to 0$ is called a
\emph{short exact sequence}. 
\end{defn}

Recall that in \cite{BRS} a category of $1$-motives with torsion was introduced. We shall
denote it here by ${}^t\M$ \index{${}^t\M$, ${}^t\M^\eff$} in order to distinguish it from $\M$. More precisely, denote by ${}^t\M^\eff$  the full subcategory of $\cM_\sab$ consisting of the objects $[L\to G]$ where
$G$ is semi-abelian: these are the effective $1$-motives with torsion (\cf Definition \ref{eff1mot}).
The category of $1$-motives with torsion ${}^t\M$ is the localisation of ${}^t\M^\eff$ with respect to quasi-isomorphisms (\cf Definition \ref{1tors}).

The main properties of ${}^t\M$ are recalled in Appendix \ref{AppendixB}. In particular, the category ${}^t\M[1/p]$ is abelian (Theorem \ref{1mtora}) and by Proposition \ref{free} we have a full embedding
\begin{equation}\label{fullem}
\M[1/p]\into{}^t\M[1/p]
\end{equation}
which makes $\M[1/p]$ an exact subcategory of ${}^t\M[1/p]$. The following lemma is a direct consequence of Proposition \ref{pstrict} and Corollary \ref{corexseq}:

\begin{lemma}\label{lexact} A complex $0\to N'\by{i} N\by{j} N''\to 0$ in $\M[1/p]$ is a short
exact sequence in the sense of Definition \ref{dexact} if and only if it is a short exact
sequence for the exact structure given by \eqref{fullem}.
\end{lemma}

\begin{proof}     Suppose that the given complex is a short exact sequence in the sense of Definition \ref{dexact}. If  $j$ is a strict morphism in the sense of Definition \ref{dstrict}, then it is an exact sequence of complexes. In general we get a factorisation of $j$ given by $N\by{\tilde j} \tilde N \to N''$ where $\tilde N$ is quasi-isomorphic to $N''$ and $N\by{\tilde j} \tilde N$ is a strict epimorphism with kernel $N'$ (\cf Proposition \ref{pstrict}). Therefore it is also exact in ${}^t\M[1/p]$. Conversely,  note that any short exact sequence in ${}^t\M[1/p]$ can be represented by a strict effective epimorphism followed by a quasi-isomorphism (see Corollary \ref{corexseq}).
\end{proof}

\begin{remark}\label{rexact} There is another, much stronger, exact
structure on $\M[1/p]$, induced by its full embedding in the abelian category $\cM_\sab[1/p]$ (Proposition \ref{nca}): it
amounts to require  a complex $[L^\cdot\to G^\cdot]$ to be exact if and
only if both complexes $L^\cdot$ and $\uG^\cdot$ are acyclic. We shall
not use this exact structure in the sequel. See also Remark \ref{r1.8.6}.
\end{remark}

\subsection{The derived category of $1$-motives}\label{1.2d}

\begin{lemma}\label{l1.5.1} A complex in $C(\M[1/p])$ is acyclic in the sense of Definition
\ref{dexact} if and only if it is acyclic with respect to the exact structure of $\M[1/p]$ provided
by Lemma
\ref{lexact} in the sense of \cite[1.1.4]{BBD} or \cite[\S 1]{neeman}.
\end{lemma}

\begin{proof} Let $X^\cdot\in C(\M[1/p])$. Viewing $X^\cdot$ as a complex of objects of $\cM_\sab[1/p]$,
we define $D^n=\im (d^n:X^n\to X^{n+1})$. Note that the $D^n$ are Deligne $1$-motives. Let
$e_n:X^n\to D^n$ be the projection and $m_n:D^n\to X^{n+1}$ be the inclusion. We have
half-exact sequences
\begin{equation}\label{eq1.5}
0\to D^{n-1}\by{m_{n-1}} X^n\by{e_n} D^n\to 0
\end{equation}
with middle cohomology equal to $H^n(X^\cdot)$. Thus, if $X^\cdot$ is acyclic in the sense of
Definition \ref{dexact}, the sequences \eqref{eq1.5} are short exact which means that $X^\cdot$
is acyclic with respect to the exact structure of $\M[1/p]$. Conversely, suppose that $X^\cdot$ is
acyclic in the latter sense. Then, by definition, we may find
${D'}^n$, $e'_n$, $m'_n$ such that $d^n=m'_ne'_n$ and that the sequences analogous to
\eqref{eq1.5} are short exact. Since $\cM_\sab[1/p]$ is abelian (Proposition \ref{nca}), ${D'}^n=D^n$ and we are done.
\end{proof}

From now on, we shall only say ``acyclic" without further precision. 

Let $K(\M[1/p])$ be the homotopy category of $C(\M[1/p])$. By \cite[Lemmas 1.1 and 1.2]{neeman}, the full
subcategory of $K(\M[1/p])$ consisting of acyclic complexes is triangulated and thick (the latter
uses the fact that $\M[1/p]$ is idempotent-complete, \cf Lemma \ref{lidco}). Thus one may define the 
derived category of $\M[1/p]$ in the usual way:

\begin{defn}\label{ddermot} \index{$D^b(\M[1/p])$}
a) The \emph{derived category of $1$-motives} is
the localisation $D(\M[1/p])$ of the homotopy category $K(\M[1/p])$ with
respect to the thick subcategory $A(\M[1/p])$ consisting of acyclic complexes. Similarly for
$D^\pm(\M[1/p])$ and $D^b(\M[1/p])$.\\
b) A morphism in $C(\M[1/p])$ is a
\emph{quasi-isomorphism} if its cone is acyclic.
\end{defn}

\subsection{Torsion objects in the derived category of $1$-motives}

Let $\cM_0$ be the category of lattices (see Definition \ref{d1.1.1}): the inclusion functor
$\cM_0[1/p]\by{A}
\M[1/p]$ provides it with the structure of an exact subcategory of $\M[1/p]$.
Moreover, the embedding
\begin{equation}\label{eq:0to1}\cM_0[1/p]\by{B} {}^t\cM_0[1/p]
\end{equation}
is clearly exact, where ${}^t\cM_0$ is the abelian category of discrete
group schemes (see Lemma \ref{discab}). In fact, we also have an exact functor
\begin{align*}
{}^t\cM_0[1/p]&\by{C}{}^t\M[1/p]\\
L&\mapsto [L\to 0].
\end{align*}

 Hence an induced diagram of
triangulated categories:
\[\begin{CD}
D^b(\cM_0[1/p])@>B>> D^b({}^t\cM_0[1/p])\\
@V{A}VV @V{C}VV\\
D^b(\M[1/p])@>D>> D^b({}^t\M[1/p]).
\end{CD}\]

\begin{thm}\label{ptors} In the
above diagram\\ 
a) $B$ and $D$ are equivalence of categories.\\ 
b) $A$ and $C$ are fully faithful; restricted to torsion objects they are equivalences of
categories.
\end{thm}

(For the notion of torsion objects, see Proposition \ref{p1.1}.)

\begin{proof} a) For $B$, this follows from Proposition
\ref{derxact} provided we check that any object $M$ in ${}^t\cM_0[1/p]$ has a
finite left resolution by objects in $\cM_0[1/p]$. In fact $M$ has a length
$1$ resolution: let $E/k$ be a finite Galois extension of group $\Gamma$
such that the Galois action on $M$ factors through $\Gamma$. Since $M$ is
finitely generated, it is a quotient of some power of $\Z[\Gamma]$, and
the kernel is a lattice. Exactly the same argument works for $D$ by considering the lattice part of a $1$-motive. (For $M =  [L\to G]$, we get $L'\onto L$ with $L' \in \cM_0[1/p]$ as above so that $M' =  [L'\to G]$ obtained by composition yields a projection $M'\onto M$ with kernel a lattice.)

b) By a) it is sufficient to prove that $C$ is fully faithful. It suffices to verify that the
criterion of Proposition \ref{pschapira} is verified by the full embedding ${}^t\cM_0[1/p]\to {}^t\M[1/p]$.

Let $[L\to 0]\into [L'\to G']$ be a monomorphism in ${}^t\M[1/p]$. We may assume that it is given
by an effective map. The assumption implies that $L\to L'$ is mono: composing 
with the projection $[L'\to G']\to [L'\to 0]$, we get the requested factorisation.

It remains to show that $A$ is essentially
surjective on torsion objects. Let $X=[C^\cdot\to G^\cdot]\in
D^b(\M[1/p])$, and let $n>0$ be such that $n 1_X=0$. Arguing as in the proof of Proposition \ref{pexact}, this implies that the cohomology sheaves of both $C^\cdot$
and $G^\cdot$ are killed by some possibly larger integer
$m$  prime to $p$. We have an exact triangle
\[[0\to G^\cdot]\to X\to [C^\cdot\to 0]\by{+1}
\]
which leaves us to show that $[0\to G^\cdot]$ is in the essential image
of $C$. Let $q$ be the smallest integer such that $G^q\ne 0$: we have an
exact triangle
\[\{G^q\to \im d^q\}\to  G^\cdot \to \{0\to G^{q+1}/\im
d^q\to\dots\}\by{+1}\]
(here we use curly braces in order to avoid confusion with the square
braces used for $1$-motives). By descending induction on $q$, the right
term is in the essential image, hence we are reduced to the case where
$G^\cdot$ is of length $1$. Then $d^q:G^q\to G^{q+1}$ is epi and
$\mu:=\ker d^q$ is finite and locally constant. Consider the
diagram in
$K^b(\cM_\sab[1/p])$
\[
[\begin{smallmatrix}0\\\downarrow\\ 0\end{smallmatrix}\to
\begin{smallmatrix}G^q\\\downarrow\\ G^{q+1}\end{smallmatrix}] \leftarrow
[\begin{smallmatrix}0\\\downarrow\\ \mu\end{smallmatrix}\to
\begin{smallmatrix}G^q\\||\\ G^{q}\end{smallmatrix}]
\leftarrow [\begin{smallmatrix}L_1\\\downarrow\\ L_0\end{smallmatrix}\to
\begin{smallmatrix}G^q\\||\\ G^{q}\end{smallmatrix}]\to
[\begin{smallmatrix}L_1\\\downarrow\\ L_0\end{smallmatrix}\to
\begin{smallmatrix}0\\\downarrow\\ 0\end{smallmatrix}]
\]
where $L_1\to L_0$ is a resolution of $\mu$ by lattices (see proof of a)).
Clearly all three maps are quasi-isomorphisms, which implies that the
left object is quasi-isomorphic to the right one on
$D^b(\M[1/p])$.
\end{proof}

\begin{cor} \label{nobox} Let $A$ be a subring of $\Q$ containing $1/p$. Then the natural functor
\[D^b(\M[1/p])\otimes A\to D^b(\M\otimes A)\]
is an equivalence of categories. These categories are idempotent-com\-ple\-te for any $A$.
\end{cor}

\begin{proof} By Proposition \ref{pB.4.1}, this is true by replacing the category $\M[1/p]$ by ${}^t\M[1/p]$. On the
other hand, the same argument as above shows that the functor $D^b(\M\otimes A)\to
D^b({}^t\M\otimes A)$ is an equivalence. This shows the first statement; the second one follows
from the fact that $D^b$ of an abelian category is idempotent-complete.
\end{proof}

\subsection{Discrete sheaves and permutation modules} The following proposition will be used in \S \ref{s2.4.1}.

\begin{propose}\label{pperm} Let $G$ be a profinite group. Denote by $D^b_c(G)$ the
derived category of finitely generated (topological discrete)
$G$-modules. Then $D^b_c(G)$ is thickly generated by $\Z$-free permutation modules.
\end{propose}

\begin{proof} The statement says that the smallest thick subcategory $\cT$ of $D^b_c(G)$ which contains permutation modules is equal to $D^b_c(G)$. Let $M$ be a finitely generated $G$-module: to prove that
$M\in \cT$, we immediately reduce to the
case where $G$ is finite. Let $\bar M = M/M_\tors$. Realise $\bar
M\otimes\Q$ as a direct summand of $\Q[G]^n$ for $n$ large enough. Up to
scaling, we may assume that the image of $\bar M$ in $\Q[G]^n$ is
contained in $\Z[G]^n$ and that there exists a submodule $N$ of $\Z[G]^n$
such that $\bar M\cap N=0$ and $\bar M\oplus N$ is of finite index in
$\Z[G]^n$. This reduces us to the case where $M$ is \emph{finite}.
Moreover, we may assume that $M$ is $\ell$-primary for some prime $\ell$.

Let $S$ be a Sylow $\ell$-subgroup of $G$. Recall that there exist two
inverse isomorphisms
\begin{gather*}
\phi:\Z[G]\otimes_{\Z[S]} M\iso \Hom_{\Z[S]}(\Z[G],M)\\
\phi(g\otimes m)(\gamma) =
\begin{cases} 
\gamma g m&\text{if $\gamma g\in S$}\\
0&\text{if $\gamma g\notin S$.}
\end{cases}\\
\psi:\Hom_{\Z[S]}(\Z[G],M)\iso \Z[G]\otimes_{\Z[S]} M\\
\psi(f)=\sum_{g\in S\backslash G}g^{-1}\otimes f(g).
\end{gather*}

On the other hand, we have the obvious unit and counit
homomorphisms
\begin{gather*}
\eta:M\to \Hom_{\Z[S]}(\Z[G],M)\\
\eta(m)(g)=gm\\
\epsilon: \Z[G]\otimes_{\Z[S]} M\to M\\
\epsilon(g\otimes m)= gm.
\end{gather*}

It is immediate that 
\[\epsilon\circ \psi\circ \eta = (G:S).\]

Since $(G:S)$ is prime to $\ell$, this shows that $M$ is a direct summand
of the induced module $\Z[G]\otimes_{\Z[S]} M\simeq
\Hom_{\Z[S]}(\Z[G],M)$. But it is well-known (see \eg \cite[\S 8.3, cor. to Prop. 26]{serrerep})
that $M$, as an $S$-module, is a successive extension of trivial
$S$-modules. Any trivial torsion $S$-module has a length $1$ resolution
by trivial torsion-free $S$-modules. Since the ``induced module" functor
is exact, this concludes the proof.
\end{proof}

\subsection{Cartier duality and $1$-motives with cotorsion}\label{1.8} Recall that the Cartier duality between tori and lattices extends to a perfect duality between discrete group schemes and groups of multiplicative type. We now extend this to the framework of $1$-motives,  introducing a new
category ${}_t\M$:

\begin{defn}\label{1cot}
We denote by ${}_t\M^\eff$ \index{${}_t\M$, ${}_t\M^\eff$} the full subcategory of $\cM_\sab$ of complexes of group schemes $[L\to G]$ such that $L$ is a lattice and by ${}_t\M$ its localisation with respect to quasi-isomorphisms. An object of ${}_t\M$ is called a \emph{$1$-motive with cotorsion}. 
\end{defn}

We shall need the following lemma:

\begin{lemma} \label{em} For any $G\in \cG_\sab$, $G_\red$ is an extension of an abelian variety by a group of multiplicative type.
\end{lemma}

\begin{proof} We may assume $G$ reduced. Let $G'=G^0\subseteq G$ be the connected component of the identity. The quotient $G/G'$ is finite and \'etale. Let $n$ be its order. If $G'$ is an extension of an abelian variety $A$ by a torus $T$, then multiplication by $n$ on $G/T$ induces an epimorphism $G/T\onto A$. Composing with $G\to G/T$, we get an epimorphism $G\onto A$, whose kernel is an extension of a finite \'etale group by $T$.
\end{proof}

Recall that Deligne \cite[\S 10.2.11-13]{D} (\cf \cite[1.5]{BSAP})
defined a self-duality on the category $\M$, that he called
\emph{Cartier duality}. The following facts elucidate the introduction
of the category ${}_t\M$.

\begin{lemma}\label{brst} Let $\Gamma$ be a $k$-group of multiplicative type,
$L$ its Cartier dual and $A$ a $k$-abelian
variety. We have a Galois-equivariant isomorphism
$$\tau:\Ext(A, \Gamma)\iso \Hom(L, \Pic^0(A))$$
given by the canonical ``pushout'' mapping. Here the $\Hom$ and $\Ext$ are computed in the category $\cG_\sab(\bar k)$.
\end{lemma}

\begin{proof} Displaying $L$ as an extension of $L_{\fr}$ by $L_{\tor}$
denote the corresponding torus by $T \df \Hom (L_{\fr},\G_m)$ and let $F\df \Hom
(L_{\tor},\G_m)$ be the dual finite group. We obtain a map of short exact sequences
\[\begin{CD}
0\to \Ext (A, T)&\to& \Ext (A, \Gamma)&\to& \Ext (A, F)\to 0\\
@V{\tau_{\fr}}VV @V{\tau}VV  @V{\tau_{\tor}}VV\\
0\to \Hom (L_{\fr}, \Pic^0(A))&\to& \Hom (L, \Pic^0(A))&\to&\Hom
(L_{\tor}, \Pic^0(A))\to 0.
\end{CD}\]

Now $\tau_{\fr}$ is an isomorphism by the classical Weil-Barsotti
formula, \ie $\Ext (A, \G_m)\cong \Pic^0(A)$, and $\tau_{\tor}$ is an
isomorphism since the N\'eron-Severi group of $A$ is free: $\Hom
(L_{\tor}, \Pic^0(A)) = \Hom (L_{\tor}, \Pic (A)) = H^1(A, F) = \Ext (A,
F)$ (\cf \cite[III.4.20]{MI}).
\end{proof}

\begin{lemma}\label{dualt}  Cartier duality on $\M(k)$
extends to a contravariant additive functor 
\[
(\ \ )^*:{}^t\M^\eff(k) \to {}_t\M^\eff(k)
\]
which sends a \qi to a \qi
\end{lemma}

\begin{proof} The key point is that
$\Ext_{\cG_\sab(\bar k)}(- , \G_m)$ vanishes on discrete group schemes (\cf \cite[III.4.17]{MI}). All $\Hom$ and $\Ext$ in this proof are computed either in $\cG_\sab(\bar k)$ or in $\cM_\sab(\bar k)$.

To define the functor, we proceed  as Deligne, \loccit (see also 
\cite[1.5]{BSAP}): starting with $M=[L\by{u}
A]\in {}^t\M^\eff(k)$, for $A$ an abelian variety, let $G^u$ be the extension of the dual abelian
variety $A^*$ by the Cartier dual $L^*$ of $L$ given by Lemma
\ref{brst} (note that $G^u\in \cG_\sab(k)$). We define $M^{*} = [0\to G^u]\in {}_t\M^\eff(k)$. For a general
$M=[L\by{u} G]\in {}^t\M^\eff(k)$, with $G$ an extension of $A$ by $T$, the
extension $M$ of $[L\by{\bar u} A]$ by the toric part $[ 0\to T]$ provides
the corresponding extension
$G^{\bar u}$ of $A'$ by $L^*$ and a boundary map
$$u^{*}:\Hom (T, \G_m)\to \Ext ([L\by{\bar u} A],[0\to \G_m]).$$ 

Lemma \ref{brst} identifies the right hand side with $G^{\bar u}(\bar k)$; this defines $M^*\in {}_t\M^\eff(k)$. This construction is clearly contravariant in $M$.

For a quasi-isomorphism $M=[L\by{u} G]\onto M'=[L'\by{u'} G']$ with kernel $[F\by{=}F]$ for a finite
group $F$, \cf \eqref{qi1mot}, the quotient $[L\by{\bar u} A]\onto
[L'\by{\bar u'} A']$ has kernel
$[F\onto F_A]$ where $F_A\df \ker (A\onto A')$. The commutative diagram of exact sequences
\[\begin{CD}
0\to \Hom (T', \G_m)&\to& \Hom (T, \G_m)&\to&\Hom (F_T,\G_m)\to 0\\
@V{(u')^*}VV @V{u^*}VV  @V{\veq}VV\\
0\to \Ext ([L'\by{\bar u'} A'],\G_m)&\to& \Ext ([L\by{\bar u} A],\G_m)
&\to&\Ext ([F\onto F_A],\G_m)\to 0
\end{CD}\]
is then equivalent to an exact sequence in $\cM_\sab$ 
\[0\to {M'}^*\to M^*\to [F'=F']\to 0\]
with $F'$ the Cartier dual of $F_T\df \ker (T\onto T')$.  
\end{proof}

\begin{propose}\label{pcd} a) The functor of Lemma \ref{dualt} induces an
anti-equivalence of categories
\[(\ \ )^*:{}^t\M[1/p] \iso {}_t\M[1/p].\]
b) The category ${}_t\M[1/p]$ is abelian; the functor of a)  and its quasi-inverse are exact.\\
c) Cartier duality on $\M[1/p]$ is an exact functor, hence induces a
triangulated self-duality on $D^b(\M[1/p])$.  
\end{propose}

\begin{proof} a)  The said functor exists by Lemma \ref{dualt}, and it is clearly additive.
Let us prove that it is i) essentially surjective, ii) faithful and iii) full.

i) Let $M=[L\to G]\in {}_t\M^\eff[1/p]$. Since $G/G_\red$ is killed by a power of $p$, the map $[L\to G_\red]\to [L\to G]$ is an isomorphism, hence we may assume $G$ reduced.  Applying Lemma \ref{em} to $G$ and proceeding as in the proof of Lemma \ref{dualt}, we construct an object $N\in {}^t\M^\eff[1/p]$, and a direct computation shows that $N^*$ is quasi-isomorphic to $M$. (Note that $N$ is well-defined only up to a \qi because we have to make a choice when applying Lemma \ref{em}.)

ii) We reduce to show that the functor of Lemma \ref{dualt} is faithful by using that the duals of Propositions
\ref{B1.1} and \ref{calfrac} are true in ${}_t\M^\eff[1/p]$ (dual proofs). By additivity, we need to prove that if
$f:M_0\to M_1$ is mapped to $0$, then $f=0$. But, by construction, $f^*$ sends the
mutiplicative type part of $M_1^*$ to that of $M_0^*$.

iii) Let $M_0=[L_0\to G_0]$, $M_1=[L_1\to G_1]$ in $ {}^t\M^\eff[1/p]$, and let $f:M_1^*\to M_0^*$
be (for a start) an effective map. We have a diagram
\[\begin{CD}
0@>>> \Gamma_1@>>> G'_1@>>> A'_1@>>> 0\\
&& && @V{f_G}VV\\
0@>>> \Gamma_0@>>> G'_0@>>> A'_0@>>> 0
\end{CD}\]
where $M_i^*=[L'_i\to G'_i]$, $A'_i$ is the dual of the abelian part of $M_i$ and $\Gamma_i$
is the dual of $L_i$. If $f_G$ maps $\Gamma_1$ to $\Gamma_0$, using Lemma \ref{brst} we get an
(effective) map $g:M_0\to M_1$ such that $g^*=f$. In general we reduce to this case: let $\mu$
be the image of $f_G(\Gamma_1)$ in $A'_0$: this is a finite group. Let now $A'_2=A'_0/\mu$, so
that we have a commutative diagram
\[\begin{CD}
0@>>> \Gamma_0@>>> G'_0@>>> A'_0@>>> 0\\
&& @VVV @V{||}VV @VVV\\
0@>>> \Gamma_2@>>> G'_0@>>> A'_2@>>> 0
\end{CD}\]
where $\mu=\ker(A'_0\to A'_2)=\coker(\Gamma_0\to \Gamma_2)$. By construction, $f_G$ induces
maps $f_\Gamma:\Gamma_1\to \Gamma_2$ and $f_A:A'_1\to A'_2$.

Consider the object $M_2=[L_2\to G_2]\in {}^t\M^\eff[1/p]$ obtained from  $(L'_0,\Gamma_2,A'_2)$
and the other data by the same procedure as in the proof of Lemma \ref{dualt}. We then have a
\qi $s:M_2\to M_0$ with kernel $[\mu = \mu]$ and a map $g:M_2\to M_1$ induced by
$(f_L,f_\Gamma,f_A)$, and $(gs^{-1})^*=f$. 

If $f$ is a \qi, clearly $g$ is a \qi; this concludes the proof of fullness.

b) Since ${}^t\M[1/p]$ is abelian, ${}_t\M[1/p]$ is abelian by a). Equivalences of abelian categories
are automatically exact.

c) One checks as for ${}^t\M[1/p]$ that the inclusion of $\M[1/p]$ into ${}_t\M[1/p]$ induces the exact
structure of $\M[1/p]$.  Then, thanks to b), Cartier duality preserves exact sequences of $\M[1/p]$,
which means that it is exact on $\M[1/p]$. 
\end{proof}

\begin{remark}\label{r1.8.6}  Cartier duality does not preserve the strong exact structure of
Remark \ref{rexact}. For example, let $A$ be an abelian variety, $a\in
A(k)$ a point of order $m>1$ and $B=A/\langle a\rangle$. Then the sequence
\[0\to [\Z\to 0]\by{m} [\Z\by{f} A]\to [0\to B]\to 0,\]
with $f(1) = a$, is exact in the sense of Definition \ref{dexact} but not in the sense of Remark
\ref{rexact}. However, its dual
\[0\to [0\to B^*]\to [0\to G]\to [0\to \G_m]\to 0\]
is exact in the strong sense. Taking the Cartier dual of the latter sequence, we come back
to the former.
\end{remark}

Dually to Theorem \ref{ptors}, we now have:

\begin{thm}\label{tstr}  The natural functor $\M[1/p]\to
{}_t\M[1/p]$ is fully faithful and induces an equivalence of categories
\[
D^b(\M[1/p])\iso D^b({}_t\M[1/p]).
\]
Moreover, Cartier duality exchanges ${}^t\M[1/p]$ and
${}_t\M[1/p]$ inside the derived category $D^b(\M[1/p])$.
\end{thm}

\begin{proof} This follows from Theorem \ref{ptors} and Proposition~\ref{pcd}.
\end{proof}

\begin{notation}\label{not} For $C\in D^b(\M[1/p])$, we write ${}^tH^n(C)$ (\resp ${}_tH^n(C)$  \index{${}^tH^n$, ${}_tH^n$} for
its cohomology objects relative to the 
$t$-structure with heart ${}^t\M[1/p]$ (\resp ${}_t\M[1/p]$). We also write ${}^tH_n$ for
${}^tH^{-n}$ and ${}_tH_n$ for ${}_tH^{-n}$.
\end{notation}

Thus we have \emph{two} $t$-structures on $D^b(\M[1/p])$ which are exchanged by Cartier duality;
naturally, these two $t$-structures coincide after tensoring with $\Q$.
In Section \ref{homotopy}, we shall introduce a third $t$-structure, of a completely different kind: see Corollary \ref{c3.3.2}.

We shall come back to Cartier duality in Section \ref{dual}.

\subsection{How not to invert $p$}\label{alessandra1} This has been done by Alessandra
Ber\-ta\-pel\-le \cite{alessandra}. She defines a larger variant of $^t\M$ by allowing finite
connected $k$-group schemes in the component of degree $0$. Computing in the fppf topology, she
checks that the arguments provided in Appendix \ref{AppendixB} carry over in this context and
yield in particular an integral analogue to Theorem \ref{1mtora}. Also, the analogue of
\eqref{fullem} is fully faithful integrally, hence an exact structure on $\M$; she also checks
that the analogue of Theorem \ref{ptors} holds integrally. 

In particular, her work provides an exact structure on $\M$, hence an integral definition of $D^b(\M)$. One could check that this exact structure can be described \emph{a priori} using Proposition \ref{pexact} and Lemma \ref{lexact}, and working with the fppf topology.

 It is likely that the duality results of \S \ref{1.8} also extend to Bertapelle's context. See also \cite{russell}.

\section{Universal realisation}

The derived category of $1$-motives up to isogeny can be realised in
Voevodsky's triangulated category of motives.  With rational coefficients, this is part of
Voevodsky's Pretheorem 0.0.18 in \cite{V0} and claimed in \cite[Sect.
3.4, on page 218]{V}. Details of this fact appear in Orgogozo \cite{OR}. In this section we
 give a $\Z[1/p]$-integral version of this theorem, where $p$ is the exponential characteristic
of $k$, using the \'etale version of Voevodsky's category.

\subsection{Statement of the theorem}

By  \eqref{eq2.2}, any $1$-motive $M =[L\to G]$ yields a
complex of objects of $\HI_\et^s$. This defines a functor
\begin{align}
\M[1/p]&\to \DM_{-,\et}^{\eff}\label{eq8}\\
M&\mapsto \underline{M}[1/p].\notag
\end{align}

\begin{defn}\label{d2.1.1} We denote by $\DM_{\gm,\et}^\eff$ \index{$\DM_{\gm,\et}^\eff$} the thick subcategory of
$\DM_{-,\et}^{\eff}$ generated by the image of $\DM_\gm^\eff$ under the functor
\[\alpha^s:\DM_-^\eff\to \DM_{-,\et}^\eff\] 
of Corollary \ref{cD.1}.  For $n\ge 0$, we write $d_{\le n}\DM_{\gm,\et}^\eff$ for  the thick
subcategory 
\index{$d_{\le n}\DM_{\gm,\et}^\eff$} of $\DM_{\gm,\et}^\eff$ 
generated by motives of smooth varieties of dimension $\le n$ (\cf \cite[\S 3.4]{V}).
\end{defn}

\begin{thm}\label{t1.2.1} Let $p$ be the exponential characteristic of $k$. The functor
\eqref{eq8} extends to a fully faithful triangulated functor
\[T:D^b(\M[1/p])\to \DM^\eff_{-,\et}\]
where the left hand side was defined in \S \ref{1.2}. Its essential image is $d_{\le 1}\DM_{\gm,\et}^\eff$. We also have 
\[T(D^b(\cM_0[1/p]))=d_{\le 0} \DM_{\gm,\et}^\eff\]
with respect to the full embedding  $D^b(\cM_0[1/p])\into D^b(\cM_1[1/p])$ of Theorem \ref{ptors}.
\end{thm}
\index{$T$, $\Tot$}

The proof is in several steps.

\subsection{Construction of $T$}\label{orgo} We follow Orgogozo \cite{OR}.
Clearly, the embedding \eqref{eq2.2} extends to a functor
\[C^b(\M[1/p])\to C^b({\HI_\et^s}^{[0,1]}).\]

By Lemma \ref{ltot}, we have a canonical functor $C^b({\HI_\et^s}^{[0,1]})\by{\Tot}D^b({\HI_\et^s})$. To get $T$,
we are therefore left to prove

\begin{lemma}\label{l2.2.1} The composite functor
\[C^b(\M[1/p])\to C^b({\HI_\et^s}^{[0,1]})\by{\Tot}D^b({\HI_\et^s})\]
factors through $D^b(\M[1/p])$.
\end{lemma}

\begin{proof} It is a general fact that a homotopy in $C^b(\M[1/p])$ is mapped to a homotopy in
$C^b({\HI_\et^s}^{[0,1]})$, and therefore goes to $0$ in $D^b({\HI_\et^s})$, so that the functor
already factors through $K^b(\M[1/p])$. The lemma now follows from Lemma \ref{l1.5.1}.
\end{proof}

\subsection{Full faithfulness of $T$}\label{2.3} It is sufficient by Proposition \ref{p1.2} to
show that $T\otimes\Q$ and $T_\tors$ are fully faithful. 

For the first fact, we note that, using Corollary \ref{nobox},  $T\otimes\Q$ factors through the functor 
\[T_\Q:D^b(\M(k)\otimes \Q)\to \DM_{-,\et}^\eff(k;\Q)\]
studied in \cite{OR}, via the full embedding
\[\DM_{-,\et}^\eff(k;\Q)\into \DM_{-,\et}^\eff(k)\otimes \Q.\]

For the reader's convenience, we sketch the 
proof given in \cite[3.3.3 ff]{OR} that $T_\Q$ is fully faithful: it first uses the equivalence of categories
\[\DM_-^\eff(k;\Q)\iso \DM_{-,\et}^\eff(k;\Q)\]
of \cite[Prop. 3.3.2]{V}.
One then reduces to show
that the morphisms
\[\Ext^i(M,M')\to \Hom(\Tot(M),\Tot(M')[i])\]
are isomorphisms for any pure $1$-motives $M,M'$ and any $i\in\Z$.
This is done by a case-by-case inspection, using the fact
\cite[3.1.9 and 3.1.12]{V} that in $\DM_-^{\eff}(k;\Q)$
\[\Hom(M(X),C)=\HH^0_{\Zar}(X,C)\]
for a bounded complex $C$ of Nisnevich sheaves and a smooth variety $X$. The key points are that 1) for such $X$ we have
$H^i_{\Zar}(X,\G_m)=0$ for
$i>1$ and for an abelian variety $A$, $H^i_{\Zar}(X,A)=0$ for $i>0$
because the sheaf $A$ is flasque, and 2) that any abelian variety is up
to isogeny a direct summand of the Jacobian of a curve. This point will  also be used for the
essential surjectivity below.

\begin{sloppypar}
For the second fact, the argument in the proof of \cite[Prop. 3.3.3 1]{V} shows that the
functor $\DM_{-,\et}^\eff\to D^-(\Shv((\Spec k)_\et))$ which takes a complex of sheaves on
$\Sm(k)_\et$ to its restriction to $(\Spec k)_\et$ is an equivalence of categories on the
full subcategories of objects of prime-to-$p$ torsion. The conclusion
then follows from Theorem \ref{ptors} b).
\end{sloppypar}

\subsection{Gersten's principle} We want to formalise here an important computational method
which goes back to Gersten's conjecture but was put in a wider perspective and systematic
use by Voevodsky. For the \'etale topology,  it can be used as a substitute to proper base change.

\begin{propose}\label{pgersten} a) Let $C$ be a complex of presheaves with transfers on $\Sm(k)$
with homotopy invariant cohomology presheaves. Suppose that $C(K)\df \varinjlim_{k(U)=K} C(U)$
is acyclic for any function field $K/k$. Then the associated complex of Zariski sheaves $C_\Zar$
is acyclic.\\ 
b) Let $f:C\to D$ be a morphism of complex of presheaves with transfers on $\Sm(k)$
with homotopy invariant cohomology presheaves. Suppose that for any function field $K/k$,
$f(K):C(K)\to D(K)$ is a quasi-isomorphism. Then $f_\Zar:C_\Zar\to D_\Zar$ is a
quasi-isomorphism.  \\
c) The conclusions of a) and b) hold for if $C,D\in\DM_{-,\et}^\eff$ and if their hypotheses are weakened
by replacing $K$ by $K_s$, a separable closure of $K$.
\end{propose}

\begin{proof} a) Let $F = H^q(C)$ for some $q\in\Z$, and let $X$ be a smooth $k$-variety with
function field $K$. By \cite[Cor. 4.18]{V2}, $F(\cO_{X,x})\into F(K)$ for any $x\in X$, hence
$F_\Zar=0$. b) follows from a) by considering the cone of $f$. c) is seen similarly by observing that for any \'etale sheaf $F$ and any function field $K/k$, the map $F(K)\to F(K_s)$ is injective.
\end{proof}

\subsection{An important computation}\label{2.5} Recall that the category $\DM_{-,\et}^{\eff}$
is provided with a partial internal Hom denoted by $\ihom_\et$, \index{$\ihom_\et$} defined on
pairs  $(M,M')$ with $M\in\DM_{\gm,\et}^{\eff}$: it is defined analogously to the one of
\cite[Prop. 3.2.8]{V} for the Nisnevich topology.\footnote{Note that $\ihom_\et(M,M')$ really belongs to $\DM_{-,\et}^{\eff}$ by the same argument as in part (2) of the proof of Theorem \ref{pD.2}.} We need:

\begin{defn}\label{dpi0} Let $X\in \Sch(k)$. We denote by $\pi_0(X)$ the universal \'etale
$k$-scheme factorising the structural morphism $X\to \Spec k$.
\index{$\pi_0(X)$}
\end{defn}

(The existence of $\pi_0(X)$ is obvious, for example by Galois descent: see
\cite[Ch. I, \S 4, no 6, Prop. 6.5]{dem-gab}.)

\begin{remark}\label{rinv} It follows from Hensel's lemma that $\pi_0(X)\iso \pi_0(Y)$ for any nilimmersion $X\into Y$ (a closed immersion defined by a nilpotent ideal sheaf).
\end{remark}

\begin{lemma}\label{normal}  Suppose $X$ is integral, with function field $K$. Then there is a canonical epimorphism $\Spec E\to\pi_0(X)$, where $E$ is the algebraic closure of $k$ in $K$; if $X$ is normal, this is an isomorphism.
\end{lemma}

(For a counterexample when $X$ is not normal, consider the $k$-scheme $\Aff^1_E/\sim$ where $E$ is a quadratic extension of $k$ and $\sim$ pinches two conjugate $E$-rational points, e.g. the affine conic $x^2+y^2=0$ over $k=\R$.)

\begin{proof} Let $f:X\to \pi_0(X)$ be the canonical morphism. Since $X$ is connected, $f$ factors through a component $y=\Spec F$ of $\pi_0(X)$, hence $y=\pi_0(X)$ by universality. The morphism $X\to y$ is dominant so we have an inclusion of fields $k\subseteq F\subseteq K$, hence $F\subseteq E$. This provides a surjective map $\Spec E\to \pi_0(X)$. Conversely, if $X$ is normal, the universal property of normalisation implies that the structural morphism $X\to \Spec k$ factors through $\Spec E$.
\end{proof}

Let $X$ be a smooth projective
$k$-variety of dimension $d$. In \cite[Th. 4.3.2]{V}, Voevodsky defines the ``class of the diagonal''
\[\Delta_X\in \Hom_{\DM_\gm^\eff}(M(X)\otimes M(X),\Z(d)[2d])\]
by apparently relying on resolution of singularities. Since we don't want to use resolution of singularities at this stage, let us recall that $\Delta_X$ can be defined elementarily using \cite[Prop. 2.1.4]{V}, which constructs a $\otimes$-functor from the category of effective Chow motives to $\DM_\gm^\eff$ (\cf \eqref{eqvf} below). We have:

\begin{propose}\label{duality} The morphism in $\DM_\gm^\eff$
\[M(X)\to \ihom(M(X),\Z(d)[2d])\]
induced by the class $\Delta_X$ is
an isomorphism.
\end{propose}

\begin{proof} This is proven in \cite[Th. 4.3.2 and Cor. 4.3.6]{V} under resolution of singularities. In  \cite[App. B]{hk}, it is explained how to avoid resolution of
singularities: recall that it uses the rigidity of the category of Chow motives plus the cancellation theorem, which is now known over any perfect field by \cite{voecan}. 
\end{proof}

We shall apply this in the next proposition when $X$ is a curve $C$: using the functor $\alpha^s$ of Corollary \ref{cD.1}, we then get a class
\begin{equation}\label{deltaC}
\Delta_C\in \Hom_{\DM_{\gm,\et}^\eff}(M_\et(C)\otimes M_\et(C),\Z_\et(1)[2]).
\end{equation}

\begin{propose}\label{lcurve}  Let $f:C\to \Spec k$ be a smooth projective
$k$-curve. In $\DM_{-,\et}^\eff$: \\
a) There is a canonical isomorphism
\[\ihom_\et(M_\et(C),\Z_\et(1)[2])\simeq R_\et f_*\G_m[1/p][1]\]
(see Definition \ref{met} for $M_\et(C)$ and  $\Z_\et(1)$).\\
b) We have
\[R^q_\et f_*\G_m[1/p]=
\begin{cases}
R_{\pi_0(C)/k}\G_m[1/p] &\text{for $q=0$}\\
\underline{\Pic}_{C/k}[1/p] &\text{for $q=1$}\\
0&\text{else.}
\end{cases}
\]
Here, $R_{\pi_0(C)/k}$ denotes the Weil restriction of scalars from $\pi_0(C)$ to $k$.\\
c) The morphism
\[M_\et(C)\to \ihom_\et(M_\et(C),\Z_\et(1)[2])\]
induced by the class $\Delta_C$ from \eqref{deltaC} is
an isomorphism.
\end{propose}

This is \cite[Cor. 3.1.6]{OR} with three differences: 1) the fppf topology should
be replaced by the \'etale topology; $p$ must be inverted (\cf Corollary \ref{cD.1}); 3) the
truncation is not necessary since $C$ is a curve.

\begin{proof}  a) is the \'etale analogue of \cite[Prop. 3.2.8]{V}
since $\Z_\et(1)=\G_m[1/p][-1]$ (see Lemma \ref{Niset}) and $f^* (\G_{m, k}) =\G_{m, C}$ for
the big \'etale sites. In b), the isomorphisms for $q=0,1$ are clear; for $q>2$, we reduce by
Gersten's principle (Prop. \ref{pgersten}) to stalks at separably closed fields, and
then the result is classical
\cite[IX (4.5)]{sga4}.

It remains to prove c).  By b) and \cite[Cor. 4.2]{VL}, the natural morphism
\begin{equation}\label{curves}
\alpha^s\ihom_\Nis(M(C), \Z (1))\to \ihom_\et(\alpha^sM(C), \Z_\et (1))
\end{equation}
is an isomorphism. Hence the result by Proposition \ref{duality}.
\end{proof}

\subsection{Essential image} 
We proceed in two steps:

\subsubsection{The essential image of $T$ is contained in 
$\cT:=d_{\le 1}\DM_{\gm,\et}^\eff$}\label{s2.4.1} It is
sufficient to prove that
$T(N)\in \cT$ for
$N$ a $1$-motive of type $[L\to 0]$, $[0\to G]$ ($G$ a torus) or $[0\to
A]$ ($A$ an abelian variety). For the first type, this follows from
Proposition \ref{pperm} which even shows that $T([L\to 0])\in d_{\le 0}\DM_{\gm,\et}^\eff$. For the second type, Proposition \ref{pperm}
applied to the character group of $G$ shows that $T([0\to G])$ is
contained in the thick subcategory generated by permutation tori, which
is clearly contained in $\cT$.

It remains to deal with the third type. If $A=J(C)$ for a smooth
projective curve $C$ having a rational $k$-point $x$, then $T([0\to
A])=A[-1]$ is the direct summand of $M_\et(C)[-1]$ (determined by $x$)
corresponding to the Chow motive $h^1(C)$, so belongs to $\cT$. If $A\to
A'$ is an isogeny, then Proposition \ref{pperm} implies that
$A[-1]\in\cT$ $\iff$ $A'[-1]\in \cT$. In general we may write $A$ as the
quotient of a jacobian $J(C)$, where $C$ is an ample curve on $A$ passing through the origin if $k$ is infinite (if not, reduce to this case by the 
standard argument using transfers). Let $B$ be the connected part of the
kernel: by complete reducibility there exists a third abelian variety
$B'\subseteq J(C)$ such that $B+B'=J(C)$ and $B\cap B'$ is finite. Hence
$B\oplus B'\in \cT$, $B'\in\cT$ and finally $A\in \cT$ since it is
isogenous to $B'$.

\subsubsection{The essential image of $T$ contains $\cT$} It suffices to
show that $M_\et(X)$ is in the essential image of $T$ if $X$ is smooth
projective irreducible of dimension $0$ or $1$. Let $E$ be the field of
constants of $X$. If $X=\Spec E$, $M_\et(X)$ is the image of $[R_{E/k}\Z\to
0]$. If $X$ is a curve, we apply Proposition \ref{lcurve}: by c) it suffices to show that the
sheaves of b) are in the essential image of $T$.
We have already observed that $R_{E/k}\G_m[1/p]$ is in the essential image of $T$. We then have
a short exact sequence of \'etale sheaves
\[0\to R_{E/k} J(X)[1/p]\to \underline{\Pic}_{X/k}[1/p]\to R_{E/k}\Z[1/p]\to 0.\]

Both the kernel and the cokernel in this extension belong to the image of $T$, and the proof
is complete.
\qed

\subsection{The universal realisation functor} 

\begin{defn}\label{tot} Define the \emph{universal realisation functor}
$$\Tot : D^b(\M[1/p]) \to \DM_{\gm,\et}^{\eff}$$ \index{$T$, $\Tot$}
to be the composition of the equivalence of categories of Theorem \ref{t1.2.1} and the embedding
$d_{\le 1}\DM_{\gm,\et}^\eff\to \DM_{\gm,\et}^\eff$.
\end{defn}

\begin{remark} \label{r2.4} In view of Theorem \ref{tstr}, the equivalence of Theorem
\ref{t1.2.1} yields \emph{two} ``motivic" $t$-structures on $d_{\leq
1}\DM_{\gm,\et}^{\eff}$: one with heart ${}^t\M[1/p]$ and the other with
heart ${}_t\M[1/p]$. In the next section we shall describe a third one, the homotopy $t$-structure.
\end{remark}

\section{$1$-motivic sheaves and the homotopy $t$-structure}\label{homotopy}

In this section, we introduce the notion of $1$-motivic sheaf: they form an abelian category $\Shv_1$ which turns out to be the heart of a $t$-structure on $D^b(\M[1/p])$ induced by the homotopy $t$-structure of $\DM_{-,\et}^\eff$ via the embedding of Theorem \ref{t1.2.1}: see Theorem \ref{t3.2.3}. The main technical result, Theorem \ref{t3.12}, is that $\Shv_1$ is well-behaved under internal Ext, see also Corollary \ref{c3.13}: we shall make use of this in the next section.

A non-finitely generated version of $\Shv_1$ is studied in \cite{ABV} by completely different methods. In particular, \cite[1.3.4]{ABV} shows that any subsheaf of a $1$-motivic sheaf in this generalised sense is $1$-motivic.  This is false for $\Shv_1$, as  shown by the example $G(k)\subset \uG$ for $G$ a nonzero semi-abelian variety when $k$ is infinite: here $G(k)$ is identified to a constant subsheaf of $\uG$. (We thank the referee for pointing this out.)

Recall that we denote by $\ES$ \index{$\ES$} the category of abelian \'etale sheaves on $\Sm (k)$. From  \S \ref{soul} to \S \ref{3.7} we  work in $\ES$. The category $\EST$ only appears from \S \ref{3.9} onwards.

\subsection{Some useful lemmas} \label{soul}
This subsection is in the spirit of \cite[Ch. VII]{gacl}.
 Let $S$ be a base scheme, $G$ be a commutative $S$-group scheme, and let us write $\uG$\index{$\uG$} for
the associated sheaf of abelian groups for a so far unspecified
Grothendieck topology (\cf \ref{uG}). Let also $\cF$ be another sheaf of abelian groups.
We then have:

\begin{itemize}
\item $\Ext^1(\uG,\cF)$ (an Ext of sheaves);
\item $H^1(G,\cF)$ (cohomology of the scheme $G$);
\item $\bar H^2(G,\cF)$: this is the homology
of the complex
\[\cF(G)\by{d^1} \cF(G\times G)\by{d^2} \cF(G\times G\times G)\]
where the differentials are the usual ones.
\end{itemize}

\begin{propose}\label{p1} There is a complex
\[\Ext^1(\uG,\cF)\by{b} H^1(G,\cF)\by{c} H^1(G\times G,\cF)\]
and an injection
\[0\to \ker b\by{a} \bar H^2(G,\cF).\]
\end{propose}

\begin{proof} Let us first define the maps $a,b,c$:

\begin{itemize}
\item $c$ is given by $p_1^*+p_2^*-\mu^*$, where $\mu$ is the group law of
$G$.
\item For $b$: let $\cE$ be an extension of $\uG$ by $\cF$. We have an
exact sequence
\[\cE(G)\to \uG(G)\to H^1(G,\cF).\]
Then $b([\cE])$ is the image of $1_G$ by the connecting homomorphism. Alternatively, we may
think of $\cE$ as an $\cF$-torsor over $G$ by forgetting its group structure.
\item For $a$: we have $b([\cE])=0$ if and only if $1_G$ has an antecedent
$s\in \cE(G)$. By Yoneda, this $s$ determines a section $s:\uG\to \cE$ of
the projection. The defect of $s$ to be a homomorphism gives a
well-defined element of  $\bar H^2(G,\cF)$ by the usual cocycle
computation: this is $a([\cE])$. 
\end{itemize}

The map $a$ is clearly injective.  We have $c\circ b=0$ because $p_1^*+p_2^*=\mu^*$ as maps $\Ext(\uG,\cF)\to \Ext(\uG\times \uG,\cF)=\Ext(\uG,\cF)\oplus \Ext(\uG,\cF)$.
\end{proof}

\begin{propose}\label{p2} a) Suppose that the map
\[\cF(G)\oplus\cF(G)\vlongby{(p_1^*,p_2^*)}\cF(G\times G)\]
is surjective. Then $ \bar H^2(G,\cF)=0$.\\
b) Suppose that  $\cF(S)\iso \cF(G^r)$ for $r=1,2$. Then the condition of a) is satisfied and 
the complex of Proposition \ref{p1} is acyclic.
\end{propose}

\begin{proof} a) Let $\gamma\in \cF(G\times G)$ be a $2$-cocycle. We may
write $\gamma=p_1^*\alpha+p_2^*\beta$. The cocycle condition implies that
$\alpha$ and $\beta$ are constant. Hence $\gamma$ is constant, and it is
therefore a $2$-coboundary (of itself).

b) The first assertion is clear. The second one follows from a direct computation identical to
the one in \cite[Ch. VIII, \S 3, no 15, proof of Th. 5]{gacl}.
\end{proof}

\begin{example}\label{ex3.1.4}
$\cF$ locally constant, $G$ smooth connected, the topology = the \'etale topology. Then
the conditions of Proposition \ref{p2} are verified. We thus get an isomorphism
\[\Ext^1(\uG,\cF)\iso H^1_\et(G,\cF)_{mult}\]
with the group of multiplicative classes in $H^1_\et(G,\cF)$.
\end{example}

\begin{lemma}\label{c3.2} Let $G$ be smooth connected and $L$ be a
locally constant $\Z$-constructible \'etale sheaf. Then  $\shom_\ES(\uG,L)=0$ and $\Hom_\ES(\uG,L)=0$. If $L$ has torsion-free
geometric fibres, then $\sext^1_\ES(\uG,L)=0$ and $\Ext^1_\ES(\uG,L)=0$.
\end{lemma}

\begin{proof} By the local to global $\Ext$ spectral sequence, it suffices to show the statement for $\shom$ and $\sext^1$. This reduces us to the case $L=\Z$ or $\Z/m$. Then the first
vanishing is obvious and the second follows from Example \ref{ex3.1.4} and the vanishing of  $H^1_\et(G\times U,\Z)$ for any smooth $k$-scheme $U$ (see \cite[IX 3.6]{sga4}). 
\end{proof}

\begin{lemma}\label{trans} For $G$ over $S =\Spec (k)$
let  $\cE\in \Ext^1_\ES(\uG,\G_m)$ and $g\in G(k)$. Denote by
$\tau_g$ the left translation by $g$. Then $\tau_g^*b(\cE)=b(\cE)$. Here $b$ is the map of
Proposition \ref{p1}.
\end{lemma}

\begin{proof} By Hilbert's theorem 90, $g$ lifts to an $h\in \cE(k)$. Then $\tau_h$ induces a
morphism from the $\G_m$-torsor $b(\cE)$ to the $\G_m$-torsor $\tau_g^*b(\cE)$: this morphism
must be an isomorphism.
\end{proof}

\subsection{Breen's method} For the proof of Theorems \ref{text} and \ref{t3.12} below we shall need the following
proposition, which unfortunately cannot be proven with the above elementary methods. 

\begin{propose}\label{ptorsion} Let $G$ be a smooth commutative algebraic $k$-group.  Let $\cF\in \ES$; assume that $\cF$ is 
\begin{thlist}
\item discrete or
\item  represented by a semi-abelian variety. 
\end{thlist}
Then, \\
a) the group $\Ext^i_\ES(\uG,\cF)$
is torsion  for any $i\ge 2$,\\
b) the sheaf $\sext^i_\ES(\uG,\cF)$ is $p$-torsion  for any $i> 2$.
\end{propose}

\begin{proof} This follows from the results and techniques of L. Breen \cite{breen}\footnote{We thank L. Illusie for pointing out this reference.}. In \cite{breen}, Breen works with the fppf topology but his methods carry over here without
any change: see remark in  \loccit top of p. 34. Case (ii) can be directly read off \cite{breen}: we explain this, then treat case (i) in some detail.

Considering the connected part $G^0$ of $G$, we reduce to the case where $G$ is connected, hence geometrically connected.  In order to prove the torsion claim of a), we may and do reduce to the case where $k$ is algebraically closed by a transfer argument; similarly for b), since this claim is local for the \'etale topology.

Suppose we are in case (ii). Using Chevalley's structure theorem for algebraic groups, we reduce to the cases where $G$ is of the form $\G_a$, $\G_m$ or an abelian variety and $\cF$ is represented by $\G_m$ or an abelian variety. If  $\cF$ is represented by $\G_m$, then a) follows from \cite[\S 7]{breen} if $G$ is an abelian variety (\resp from \cite[\S 8]{breen} if $G=\G_a$ or $\G_m$); if $\cF$ is represented by an abelian variety, it follows from \cite[\S 9]{breen} (even for $i\ge 1$). This proves a) in this case. 

For b), we use the fact that a) is true over any regular Noetherian base $S$: this is explicitly written in \cite{breen}, except in the case where $\cF$ is represented by an abelian variety (\loccit, \S 9). In this case however, Breen uses rigidity results for an abelian variety over a field, which can be extended over such an $S$ by using \cite[p. 116, Cor. 6.2]{GIT} (or obtained directly, as our  $S$-group schemes come from $k$-group schemes). This shows that the sheaf $\sext^i_\ES(\uG,\cF)$ is torsion for $i\ge 2$. To get the $p$-torsion statement, we argue as in \cite[\S 10]{breen}, using the torsion subsheaves ${}_m \uG$ of $\uG$ for $m$ prime to $p$.  For later use, let us repeat this argument here: since ${}_m\uG$ is locally constant for the \'etale topology and $\cF$ is $m$-divisible, we have $\sext_\ES^{j}({}_m\uG, \cF)=0$ for $j>0$ and the exact sequence
\begin{equation}\label{breen10}
\sext_\ES^{i-1}({}_m\uG, \cF)\by{\delta} \sext_\ES^i(\uG, \cF)\by{m} \sext_\ES^i(\uG, \cF)
\end{equation}
yields the claimed vanishing for $i\ge 2$.  This proves b) in case (ii).

It remains to treat Case (i). We go back to the method of \cite{breen} that we now summarise: using essentially the Eilenberg-Mac Lane spectrum associated to $\uG$, Breen gets two spectral sequences $'E_r^{p,q}$ and $''E_r^{p,q}$ converging to the same abutment, with
\begin{itemize}
\item $''E_2^{p,1}=\Ext^p_\ES(\uG,\cF)$;
\item $''E_2^{p,q}$ is torsion for $q\ne 1$;
\item $'E_2^{p,q}$ is the $p$-th cohomology group of a complex involving terms of the form $H^q_\et(G^a,\cF)$.
\end{itemize}

In case (i), it follows for example from \cite[(2.1)]{deninger} that
$H^q_\et(G^a,\cF)$ is torsion for any $q>0$: to see this easily, reduce to the case where $\cF$ is
constant by a transfer argument involving a finite extension of $k$. Hence $'E_2^{p,q}$ is
torsion for $q>0$. On the other hand, since $G$ is geometrically connected, so are its powers
$G^a$, which implies that $H^0(G^a,\cF)=H^0(k,\cF)$ for any $a$. Since the complex giving
$'E_2^{*,0}$ is just the bar complex, we get that $'E_2^{0,0} = \cF(k)$ and $'E_2^{p;0}=0$ for
$p>0$. Thus all degree $>0$ terms of the abutment are torsion, and the conclusion follows for a). For b) we argue similarly to Case (ii), replacing $i\ge 2$ by $i>2$ in \eqref{breen10}.
\end{proof}

\begin{remarks}\label{r3.2}\
\begin{enumerate}
\item The results of \S \ref{soul} could also be deduced from \cite{breen}, see its introduction.
\item Proposition \ref{ptorsion} b) is false for $i=2$, as  shown by the example $\sext^2_\ES(\G_m,\Z)[1/p]\allowbreak\simeq (\Q/\Z)'(-1)$. For this computation, note that $\Ext^1_\ES(\G_m,\Z)\allowbreak=0$ by Lemma \ref{c3.2}, but $\Ext^1_\ES(\G_m,\Z/m)\simeq \Z/m(-1)$ for $m$ prime to $p$ if $k$ is separably closed by Example \ref{ex3.1.4}.
\end{enumerate}
\end{remarks}

\subsection{$1$-motivic sheaves} 

\begin{defn}\label{d3.1} 1) An \'etale sheaf $\cF$ on $\Sm(k)$ is \emph{discrete} (or  \emph{$0$-motivic}) if 
 it is of the form $\underline{M}$, where $M\in {}^t\cM_0$ is a discrete group scheme (see Definition \ref{d1.1.1}). We write $\Shv^\star_0\subset \ES$ for the full subcategory of discrete sheaves.\\
2) An \'etale sheaf $\cF$ on $\Sm(k)$ is
\emph{$1$-motivic} if  there is a morphism of sheaves
\begin{equation}\label{2}
\uG\by{b} \cF
\end{equation}
where $G$ is a semi-abelian variety and $\ker b, \coker b$ are discrete. 
We denote by 
$\Shv^\star_1$ the full subcategory of   $\ES$ consisting of
$1$-motivic sheaves.
\end{defn}

Unfortunately, the category $\Shv^\star_1$ is not abelian if $k$ has positive characteristic $p$: multiplication by $p$ on $\G_m$ has trivial kernel and cokernel in $\Shv^\star_1$, but is not an isomorphism. For this reason, we introduce:

\begin{defn}\label{d3.1a} Let $p$ be the exponential characteristic of $k$. We write  $\Shv_0\df \Shv^\star_0[1/p]$ and $\Shv_1\df \Shv^\star_1[1/p]$.\index{$\Shv_0$, $\Shv_1$, $\Shv^\star_0$, $\Shv^\star_1$} We have a functor
\begin{align*}
\rho:\Shv_1&\to \ES\\
\cF&\mapsto \cF\otimes \Z[1/p].
\end{align*} 
\end{defn}



 We need the following lemma, in which $(-)\{p\}$ denotes $p$-primary torsion:

\begin{lemma}\label{ltorsion} Suppose that $\car k=p>0$. 
Let $\cF_1,\cF_2\in \Shv^\star_1$. Then the groups $\Hom_\ES(\cF_1,\cF_2\{p\})$, $\Hom_\ES(\cF_1,\cF_2\otimes \Q_p/\Z_p)$, $\Ext^1_\ES(\cF_1,\cF_2\{p\})$, $\Ext^1_\ES(\cF_1,\cF_2\otimes \Q_p/\Z_p)$  and $\Ext^2_\ES(\cF_1,\cF_2\{p\})$ are $p$-primary torsion.
\end{lemma}

\begin{proof} 
By d\'evissage, we reduce to the case where $\cF_1$ is discrete or is representable by a semi-abelian variety. The discrete case is easy and left to the reader. Assume $\cF_1=\uG_1$ for some semi-abelian variety $G_1$. By d\'evissage, the sheaf $\cF_2\{p\}$ is locally constant so the result follows from Example \ref{ex3.1.4}. For $\cF_2\otimes \Q_p/\Z_p$, we reduce using Proposition \ref{ptorsion} to the crucial case $\cF_2=\uG_2$. By Yoneda's lemma
\[\Hom_\ES(\uG_1,\uG_2\otimes \Q_p/\Z_p)\subset \uG_2(G_1)\otimes \Q_p/\Z_p\]
hence the first case. For the second case, we use the fact that any $k$-morphism from $G_1$ to $G_2$ is the sum of a constant morphism and a homomorphism (\eg \cite[Lemma 4.1]{toresnis}). It follows that $G_2\otimes \Q_p/\Z_p$ verifies the condition of Proposition \ref{p2} a).  Therefore $\Ext^1_\ES(\uG_1,\uG_2\otimes \Q_p/\Z_p)\subset H^1_\et(G_1,\uG_2\otimes \Q_p/\Z_p)$ and we conclude again. The case of $\Ext^2_\ES(\cF_1,\cF_2\{p\})$ follows from Proposition \ref{ptorsion} a).
\end{proof}

\begin{propose}\label{p3.1} a) In Definition \ref{d3.1} we may choose $b$ such that $\ker
b$ is torsion-free: we then say that $b$ is \emph{normalised}. We thus have an exact sequence
\[0\to L\by{a} \uG\by{b} \cF\by{c} E_1\to  0\]
with $L,E\in \Shv^\star_0$ and $L$ torsion-free.\\ 
b)  Given two $1$-motivic sheaves $\cF_1,\cF_2$, normalised morphisms
$b_i:\uG_i\to \cF_i$ and a map $\phi:\cF_1\to \cF_2$,
there exists a unique homomorphism of group schemes 
$\phi_G:G_1\to G_2$ such that the diagram
\[\begin{CD}
\uG_1@>{b_1}>> \cF_1\\
@V{\phi_G}VV@V{\phi}VV\\
\uG_2@>{b_2}>> \cF_2
\end{CD}\]
commutes.\\
c) Given a $1$-motivic sheaf $\cF$, a pair $(G,b)$ with $b$ normalised
is uniquely determined by $\cF$.\\
  d) The functor $\rho$ of Definition \ref{d3.1a} is fully faithful.\\ 
e) In view of d), the subcategories $\Shv_0$ and $\Shv_1$ of  $\ES$ are closed under kernels and cokernels; hence they are abelian subcategories, and the inclusion functors are exact.
\end{propose}
\begin{proof}   Recall that $G\mapsto \uG$ is fully faithful when restricted to smooth $k$-group schemes (Lemma \ref{l1.4}).

a) If $\ker b$ is not torsion-free, simply divide $G$ by the image of its torsion. 

b) We want to construct a commutative diagram
\begin{equation}\label{eq4.2}
\begin{CD}
0@>>> L_1@>a_1>> \uG_1@>{b_1}>> \cF_1@>{c_1}>> E_1@>>> 0\\
&&@V{\phi_L}VV@V{\phi_G}VV@V{\phi}VV@V{\phi_E}VV\\
0@>>> L_2@>a_2>> \uG_2@>{b_2}>> \cF_2@>{c_2}>> E_2@>>> 0
\end{CD}
\end{equation}
where $L_i=\ker b_i$ and $E_i=\coker b_i$.   By Lemma \ref{c3.2}, 
$c_2\phi b_1=0$: this proves the existence of $\phi_E$. We also get a
homomorphism of sheaves $\uG_1\to \uG_2/L_2$, which lifts to
$\phi_G:\uG_1\to \uG_2$ by Lemma \ref{c3.2} again, hence $\phi_L$.

From the construction, it is clear that $\phi_E$ is uniquely determined
by $\phi$ and that $\phi_L$ is uniquely determined by $\phi_G$. It
remains to see that $\phi_G$ is unique. Let $\phi'_G$ be another choice.
Then $b_2(\phi_G-\phi'_G)=0$, hence $(\phi_G-\phi'_G)(\uG_1)\subseteq
L_2$, which implies that $\phi_G=\phi'_G$.

c) Follows from b).

d) The statement is trivial in characteristic $0$ ($p=1$). If $p>1$, we have to show that the map
\[\Hom_\ES(\cF_1,\cF_2)\otimes \Z[1/p]\by{a} \Hom_{\ES}(\rho(\cF_1),\rho(\cF_2))\]
is bijective for any  $\cF_1,\cF_2\in \Shv_1$. Considering the morphism $\cF_2\to \rho(\cF_2)$, we deduce from Lemma \ref{ltorsion} that the map
\[\Hom_\ES(\cF_1,\cF_2)[1/p]\by{b} \Hom_\ES(\cF_1,\rho(\cF_2))[1/p]=\Hom_\ES(\cF_1,\rho(\cF_2))\]
is bijective. Using now the morphism $\cF_1\to \rho(\cF_1)$,  the map
\[\Hom_\ES(\rho(\cF_1),\rho(\cF_2))\by{c} \Hom_\ES(\cF_1,\rho(\cF_2))\]
is clearly bijective. It remains to observe that $b=ca$, which is left to the reader.

e) The case of $\Shv_0$ is obvious. For $\Shv_1$, we may reduce to a map $\phi$ as in
b). We want to show that
$\cF_3=\ker \rho(\phi)$ and $\cF_4=\coker\rho(\phi)$ are in $\rho(\Shv_1)$. Let $G_3=(\ker\phi_G)^0$ and
$G_4=\coker\phi_G$: we get induced maps $b_i:\uG_i [1/p]\to \cF_i$ for $i=3,4$.
An easy diagram chase shows that $\ker b_i$ and $\coker b_i$ are both  in $\rho(\Shv_0)$. 
\end{proof}

Here is an extension of Proposition \ref{p3.1} which elucidates the
structure of $\Shv^\star_1$ somewhat:

\begin{thm}\label{t3.2.4} a) Let $\SAb$ be the
category of
semi-abelian $k$-varie\-ties (\cf Definition \ref{dsabt}). Then the fully faithful functor
\begin{align*}
\SAb&\to \Shv^\star_1\\
G&\mapsto \underline{G}
\end{align*}
has a faithful right adjoint $\gamma$; the counit of this adjunction
is given by \eqref{2} (with $b$ normalised).  The functor   $\gamma[1/p]:\Shv_1\to \SAb[1/p]$ is still faithful and right adjoint to the full embedding $\SAb[1/p]\into \Shv_1$; it is ``exact up to isogenies". For a morphism
$\phi$ of $\Shv^\star_1$, $\gamma(\phi)=\phi_G$ is an isogeny if and only if $\ker\phi$ and
$\coker\phi\in \Shv^\star_0$. In particular, $\gamma$ induces an
equivalence of categories 
\[\Shv_1/\Shv_0\iso\SAb\otimes\Q\] 
where $\SAb\otimes\Q$ is the category of semi-abelian varieties up to
isogenies.\\ 
b) The inclusion functor $\Shv^\star_0\to \Shv^\star_1$ has a left
adjoint $\pi_0$; for $\uG\by{b} \cF\in  \Shv^\star_1$ as in \eqref{2} the unit of this adjunction is given by
$\cF\to \coker b$, and $\pi_0(\cF)=  \coker b$.  The functor $\pi_0[1/p]$ is left adjoint to the inclusion $\Shv_0\into \Shv_1$. The  induced right exact functor
\[(\pi_0)_\Q:\Shv_1\to \Shv_0\otimes\Q\]
has one left derived functor $(\pi_1)_\Q$ given by $\ker b$ in \eqref{2}.
\end{thm}

\begin{proof}  For the assertions on $-[1/p]$, see Lemma \ref{lB.1}.

a) The only delicate thing is the exactness of $\gamma$ up to isogenies. This
means that, given a short exact sequence $0\to \cF'\to \cF\to \cF''\to 0$  in $\Shv_1$, the sequence
\[0\to \gamma(\cF')[1/p]\to \gamma(\cF)[1/p]\to\gamma(\cF'')[1/p]\to 0\]
is half exact and the middle homology is finite.  We may reduce to the case where the maps come from $\Shv^\star_1$;  then this follows from a chase in the diagram
\[\begin{CD}
0@>>> L'@>a'>> \uG'@>{b'}>> \cF'@>{c'}>> E'@>>> 0\\
&&@VVV@VVV@VVV@VVV\\
0@>>> L@>a>> \uG@>{b}>> \cF@>{c'}>> E@>>> 0\\
&&@VVV@VVV@VVV@VVV\\
0@>>> L''@>a''>> \uG''@>{b''}>> \cF''@>{c''}>> E''@>>> 0
\end{CD}\]
of which we summarize the main points: (1) $\uG'\to \uG$ is injective because its kernel is the
same as $\ker(L'\to L)$. (2) $\uG\to \uG''$ is surjective because (i) $\Hom(\uG''\to \coker(E'\to
E))=0$ and (ii) if $L''\to \coker(\uG\to \uG'')$ is onto, then this cokernel is $0$. (3) The
middle homology is finite because the image of $\ker(\uG\to \uG'')\to E'$ must be finite.

In b), the existence and characterisation of $(\pi_1)_\Q$ follows from the exactness of $\gamma[1/p]$
in a).
\end{proof}

\begin{remark} One easily sees that  the left derived functor of $\pi_0$ does not exist integrally. Rather, it exists as a
functor to the category of pro-objects of $\Shv_0$. (Actually to a finer subcategory:  compare
\cite[\S 5.3, Def. 1]{proalg}). 
\end{remark}

\subsection{Extensions of $1$-motivic sheaves}  By Proposition \ref{p3.1} d), we now view $\Shv_0$ and $\Shv_1$ as full subcategories of $\ES$ via the functor $\rho$ of Definition \ref{d3.1a}.
The aim of this subsection is to prove:

\begin{thm}\label{text}
The categories $\Shv_0$ and $\Shv_1$ are thick in $\ES$.
\end{thm}

\begin{proof}By Proposition \ref{p3.1} e), this amounts to show that $\Shv_0$ and $\Shv_1$ are stable under extensions in $\ES$. 

The statement is obvious for $\Shv_0$. Let us now show that $\Shv_1$ is closed
under extensions in $\ES$. Let
$\cF_1,\cF_2$ be as in
\eqref{eq4.2} (no map given between them). As the functor $\rho$ is exact and fully faithful  we  have
an injection
$\Ext^1_{\Shv_1}(\cF_2,\cF_1)\into \Ext^1_\ES(\rho(\cF_2),\rho (\cF_1))$. The same argument as in the proof of Proposition \ref{p3.1} d) shows that the map
$\Ext^1_\ES(\cF_2,\cF_1)\otimes \Z[1/p]\to \Ext^1_\ES(\rho(\cF_2),\rho (\cF_1))$ 
is bijective (note that the restriction map $\Ext^1_\ES(\rho(\cF_2),\rho (\cF_1))\to \Ext^1_\ES(\cF_2,\rho (\cF_1))$ has the inverse $\cE\mapsto \cE[1/p]$). We are then left to show that the injection
\begin{equation}\label{eq3.1}
\Ext^1_{\Shv_1}(\cF_2,\cF_1)\into \Ext^1_\ES(\cF_2,\cF_1)[1/p]
\end{equation}
 is surjective.
This is certainly so in the following special cases:
\begin{enumerate}
\item $\cF_1$ and $\cF_2$ are semi-abelian varieties;
\item $\cF_2$ is semi-abelian and $\cF_1$ is discrete (see Example \ref{ex3.1.4}).
\end{enumerate}

We proceed from these cases by increasing degree of complexity of $\cF_1,\cF_2$. For $m>1$, consider 
\[\cF^m=\coker(L_1\vlongby{(a_1,m)}\uG_1\oplus L_1)\]
so that we have two exact sequences
\[\begin{CD}
0@>>> \uG_1@>(1_{G_1},0)>> \cF^m@>>> L_1/m@>>> 0\\
0@>>> L_1@>(a_1,0)>> \cF^m@>>> \uG_1/L_1\oplus L_1/m@>>> 0.
\end{CD}\] 

Then first one shows that \eqref{eq3.1} is surjective for $(\cF_2,\cF_1)=(\uG_2,\cF^m)$. Let us now consider the commutative diagram with exact rows associated to the second one, for an unspecified $m$:
\begin{equation}
\begin{CD}
\Ext^1_{\Shv_1}(\uG_2,\cF^m)&\to& \Ext^1_{\Shv_1}(\uG_2,\uG_1/L_1\oplus L_1/m)&\longby{}& \Ext^2_{\Shv_1}(\uG_2,L_1)\\
@V{\wr}VV @VVV @VVV\label{eq3.2}\\
\Ext^1_\ES(\uG_2,\cF^m)&\to& \Ext^1_\ES(\uG_2,\uG_1/L_1\oplus L_1/m)&\longby{\delta^m}& \Ext^2_\ES(\uG_2,L_1).
\end{CD}
\end{equation}

Note that the composition 
\[\Ext^1_\ES(\uG_2,\uG_1/L_1)\to \Ext^1_\ES(\uG_2,\uG_1/L_1\oplus L_1/m)\longby{\delta^m} \Ext^2_\ES(\uG_2,L_1)\]
coincides with the boundary map $\delta$ associated to the exact sequence
\[0\to L_1\to \uG_1\to \uG_1/L_1\to 0.\]

Let $e\in \Ext^1_\ES(\uG_2,\uG_1/L_1)$. By Proposition \ref{ptorsion} a), $f=\delta(e)$ is torsion. Choose now $m$ such that $mf=0$. Then there exists $e'\in \Ext^1_\ES(\uG_2,L_1/m)$ which bounds to $f$ via the Ext exact sequence associated to the exact sequence of sheaves
\[0\to L_1\by{m} L_1\to L_1/m\to 0.\]

Since $\delta^m(e,-e')=0$, \eqref{eq3.2} shows that $(e,-e')$ comes from the left, which shows that \eqref{eq3.1} is surjective for $(\cF_2,\cF_1)=(\uG_2,\uG_1/L_1)$.

By Lemma \ref{c3.2}, in the commutative diagram
\[\begin{CD}
\Ext^1_{\Shv_1}(\uG_2,\uG_1/L_1)@>>> \Ext^1_{\Shv_1}(\uG_2,\cF_1)\\
@V{\wr}VV @VVV\\
\Ext^1_\ES(\uG_2,\uG_1/L_1)@>>> \Ext^1_\ES(\uG_2,\cF_1)
\end{CD}\]
the horizontal maps are isomorphisms. Hence  \eqref{eq3.1} is surjective for $\cF_2=\uG_2$ and any $\cF_1$. 

To conclude, let $\cF$ be an extension of $\cF_2$ by $\cF_1$ in $\ES$. By the above, $\cF'\df b_2^*\cF$ is $1$-motivic as an extension of $\uG_2$ by $\cF_1$, and we have an exact sequence
\[0\to L_2\to \cF'\to\cF\to E_2\to 0.\]

Let $b':\uG\to \cF'$ be a normalised map (in the sense of Proposition \ref{p3.1}) from a
semi-abelian variety to $\cF'$ and let $b:\uG\to \cF$ be its composite with the above map. It
is then an easy exercise to check that $\ker b$ and $\coker b$ are both discrete. Hence $\cF$
is $1$-motivic.
\end{proof}

\begin{remark}\label{fppf} We may similarly define $1$-motivic sheaves for the fppf
topology over $\Spec k$; as one easily checks, all the above results hold equally
well in this context. This is also the case for \S \ref{3.7} below. 

In fact, let $\Shv_1^{\rm fppf}$ be the 
category of fppf
$1$-motivic sheaves and $\pi:\Sch(k)_{\rm fppf}\to \Sm(k)_\et$ be the projection functor. Then the
functors $\pi^*$ and $\pi_*$ induce \emph{quasi-inverse equivalences of categories} between
$\Shv_1$ and $\Shv_1^{\rm fppf}$. Indeed it suffices to check that $\pi_*\pi^*$ is naturally
isomorphic to the identity on $\Shv_1$: if $\cF\in \Shv_1$ and we consider its normalised
representation, then in the commutative diagram
\[\begin{CD}
0@>>> L@>>> \uG@>>> \cF@>>> E@>>> 0\\
&& @VVV @VVV @VVV @VVV\\
0@>>> \pi_*\pi^*L@>>> \pi_*\pi^*\uG@>>> \pi_*\pi^*\cF@>>> \pi_*\pi^*E@>>> 0
\end{CD}\]
the first, second and fourth vertical maps are isomorphisms and the lower
sequence is still exact: both facts follow from \cite[p. 14, Th. III.3.9]{MI}.

In particular the restriction of $\pi_*$ to $\Shv_1^{\rm fppf}$ is exact. Actually,
$(R^q\pi_*)_{|\Shv_1^{\rm fppf}}=0$ for $q>0$ (use same reference).  
\end{remark}

\subsection{A basic example}\label{s:basic}
For $\pi : X\to k$ we shall denote $\pi_*\G_m$ by $\G_{X/k}$ and $R^1\pi_*\G_m$ by $\underline{\Pic}_{X/k}$, both considered as sheaves on $\Sm(k)_\et$.

\begin{propose}\label{p3.3.1} Let $X\in \Sm(k)$, with structural morphism $\pi$. Then $\G_{X/k}$ and $\underline{\Pic}_{X/k}$ are $1$-motivic (\ie $\pi_*\G_m[1/p]$ and $R^1\pi_*\G_m[1/p]$  belong to $\rho(\Shv_1)$).
\index{$\underline{\Pic}_{X/k}$, $\NS_{X/k}$}
\end{propose}

\begin{proof} We may assume $X$ irreducible. We shall use Theorem \ref{text} (thickness of $\Shv_1$) repeatedly in this proof, a fact that we won't recall.

Suppose first that $X$ is smooth projective, with field of constants $E=\pi_0(X)$. Then  $\G_{X/k}=\G_{E/k}$ and $\underline{\Pic}_{X/k}=R_{E/k} \underline{\Pic}_{X/E}$
 where $\underline{\Pic}_{X/E}$ is an extension of the discrete sheaf $\NS_{X/E}$ (N\'eron-Severi) by
the abelian variety $\underline{\Pic}^0_{X/E}$ (Picard variety). Both sheaves are clearly $1$-motivic over $E$, and so are their restrictions of scalars to $k$. (Note that this restriction of scalars is exact.)

In general, we apply  de Jong's  theorem  \cite[Th. 4.1]{DJ}: there exists a
diagram
\[\begin{CD}
\tilde U @>>> \bar X\\
@V{\pi}VV\\
   U  @>>>   X
\end{CD}\]
where the horizontal maps are open immersions, $\bar X$ is smooth
projective and the vertical map is finite \'etale. Then we get a
corresponding diagram of units and Pics

\begin{equation}\label{dejong}
\begin{CD}
0&\longleftarrow&\underline{\Pic}_{\tilde U/k} &\longleftarrow& \underline{\Pic}_{\bar X/k}&\longleftarrow&\displaystyle \bigoplus_{x\in  \bar X^{(1)}\cap \tilde Z}\Z&\longleftarrow& \G_{\tilde U/k} &\longleftarrow& \G_{\bar X/k} &\longleftarrow& 0\\
&&@A{\pi^*}AA&&&&@A{\pi^*}AA\\
0&\longleftarrow&\underline{\Pic}_{U/k}  &\longleftarrow&   \underline{\Pic}_{X/k}&\longleftarrow&\displaystyle \bigoplus_{x\in  X^{(1)}\cap Z}\Z&\longleftarrow& \G_{U/k}  &\longleftarrow&   \G_{X/k}&\longleftarrow& 0
\end{CD}
\end{equation}
with exact rows, where $Z=X-U$ and $\tilde Z=\bar X-\tilde U$. This already shows  that $\G_{\tilde U/k}\in \Shv_1$ and $\underline{\Pic}_{\tilde U/k}\in \Shv_1$. 

Consider the \v Cech spectral sequence associated to the \'etale hypercover $\cU$ associated to $\pi$. It yields an injection and an exact sequence
\begin{gather}0\to \G_{U/k}\by{\pi^*}\G_{\tilde U/k}\label{eqG}\\
0\to \check H^1(\cU,\G_{?/k})\to \underline{\Pic}_{U/k}\by{\pi^*} \underline{\Pic}_{\tilde U/k}\label{eqpic}
\end{gather}
where $\check H^1(\cU,\G_{?/k})$ is the cohomology of the complex
\[\G_{\tilde U/k}\to \G_{\tilde U\times_U\tilde U/k}\to \G_{\tilde U\times_U\tilde U\times_U\tilde U/k}. \]

The pull-back maps $\pi^*$ have ``almost retractions'' $\pi_*$ such that $\pi_*\pi^*=n$ where $n=\deg(p)$. Thus \eqref{eqG} refines to an injection
\[\G_{U/k}\into \ker(n-\pi^*\pi_*)\]
with cokernel killed by $n$. Lemma \ref{l3.5.2} below then implies that $\G_{U/k}\in \Shv_1$. Applying this to $\tilde U\times_U\tilde U$ and $\tilde U\times_U\tilde U\times_U\tilde U$, we find that $\check H^1(\cU,\G_{?/k})\in \Shv_1$. Similar to the above, we have a map
\[\underline{\Pic}_{U/k}\to \ker(n-\pi^*\pi_*)\]
with cokernel killed by $n$; applying Lemma \ref{l3.5.2} again, we get $\underline{\Pic}_{U/k}\in \Shv_1$. Then we are done by considering the lower row of \eqref{dejong}.
\end{proof}

\begin{lemma}\label{l3.5.2} Let $\cF\in\Shv_1$. Then any quotient of $\cF$ which is of finite exponent $n>0$ is in $\Shv_0$.
\end{lemma}

\begin{proof} Take  $\phi=$ multiplication by $n$ in \eqref{eq4.2} and use that the semi-abelian part of $\cF$ is divisible.
\end{proof}

\subsection{Application: the N\'eron-Severi group of a smooth scheme} The following extends its definition from smooth projective to smooth varieties:

\begin{defn}\label{dNS} Let $X\in \Sm(k)$.\\
a) Suppose that $k$ is algebraically closed. Then we write $\NS(X)$ for the group of cycles of
codimension $1$ on $X$ modulo algebraic equivalence.\\ b) In general, we define $\NS_{X/k}$ as
the discrete \'etale sheaf associated to the $Gal(\bar k/k)$-module $\NS(X\otimes_k \bar k)$.
\index{$\underline{\Pic}_{X/k}$, $\NS_{X/k}$}
\end{defn}

 Note that by the rigidity of cycles modulo algebraic equivalence,
\[\NS_{X/k}(U)=\NS(X\times_k \overline{k(U)})^G\]
if $U\in \Sm(k)$ is irreducible, $\overline{k(U)}$ is a separable closure of $k(U)$ and
$G=Gal(\overline{k(U)}/k(U))$.

\begin{propose} \label{belong} The natural map
$e:\underline{\Pic}_{X/k}\to \NS_{X/k}$ identifies $\NS_{X/k}$ with
$\pi_0[1/p](\underline{\Pic}_{X/k})$ (\cf Theorem \ref{t3.2.4} b)). In particular, $\NS_{X/k}\in
\Shv_0$. (Here we use that $\underline{\Pic}_{X/k}\in \Shv_1$, see Proposition \ref{p3.3.1}.)
\end{propose}

\begin{proof}  Proposition \ref{pi0} below  
implies that $e$ induces a map 
\[\bar e:\pi_0(\underline{\Pic}_{X/k})\allowbreak\to \NS_{X/k}\]
which is
evidently epi. But let $\underline{\Pic}^0_{X/k}\df\ker e$: by \cite[Lemma 7.10]{bo},
$\Pic^0(X_{\bar k})=\underline{\Pic}^0_{X/k}(\bar k)$ is divisible, which forces $\bar e$ to be
an isomorphism in $\Shv_0$.
\end{proof}

\begin{remarks} 1) Proposition \ref{belong} implies in particular that  $\NS(X)[1/p]$ is  a finitely generated $\Z[1/p]$-module for any $X\in \Sm(k)$ if $k$ is algebraically closed  (\cf \cite[Th. 3]{pic} for another proof  of a stronger result).

2) In \cite[Cor. 1.2.5]{ABV}, the functor $\pi_0$ is extended to the category $\EST$, with values in the category of ind-objects of $\Shv_0$. When applied to the  \'etale sheaf ${\rm \underline{CH}}^{r,\et}_{X/k}$ associated to the relative $r$-th Chow presheaf of $X$, one gets  $\NS_{X/k}^{r, \et}$, the  \'etale sheaf associated to the relative presheaf of cycles of codimension $r$ modulo algebraic equivalence \cite[Theorem 3.1.4]{ABV}. This generalises  Proposition \ref{belong}.
\end{remarks}

\subsection{Technical results on $1$-motivic sheaves}

\begin{propose}\label{p4.2} The evaluation functor
\begin{align*}
ev: \Shv_1&\to \Mod\Z[1/p]\\
\cF&\mapsto \cF(\bar k)[1/p]
\end{align*}
 to the category $\Mod\Z[1/p]$ of $\Z[1/p]$-modules is exact and faithful, hence (\cf \cite[Ch. 1, p. 44,
prop. 1]{bbki}) ``faithfully exact": a se\-quen\-ce $\cE$ is exact if and only if $ev(\cE)$ is
exact.
\end{propose}

\begin{proof} The exactness of $ev$ is clear. For faithfulness, let
$\phi:\cF_1\to \cF_2$ be such that $ev(\phi)=0$. In $ev$\eqref{eq4.2}, we
have $\phi_G(\uG_1(\bar k))\subseteq L_2(\bar k)$; since the former group
is divisible and the latter is finitely generated, $ev(\phi_G)=0$. Hence
$\phi_G=0$. On the other hand, $ev(\phi_E)=0$, hence $\phi_E=0$. This
implies that $\phi$ is of the form $\psi c_1$ for $\psi:E_1\to \cF_2$.
But $ev(\psi)=0$, which implies that $\psi=0$.
\end{proof}

The following strengthens Theorem \ref{t3.2.4} b):

\begin{propose} \label{pi0} a) Let $G$ be a  connected commutative algebraic
$k$-group and let $E$ be a $Gal(\bar k/k)$-module, viewed as an \'etale
sheaf over $\Sm(k)$ ($E$ is not supposed to be constructible). Then $\Hom(\uG,E)=0$.\\ b) Let
$\cF\in
\Shv^\star_1$ and
$E$ as in a). Then any morphism $\cF\to E$ factors canonically through
$\pi_0(\cF)$.
\end{propose}

\begin{proof} a) Thanks to Proposition \ref{p4.2} we may assume $k$
algebraically closed. By Yoneda, $\Hom(\uG,E)$ is a subgroup of $E(G)$
(it turns out to be the subgroup of multiplicative sections but we don't
need this). Since $E(k)\iso E(G)$, any homomorphism from $\uG$ to $E$ is
constant, hence $0$.

b) follows immediately from a) and Proposition \ref{p3.1}.
\end{proof}

\begin{propose}\label{pladj} The fully faithful functor
\begin{align*}
{}^t\AbS&\to \Shv^\star_1\\
G&\mapsto \uG
\end{align*}
has a left adjoint $\Omega$. (See Definition \ref{dsabt} for ${}^t\AbS$.)
\end{propose}

\begin{proof} Let $\cF\in \Shv^\star_1$ with normalised representation 
\begin{equation}\label{eqnorm}
0\to L\to \uG\by{b} \cF\to E\to 0.
\end{equation}

 As the set of
closed subgroups of $H\subseteq G$ is Artinian, there is a minimal $H$ such that the composition
\[L\to \uG\to \uG/\underline{H}\]
is trivial. Then $\cF/b(\underline{H})$  is represented by an object $\Omega(\cF)$ of ${}^t\AbS$ and
it follows from Proposition \ref{p3.1} b) that the universal property is satisfied.
(In other words,  $\Omega(\cF)$ is the quotient of $\cF$ by the Zariski closure of $L$ in $G$.)
\end{proof}

\begin{propose}\label{p3.6} Let $f:\cF_1\to \cF_2$ be a morphism in $\Shv^\star_1$.
Assume that for any $n>1$, $f$ is an isomorphism on $n$-torsion and
injective on
$n$-cotorsion. Then $f$ is injective with lattice cokernel. If $f$ is even
bijective on $n$-cotorsion, it is an isomorphism.
\end{propose}

\begin{proof} a) We first treat the special case where $\cF_1=0$. Consider multiplication by $n$ on the normalised presentation of $\cF_2$:
\[\begin{CD}
0@>>> L@>>> \uG@>>> \cF_2@>>> E@>>> 0\\
&& @V{n_L}VV @V{n_G}VV @V{n}VV @V{n_E}VV\\
0@>>> L@>>> \uG@>>> \cF_2@>>> E@>>> 0.
\end{CD}\]

Since $L$ is torsion-free, $n_G$ is injective for all $n$, hence $G=0$ and $\cF_2=E$. If moreover multiplication by $n$ is surjective for any $n$, we have $\cF_2=0$ since $E$ is finitely generated. 

b) The general case. Split $f$ into two short exact sequences:
\begin{gather*}
0\to K\to \cF_1\to I\to 0\\
0\to I\to \cF_2\to C\to 0.
\end{gather*}

We get torsion/cotorsion exact sequences
\begin{gather*}
0\to {}_nK\to {}_n\cF_1\to {}_nI\to K/n\to \cF_1/n\to I/n\to 0\\
0\to {}_nI\to {}_n\cF_2\to {}_nC\to I/n\to \cF_2/n\to C/n\to 0.
\end{gather*}

A standard diagram chase successively yields ${}_n K=0$, ${}_n\cF_1\iso {}_nI\iso {}_n\cF_2$, $\cF_1/n\iso I/n$, $K/n=0$ and ${}_n C=0$. By a), we find $K=0$ and $C$ a lattice, which is what we wanted.
\end{proof}

\subsection{Presenting $1$-motivic sheaves by group schemes}\label{3.7} In this subsection, we
give another description of the category $\Shv^\star_1$; it will be used in the next subsection.

\begin{defn}\label{s1a} We denote by $S_1^\eff$ \index{$S_1$, $S_1^\eff$} the full subcategory of ${}^t\AbS^{[-1,0]}$ consisting of those complexes $F_\cdot=[F_1\to F_0]$ such that
\begin{thlist}
\item $F_1$ is discrete (\ie in ${}^t\cM_0$);
\item $F_0$ is of the form $L_0\oplus G$, with $L_0\in {}^t\cM_0$ and $G\in \SAb$;
\item $F_1\to F_0$ is a monomorphism;
\item $\ker(F_1\to L_0$) is free. 
\end{thlist}
We call $S_1^\eff$ the \emph{category of presentations}.
\end{defn}

We shall view $S_1^\eff$ as a full subcategory of ${\Shv^\star_1}^{[-1,0]}$ via the functor $G\mapsto
\underline{G}$ which sends a group scheme to the associated representable sheaf. In this light,
$F_\cdot$ may be viewed as a
\emph{presentation} of
$\cF:=H_0(\underline{F_\cdot})$. In the next definition, quasi-isomorphisms are also understood
from this viewpoint.

\begin{defn}\label{s1b} We denote by $\Sigma$ the collection of quasi-iso\-morph\-isms
of $S_1^\eff$, by ${\bar S}_1^\eff$ the homotopy
category of $S_1^\eff$ (Hom groups quotiented by homotopies) and by
$S_1=\Sigma^{-1}{\bar S}_1^\eff$ the localisation of ${\bar S}_1^\eff$ with respect to (the
image of) $\Sigma$.
\end{defn}

The functor $F_\cdot\mapsto H_0(F_\cdot)$ induces a functor 
\begin{equation}\label{ftot}
h_0:S_1\to \Shv^\star_1.
\end{equation}

Let $F_\cdot=(F_1,L_0,G)$ be a presentation of $\cF\in \Shv^\star_1$. Let $L=\ker(F_1\to
L_0)$ and $E=\coker(F_1\to L_0)$. Then we clearly have an exact sequence
\begin{equation}\label{eq2.3}
0\to L\to \uG\to\cF\to E\to 0.
\end{equation}

\begin{lemma}\label{l2.3.1} Let $F_\cdot=(F_1,L_0,G)\in S_1^\eff$. Then, for any finite
Galois extension
$\ell/k$ such that $L_0$ is constant over $\ell$, there exists a \qi $\tilde F_\cdot\to
F_\cdot$, with $\tilde F_\cdot=[\tilde F_1\by{u_0} \tilde L_0\oplus G]$ such that  $\tilde L_0$ is a free $Gal(\ell/k)$-module.
\end{lemma}

\begin{proof} Just take for $\tilde L_0$ a free module projecting onto $L_0$ and for
$\tilde F_1\to \tilde L_0$ the pull-back of $F_1\to L_0$.
\end{proof}

\begin{lemma}\label{l2.3.2} The set $\Sigma$ admits a calculus of right
fractions within $\bar S_1^\eff$ in the sense of (the dual of) \cite[Ch. I, \S 2.3]{GZ}.
\end{lemma}

\begin{proof} The statement is true by Lemma \ref{lloc} if we replace $S_1^\eff$ by
${\Shv^\star_1}^{[-1,0]}$; but one easily checks that the constructions in the proof of Lemma
\ref{lloc} preserve $S_1^\eff$.
\end{proof}

\begin{propose}\label{p3.4.6} The functor $h_0$ of \eqref{ftot} is an equivalence of categories. In
particular,  $S_1 [1/p]$ is abelian.
\end{propose}

\begin{proof} {\it Step 1.} $h_0$ is essentially surjective. Let $\cF\in \Shv^\star_1$ and let
\eqref{eq2.3} be the exact sequence attached to it by Proposition \ref{p3.1} b). We shall
construct a presentation of $\cF$ from \eqref{eq2.3}. Choose elements $f_1,\dots,f_r\in
\cF(\bar k)$ whose images generate
$E(\bar k)$. Let
$\ell/k$ be a finite Galois extension such that all $f_i$ belong to $\cF(\ell)$, and let
$\Gamma=Gal(\ell/k)$. Let $\tilde L_0=\Z[\Gamma]^r$ and define a morphism of sheaves $\tilde
L_0\to \cF$ by mapping the $i$-th basis element to $f_i$. Then $\ker(\tilde L_0\to E)$ maps to
$\uG/L$. Let $M_0$ be the kernel of this morphism, and let $L_0= \tilde L_0/M_0$. Then $\tilde
L_0\onto E$ factors into a morphism $L_0\onto E$, whose kernel $K$ injects into $\uG/L$.

Pick now elements $g_1,\dots,g_s\in G(\bar k)$ whose image in $G(\bar k)/L(\bar k)$ generate
the image of $K(\bar k)$, and $g_{s+1},\dots,g_t\in G(\bar k)$ be generators of the image of
$L(\bar k)$. Let $\ell'/k$ be a finite Galois extension such that all the $g_i$ belong to
$G(\ell')$, and let $\Gamma'=Gal(\ell'/k)$. Let $\tilde F_1=\Z[\Gamma']^t$, and define a
map
$f:\tilde F_1 \to G$ by mapping the $i$-th basis element to $g_i$.  By construction,
$f^{-1}(L)=\ker(\tilde F_1\onto K)$ and $f':f^{-1}(L)\to L$ is onto. Let $M_1$
be the kernel of $f'$ and $F_1=\tilde F_1/M_1$: then $\tilde F_1\to K$ factors through
$F_1$ and
$\ker(F_1\onto K)=\ker(F_1\to L_0)\iso L$. In particular, condition (iii) of Definition
\ref{s1a} is verified.

{\it Step 2.} $h_0$ is faithful. Let $f:F_\cdot\to F'_\cdot$ be a map in $S_1$ such that
$h_0(f)=0$. By Lemma \ref{l2.3.2}, we may assume that $f$ is an effective map (\ie comes from
$S_1^\eff$). We have
$f(L_0\oplus G)\subseteq
\im(L'_1\to  L'_0\oplus G')$, hence $f_{|G} = 0$ and 
$f(L_0)$ is contained in $\im(L'_1\to L'_0\oplus G')$. Pick a finite Galois extension $\ell/k$
such that
$L_0$ and $L'_1$ are constant over $\ell$. By Lemma \ref{l2.3.1}, take a \qi $u:
[\tilde F_1\to \tilde L_0]\to [F_1\to L_0]$ such that $\tilde L_0$ is $Gal(\ell/k)$-free.
Then the composition
$\tilde L_0\to L_0\to \im(L'_1\to L'_0\oplus G')$ lifts to a map $s:\tilde L_0\to L'_1$,
which defines a homotopy between $0$ and $fu$.

{\it Step 3.} $h_0$ is full. Let $F_\cdot,F'_\cdot\in S_1$ and let $\phi:\cF\to \cF'$, where
$\cF=h_0(F_\cdot)$ and $\cF'=h_0(F'_\cdot)$. In particular, we get a map $\phi_G:\uG\to
\uG'$ and a map $\psi:L_0\to L'_0\oplus \uG'/F'_1$. Let
$\ell/k$ be a finite Galois extension such that $F'_1$ is constant over $\ell$. Pick a \qi
$u:\tilde F_\cdot\to F_\cdot$ as in Lemma
\ref{l2.3.1} such that $\tilde L_0$ is
$Gal(\ell/k)$-free. Then $\psi\circ u$ lifts to a map $\tilde\psi:\tilde L_0\to L'_0\oplus
\uG'$. The map
\[f=(\tilde\psi,\phi_G):\tilde L_0\oplus G\to L'_0\oplus G'\]
sends $\tilde F_1$ into $F'_1$ by construction, hence yields a map $f:\tilde F_\cdot\to
F'_\cdot$ such that $h_0(f u^{-1})=\phi$.
\end{proof}

\begin{cor} The obvious functor
\[S_1[1/p]\to D^b(\Shv_1)\]
is fully faithful.
\end{cor}

\begin{proof} The composition of this functor with $H_0$ is the equivalence $h_0$ of
Proposition \ref{p3.4.6}. Therefore it suffices to show that the restriction of $H_0$ to the
image of $S_1[1/p]$ is faithful. This is obvious, since the objects of this image are homologically
concentrated in degree $0$.
\end{proof}

\subsection{The transfer structure on $1$-motivic sheaves}\label{3.9}
Recall the category  $\HI_\et^s$ from Definition
\ref{his}; it is a thick $\Z[1/p]$-linear subcategory of $\EST$. Recall that we also have denoted by $\rho$ the functor $G\mapsto \uG [1/p]$ from $\cG^*_{\rm hi}$ to $\HI_\et^s$, see \eqref{rho}. 

\begin{propose}\label{ptransf} The restriction of \eqref{rho} to ${}^t\AbS$ extends to a full embedding 
\[\rho:\Shv_1\into\HI_\et^s.\]
This functor is exact with thick image (\ie stable under extensions).
\end{propose}

\begin{proof} By Proposition \ref{p3.4.6}, it suffices to construct a functor $\rho:S_1[1/p]\allowbreak\to
\HI_\et^s$. First define a functor  $\tilde \rho:S_1^\eff[1/p]\to \HI_\et^s$ by
\[\tilde\rho([F_1\to F_0])= \coker(\rho(F_1)\to \rho(F_0)),\]
using \eqref{rho}. Note that the forgetful functor 
\[f:\HI_\et^s\to \ES\]
 is faithful and exact, hence
conservative. This first gives that
$\tilde\rho$ factors into the desired $\rho$.

Proposition \ref{p3.1} d) -- e) says that $f\rho$ is fully faithful and exact.
Since
$f$ is faithful, $\rho$ is fully faithful and exact. 

It remains to show that $\rho$ is thick. By
Theorem \ref{text}, $\Shv_1$ is thick in $\ES$. Since $f$ is exact, the proof is concluded by the lemma below.
\end{proof}

\begin{lemma}\label{l3.8} The functor $f:\HI_\et^s\to \ES$ is fully faithful. In particular, the transfer structure on a sheaf $\cF\in \HI_\et^s$ is unique.
\end{lemma}

\begin{proof} It is a variation on the Gersten principle of Proposition \ref{pgersten}. Let $\cF_1,\cF_2\in \HI_\et^s$ and let $\phi:f\cF_1\to f\cF_2$ be a morphism. Thus, for $X,Y\in
\Sm(k)$, we have a diagram
\[\begin{CD}
\cF_1(X)\otimes c(Y,X)@>>> \cF_1(Y)\\
@V\phi_X \otimes 1VV @V\phi_Y VV\\
\cF_2(X)\otimes c(Y,X)@>>> \cF_2(Y)
\end{CD}\]
and we want to show that it commutes. We may clearly assume that $Y$ is irreducible.

If $F=k(Y)$ is the function field of $Y$ then the map $\cF_2(Y)\to \cF_2(F)$ is injective by \cite[Cor. 4.19]{V2} since $\cF_2$ is a homotopy invariant Zariski sheaf with transfers. Thus we may replace $Y$ by $F$.

Moreover, since $\cF_2$ is an \'etale sheaf,
we have an injection $\cF_2(F)\into \cF_2(F_s)$, where $F_s$ is a separable closure of $F$. By a transfer argument, $\ker(\cF_2(F_s)\to \cF_2(\bar F))$ is $p$-primary torsion if  $\bar F$ is an algebraic closure of $F$,  hence $0$ since $\cF_2$ is a sheaf  of $\Z[1/p]$-modules. So we we may even replace $F_s$ by $\bar F$. Thus, we may even
replace $Y$ by $\bar F$.\footnote{Note that $\Spec \bar F$ is a pro-object of $\Sm(k)$: since $k$ is perfect, any regular $k$-scheme of finite type is smooth.}

Then the group $c(Y,X)$ is replaced by $c(\bar F,X) = Z_0(X_{\bar F})$. Since $\bar F$ is
algebraically closed, all closed points of $X_{\bar F}$ are rational, hence all finite
correspondences from $\Spec \bar F$ to $X$ are linear combinations of morphisms. Therefore the diagram commutes on them.
\end{proof}

\subsection{$1$-motivic sheaves and $\DM$}  The introduction of $\Shv_1$ is now made clear by the following 

\begin{thm} \label{t3.2.3} a) The embedding of Proposition \ref{ptransf} sends $\Shv_1$ into $d_{\le 1}\DM_{\gm,\et}^{\eff}\subset \DM_{-,\et}^\eff$.\\
b) Let
$M\in d\1 \DM_{\gm,\et}^{\eff}$. Then for all $i\in\Z$, $\sH^i(M)\in \Shv_1$, where $\sH^i$ is computed with respect to the homotopy $t$-structure (see Definition \ref{d1.7}). \index{$\sH^n$, $\sH_n$}
In particular, the latter induces a $t$-structure on $d\1 \DM_{\gm,\et}^{\eff}$, with heart $\Shv_1$.
\end{thm}

\begin{proof} a) is proven as in \S \ref{s2.4.1}. b)  By the thickness of $\Shv_1$ in $\HI_\et^s$
(Proposition
\ref{ptransf}), we reduce to the case
$M=M_\et(C)$, $C\by{p}\Spec k$ a smooth projective curve. By Proposition \ref{lcurve}, the
cohomology sheaves of $M_\et(C)$ belong to $\Shv_1$: for $\cH^{-1}$ this is clear and for $\cH^0$ it
is a (trivial) special case of Proposition \ref{p3.3.1}.
\end{proof}

Note that the functor $\M[1/p]\to {\HI_\et^s}^{[0,1]}$ of \eqref{eq2.2} refines to a functor 
\[\M[1/p]\allowbreak\to {\Shv_1}^{[0,1]}.\]

Hence, using Lemma \ref{ltot} as in \S \ref{orgo}, we get a composed triangulated functor
\begin{equation}\label{hts}
\tot:D^b(\M[1/p])\to D^b({\Shv_1}^{[0,1]})\to D^b(\Shv_1)
\end{equation}\index{$\tot$}
refining the one from Lemma \ref{l2.2.1} (same proof). We then have:

\begin{cor} \label{c3.3.2} The two functors
\[\begin{CD}
D^b(\M[1/p])@>\tot>>D^b(\Shv_1)\to d\1\DM_{\gm,\et}^\eff
\end{CD}\]
are equivalences of categories. \end{cor}

\begin{proof} For the composition, this is Theorem \ref{t1.2.1}. This implies that the second functor is full and essentially surjective, and to conclude, it suffices by Lemma \ref{lA.2} to see that it is conservative. But this follows immediately from Proposition \ref{ptransf} and Theorem \ref{t3.2.3}.
\end{proof}

\begin{defn}\label{dhts} We call the $t$-structure defined on $D^b(\M[1/p])$ or on
$d\1\DM_{\gm,\et}^\eff$ by Corollary \ref{c3.3.2} the \emph{homotopy $t$-structure}.
\end{defn}

\begin{remark}[\cf \S \protect{\ref{alessandra1}}]\label{alessandra2}  In
\cite{alessandra}, A. Bertapelle defines  a variant  $\Shv_1^\prime$ of
the category $\Shv_1^{\rm fppf}$ from Remark \ref{fppf} allowing non-reduced finite commutative
group schemes and constructs an equivalence of categories 
\[D^b(\M)\simeq D^b(\Shv_1^\prime)\]
without inverting $p$ (not going via $\DM$). Hence the homotopy $t$-structure of Definition \ref{dhts} exists integrally even  over a perfect field of positive characteristic.
\end{remark}

\subsection{Comparing $t$-structures}\label{3.6} In this subsection, we want
to compare the homotopy $t$-structure of Definition \ref{dhts} with the motivic $t$-structure of Theorem \ref{ptors} a).

Let $C\in D^b(\M[1/p])$. Recall (from \ref{not}) the notation ${}^tH_n(C)\in {}^t\M[1/p]$ for its homology relative to the torsion $1$-motivic
$t$-structure from Theorem \ref{ptors}. Recall (from \ref{d1.7}) that we also write $\sH^n(C)\allowbreak\in
\Shv_1$ for its cohomology objects  relative to the homotopy $t$-structure. \index{$\sH^n$, $\sH_n$} \index{${}^tH^n$, ${}_tH^n$}

Consider the functor $\tot$ of \eqref{hts}. Let $\cF$ be a $1$-motivic sheaf and
$(G,b)$ its associated normalised pair (see Proposition \ref{p3.1} a)). Let $L=\ker
b$ and $E=\coker b$. In $D^b(\M[1/p])$, we have an exact triangle
\[[L\to G][1]\to \tot^{-1}(\cF)\to [E\to 0]\by{+1}\]
(see Corollary \ref{c3.3.2}). This shows:

\begin{lemma} \label{l4.3.1} We have
\begin{align*}
{}^tH_0(\tot^{-1}(\cF))&=[E\to 0]\\
{}^tH_1(\tot^{-1}(\cF))&=[L\to G]\\
{}^tH_q(\tot^{-1}(\cF))&=0 \text{ for } q\ne 0,1.\qed
\end{align*}
\end{lemma}

On the other hand, given a 1-motive (with torsion or cotorsion) $M=[L\by{f} G]$, we clearly have
\begin{align}
\sH^0(M)&=\ker f\notag\\
\sH^{1}(M)&=\coker f\label{eq3.9}\\
\sH^q(M)&=0 \text{ for } q\ne 0,1.\notag
\end{align}
by considering it as a complex of length $1$ of $1$-motivic sheaves.

In particular, ${}^t\M[1/p]\cap \Shv_1= \Shv_0$,
${}^t\M[1/p]\cap \Shv_1[-1]$ consists of quotients of semi-abelian
varieties by discrete subsheaves and ${}^t\M[1/p]\cap \Shv_1[q]\allowbreak=0$
for $q\ne 0,-1$.

Here is a more useful result relating $\sH^i$ with the two motivic $t$-structures:

\begin{propose}\label{p3.10} Let $C\in D^b(\M[1/p])$; write $[L_i\by{u_i} G_i]$ for ${}_tH_i(C)$ and $[L^i\by{u^i} G^i]$ for ${}^tH^i(C)$\footnote{Note that $(L_i,G_i)$ and $(L^i,G^i)$ are determined only up to the relevant \qi's.}. Then we have exact sequences in $\Shv_1$:
\begin{gather*}
\dots\to L_{i+1}\by{u_{i+1}} G_{i+1}\to \sH_i(C)\to L_{i}\by{u_{i}} G_{i}\to\dots\\
\dots\to L^{i-1}\by{u^{i-1}} G^{i-1}\to \sH^i(C)\to L^{i}\by{u^{i}} G^{i}\to\dots
\end{gather*}
\end{propose}

\begin{proof} For the first one, argue by induction on the length of $C$ with respect to the motivic $t$-structure with heart ${}_t\M[1/p]$ (the case of length $0$ is \eqref{eq3.9}). For the second one, same argument with the other motivic $t$-structure.
\end{proof}

Note finally that the homotopy $t$-structure is far from being invariant under Cartier duality:
this can easily be seen by using Proposition \ref{p3.4.6}.

\subsection{Global $\Ext^i$ with transfers}\label{gtr}
 For the sake of notation (see  also \S \ref{1.5}) we write $\Ext^i_\tr (\cF_1 , \cF_2)$ for $\Ext^i_{\EST}(\cF_1 , \cF_2)$ and sheaves $\cF_1 , \cF_2\in \EST$. 

\begin{lemma}\label{l3.13} The group $\Ext_\tr^i(\cF_1, \cF_2)$ is torsion for any $i\geq 2$ and any $\cF_1 \in \Shv_1$, $\cF_2\in \HI^s_\et$.
\end{lemma}

\begin{proof} By a transfer argument, we reduce to $k$ algebraically closed. Given the structure of $\cF_1$, we reduce by d\'evissage to three basic cases: $\cF_1\in \Shv_0$, $\cF_1$ a torus and $\cF_1$ an abelian variety. The first case further reduces to (i) $\cF_1=\Z$, the second one to (ii) $\cF_1=\G_m$ and the third to (iii) $\cF_1=J(C)$ for $C$ a smooth projective curve. In case (i), we find $\Ext_\tr^i(\cF_1, \cF_2)=H^i_\et(k,\cF_2)$ which is $0$ for $i>0$. In case (ii), $\G_m$ is a direct summand of $M_\et(\G_m)$ in $\DM_{-,\et}^\eff$ thus $\Ext_\tr^i(\cF_1, \cF_2)$ is a direct summand of $H^i_\et(\G_m,\cF_2)$ which is torsion for $i>1$. In case (iii), $\Ext_\tr^i(\cF_1, \cF_2)$ is similarly a direct summand of $H^i_\et(C,\cF_2)$ \cite[Th. 3.4.2]{V}, which is again torsion for $i>1$. Here we use the fact that the \'etale cohomological dimension of a Noetherian scheme $X$  for sheaves of $\Q$-vector spaces is $\le \dim (X)$: for this, one may reduce to the case of fields as in  \cite[Exp. X, proof of Th. 4.1]{sga4}. \end{proof}

\begin{remark} This simple proof was suggested independently by the referee and Joseph Ayoub, that we wish to thank here. Let us mention two other proofs, which only work for $\cF_2\in \Shv_1$ (this is the only case we shall need):
\begin{itemize}
\item Adapt Breen's technique to the framework of \'etale sheaves with transfers (this was our original proof).
\item Play around with the exact triangles of \S \ref{3.6}, using mainly Proposition \ref{iso1}.
\end{itemize}
\end{remark}

Here is a refinement of Lemma \ref{l3.13}:

\begin{lemma}\label{l3.12} Suppose $k=\bar k$. Then, we have  $\Ext_\tr^i(\cF_1, \cF_2)=0$ for $\cF_1 \in \Shv_1$, $\cF_2\in \HI^s_\et$ and any $i> 2$   (remember that we work with $p$ inverted); moreover, $\Ext_\tr^2(\cF_1, \cF_2)$ is divisible. 
\end{lemma} 

\begin{proof}  Note first that $\Ext^i_\tr(\cF_1,\cF_2)$ is a $\Z[1/p]$-module for all $i$, since $\cF_1,\cF_2\in \HI^s_\et$.

Assume first that  $\cF_1\in \Shv_0$. A standard d\'evissage  reduces us to the following basic cases for: (i) $\cF_1 =\Z$ or (ii) $\cF_1=\Z/m$. For $\cF_1=\Z$, we have $\Ext^i_\tr(\cF_1,\cF_2)=H^i_\et(k,\cF_2)=0$ for $i>0$. For $\cF_1=\Z/m$, the exact sequence
\[\Ext_\tr^{i-1}(\Z, \cF_2)\to \Ext_\tr^i(\Z/m, \cF_2)\to \Ext_\tr^i(\Z, \cF_2)\by{m} \Ext_\tr^i(\Z, \cF_2)\]
and (i) gives the claimed vanishing in the case (ii) for $i>1$. 

In general, by considering a normalized morphism $b_1:\uG_1\to \cF_1$ with $L_1=\ker b_1$ free, using (i) and (ii) above we get
\[\Ext_\tr^i(\cF_1, \cF_2)\iso \Ext_\tr^i(\uG_1/L_1, \cF_2)\iso \Ext_\tr^i(\uG_1, \cF_2)\]
for $i>1$. The same argument as in the proof of Proposition \ref{ptorsion}, using exact sequences analogous to \eqref{breen10} then  yields, using Lemma \ref{l3.13}, the claimed vanishing for $i>2$. 
The divisibility of this group for $i=2$ is clear since $\Ext_\tr^2({}_m\uG_1, \cF_2)=0$ for any $m$ prime to $p$.
\end{proof}

 \begin{remark} For an example where $\Ext_\tr^2(\cF_1, \cF_2)\ne 0$, we may take $(\cF_1,\cF_2)=(\G_m,\Z)$ (compare Remark \ref{r3.2} (2)).
\end{remark}

\subsection{Local $\Ext^i$ with transfers}\label{ltr}

We now get back to the case of an arbitrary perfect field $k$. Recall that the category $\EST$ has enough injectives \cite[6.19]{VL} so that we can define $\sext^i_\tr(\cF,\cG)\in \EST$ as the derived functors of $\shom_\tr(\cF,\cG)$ (compare \S \ref{1.5}). 

Recall that the derived tensor product $\oo^L$ of $D^-(\EST)$ \cite[Prop. 8.8]{VL} has a right adjoint $\srhom(\cF,\cG)\in D^+(\EST)$ defined at least for $\cF$ representable, and computable from $\shom$ by injective resolutions \cite[Rk. 8.21]{VL}; this immediately extends to any $\cF$ by using Voevodsky's canonical resolutions. In particular,
\[\cH^i(\srhom(\cF,\cG))=\sext^i_\tr(\cF,\cG)\]
with $\sext^i_\tr$ as above.

\begin{lemma} If $\cF_1,\cF_2\in \HI_\et^s$, we have $\sext^i_\tr(\cF_1,\cF_2)\in \HI_\et^s$ for all $i\ge 0$.
\end{lemma}

\begin{proof} Indeed, the above $\srhom$ restricts to the partial internal Hom of $\DM_{-,\et}^\eff$ already considered in \S \ref{2.5} (compare \cite[Rk. 14.12]{VL}). 
\end{proof}

Let $K/k$ be an algebraically closed extension. In the rest of this subsection, we shall work with \'etale sheaves with transfers over $k$ and $K$, so we exceptionally specify this in the notation $\EST(k)$ and $\EST(K)$. The ``direct image'' functor
\begin{align*}
&\EST(K)\to \EST(k)\\
&\cF\mapsto \cF_{|k}\\ 
&\cF_{|k}(X)\df \cF(X_K)\hspace*{0.7cm}  X\in \Sm(k)
\end{align*}
has an exact left adjoint $\cF\mapsto \cF_K$. 

\begin{lemma}\label{simpler} For $\cF\in \EST(k)$ and $X\in \Sm(K)$, we have an isomorphism
\begin{equation}\label{eq3.14.3}
\colim_{X\to \cX}\cF(\cX)\iso \cF_K(X)
\end{equation}
where the colimit is taken over the filtering system of $k$-morphisms
$X\to\cX$ for $\cX\in \Sm(k)$.
\end{lemma}

\begin{proof} Let $X\to\cX$ be as in the lemma, whence a $K$-morphism $X\to \cX_K$. The unit morphism $\cF\to (\cF_K)_{|k}$ yields a map
\[\cF(X)\to \cF_K(\cX_K)\to \cF_K(X)\]
hence a map \eqref{eq3.14.3} in the limit. To show that it is an isomorphism, we may reduce to representable sheaves $L(Y)$, and then it reduces to the tautological formula
\[L(Y)_K=L(Y_K).\]
\end{proof}

Ideally we would like a computation of $\sext^i(\cF,\cG)(K)$ similar to the classical one for stalks on small \'etale sites \cite[Ch. III, Ex. 1.31 (b)]{MI}. We don't know how to do this, mainly because we don't know if $\cI_K$ is injective when $\cI$ is. Instead, we have the following

\begin{propose}\label{l:comm} For $\cF,\cG\in \HI_\et^s(k)$ with $\cF\in \DM_{\gm,\et}^\eff(k)$. Then  there are  isomorphisms, natural in $\cF,\cG,K$:
\[\Ext^i_\tr(\cF_K,\cG_K)\simeq \sext^i_\tr(\cF,\cG)(K). \]
\end{propose}

\begin{proof}  By the above, we have
\begin{equation}\label{eq3.14.1}
\sext^i_\tr(\cF,\cG)=\cH^i(\ihom_\et (\cF,\cG[i])).
\end{equation}

Let $U\in \Sm(k)$ and let $\phi:\Spec K\to  U$ be a $k$-morphism. It induces a morphism in $\DM_{-,\et}^\eff(K)$
\[\phi_*:M_\et(\Spec K)\to M_\et(U_K)=M_\et(U)_K.\]

(Note that $\cF\mapsto \cF_K$ extends to triangulated functors on the derived categories.)

For $\cG\in \HI_\et^s(k)$,  we get a composition
\[\ihom_\et (M_\et(U),\cG)_K\to \ihom_\et (M_\et(U)_K,\cG_K)\longby{\phi^*} \cG_K  \]
hence by adjunction
\[\ihom_\et (M_\et(U),\cG)\to  R_{|k}(\cG_K)  \]
where $R_{|k}$ is the total derived functor of $\cF\mapsto \cF_{|k}$\footnote{Note that in fact $R_{|k}(\cG_K)=(\cG_K)_{|k}$ since $K$ has cohomological dimension $0$, which implies in particular that $R_{|k}(\cG_K)\in \DM_{-,\et}^\eff(k)$.}. Applying $\Hom(M,-[i])$ to this morphism for $M\in  \DM_{-,\et}^\eff(k)$, we get
\[\Hom(M,\ihom_\et (M_\et(U),\cG)[i])\to  \Hom(M,R_{|k}(\cG_K)[i]).\] 

The right hand side may be rewritten as
\[\Hom(M,R_{|k}(\cG_K)[i])\simeq \Hom(M_K,\cG_K[i])\]
while the left hand side may be rewritten as
\begin{multline*}
\Hom(M,\ihom_\et (M_\et(U),\cG)[i])\simeq \Hom(M\oo^L M_\et(U),\cG[i])\\ \simeq \Hom(M_\et(U),\ihom_\et (M,\cG)[i])\simeq \HH^i_\et(U,\ihom_\et (M,\cG)).
\end{multline*}

Here we used the \'etale analogue of \cite[Prop. 3.2.8]{V}, which was already used in the proof of Proposition \ref{lcurve}: note that there is no problem of convergence for the hypercohomology spectral sequences, since $\srhom(\cF,\cG)$ is bounded below.

Thus we get compatible morphisms 
\[\HH^i_\et(U,\ihom_\et (M,\cG))\longby{\phi^*} \Hom_\et (M_K,\cG_K[i])\]
and, passing to the limit, a morphism
\begin{equation}\label{eq3.14}
 \HH^i_\et(K,\ihom_\et (M,\cG)) \to \Hom(M_K,\cG_K[i]).
\end{equation}
Here we used the fact (see \cite[III.1.16]{MI}) that \'etale cohomology commutes with filtered limits of schemes (with affine transition morphisms).

In view of \eqref{eq3.14.1}, the proposition will follow from the slightly more general statement that \eqref{eq3.14} is an isomorphism for any $M\in \DM_{\gm,\et}^\eff(k)$. Since this assertion  is stable under cones, we reduce to $M=M_\et(X)$ for some $X\in\Sm(k)$.

In this case, the right hand side of \eqref{eq3.14} is  $H^i_\et(X_K,\cG_K)$, while the left hand side is 
\begin{multline*}
\colim_U \HH^i_\et(U,\ihom_\et (M_\et(X),\cG))\simeq\colim_U \Hom(M_\et(U),\ihom_\et (M_\et(X),\cG))\\
\simeq\colim_U \Hom(M_\et(U\times X),\cG)\simeq\colim_U H^i_\et(U\times X,\cG)\simeq H^i_\et(X_K,\cG)
\end{multline*}
and the conclusion follows from Lemma \ref{simpler}.
\end{proof}

\begin{cor}\label{l3inv} For $\cF\in \Shv_1$ and $\cG\in \HI^s$, we have $\sext^i_\tr(\cF,\cG)=0$ for $i> 2$.
\end{cor}

\begin{proof} We apply  Gersten's principle (Proposition \ref{pgersten} c)) to  $\cE=\sext^i_\tr(\cF,\cG)$. This says that $\cE=0$ if and only if  $\cE(K_s)=0$ for any separable closure $K_s$ of the function field $K$ of a smooth $K$-variety.  We  may even replace $K_s$ by its algebraic closure $\bar K$ (see proof of Lemma \ref{l3.8}). Since $\cF\in \DM_{\gm,\et}^\eff$ (Theorem \ref{t3.2.3} a)),  Proposition \ref{l:comm} yields $\Ext^i_\tr(\cF_{\bar K},\cG_{\bar K})\simeq \sext^i_\tr(\cF,\cG)(\bar K)$, and the left hand side is $0$ for $i> 2$ by   Lemma \ref{l3.12}.   
\end{proof}

\subsection{$\Ext^n$ with and without transfers}\label{3.11}  Here we set $\Ext^n=\Ext^n_\ES$ and $\Ext^n_\tr=\Ext^n_\EST$ and keep the notation from the previous \S \ref{ltr} for local $\sext$.

 \begin{propose}\label{p3.10.1} Let $\cF_1,\cF_2\in \Shv_1$. Then the natural map
\[\Ext^n_\tr(\cF_1,\cF_2)\to \Ext^n(\cF_1,\cF_2)\]
is bijective for $n=0,1$. Here we implicitly used the full embedding $\Shv_1\into \HI^s_{\et}\into \EST$ from Proposition \ref{ptransf}. 
\end{propose}

\begin{proof} The case $n=0$ follows from Lemma \ref{l3.8}. Let us do $n=1$. Injectivity: let $0\to \cF_2\to \cF\to \cF_1\to 0$ be an extension in $\EST$ which becomes split in $\ES$. Let $f:\cF_1\to \cF$ be a section of the projection in $\ES$. By the case $i=0$, $f$ is a morphism in $\EST$, hence $\cF$ is split in $\EST$. Surjectivity: let $0\to \cF_2\to \cF\to \cF_1\to 0$ be an extension in $\ES$. By Theorem \ref{text}, $\cF\in \Shv_1$.
\end{proof}

\begin{thm}\label{t3.12} Let $\cF_1,\cF_2\in \Shv_1$. Then\\
a) The natural homomorphism of sheaves
\[f^n:\omega\sext^n_\tr(\cF_1,\cF_2)\to \sext^n(\cF_1,\cF_2)\]
is an isomorphism for all $n\ge 0$, where $\omega:\EST \to \ES$ is the (exact) forgetful functor. These sheaves are $0$ for $n> 2$.\\
b) We have $\sext^n_\tr(\cF_1,\cF_2)\in \Shv_1$ for  $n = 0, 1$, and $\sext^2_\tr(\cF_1,\cF_2)$ is a divisible torsion ind-$0$-motivic sheaf.\\
c) The natural homomorphism of abelian groups
\[f^n:\Ext^n_\tr(\cF_1,\cF_2)\to \Ext^n(\cF_1,\cF_2)\]
is an isomorphism for all $n\ge 0$.\\
\end{thm}

\begin{proof} Since c) follows from a) by the local to global spectral sequences, we are left to prove a) and b). The assertions are local for the \'etale topology, so we reduce by the same d\'evissage as in the proof of Lemma \ref{l3.13} to the basic cases $\cF_i=\Z,\G_m$ or an abelian variety $A$. By the same technique as in \ref{s2.4.1}, we may further reduce to the case where $A$ is of the form $J(C)$ for $C$ a smooth projective curve with a rational point.

If   $n> 2$, the right hand side in a) is $0$ by Proposition \ref{ptorsion} b) and  the same d\'evissage as in the proof of Lemma \ref{l3.12}. The left hand side is also $0$ by Corollary \ref{l3inv}. 

Suppose now  $n\le 2$.  We may argue as for   $n> 2$ whenever we know that $\sext^n(\cF_1,\cF_2)=0$ and $\Ext^n_\tr((\cF_1)_K,(\cF_2)_K)=0$ for any algebraically closed extension $K$ of $k$. By Proposition \ref{p3.10.1} and Lemma \ref{c3.2}, this is the case except when:
\begin{description}
\item[$n=0$] $\cF_1=\Z$, $(\cF_1,\cF_2)=(\G_m,\G_m)$, $(\cF_1,\cF_2)=(A,B)$, $A,B$ abelian varieties. 
\item[$n=1$] $(\cF_1,\cF_2)=(A,\G_m)$, $A$ an abelian variety.
\item[$n=2$] $(\cF_1,\cF_2)=(\G_m,\Z)$ or $(A,\Z)$, $A$ an abelian variety (see last statement of Lemma \ref{l3.12}).
\end{description}

When $n=0$ and $\cF_1=\Z$, $f^n$ is the identity map $\cF_2\to \cF_2$. If $\cF_1=\G_m$ or $J(C)$ and $n\le 1$, we may write both sides in a) as direct summands of $R^n_\et\pi_* \cF_2$ for $\pi:X\to \Spec k$ with $X=\Aff^1-\{0\}$ or $C$, and the isomorphism is clear.  Finally, when $n=2$, for $G=\G_m$ or $A$, in the commutative diagram
\[\begin{CD}
\omega\sext^1_\tr({}_m\uG,\Z)@>f^1>> \sext^1({}_m\uG,\Z)\\
@V VV @V VV\\
{}_m\omega\sext^2_\tr(\uG,\Z)@>f^2>> {}_m\sext^2(\uG,\Z)
\end{CD}
\]
$f^1$ is an isomorphism because ${}_m\uG$ is locally constant and the vertical maps are isomorphisms because $\omega\sext^1_\tr(\uG,\Z)=\sext^1(\uG,\Z)=0$  (as noted for the case $n=1$ above), hence $f^2$ is an isomorphism. This completes the proof of a). 
b) is obvious for $n> 2$ and is proven for  $n=0,1$ by the same d\'evissage as above;  for $n=2$ it follows from the last statement of Lemma \ref{l3.12}.
\end{proof}

\begin{remark} Using the category $\DA_\et(k)$ of \cite{real.etale}, one can probably extend Theorem \ref{t3.12} a) and c) to all $\cF_2\in \HI^s_\et$, with a more reasonable proof (compare Lemma \ref{l3.12}). This would extend Proposition \ref{ext=ext} below to all $C_3\in \DM_{\gm,\et}^\eff$, but  would take us too far here.
\end{remark}

Now consider $\ihom_{\et}$ the partial internal Hom of $\DM_{-,\et}^\eff$ (following the notation adopted in \S \ref{2.5}). We obtain:

\begin{cor}\label{c3.13} Let $C_1,C_2\in d_{\le 1}\DM_{\gm,\et}^\eff\otimes \Q$. Then $\ihom_{\et}(C_1,C_2)\allowbreak\in d_{\le 1}\DM_{\gm,\et}^\eff\otimes \Q$.
\end{cor}

\begin{proof} This follows from Theorem \ref{t3.12} by d\'evissage, using Theorem \ref{t3.2.3}. 
\end{proof}

\begin{remark} \label{rkc3.13}
 A version of Corollary \ref{c3.13} remains true integrally: $\ihom_{\et}(C_1,C_2)\in d_{\le 1}\DM_{-,\et}^\eff$ where $d_{\le 1}\DM_{-,\et}^\eff$ is the localising subcategory of $\DM_{-,\et}^\eff$ generated by $d_{\le 1}\DM_{\gm,\et}^\eff$. This follows by d\'evissage from Theorem \ref{t3.12}. On the other hand, the example of $\shom_\tr(\Q/\Z,\Q/\Z)$ shows that the situation becomes unpleasant if one allows $C_1$ and $C_2$ to run through $d_{\le 1}\DM_{-,\et}^\eff$. This has some similarity with Remark \ref{noleft} 3) below.
\end{remark} 

\subsection{$t$-exactness}\label{s3.14} Theorem \ref{t1.2.1} and Corollary \ref{c3.13} provide the category $D^b(\M\otimes \Q)$ with an internal Hom that we denote by $\ihom_1$.
\index{$\ihom_1$} It is by construction left exact with respect to the homotopy $t$-structure of Definition \ref{dhts}. We now show:

\begin{thm}\label{t3.14} The bifunctor $\ihom_1$ is  $t$-exact with respect to the canonical $t$-structure of $D^b(\M\otimes \Q)$.
\end{thm}

\begin{proof} By d\'evissage, it suffices to check that $\ihom_1(M,N)$ is $t$-con\-centr\-at\-ed in degrees $0$ if $M,N$ are $1$-motives of pure weight.

If $L$ is discrete, write $L$ for $[L\to 0]$, and if $G$ is semi-abelian,  write $G[-1]$ for $[0\to G]$. Also write $A$ for an abelian variety and $T$ for a torus and $\shom$, $\sext$ the \'etale sheaves Hom and Ext. Then
\begin{align*}
{}^tH^i(\ihom_1(L,L')) &= 
\begin{cases}
\shom(L,L') &\text{for $i= 0$}\\
 0& \text{else.}
 \end{cases}\\
{}^tH^i(\ihom_1(L,G'[-1])) &= 
\begin{cases}
\shom(L
,G')[-1] &\text{for $i=0$}\\
0 & \text{else.}
\end{cases}\\
{}^tH^i(\ihom_1(G[-1],L')) &= 0
\\
{}^tH^i(\ihom_1(T[-1],T'[-1])) &= 
\begin{cases}
\shom(T,T')& \text{if $i= 0$}\\
0& \text{else.}
\end{cases}\\
\ihom_1(T[-1],A'[-1]) &= 0\\
{}^tH^i(\ihom_1(A[-1],T'[-1])) &= 
\begin{cases}
\sext(A,T')[-1]& \text{if $i= 0$}\\
0&\text{else.}
\end{cases}\\
{}^tH^i(\ihom_1(A[-1],A'[-1])) &= 
\begin{cases}
\shom(A,A')& \text{if $i=0$}\\
0& \text{else.}
\end{cases}
\end{align*}

In this display, we use for example that $\sext(A,T')$ has the structure of an abelian variety when $A$ is an abelian variety and $T'$ a torus, and that $\shom(A,A')$ is a lattice when $A,A'$ are two abelian varieties. This completes the proof.
\end{proof}

\section{Comparing two dualities}\label{dual}

In this section, we show that the classical Cartier duality for $1$-motives is compatible with a
``motivic Cartier duality" on triangulated motives, described in Definition \ref{cdef} below. 

\subsection{Biextensions of $1$-motives}\label{sbiext}This material is presumably well-known
to  experts, and the only reason why we write it up is that we could not find
it in the literature. Exceptionally, we put $1$-motives in degrees $-1$ and $0$ in this subsection and in the next one, for compatibility with Deligne's conventions in \cite{D}.

Recall (see \cite[\S 10.2]{D}) that for $M_1= [L_1\by{u_1} G_1]$ and 
$M_2= [L_2\by{u_2} G_2]$ two complexes of abelian sheaves over some site $\cS$, concentrated in
degrees
$-1$ and
$0$, a \emph{biextension} of
$M_1$ and
$M_2$ by an abelian sheaf $H$ is given by a (Grothendieck) biextension $P$
of $G_1$ and $G_2$ by $H$ and a pair of compatible trivializations of the
biextensions of
$L_1\times G_2$ and $G_1\times L_2$ obtained by pullbacks.
Let $\Biext (M_1,M_2;H)$ denote the group of isomorphism  classes
of biextensions. We have the following fundamental formula (see \cite[\S
10.2.1]{D}):\index{$\Biext$}
\begin{equation}\label{biextform}
\Biext (M_1,M_2;H) = \EExt^1_\cS(M_1\oo^L M_2,H).
\end{equation}

\begin{sloppypar}
Suppose now that $M_1$ and $M_2$ are two Deligne $1$-motives. Since $G_1$ and $G_2$ are smooth,
we may compute biextensions by using the \'etale topology. Hence, we shall take here
\[\cS=\Sm(k)_\et.\]

Let $M_2^*$ denote the Cartier dual of $M_2$ as constructed by Deligne
(see \S \ref{1.8},  \cf \cite[\S 10.2.11]{D} and \cite[\S 0]{BSAP}) along with the
Poincar\'e biextension $P_{M_2}\in \Biext (M_2,M_2^*;\G_m)$. We also
have  the transpose ${}^tP_{M_2}=P_{M_2^*}\in \Biext (M_2^*,M_2;\G_m)$.
Pulling back ${}^tP_{M_2}$ yields a map
\begin{align}
\gamma_{M_1,M_2} : \Hom (M_1, M_2^*)& \to \Biext (M_1,M_2;\G_m)\label{eqbiext}\\
\phi&\mapsto(\phi\times 1_{M_1})^*({}^tP_{M_2})\notag
\end{align}
which is clearly additive and natural in $M_1$.
\end{sloppypar}

\begin{propose} \label{biext} The map $\gamma_{M_1,M_2}$ yields an isomorphism of functors from $1$-motives to abelian groups, \ie
the functor 
$$M_1\mapsto \Biext (M_1,M_2;\G_m)$$
on $1$-motives is representable by the Cartier dual $M_2^*$. Moreover,
$\gamma_{M_1,M_2}$ is also natural in $M_2$. 
\end{propose}

\begin{proof}  We start with a few lemmas:

\begin{lemma}\label{ld0} For $q\le 0$, we have
\[\Hom_{\bar k}(M_1\oo^L M_2,\G_m[q])=0.\]
\end{lemma}

\begin{proof} For $q<0$ this is trivial and for $q=0$ this is \cite[Lemma 10.2.2.1]{D}.
\end{proof}

\begin{lemma} \label{kalg} Let $\bar k$ be an algebraic closure of $k$
and $G=Gal(\bar k/k)$. Then
\begin{align*}
\Hom_k (M_1, M_2^*)&\iso \Hom_{\bar k} (M_1, M_2^*)^G\\
\Biext_k (M_1,M_2;\G_m)&\iso \Biext_{\bar k} (M_1,M_2;\G_m)^G.
\end{align*}
\end{lemma}

\begin{sloppypar}
\begin{proof} The first isomorphism is obvious. For the second, thanks
to \eqref{biextform} we may use the spectral sequence
\[H^p(G,\Hom_{\bar k}(M_1\oo^L M_2,\G_m[q]))\Rightarrow \Hom_k(M_1\oo^L
M_2,\G_m[p+q]).\]

(This is the only place in the proof of Proposition \ref{biext} where we
shall use \eqref{biextform}.) The assertion then follows from Lemma \ref{ld0}. \end{proof}
\end{sloppypar}

Lemma \ref{kalg}, reduces the proof of Proposition \ref{biext} to the
case where \emph{$k$ is algebraically closed}, which we now assume. The following is a special
case of this proposition:

\begin{lemma} \label{biexta} The map $\gamma_{M_1,M_2}$ is an
isomorphism when $M_1$ and $M_2$ are abelian varieties $A_1$ and $A_2$,
and is natural in $A_2$.
\end{lemma}

Again this is certainly well-known and explicitly mentioned as such in
\cite[VII, p. 176, (2.9.6.2)]{sga7}. Unfortunately we have not been able to
find a proof in the literature, so we provide one for the reader's
convenience.

\begin{proof}  We shall use the universal property of the Poincar\'e
bundle \cite[Th. p. 125]{mumford}. Let $P\in \Biext(A_1,A_2)$. Then 
\begin{enumerate}
\item $P_{|A_1\times\{0\}}$ is trivial;
\item $P_{|\{a\}\times A_2}\in \Pic^0(A_2)$ for all $a\in A_1(k)$.
\end{enumerate}

Indeed, (1) follows from the multiplicativity of $P$ on the $A_2$-side.
For (2) we offer two proofs (note that they use multiplicativity on
different sides): 

\begin{sloppypar}
\begin{itemize}
\item By multiplicativity on the $A_1$-side, $a\mapsto P_{|\{a\}\times
A_2}$ gives a homomorphism $A_1(k)\to \Pic(A_2)$. Composing with the
projection to $\NS(A_2)$ gives a homomorphism from a divisible group to a
finitely generated group, which must be trivial.
\item (More direct but more confusing): we have to prove that
$T_b^*P_{|\{a\}\times A_2}=P_{|\{a\}\times A_2}$ for all $b\in A_2(k)$.
Using simply $a$ to denote the section $\Spec k\to A_1$ defined by $a$,
we have a commutative diagram
\[\begin{CD}
A_2@>a\times 1_{A_2}>> A_1\times A_2\\
@V{T_b}VV @VV{1_{A_1}\times T_b}V \\
A_2@>a\times 1_{A_2}>> A_1\times A_2.
\end{CD}\]

Let $\pi_1:A_1\to\Spec k$ and $\pi_2:A_2\to\Spec k$ be the two
structural maps. Then by multiplicativity on the $A_2$-side, an easy
computation gives
\[(1_{A_1}\times T_b)^*P=P\otimes \left(1_{A_1}\times (\pi_2\circ
b)\right)^*P.\]

Applying $(a\times 1_{A_2})^*$ to this gives the result since $ (a\times
1_{A_2})^*\circ\left(1_{A_1}\times (\pi_2\circ b)\right)^*P=\pi_{A_2}^*
P_{a,b}$ is trivial.
\end{itemize}
\end{sloppypar}

By the universal property of the Poincar\'e bundle, there exists a
unique morphism\footnote{For convenience we denote here by $A'$ the dual of an
abelian variety $A$ and by $f'$ the dual of a homomorphism $f$ of abelian
varieties.} $f:A_1\to A_2'$ such that $P\simeq (f\times
1_{A_2})^*({}^tP_{A_2})$. It remains to see that $f$ is a homomorphism:
for this it suffices to show that $f(0)=0$. But
\begin{multline*}
\sO_{A_2}\simeq P_{|\{0\}\times A_2}=(0\times 1_{A_2})^*\circ (f\times
1_{A_2})^*({}^t P_{A_2})\\ =(f(0)\times 1_{A_2})^*({}^t
P_{A_2})=(P_{A_2})_{|A_2\times\{f(0)\}}=f(0)
\end{multline*}
where the first isomorphism holds by multiplicativity of $P$ on the
$A_1$-side.

Finally, the naturality in $A_2$ reduces to the fact that, if $f:A_1\to A_2'$, then $(f\times
1_{A_2})^*({}^tP_{A_2})\simeq (1_{A_1}\times f')^* (P_{A_1})$. This follows from the
description of
$f'$ on $k$-points as the pull-back by $f$ of line bundles.
\end{proof}

We also have the following easier

 \begin{lemma}\label{kalg1} Let $L$ be a lattice and $A$ an abelian
variety. Then the natural map
\begin{align*}
\Hom(L,A')&\to \Biext(L[0],A[0];\G_m)\\
f&\mapsto (1\times f)^*({}^tP_A)
\end{align*}
is bjiective.
\end{lemma}

\begin{proof} Reduce to $L=\Z$; then the right hand side can be identified with $\Ext(A,\G_m)$
and the claim comes from the Weil-Barsotti formula.
\end{proof}

Let us now come back to our two $1$-motives $M_1,M_2$. We denote by $L_i,
T_i$ and
$A_i$ the discrete, toric and abelian parts of $M_i$ for $i= 1,2$. Let us
further denote by 
$u'_i:L_i'\to A_i'$ the map corresponding to $G_i$ under the isomorphism
$\Ext (A_i,T_i) \simeq\Hom (L_i',A_i')$ where $L_i' =\Hom (T_i,\G_m)$ and
$A_i' =\Pic^0 (A_i)$. 

We shall use the symmetric avatar  $(L_i,A_i,L_i',A_i',\psi_i)$ of $M_i$
(see \cite[10.2.12]{D} or \cite[p. 17]{BSAP}): recall that
$\psi_i$ denotes a certain section of the Poincar\'e biextension
$P_{A_i}\in \Biext (A_i,A_i';\G_m)$ over $L_i\times L_i'$. The symmetric
avatar of the Cartier dual is 
$(L_i',A_i',L_i,A_i,\psi^t_i)$. By \loccit a map of 1-motives 
$\phi: M_1 \to M_2^*$ is equivalent to a homomorphism $f: A_1\to
A_2'$ of abelian varieties and, if $f'$ is the dual of $f$, liftings $g$
and $g'$ of $fu_1$ and $f'u_2$ respectively, \ie to the following
commutative squares
\begin{equation}\label{squares}
\begin{CD}
L_1@>g>> L_2'\\
@V{u_1}VV  @V{u'_2}VV  \\
 A_1 @>f>>  A_2'
\end{CD}\qquad \text{and} \qquad
\begin{CD}
 L_2@>g'>> L_1'\\
@V{u_2}VV  @V{u_1'}VV  \\
 A_2 @>f'>>  A_1'\\[4pt]
\end{CD}
\end{equation}
under the condition that 
\begin{equation}\label{cond}
(1_{L_1}\times g')^*\psi_1= (g\times 1_{L_2})^*{}^t\psi_2
\text{ on } L_1\times L_2.
\end{equation}

Now let $(P,\tau, \sigma)$ be a biextension of $M_1$ and
$M_2$ by $\G_m$, \ie  a biextension $P\in\Biext (G_1,G_2;\G_m)$, a
section $\tau$ on $L_1\times G_2$ and a section $\sigma$ on $G_1\times
L_2$ such that 
\begin{equation}\label{eq7}
\tau\mid_{L_1\times L_2}= \sigma\mid_{L_1\times L_2}.
\end{equation} 

We have to show that $(P,\tau,\sigma)=(\phi\times 1)^*({}^tP_{A_2},\tau_2,\sigma_2)$
for a unique $\phi:M_1\to M_2^*$, where $\tau_2$ and $\sigma_2$ are the
universal trivializations.

Recall that $\Biext (G_1,G_2;\G_m)= \Biext (A_1,A_2;\G_m)$ (\cf
\cite[10.2.3.9]{D}) so that,  by Lemma \ref{biexta}, $P$ is the pull-back
to $G_1\times G_2$ of $(f\times 1_{A_2})^*({}^tP_{A_2})$ for a unique
homomorphism $f:A_1\to A_2'$. We thus have obtained the map $f$ and its
dual $f'$ in \eqref{squares}, and we now want to show that the extra data
$(\tau,\sigma)$ come from a pair $(g,g')$ in a unique way.

We may view $E=(fu_1\times 1_{A_2})^*({}^tP_{A_2})$ as an extension of
$L_1\otimes A_2$ by $\G_m$. Consider the commutative diagram of exact
sequences
\begin{equation}\label{diag}
\begin{CD}
&&0&& 0&& 0\\
&&@VVV @VVV @VVV \\
0@>>> 0@>>> L_1\otimes T_2@= L_1\otimes T_2@>>> 0\\
&&@VVV @VVV @V{1_{L_1}\otimes i_2}VV\\
0@>>> \G_m@>i>> Q@>\pi'>> L_1\otimes G_2@>>> 0\\
&&||&&@V{\pi}VV @V{1_{L_1}\otimes p_2}VV\\
0@>>> \G_m@>>> E@>>> L_1\otimes A_2@>>> 0\\
&&@VVV @VVV @VVV \\
&&0&& 0&& 0
\end{CD}
\end{equation}
where $i_2$  (\resp $p_2$) is the inclusion $T_2\into G_2$ (\resp the
projection $G_2\to A_2$). The section $\tau$ yields a retraction
$\tilde\tau:Q\to
\G_m$ whose restriction to $L_1\otimes T_2$ yields a homomorphism
\[\tilde g:L_1\otimes T_2\to \G_m\]
which in turn defines a homomorphism as in \eqref{squares}. We denote
the negative of this morphism by $g$.

\begin{lemma}\label{standard} With this choice of $g$, the left square of
\eqref{squares} commutes and $\tau=(g\times 1_{G_2})^*\tau_2$.
\end{lemma}

\begin{proof} To see the first assertion, we may apply $\Ext^*(-,\G_m)$
to \eqref{diag} and then apply \cite[Lemma 2.8]{bs} to the corresponding
diagram. Here is a concrete description of this argument: via the map of
Lemma \ref{kalg1}, $u'_2g$ goes to the following pushout
\[\begin{CD}
0@>>> L_1\otimes T_2@>{1\otimes i_2}>> L_1\otimes G_2@>{1\otimes p_2}>>
L_1\otimes A_2@>>> 0\\ 
&&@V{-\tilde g}VV @V{\pi\circ\tau}VV@V{||}VV\\
0@>>>\G_m@>>> E @>>> L_1\otimes A_2@>>> 0. 
\end{CD}\] 
because, due to the relation $i\tilde\tau + \tau \pi'=1$,  the left square
in this diagram commutes.

In particular, we have
\begin{multline*}
Q=(1\otimes p_2)^*(fu_1\otimes
1)^*{}^tP_{A_2}=(fu_1\otimes p_2)^*{}^tP_{A_2}\\
=(u'_2g\otimes
p_2)^*{}^tP_{A_2}=(g\otimes 1)^*(u'_2\otimes p_2^*){}^tP_{A_2}.
\end{multline*}

For the second assertion, since $\Hom(L_1\otimes A_2,\G_m)=0$ it suffices
to check the equality after restricting to $L_1\otimes T_2$. This is clear
because under the isomorphism $\Hom (L'_2\otimes T_2,\G_m)= \Hom
(L'_2,L'_2)$, the canonical trivialization ${}^t\psi_2$
corresponds to the identity.
\end{proof}

Note that if we further pullback we obtain that 
\begin{equation}\label{eq4}
\tau\mid_{L_1\times L_2}=\psi_2^t\mid_{L_1\times L_2}.
\end{equation} 

The same computation with $\sigma$ yields a map
\[g':L_2\to L'_1\]
and the same argument as in Lemma \ref{standard} shows that with this
choice of $g'$ the right square of \eqref{squares} commutes. We now use
that $P=(1_{A_1}\times f')^*(P_{A_1})$, which follows from the
naturality statement in Lemma \ref{biexta}. As in the proof of Lemma
\ref{standard}, this implies that its  trivialization
$\sigma$ on $G_1\times L_2$ is the pullback of the canonical
trivialization $\psi_1$ on $G_1\times L_1'$ along 
$1_{G_1}\times g': G_1\times L_2\to G_1\times L_1'$. In particular: 
\begin{equation}\label{eq5}
\sigma\mid_{L_1\times L_2}=\psi_1\mid_{L_1\times L_2}.
\end{equation}

Put together, \eqref{eq7}, \eqref{eq4} and \eqref{eq5} show that
Condition \eqref{cond} is verified: thus we get a morphism $\phi:M_1\to
M_2^*$. Let $h:G_1\to G'_2$ be its group component. It remains to check
that $\sigma= (h\times 1_{L_2})^*\sigma_2$. As in the proof of Lemma
\ref{standard} we only need to check this after restriction to
$T_1\otimes L_2$. But the restriction of $h$ to the toric parts is
the  Cartier dual of $g'$, so we conclude by the same argument.

Finally, let us show that $\gamma_{M_1,M_2}$ is natural in $M_2$. This
amounts to comparing two biextensions. For the bitorsors this follows from Lemma \ref{biexta} and for the sections we may argue again as in the proof of Lemma \ref{standard}.
\end{proof}

\subsection{Biextensions of complexes of $1$-motives}

Let $\cA$ be a category of abelian sheaves, and consider two bounded complexes $C_1,C_2$ of objects of $\cA^{[-1,0]}$. Let $H\in \cA$. We have a double complex
\[\underline{\Biext}(C_1,C_2;H)^{p,q}\df\Biext(C_1^p,C_2^q;H).\]

\begin{defn} A \emph{biextension of $C_1$ and $C_2$ by $H$} is an element of the group of cycles 
\[\Biext(C_1,C_2;H)\df Z^0(\mathsf{Tot}\ \underline{\Biext}(C_1,C_2;H)).\]
Here $\mathsf{Tot}$ denotes the total complex associated to a double complex.
\end{defn}
\index{$\mathsf{Tot}$}
\index{$\underline{\Biext}$}
Concretely: such a biextension $P$ is given by a collection of biextensions $P_p\in \Biext(C_1^p,C_2^{-p};H)$ such that, for any $p$,
\[(d_1^p\otimes 1)^* P_{p+1}= (1\otimes d_2^{-p-1})^*P_p
\]
where $d_1^\cdot$ (\resp $d_2^\cdot$) are the differentials of $C_1$ (\resp of $C_2$).

Now suppose that $\cA$ is the category of fppf sheaves, that $H=\G_m$ and that all the $C_i^j$ are Deligne $1$-motives. By Lemma \ref{ld0}, 
we have
\[\EExt^i(C_1^p,C_2^q;\G_m)=0\text{ for } i\le 0.\]

Therefore, a spectral sequence argument yields an edge homomorphism
\begin{equation}\label{edge}
\Biext(C_1,C_2;\G_m)\to \EExt^1(C_1\oo^L C_2,\G_m).
\end{equation}

Recall that Deligne's Cartier duality \cite{D} provides an exact functor
$$M\mapsto M^*: \M[1/p]\to \M[1/p]$$
yielding by Proposition \ref{pcd} a triangulated functor
\begin{equation}\label{dcartier}
(\ \ )^*: D^b(\M[1/p])\to D^b(\M[1/p]).
\end{equation}

Note that for a complex of $1$-motives
\[C=(\cdots \to M^{i}\to M^{i+1}\to \cdots) \]
we can compute $C^*$ by
means of the complex
\[C^*=(\cdots \to (M^{i+1})^*\to (M^{i})^*\to \cdots) \]
of Cartier duals here placed in degrees ..., $-i-1$, $-i$, etc.

Let us now take in \eqref{edge} $C_1=C$, $C_2 = C^*$. For each $p\in\Z$, we have the
Poincar\'e biextension $P_p\in \Biext(C^p,(C^p)^*;\G_m)$. By Proposition \ref{biext}, the
$\{P_p\}$ define a class in $\Biext(C,C^*;\G_m)$.

\begin{defn}\label{dpc} This class $P_C$ is the \emph{Poincar\'e biextension of the
complex $C$}.
\end{defn}

Let $C_1,C_2\in C^b(\M)$. As in Subsection \ref{sbiext}, pulling back ${}^tP_{C_1}=P_{C_1^*}\in \Biext(C_1,C_1^*;\G_m)$ yields a map generalising \eqref{eqbiext}:
\begin{align}
\gamma_{C_1,C_2} : \Hom (C_1, C_2^*)& \to \Biext (C_1,C_2;\G_m)\label{eqbiextc}\\
\phi&\mapsto(\phi\times 1_{C_1})^*({}^tP_{C_2}).\notag
\end{align}
which is clearly additive and natural in $C_1$. We then have the following trivial extension
of the functoriality in Proposition \ref{biext}: 

\begin{propose}\label{nat} $\gamma_{C_1,C_2}$ is also natural in $C_2$.\qed
\end{propose}

\subsection{A pairing with finite coefficients} In this section, we assume that $k$ is algebraically closed.

If $C$ is a complex of $1$-motives and $n>1$ is prime to $\car k$, we define
\[C/n = \text{cone}(C\by{n} C)\]
the mapping cone of multiplication by $n$. This defines a functor on $C^b(\M)$; we clearly
have a natural isomorphism
\[C^*/n\simeq (C/n[-1])^* \simeq (C/n)^*[1].\]

This functor and this natural isomorphism are easily seen to pass to $D^b(\M)$.
Composition of morphisms now yields a pairing ($\Hom$ groups computed in $D^b(\M)$):
\begin{multline*}
\Hom(\Z,C/n)\times \Hom(\Z,C^*/n)=\Hom(\Z,C/n)\times \Hom(\Z,(C/n)^*[1])\\
\simeq\Hom(\Z,C/n)\times \Hom(C/n,\Z^*[1])\to \Hom(\Z,\Z^*[1]) = k^*.
\end{multline*}

\begin{lemma} For any $C\in C^b(\M)$, the map $C/n\by{n}C/n$ is homotopic to $0$.
\end{lemma}

\begin{proof} We may embed $1$-motives in a category of complexes of sheaves of length $1$,
and then the proof is standard.
\end{proof}

This lemma implies that the above pairing refines into a pairing
\begin{equation}\label{eq4.1}
\Hom(\Z,C/n)\times \Hom(\Z,C^*/n)\to \mu_n.
\end{equation}

\begin{thm} \label{t4.3.2} This pairing is perfect.
\end{thm}

\begin{proof} Convert \eqref{eq4.1} into a morphism
\[\Hom(\Z,C/n)\to \Hom(\Hom(\Z,C^*/n), \mu_n).\]

This map is clearly natural in $C$, hence by d\'evissage we may check that it is an
isomorphism on ``generators" $C=N[i]$, $i\in  \Z$, where $N$ is a $1$-motive. We may further
reduce to $N=[\Z\to 0]$, $[0\to \G_m]$ or $[0\to A]$ where $A$ is an abelian variety.

It is convenient to replace $D^b(\M)$ by $D^b({}^t\M)$ (Theorem \ref{ptors}), which allows us to
represent
$N/n$ by
\[\begin{cases}
[\Z/n\to 0][0] &\text{if $N=[\Z\to 0]$}\\
[\mu_n\to 0][0] &\text{if $N=[0\to \G_m]$}\\
[{}_n A\to 0][0] &\text{if $N=[0\to A]$.}
\end{cases}\]

This implies that $\Hom(\Z,N[i]/n)=0$ if $i\ne 0$ (the best way to see this is to use the functor $\Tot$). Suppose $i=0$. In the cases $N=[\Z\to 0]$ or $[0\to \G_m]$, the pairing is easily seen to be the obvious pairing $\Z/n\times \mu_n\to \mu_n$ or $\mu_n\times \Z/n\to \mu_n$, which is clearly perfect. In the case $N=[0\to A]$, so that $N^*=[0\to A']$ is the dual abelian variety, we get a pairing
\[{}_n A\times {}_n A'\to \mu_n\]
which is by construction the Weil pairing (see \cite[IV.20]{mumford} and \cite[\S 16]{milne}). Therefore it is perfect too.
\end{proof}

\subsection{Comparing two Ext groups} The aim of this subsection is to prove:

\begin{propose}\label{ext=ext} Let $C_1,C_2,C_3\in d_{\le 1}\DM_{\gm,\et}^\eff$. Then the forgetful triangulated
functors
\[\DM_{-,\et}^{\eff}\by{i} D^{-}(\EST)\by{\omega}
D^{-}(\ES)\]
induce an isomorphism
\[\Hom_{\DM_{-,\et}^{\eff}}(C_1\otimes C_2,C_3[q])\iso
\Hom_{D^{-}(\ES)}(C_1\oo^L C_2, C_3[q])\]
for any $q\in\Z$.
\end{propose}

\begin{proof} Let $\ihom_\et$ denote the internal Hom of $\DM_{-,\et}^{\eff}$ (notation adopted in \S\ref{2.5}). By adjunction, it is enough to provide isomorphisms
\begin{multline*}
\Hom_{\DM_{-,\et}^{\eff}}(C_1,\ihom_\et(C_2,C_3)[q])\\
\iso
\Hom_{D^{-}(\ES)}(C_1,\ihom(C_2,C_3)[q])
\end{multline*}
where the right hand $\ihom$ is the (partially defined) internal Hom of $D^{-}(\Shv_{\et}(\Sm(k))$.
For this, it suffices to prove that the composite functor $\omega i$ of Proposition \ref{ext=ext} carries $\ihom_\et(C_2,C_3)$ to $\ihom(C_2,C_3)$. 

We first note that $i$ carries $\ihom_\et(C_2,C_3)$ to the internal Hom of $D^{-}(\EST)$  (compare \cite[Rk. 9.28]{VL}). Since $\ihom_\et(iC_2,iC_3)$ belongs to $D^{+}(\EST)$, it suffices to show that the natural map in $D^+(\ES)$
\[\omega\ihom_\et(iC_2,iC_3)\to \ihom(\omega iC_2,\omega iC_3)\]
is an isomorphism. By d\'evissage, we reduce to the case where $C_2$ and $C_3$ are single sheaves of $\Shv_1$ concentrated in degree $0$, and the claim follows from Theorem \ref{t3.12} a).
\end{proof}

\subsection{Two Cartier dualities} Recall the internal Hom $\ihom_\et$ from \S \ref{2.5}. We
define
\begin{equation}\label{D1}
D\1^\et (M) \df\ihom_\et (M, \Z_\et (1))
\end{equation}
for any object $M\in \DM_{\gm,\et}^{\eff}$.\index{$D\1^\et$, $D\1^\Nis$}

We now want to compare the duality \eqref{dcartier} with the following duality on triangulated
$1$-motives:

\begin{sloppypar}
\begin{propose}\label{cd} 
The functor $D\1^\et$
restricts to a self-duality $(\ )^\vee$ (anti-equivalence of categories) on
$d\1\DM_{\gm,\et}^\eff$.
\end{propose}
\end{sloppypar}

\begin{proof} It suffices to compute on motives of smooth
projective curves $M_\et (C)$. Then it is obvious in view of Proposition \ref{lcurve} c).
\end{proof}

\begin{defn}\label{cdef} For $M\in d\1\DM_{\gm,\et}^{\eff}$, we say that $M^\vee$ is the
\emph{motivic Cartier dual} of $M$.
\end{defn}

Note that motivic Cartier duality exchanges Artin motives and Tate motives, \eg $\Z_\et(0)^{\vee} = \Z_\et (1)$. We are going to compare
it with the Cartier duality on $D^b(\M[1/p])$ (see Proposition \ref{pcd}) via Theorem \ref{t1.2.1}.

For two complexes of $1$-motives
$C_1$ and 
$C_2$, by composing \eqref{eqbiextc} and \eqref{edge} and applying Proposition \ref{nat}, we get
a bifunctorial morphism
\begin{equation}\label{bimap}
\Hom(C_1,C_2^*) \to \Biext(C_1,C_2;\G_m)\to\Hom(C_1\oo^L C_2,\G_m[-1])
\end{equation}
where the right hand side is computed in the derived category of \'etale
sheaves. This natural transformation trivially factors through $D^b(\M[1/p])$.

From Proposition \ref{ext=ext} and Lemma \ref{Niset} (\cf \cite[Thm. 4.1]{VL}), taking $C_3 = \Z_\et (1)\cong \G_m[1/p][-1]$, it follows that the map \eqref{bimap} may be reinterpreted as a
natural transformation
\[\Hom_{D^b(\M[1/p])}(C_1,C_2^*)\to \Hom_{d\1\DM_{\gm,\et}^{\eff}}(\Tot(C_1),\Tot(C_2)^\vee).\]

Now we argue \`a la Yoneda: taking $C_1=C$ and $C_2=C^*$, the image of the identity
yields a canonical morphism of functors:
\[\eta_C:\Tot(C^*)\to \Tot(C)^\vee.\]

\begin{thm}\label{teq} The natural transformation $\eta$ is an isomorphism of functors.
\end{thm}

\begin{proof} It suffices to check this on $1$-motives, since they are dense in the
triangulated category $D^b(\M[1/p])$. Using Yoneda again and the previous discussion, it then follows from
Theorem \ref{t1.2.1} and the isomorphisms \eqref{biextform} and \eqref{eqbiext} (the latter
being proven in Proposition \ref{biext}). The following commutative diagram explains this:
\[\begin{CD}
\Hom(N,M^*)@>\text{Th. \ref{t1.2.1}}>\sim> \Hom(\Tot(N),\Tot(M^*))\\
@V{\wr}V{\eqref{biextform}+\eqref{eqbiext}}V @V{\eta_*}VV\\
\EExt^1_{\ES}(N\oo^L M,\G_m[1/p])@<\text{Prop. \ref{ext=ext}}<\sim<
\Hom(\Tot(N),\Tot(M)^\vee) .
\end{CD}\]
\end{proof}

\newpage
\part{The functors $\LAlb$ and $\RPic$}

\section{Definition of $\LAlb$ and $\RPic$}\label{lalb}

The aim of this section is to construct the closest integral approximation to a left adjoint of the
full embedding $\Tot$ of Definition \ref{tot}. In order to work it out, we first recollect some ideas from
\cite{V0}.

We shall show in Theorem \ref{ladj} that the functor $\LAlb$ of Definition \ref{LAlb} does
provide a left adjoint to $\Tot$ after we tensor Hom groups with $\Q$: this will provide a proof
of results announced in \cite[Preth. 0.0.18]{V0} and \cite{V1}. See Remark \ref{noleft}
for an integral caveat.

\subsection{Motivic Cartier duality} \label{mcd}

Recall the functor $D\1^\et:\DM_{\gm,\et}^{\eff}\to \DM_{-,\et}^\eff$ of \eqref{D1}. On the
other hand, by Proposition \ref{cD.2}, we may consider truncation on
$\DM_{-,\et}^\eff$ with respect to the homotopy $t$-structure. We have:

\begin{lemma}\label{l2.0} Let $X$ be a smooth $k$-variety. Then the truncated complex
$\tau_{\le 2}D\1^\et(M_\et(X))$ belongs to $d\1\DM_{\gm,\et}^\eff$.
\end{lemma}

\begin{proof} Recall that $\Z_\et (1) = \G_m[1/p][-1]$ so that this is a consequence of an analogue of Proposition \ref{lcurve} a) and b) in higher dimension.  In fact, the nonvanishing cohomology sheaves are $\sH^{-1}=\G_{X/k}[1/p]$ (see \S\ref{s:basic}) and
$\sH^0=\underline{\Pic}_{X/k}[1/p]$. Both belong to $\Shv_1$ by Proposition
\ref{p3.3.1}, hence the claim follows from Theorem \ref{t3.2.3}.
\end{proof}

Unfortunately, $\sH^i(D\1^\et(M_\et(X)))$ does not belong to $d\1\DM_{\gm,\et}^\eff$ for $i>2$
in general: indeed, it is well-known that this is a torsion sheaf of cofinite type, with
nonzero  divisible part in general (for $i\ge 3$ and in characteristic $0$, its corank is equal
to the
$i$-th Betti number of $X$).\footnote{One should compare this situation with that of Lemma \ref{l3.12} and the remark following it.} It might be considered as an ind-object of
$d_{\le 0}\DM_{\gm,\et}^\eff$, but this would take us too far. To get around this problem, we
shall restrict to the standard category of geometric triangulated motives of \cite{V},
$\DM_\gm^\eff$.

Let us denote by $D\1^\Nis$ \index{$D\1^\et$, $D\1^\Nis$} the same functor as $D\1^\et$ in the category $\DM_-^\eff$, defined
with the Nisnevich topology. Let as before
$\alpha^s:\DM_-^\eff\to \DM_{-,\et}^\eff$ denote the ``change of topology" functor.

\begin{lemma}\label{l2.1} a) For any smooth $X$ with motive $M(X)\in \DM_\gm^\eff$, we have 
\[\alpha^s D_{\le 1}^\Nis M(X) \iso \tau_{\le 2} 
D_{\le 1}^\et \alpha^s M(X).\]
b) The functor $\alpha^s D_{\le 1}^\Nis$ induces a triangulated functor 
\[\alpha^s D_{\le 1}^\Nis:\DM_\gm^\eff\to d\1\DM_{\gm,\et}^\eff.\]
\end{lemma}

\begin{proof} a) This is the weight $1$ case of the Beilinson-Lichtenbaum conjecture (here 
equivalent to Hilbert's theorem 90.) b) follows from a) and Lemma \ref{l2.0}.
\end{proof}

\begin{defn}\label{d5.1} We denote by $d\1:\DM_\gm^\eff\to d\1\DM_{\gm,\et}^\eff$ the composite functor $D\1^\et\circ \alpha^s\circ D\1^\Nis$.\index{$d\1$}
\end{defn}

Thus, for $M\in\DM_{\gm}^{\eff}$, we have
\begin{equation}\label{D2}
d\1(M)=\ihom_\et (\alpha^s\ihom_\Nis (M, \Z (1)), \Z_\et (1)).
\end{equation}

The evaluation map $M \otimes\ihom_\Nis (M, \Z (1))\to \Z (1)$ then yields a canonical map
\begin{equation}\label{amap}
a_M :\alpha^sM \to d\1 (M)
\end{equation}
for any object $M\in \DM_{\gm}^{\eff}$. We call $a_M$ the \emph{motivic Albanese map}
associated to $M$ for reasons that will appear later.

\begin{propose}\label{cd2} 
The restriction of \eqref{amap} to $d\1\DM_{\gm}^{\eff}$ is an
isomorphism of functors. In particular, we have equalities
\[\alpha^sD^\Nis_{\le 1}(\DM_{\gm}^{\eff}) = \alpha^sd\1\DM_{\gm}^{\eff}=d\1\DM_{\gm , \et}^{\eff}.\]
\end{propose}
\begin{proof} For the first claim, we reduce to the case $M=M(C)$ where $C$ is a smooth proper
curve. Then it follows from the proof of Proposition \ref{lcurve} c) (see \eqref{curves}).
The other claim is then clear (the second equality is true essentially by definition).
\end{proof}

\subsection{Motivic Albanese} \label{s.5.2}

\begin{defn} \label{LAlb} The \emph{motivic Albanese functor}
\[\LAlb : \DM_\gm^\eff\to D^b(\M[1/p])\]
is the composition of $d\1$ with a quasi-inverse to the equivalence of categories of  Theorem \ref{t1.2.1}. \end{defn}\index{$\LAlb$}

\begin{remark}\label{noleft} 
1) With this definition we get the following form of \eqref{amap} 
\[a_M :\alpha^sM \to \Tot\LAlb  (M)\]
where $\Tot$ is the functor  of Definition \ref{tot}; its restriction to $d\1\DM_{\gm}^{\eff}$ is an
isomorphism of functors.

2) By Theorem \ref{t1.2.1}, we then have the following relationship between $\LAlb$ and $\Tot$: for $M\in \DM_\gm^\eff$ and $N\in D^b(\M[1/p])$,
the map $a_M$ of \eqref{amap}  induces a map
\begin{equation}\label{lalbuniv}
\Hom (\LAlb M, N)\to \Hom (\alpha^sM, \Tot(N)).
\end{equation}
In Section \ref{qlalb}, we prove that this map is an isomorphism rationally,  showing   that
$\LAlb$ yields a left adjoint of $\Tot$ after Hom groups have been tensored  with $\Q$.  By 1), it is an isomorphism integrally if $M\in d\1\DM_{\gm}^{\eff}$. However, it is not so in general: 

Take $M=\Z(2)$, $N=\Z/n$ ($n$ prime to $p$). Then $\LAlb M=0$ because $\ihom_\Nis(\Z(2),\Z(1))\allowbreak=0$, but
\[\Hom (\alpha^sM, \Tot(N))= H^0_\et(k,\Z/n(-2))\]
which is nonzero \eg if $k=\bar k$.

3) The same example shows that $\Tot$ does not have a left adjoint. 
Indeed, suppose that such a left adjoint exists, and let us denote it by $\LAlb^\et$. For
simplicity, suppose $k$ algebraically closed. Let $n\ge 2$.  For any $m>0$, the exact triangle
in
$\DM_{\gm,\et}^\eff$
\[\Z(n)\by{m}\Z(n)\to \Z/m(n)\by{+1}\]
must yield an exact triangle
\[\LAlb^\et \Z(n)\by{m}\LAlb^\et\Z(n)\to \LAlb^\et\Z/m(n)\by{+1.}\]

Since $\Tot$ is an equivalence on torsion objects, so must be $\LAlb^\et$.  Since $k$ is
algebraically closed, $\Z/m(n)\simeq \mu_m^{\otimes n}$ is constant, hence we must have $\LAlb^\et\Z/m(n)\allowbreak\simeq [\Z/m\to 0]$.
Hence, multiplication by $m$ must be bijective on the $1$-motives $H^q(\LAlb^\et(\Z(n)))$ for
all $q\ne 0,1$, which forces these $1$-motives to vanish. For $q=0,1$ we must have exact
sequences
\begin{multline*}
0\to H^0(\LAlb^\et(\Z(n)))\by{m}H^0(\LAlb^\et(\Z(n)))\to [\Z/m\to 0]\\
\to H^1(\LAlb^\et(\Z(n)))\by{m}H^1(\LAlb^\et(\Z(n)))\to 0
\end{multline*}
which force either $H^0=[\Z\to 0]$, $H^1=0$ or $H^0=0$, $H^1=[0\to\G_m]$. But both cases are
impossible as one easily sees by computing
\begin{multline*}
\Hom(M(\P^n),\Tot([\Z\to 0])[2n+1])= H^{2n+1}_\et(\P^n,\Z)[1/p]\\
\simeq H^{2n}_\et(\P^n,(\Q/\Z)')\simeq (\Q/\Z)'
\end{multline*}
via the trace map, where $(\Q/\Z)'=\bigoplus_{l\ne p}\Q_\ell/\Z_\ell$.

Presumably, $\LAlb^\et$ does exist with values in a suitable pro-category containing $D^b(\M[1/p])$, and
sends $\Z(n)$ to the complete Tate module of $\Z(n)$ for $n\ge 2$. Note that, by
\ref{tatetwist} below, $\LAlb(\Z(n))=0$ for $n\ge 2$, so that the natural transformation
$\LAlb^\et(\alpha^s M)\to \LAlb(M)$ will not be an isomorphism of functors in general.
\end{remark}

\subsection{Motivic Pic}

\begin{defn} \label{RPic}
The \emph{motivic Picard functor} (a contravariant functor) is the functor
\[\RPic : \DM_{\gm}^{\eff}\to D^b(\M[1/p])\]
given by $\Tot^{-1}\alpha^sD\1^\Nis$ (\cf Definition \ref{LAlb}).
\end{defn}\index{$\RPic$}

For $M\in \DM_{\gm}^{\eff}$ we then have the following
tautology
\[(\Tot \RPic (M))^{\vee} = \Tot \LAlb (M).\]

Actually, from Theorem \ref{teq} we deduce:

\begin{cor}\label{rpdla}
For $M\in \DM_{\gm}^{\eff}$ we have
\[\RPic (M)^* = \LAlb (M).\qed\]
\end{cor}

Therefore we get ${}^tH^i(\RPic (M)) = ({}_tH_i(\LAlb (M)))^*$.

\subsection{Motivic $\pi_0$}\label{s5.4} \index{$L\pi_0$} All results of \S \ref{mcd} hold when replacing $D_{\le 1}$ by $D_{\le 0}=\ihom(-,\Z(0))$, with similar (and easier) proofs. In particular, we get a triangulated functor
\[d_{\le 0}=D_{\le 0}^\et\circ \alpha^s\circ D_{\le 0}^\Nis: \DM_\gm^\eff\to d_{\le 0} \\DM_{\gm,\et}^\eff\]
 hence, using the dimension $0$ case of Theorem \ref{t1.2.1},  a triangulated functor
\[L\pi_0:\DM_\gm^\eff\to D^b(\cM_0[1/p])\]
and a natural transformation
\[\alpha^s M \to \Tot L\pi_0(M)\]
for $M\in \DM_\gm^\eff$.

\begin{propose}\label{p5.4} If $X\in \Sm(k)$, the natural morphism $X\to \pi_0(X)$ induces an isomorphism  $L\pi_0(M(X))\simeq \Z[\pi_0(X)][0]$.
\end{propose}

\begin{proof} This is obvious if $\dim X=0$; hence it is enough to show that $D_{\le 0}^\Nis$ converts the morphism $M(X)\to M(\pi_0(X))$ into an isomorphism. This statement means:
\[H^i_\Nis(\pi_0(X)\times Y,\Z)\iso H^i_\Nis(X\times Y,\Z) \quad \forall Y\in\Sm(k), \forall i\in\Z.\]

For $Y=\Spec k$, this is true because the constant Nisnevich sheaf $\Z$ is flasque, $X$ is locally irreducible,  and flasque sheaves have trivial Nisnevich cohomology \cite[Lemma 1.40]{rioudea}. The general case reduces to this one by considering the composition
\[H^i_\Nis(\pi_0(X)\times \pi_0(Y),\Z)\to H^i_\Nis(\pi_0(X)\times Y,\Z)\to H^i_\Nis(X\times Y,\Z)\]
and noting that $\pi_0(X)\times \pi_0(Y)\simeq \pi_0(X\times Y)$.
\end{proof}

\subsection{$\LAlb$ and Chow motives} Let $\Chow$ be the category of Chow motives over $k$  with integral coefficients and $\Chow^\eff$ the full subcategory of effective 
Chow motives. We take the covariant convention for composition of correspondences: the functor
\[X\mapsto h(X)\]
from smooth projective varieties to $\Chow^\eff$ is covariant. Recall that Voevodsky \cite[2.1.4]{V2} defined a functor
\begin{equation}\label{eqvf}
\Phi:\Chow^\eff \to \DM_\gm^\eff
\end{equation}
such that $\Phi(h(X))=M(X)$ for any smooth projective $X$: this is a full embedding by the cancellation theorem of \cite{voecan}, see \cite[Cor. 6.7.3]{bv}. Hence a composition
\[\Phi_{(1)}:d_{\le 1}\Chow^\eff\by{\iota}
\Chow^\eff\by{\Phi}\DM_\gm^\eff\by{\LAlb}D^b(\M[1/p])\]
where $d_{\le 1}\Chow^\eff$ is the thick subcategory of
$\Chow^\eff$ (\ie full, stable under direct sumands), generated by motives of curves.

\begin{propose}\label{l18.1} 
The functor $\Phi_{(1)}[1/p]$
is fully faithful. It yields a naturally commutative diagram
\[\begin{CD}
d_{\le 1}\Chow^\eff[1/p]@>\Phi_{(1)}[1/p]>> D^b(\M[1/p])\\
@V{\iota}VV @V{\Tot}VV\\
\Chow^\eff[1/p]@>\alpha^s\circ \Phi>> \DM_{\gm,\et}^\eff
\end{CD}\]
where all functors except $\alpha^s\circ \Phi$ are  full embeddings.
\end{propose}

\begin{proof}  The diagram is naturally commutative by Remark \ref{noleft} 1). Since $\Tot$ is fully faithful (see Definition \ref{tot}), it then suffices to check that $\alpha^s \circ \Phi\circ \iota$ is fully faithful. If $C,C'$ are two smooth projective curves, this functor induces a homomorphism
\[\Hom_{d_{\le 1}\Chow^\eff[1/p]}(h(C),h(C'))\to \Hom_{\DM_{\gm,\et}^\eff}(M_\et(C),M_\et(C')).\]

The left group is $CH^1(C\times C')[1/p]$. The right one can be computed by Poincar\'e duality (see \cite[App. B]{hk}):
\begin{multline*}
\Hom_{\DM_{\gm,\et}^\eff}(M_\et(C),M_\et(C')) \simeq \Hom_{\DM_{\gm,\et}^\eff}(M_\et(C\times C'),\Z_\et(1)[2])\\
=\Pic(C\times C')[1/p]
\end{multline*}
and the map is clearly the identity.
\end{proof}

\begin{remark}\label{r16.3} The functor $\alpha^s\circ \Phi$ becomes a full embedding after $\boxtimes \Q$;  $d_{\le 1}\Chow^\eff\boxtimes\Q$ consists of those objects that may be
written as a direct sum of Chow motives of the following type:
\begin{itemize}
\item an Artin motive;
\item a motive of the form $h_1(A)$ for $A$
an abelian variety;
\item a motive of the form $M\otimes \L$, where $M$ is an Artin motive and $\L$
is the Lefschetz motive.
\end{itemize}

This is clear by the Chow-K\"unneth decomposition for motives of curves and the
fact that any abelian variety is a direct summand of the Jacobian of a curve, up to isogeny.

We could then define $\Phi_{(1)}\boxtimes \Q$ without reference to $\DM_\gm^\eff$ or $\LAlb$: this functor sends
\begin{enumerate}
\item[a)] An Artin motive $M$ to the $1$-motive $[L\to 0]$, where
$L$ is the permutation Galois-module associated to $M$. 
\item[b)] If $A$ is an abelian variety, $h_1(A)$ to $[0\to A][1]$.
\item [c)] A Lefschetz motive $M\otimes\L$ to $[0\to L\otimes \G_m][2]$, where $L$ is as in a).
\end{enumerate}
It is not quite clear how to define $\Phi_{(1)}$ integrally without using $\LAlb$.
\end{remark}

The birational version of $d\1\Chow^\eff$ was described in \cite[Prop. 7.2.4]{ks}.

\section{The adjunction $\LAlb-\Tot$  with rational coefficients}\label{qlalb}

Throughout this section, we use the notations $\otimes\Q$ and $\boxtimes\Q$ from Definition
\ref{1mot}.
\subsection{Rational coefficients revisited}
Let
$\DM_-^\eff(k;\Q)$ and $\DM_{-,\et}^\eff(k;\Q)$ denote the full subcategories of
$\DM_-^\eff(k)$ and $\DM_{-,\et}^\eff(k)$ formed of those complexes whose cohomology sheaves
are uniquely divisible. Recall that by \cite[Prop. 3.3.2]{V}, the change of topology functor
\[\alpha^s:\DM_-^\eff(k)\to \DM_{-,\et}^\eff(k)\]
induces an equivalence of categories
\[\alpha^s_\Q:\DM_-^\eff(k;\Q)\iso \DM_{-,\et}^\eff(k;\Q).\]

Beware that in \loccit, these two categories are respectively denoted by
$\DM_-^\eff(k)\otimes\Q$ and $\DM_{-,\et}^\eff(k)\otimes\Q$, while we use this notation here
according to Definition \ref{1mot}.  With the notation adopted in this book, we have a commutative diagram
\begin{equation}\label{eq6.1}\begin{CD}
\DM_-^\eff(k;\Q)@>>> \DM_-^\eff(k)@>>> \DM_-^\eff(k)\otimes\Q\\
@V{\alpha^s_\Q}VV @V{\alpha^s}VV @V{\alpha^s\otimes\Q}VV\\
\DM_{-,\et}^\eff(k;\Q)@>>> \DM_{-,\et}^\eff(k)@>>> \DM_{-,\et}^\eff(k)\otimes\Q
\end{CD}\end{equation}
whose horizontal compositions are fully faithful but  not essentially surjective. The functor $\alpha^s\otimes\Q$ is not
essentially surjective, nor (a priori) fully faithful. Nevertheless, these two horizontal compositions have a left adjoint $C\mapsto C\otimes\Q$, and


\begin{propose}\label{ptensq} a) The compositions
\begin{gather*}
\DM_\gm^\eff(k)\otimes\Q\to \DM_-^\eff(k)\otimes\Q\longby{\otimes\Q} \DM_-^\eff(k;\Q)\\
\DM_{\gm,\et}^\eff(k)\otimes\Q\to \DM_{-,\et}^\eff(k)\otimes\Q\longby{\otimes\Q}
\DM_{-,\et}^\eff(k;\Q)
\end{gather*}
are fully faithful.\\
b) Via these full embeddings, the functor $\alpha^s$ induces equivalences of categories
\begin{gather*}
\DM_\gm^\eff(k)\boxtimes\Q\iso \DM_{\gm,\et}^\eff(k)\boxtimes\Q\\
d\1\DM_\gm^\eff(k)\boxtimes\Q\iso d\1\DM_{\gm,\et}^\eff(k)\boxtimes\Q.
\end{gather*}
Here $d\1\DM_\gm^\eff(k)$ is the thick subcategory of $\DM_\gm^\eff(k)$ generated by motives of smooth curves.
\end{propose}

\begin{proof} a) In the first case, let $M,N\in\DM_\gm^\eff(k)$: we have to prove that the obvious map
\[\Hom(M,N)\otimes\Q\to \Hom(M\otimes \Q,N\otimes\Q)\]
is an isomorphism. We shall actually prove this isomorphism for any $M\in \DM_\gm^\eff(k)$ and
any $N\in \DM_-^\eff(k)$. By adjunction, the right hand side coincides with $\Hom(M,N\otimes
\Q)$ computed in $\DM_-^\eff(k)$. We may reduce to $M=M(X)$ for $X$ smooth. By \cite[Prop.
3.2.8]{V}, we are left to see that the map
\[H^q_\Nis(X,N)\otimes \Q\to H^q_\Nis(X,N\otimes\Q)\]
is an isomorphism for any $q\in\Z$. By hypercohomology spectral sequences, we reduce to the
case where $N$ is a sheaf concentrated in degree $0$; then the assertion follows from the fact
that Nisnevich cohomology commutes with filtering direct limits of sheaves.

This proof works in the \'etale topology as long as $cd(k)<\infty$: we thank the referee for pointing out this issue and suggesting the following argument in general (\cf \cite[proof of Prop. 2.1 b)]{kmilnor}). Using the exact triangle $N\to N\otimes \Q\to N\otimes \Q/\Z\by{+1}$, we reduce to proving the isomorphism for $N\otimes \Q$ (in which case it is trivial) and for $N\otimes\Q/\Z$. Since $N\in \DM_{\gm,\et}^\eff(k)$, we may reduce to $N=M(X)$ for some smooth $X$. By rigidity \cite{suvo}, the cohomology sheaves of $N\otimes \Q/\Z$ are locally constant and $0$ outside some finite interval, and the previous argument goes through.

b) It is clear that the two compositions commute with $\alpha^s$, which sends
$\DM_\gm^\eff(k)\otimes\Q$ into $\DM_{\gm,\et}^\eff(k)\otimes\Q$. By a) and \cite[Prop.
3.3.2]{V}, this functor is fully faithful, and the induced functor on the $\boxtimes$
categories remains so and is essentially surjective by definition of the two categories. Similarly for the $d\1$ categories.
\end{proof}

\begin{remarks} 1) In fact, 
$d\1\DM_{\gm,\et}^\eff(k)\otimes\Q=d\1\DM_{\gm,\et}^\eff(k)\boxtimes\Q$ thanks to Corollary
\ref{nobox} and Theorem \ref{t1.2.1}. We don't know whether the same is true for the other
categories.\\
2) See \cite[A.2.2]{riou} for a different, more general approach to Proposition \ref{ptensq}.
\end{remarks}

\begin{defn}   \label{d5.1a} With rational coefficients, we identify  $\DM_\gm^\eff(k)\boxtimes\Q$ with $\DM_{\gm,\et}^\eff(k)\boxtimes\Q$ via $\alpha^s$ (using Proposition \ref{ptensq} b)) and define
\begin{align*}D\1&=D\1^\Nis=D\1^\et\\ 
d\1&=D\1^2\\
\LAlb^\Q&=\Tot^{-1}\circ d\1.
\end{align*}
\end{defn}

\begin{remark} \label{formq} This definition is compatible with  the formula $d\1=D\1^\et\circ \alpha^s\circ D\1^\Nis$ of Definition \ref{d5.1}.
\end{remark}

\subsection{The functor $\LAlb^\Q$} We now get the announced adjunction by
taking  \eqref{D2} with rational coefficients, thanks to Corollary 
\ref{nobox} and Proposition \ref{ptensq}.

\begin{thm}\label{ladj} Let $M\in \DM_{\gm}^{\eff}(k)\boxtimes\Q$. Then the map
 $a_M$ from \eqref{amap} induces an isomorphism
\[\Hom (d\1 M, M')\iso \Hom (M, M')\]
for any $M'\in d\1\DM_{\gm}^{\eff}(k)\boxtimes\Q$. Equivalently, \eqref{lalbuniv} is an isomorphism with rational coefficients.
\end{thm}

\begin{proof}  By
Proposition \ref{cd}, $M'$ can be
written as $N^\vee=D\1 (N)$ for some
$N\in d\1\DM_{\gm}^{\eff}(k)\boxtimes\Q$. We have the following commutative diagram
\[\begin{CD}
\Hom (M, D\1 (N)) &=& \Hom (M\otimes N, \Z (1))  &=&\Hom (N, D\1 (M))\\
@A{a_M^*}AA @A{(a_M\otimes 1_N)^*}AA  @A{D\1(a_M)_*}AA\\
 \Hom (d\1M, D\1 (N)) &=& \Hom (d\1M\otimes N, \Z (1)) &=&\Hom (N,
D\1(d\1 M)).
\end{CD}\]

But $D\1(a_M)\circ a_{D\1 M}=1_{D\1M}$ \cite[p. 56, (3.2.3.9)]{saa}\footnote{Note that this proof carries over in our case.} and $a_{D\1 M}$ is an isomorphism by Proposition
\ref{cd2}, which proves the claim.\end{proof}

\begin{cor}\label{cet} The functor $d\1$ of \eqref{D2} induces a left adjoint to the embedding
$d\1\DM_\gm^\eff(k)\boxtimes\Q\into\DM_\gm^\eff(k)\boxtimes\Q$. The Voevodsky-Orgogozo full
embedding
$\Tot:D^b(\M\otimes\Q)\to \DM_\gm^\eff\boxtimes\Q$ has a left adjoint
$\LAlb^\Q$.\qed
\end{cor}

\section{A tensor structure on $D^b(\M\otimes\Q$)}\label{tens}

In this section, coefficients are tensored with $\Q$ and we use the functor $\LAlb^\Q$ of
Corollary \ref{cet}.

\subsection{Tensor structure}

\begin{lemma}\label{biextrigid} Let $G_1,G_2$ be two semi-abelian varieties. Then, we have in
$\DM_{\gm,\et}^\eff\boxtimes \Q$:
\[\cH^q(D\1(\uG_1[-1]\otimes \uG_2[-1]))=
\begin{cases}
\Biext(G_1,G_2;\G_m)\otimes\Q&\text{if $q=0$}\\
0&\text{else.}
\end{cases}
\]
\end{lemma}

\begin{proof} By Gersten's
principle (Proposition \ref{pgersten}), it is enough to show that the isomorphisms are valid
over function fields
$K$ of smooth $k$-varieties and that $\sH^0$ comes from the small \'etale site of $\Spec k$.
Since we work up to torsion, we may even replace $K$ by its perfect closure. Thus, without loss
of generality, we may assume $K=k$ and we have to show the lemma for sections over $k$.

For $q\le  0$, we use Proposition \ref{ext=ext}: for $q<0$ this follows from Lemma \ref{ld0}, while for $q=0$ it follows from the
isomorphisms \eqref{biextform} and \eqref{eqbiext} (see Proposition \ref{biext}), which show that $\Biext(G_1,G_2;\G_m)$ is rigid. 

For $q>0$, we use the formula
\[D\1(\uG_1[-1]\otimes\uG_2[-1])\simeq \ihom(\uG_1[-1],\Tot([0\to G_2]^*))\]
coming from Theorem \ref{teq}. Writing $[0\to G_2]^*=[L_2\to A_2]$ with $L_2$ a lattice and $A_2$ an abelian variety, we are left to show that
\begin{gather*}
\Hom_{\DM_{\gm,\et}^\eff\boxtimes \Q}(\uG_1,L_2[q+1])=0\text{ for } q>0\\
\Hom_{\DM_{\gm,\et}^\eff\boxtimes \Q}(\uG_1,\underline{A}_2[q])=0\text{ for } q> 0.
\end{gather*}

For this, we may reduce to the case where $G_1$
is either an abelian variety or $\G_m$.  If $G_1=\G_m$, $\uG_1$ is a direct summand of $M(\P^1)[-1]$ and the result follows. If $G_1$ is an abelian variety, it is isogenous to a direct summand of
$J(C)$ for
$C$ a smooth projective geometrically irreducible curve. Then $\uG_1$ is a direct
summand of
$M(C)$, and the result follows again since $L_2$ and $\underline{A}_2$ define locally constant
(flasque) sheaves for the Zariski topology.
\end{proof}

\begin{propose}\label{tens1} a) The functor $\LAlb^\Q:\DM_\gm^\eff\boxtimes\Q\to
D^b(\M\otimes\Q)$ is a localisation functor; it carries the tensor
structure $\otimes$ of $\DM_\gm^\eff\boxtimes\Q$ to a tensor structure $\otimes_1$\index{$\otimes_1$} on
$D^b(\M\otimes\Q)$, which is left adjoint to the internal Hom $\ihom_1$ of \S \ref{s3.14}. \\
b) For $(M,N)\in \DM_\gm^\eff\boxtimes\Q\times D^b(\M\otimes\Q)$, we have
\[\LAlb^\Q(M\otimes \Tot(N))\simeq \LAlb^\Q(M)\otimes_1 N.\]
c) We have
\[
[\Z\to 0]\otimes_1 C = C\]
for any  $C\in D^b(\M\otimes\Q)$;
\[N_1\otimes_1 N_2 =[L\to G]\] 
for two Deligne $1$-motives  $N_1=[L_1\to G_1]$, $N_2=[L_2\to G_2]$, where
\[L = L_1\otimes L_2;\]
there is an extension
\[0\to \Biext(G_1,G_2;\G_m)^*\to G\to L_1\otimes G_2\oplus L_2\otimes G_1\to 0.
\]
d) The tensor product $\otimes_1$ is exact with respect
to the motivic $t$-structure and respects the weight filtration. Moreover, it is right exact
with respect to the homotopy $t$-structure.\\
e) For two $1$-motives $N_1,N_2$ and a semi-abelian variety $G$, we have
\[\Hom(N_1\otimes_1 N_2,[0\to G])\simeq \Biext(N_1,N_2;G)\otimes\Q.\]
\end{propose}

\begin{proof} a) The first statement is clear since $\LAlb^\Q$ is left adjoint
to the fully faithful functor $\Tot$. For the second, it suffices to see
that if $\LAlb^\Q(M)=0$ then $\LAlb^\Q(M\otimes N)=0$ for any $N\in
\DM_\gm^\eff\otimes\Q$. We may check this after applying $\Tot$. Note
that, by Proposition \ref{cd} and Remark \ref{formq} 3), $\Tot\LAlb^\Q(M)=d\1(M)=0$ is equivalent
to
$D\1(M)=0$.  We have:
\[
D\1(M\otimes N)=\ihom(M\otimes N,\Z(1))
=\ihom(N,\ihom(M,\Z(1)))=0.
\]

The last statement follows by adjunction from the fact that $\Tot$ commutes with internal Homs.

b) Let $M' = fibre(\Tot\LAlb^\Q(M)\to M)$: then $\LAlb^\Q(M')=0$. By definition of $\otimes_1$
we then have
\begin{multline*}
\LAlb^\Q(M)\otimes_1 N =\LAlb^\Q(\Tot\LAlb^\Q(M)\otimes \Tot(N))\\
\iso \LAlb^\Q(M\otimes\Tot(N)).
\end{multline*}

c) The first formula is obvious. For the second, we have an exact triangle
\begin{multline*}
G_1[-1]\otimes G_2[-1]\to \Tot(N_1)\otimes \Tot(N_2)\\
\to \Tot([L_1\otimes L_2\to
L_1\otimes G_2\oplus L_2\otimes G_1])\by{+1}
\end{multline*}
hence an exact triangle
\begin{multline*}
\ihom(\Tot([L_1\otimes L_2\to
L_1\otimes G_2\oplus L_2\otimes G_1],\Z(1))\\
\to \ihom(\Tot(N_1)\otimes
\Tot(N_2),\Z(1))
\to \ihom(G_1[-1]\otimes G_2[-1],\G_m) \by{+1}
\end{multline*}

By Lemma \ref{biextrigid}, the last term is $\Biext(G_1,G_2;\G_m)$, hence the claim.

d) Exactness and compatibility with weights follow from the second formula of c); right
exactness for the homotopy $t$-structure holds because it holds on $\DM_-^\eff\otimes\Q$.

e) We have:
\begin{multline*}
\Hom_{\M\otimes\Q}(N_1\otimes_1 N_2,[0\to G]) =
\Hom_{d\1\DM\otimes\Q}(\Tot(N_1\otimes_1 N_2),G[-1])\\
=\Hom_{\DM\otimes\Q}(\Tot(N_1)\otimes
\Tot(N_2),G[-1])=\Biext(N_1,N_2;G)\otimes\Q
\end{multline*}
by Proposition \ref{ext=ext} and formula \eqref{biextform}.
\end{proof}

\begin{remarks} 1) By the same argument as in Remark \ref{noleft}, one can see that $\otimes_1$ is not defined integrally on $(\Z(1),\Z(1))$. Details are left to the reader.

2) In \cite{bermaz}, Cristiana Bertolin and Carlo Mazza generalise Proposition \ref{tens1} e) to an isomorphism
\[\Hom_{\M\otimes\Q}(N_1\otimes_1 N_2,N) = \Biext(N_1,N_2;N)\otimes\Q\]
for three $1$-motives $N_1,N_2,N$, where the right hand side is the biextension group
introduced by Bertolin \cite{bert}. This puts in perspective her desire to interpret these
groups as Hom groups in the (future) tannakian category generated by $1$-motives.

More precisely, one expects that $\DM_\gm\boxtimes\Q$ carries a motivic $t$-struct\-ure whose
heart $\MM$ would be the searched-for abelian category of mixed motives. Then $\M\otimes\Q$
would be a full subcategory of $\MM$ and we might consider the thick tensor subcategory
$\M^\otimes\subseteq \MM$ generated  by $\M\otimes\Q$ and the Tate motive (inverse to the
Lefschetz motive): this is the putative category Bertolin has in mind.

Since the existence of the abelian category of mixed Tate motives (to be contained in
$\M^\otimes$!) depends on the truth of the Beilinson-Soul\'e conjecture, this basic obstruction
appears here too.

Extrapolating from Corollary \ref{cet} and Proposition \ref{tens1}, it seems that the embedding
$\M\otimes\Q\into \MM^\eff$ (where $\MM^\eff$ is to be the intersection of $\MM$ with
$\DM_\gm^\eff\boxtimes\Q$) is destined to have a left adjoint
$\Alb^\Q=H_0\circ \LAlb^\Q_{|\MM^\eff}$, which would carry the tensor product of $\MM^\eff$ to
$\otimes_1$. Restricting
$\Alb^\Q$ to $\M^\otimes\cap \MM^\eff$ would provide the link between Bertolin's ideas and
Proposition
\ref{tens1} e).
\end{remarks}

\subsection{A formula for the internal Hom}

\begin{propose} \label{ihom1} We have
\[\ihom_1(C_1,C_2) = (C_1\otimes_1 C_2^*)^*\]
for $C_1,C_2\in D^b(\M\otimes\Q)$.
\end{propose}

\begin{proof} In view of Theorem
\ref{teq}, we are left to show that
$\ihom(M_1,M_2)\allowbreak\simeq (M_1\otimes_1,M_2^\vee)^\vee$ for $M_1,M_2\in d\1\DM_\gm^\eff\otimes\Q$. By duality, we may replace $M_2$ by $M_2^\vee$.
Then:
\begin{multline*}\ihom(M_1,M_2^\vee)=\ihom(M_1,\ihom(M_2,\Z(1))\\
\simeq \ihom(M_1\otimes M_2,\Z(1))\simeq \ihom(M_1\otimes_1 M_2,\Z(1)) = (M_1\otimes_1 M_2)^\vee
\end{multline*}
where the second isomorphism follows from Proposition \ref{tens1} b).
\end{proof}

\section{The Albanese complexes and their basic properties}\label{6}

We now introduce homological and Borel-Moore Albanese complexes of an
algebraic variety. We also consider a slightly more sophisticated cohomological
Albanese complex $\LAlb^* (X)$ which is only contravariantly
functorial for maps between schemes of the same dimension. All
these complexes coincide for smooth proper schemes.

\subsection{Motives of singular schemes}\label{s8.1}

Let $X$ be a smooth $k$-variety. The assignment $X\mapsto M(X)$ defines a covariant functor from $\Sm(k)$ to $\DM_\gm^\eff$.
The image of
$M(X)$ via the full embedding $\DM_\gm^\eff\to
\DM_-^\eff$ is given by the Suslin complex $C_*$ of the
representable Nisnevich sheaf with transfers $L(X)$ associated to $X$.

For $X\in \Sch(k)$, the formula
$M(X)=C_*(L(X))$ still defines an object of $\DM_-^\eff$. Similarly, we have the motive with
compact support of $X$, denoted by $M^c (X)\in\DM_-^{\eff}$ (\ibid): it is the Suslin complex of the presheaf with transfers $L^c (X)$ given by quasi-finite correspondences. Since finite implies quasi-finite we have a canonical map $M (X) \to M^c (X)$ which is an isomorphism if $X$ is proper over
$k$. $M$ is covariant for any morphism in $\Sch(k)$ while $M^c$ is covariant for proper morphisms between $k$-schemes of finite type.

If $\car k=0$, $M(X)$ enjoys cohomological properties like Mayer-Vietoris and blow-up exact triangles, while $M^c(X)$ enjoys homological properties like localisation exact triangles, by \cite[\S 4.1]{V}. This implies that $M(X), M^c(X)\in \DM_\gm^\eff$ for any $X\in \Sch(k)$.

Voevodsky's arguments in \loccit rely on (the strong form of) Hironaka's resolution of singularities. The assignments $X\mapsto M(X), M^c(X)$ can be extended to characteristic $p$ (as functors with values in $\DM_\gm^\eff[1/p]$) in two fashions:
\begin{itemize}
\item Using the 6 operations of Voevodsky-Ayoub \cite[\S 6.7]{fullfaith}.
\item By Kelly's thesis \cite{kelly}, which extends Voevodsky's arguments to characteristic $p$.
\end{itemize}

Both approaches rely on Gabber's refinement of de Jong's theorem. Either one is sufficient for the present section. On the other hand, some of the arguments of Sections \ref{comps} and \ref{12}, in characteristic $0$, use Hironaka's resolution of singularities. They could probably be extended to characteristic $p$ by using the de Jong/Gabber alteration theorem, but modifying the proofs would be tedious and we prefer to leave it to the interested reader. So we shall deal with singular schemes only in characteristic $0$.\\

\noindent{\bf Convention.} In the rest of this book, ``scheme" means separated $k$-scheme
of finite type if $\car k=0$ and smooth (separated) $k$-scheme of finite type if $\car k>0$.

\subsection{The homological Albanese complex}
\begin{defn}\label{LAlb-} We define the {\it homological Albanese
complex} of $X$ by
$$\LAlb(X)\df \LAlb (M (X)).$$
Define, for $i\in\Z$
$$\LA{i}(X)\df {}_tH_i(\LAlb(X))$$ the 1-motives with cotorsion (see Definition \ref{1cot} and
Notation \ref{not}) determined by the homology of the Albanese complex.
\end{defn}

\begin{remark} We could have chosen to define the homology with respect to the dual
$t$-structure, corresponding to $1$-motives with torsion according to Theorem~\ref{tstr}.
The reason for our choice appears in Sections \ref{comp} -- \ref{12}, where it works well especially for $\LA{1}$ which turns out to yield a Deligne (\ie cotorsion-free) $1$-motive in all cases considered. This is not always true with the other $t$-structure, already for curves. The interested reader is invited to investigate the finite groups appearing in
${}_tH_i(\LAlb(X))$
and ${}^tH_i(\LAlb(X))$, taking care of both $t$-structures.
Of course there is no difference with rational coefficients, as these two $t$-structures coincide 
after tensoring with $\Q$.\end{remark}

The functor $\LAlb$ has the following properties, easily
deduced from \cite[2.2]{V}:

\subsubsection{Homotopy invariance} For any scheme $X$
the map
$$\LAlb(X\times \Aff^1)\to \LAlb(X)$$ is an isomorphism, thus
$$\LA{i}(X\times \Aff^1)\iso \LA{i}(X)$$
for all $i\in\Z$.

\subsubsection{Mayer-Vietoris} For a scheme $X$ and an open
covering
$X = U\cup V$ there is a distinguished triangle
\[\begin{CD}
\LAlb (U\cap V)  @>>> \LAlb(U)\oplus \LAlb(V)\\
{\scriptstyle +1}\nwarrow &&\swarrow\\
&\LAlb(X)&
\end{CD}\]
and therefore a long exact sequence of 1-motives with cotorsion
$$\cdots\to \LA{i}(U\cap V) \to\LA{i}(U)\oplus \LA{i}(V)\to
\LA{i}(X)\to \cdots$$

\subsubsection{Tate twists}\label{tatetwist} If $X$ is a
smooth scheme and
$n>0$, then
\[\Tot\LAlb(M(X)(n))=
\begin{cases}
0&\text{if $n>1$}\\
M(\pi_0(X))(1)&\text{if $n=1$}
\end{cases}
\]
where $\pi_0(X)$ is the scheme of constants of $X$, see Definition \ref{dpi0}. Indeed
\begin{multline*}
\Tot\LAlb(M(X)(n))=\ihom_\et\alpha^s(\ihom_\Nis(M(X)(n),\Z(1)),\Z(1))\\=
\ihom_\et\alpha^s(\ihom_\Nis(M(X)(n-1),\Z),\Z(1))
\end{multline*}
by the cancellation theorem \cite{voecan}. Now
\[\ihom_\Nis(M(X)(n-1),\Z)=
\begin{cases}
0&\text{if $n>1$}\\
\ihom_\Nis(M(\pi_0(X)),\Z)&\text{if $n=1$}.
\end{cases}
\]

The last formula should follow from \cite[Lemma 2.1 a)]{kmot} but
the formulation there is wrong; however, the formula
immediately follows from the argument in the proof of \loccit,
\ie considering the Zariski cohomology of $X$ with coefficients in
the flasque sheaf $\Z$.

This gives
\begin{equation}\label{eq8.1}
\LA{i}(M(X)(1))=
\begin{cases}
[0\to \Z[\pi_0(X)]\otimes \G_m]&\text{if $i=0$}\\
0&\text{else.}
\end{cases}
\end{equation}

More generally,  using the functor $L\pi_0$ of \S \ref{s5.4}:

\begin{propose}\label{ptate} For any $M\in \DM_{\gm}^{\eff}$, 
\begin{thlist}
\item There is a natural isomorphism 
\[\LAlb(M(1))\simeq L\pi_0(M)(1),\]
where on the right hand side, $C\mapsto C(1)$ denotes the functor 
\[
D^b(\cM_0[1/p])\to D^b(\cM_1[1/p])
\]
induced by $L \mapsto [0\to L\otimes \G_m]$.
\item $L\pi_0(M(1))=0$.
\item $\LAlb(M(n))=0$ for $n\ge 2$.
\end{thlist}
\end{propose}

\begin{proof} Let $M\in \DM_\gm^\eff$. The cancellation theorem yields an isomorphism
\[D_{\le 0}^\Nis(M)\iso D_{\le 1}^\Nis(M(1)).\]

Let $N\in \DM_{\gm,\et}^\eff$. Tensoring the evaluation map
\[\ihom_\et(N,\Z)\otimes N\to \Z\]
with $\Z(1)$ and using adjunction, we get a morphism
\[D_{\le 0}^\et(N)(1)=  \ihom_\et(N,\Z)\otimes\Z(1)\to \ihom_\et(N,\Z(1))=D_{\le 1}^\et(N).\]
 
 For $N\in d_{\le 0}\DM_{\gm,\et}^\eff$, this map is an isomorphism by reduction to $N=M(\Spec E)$ for a finite extension $E/k$. Applying this to $N=\alpha^s D_{\le 0}^\Nis(M)$, we get a composite isomorphism
 \[(d_{\le 0} M)(1)\iso D_{\le 1}^\et \alpha^s D_{\le 0}^\Nis(M)\iso d_{\le 1}(M(1))\]
 from which (i)  follows.

For (ii), we reduce to $M=M(X)$ for $X\in\Sm(k)$; then this follows from Proposition \ref{p5.4} applied to $X$ and $X\times \P^1$. Finally, (iii) follows from (i) and (ii).
\end{proof}

\subsubsection{$\LAlb$ with supports} \label{supports1} Let
$X\in \Sm(k)$, $U$ an open subset of $X$ and $Z=X-U$ (reduced
structure). In $\DM_\gm^\eff$, we have the \emph{motive with supports} $M^Z(X)$
fitting in an exact triangle
\[M(U)\to M(X)\to M^Z(X)\by{+1}\]
hence the homological complex with supports
\[\LAlb^Z(X):=\LAlb(M^Z(X))\] fitting in an exact triangle
\[\LAlb(U)\to \LAlb(X)\to \LAlb^Z(X)\by{+1}.\]

\subsubsection{Gysin}\label{Gysin} Keep the notation of \S\ref{supports1}. When $Z$ is smooth,  purely of codimension $c$, we have the Gysin isomorphism \cite[Prop. 3.5.4]{V}
\begin{equation}\label{eq:gysin}
M^Z(X)\simeq M(Z)(c)[2c]
\end{equation}

In general this implies:

\begin{lemma} \label{lsupports1} a) Let $c=\codim_X(Z)$. Then there is an effective motive $M\in \DM_\gm^\eff$ such that $M^Z(X)\simeq M(c)$.\\
b) For $c>1$, we have an isomorphism $\LAlb(U)\iso \LAlb(X)$, while for $c=1$ we have an exact triangle
\[\LAlb(U)\to \LAlb(X)\to [0\to\Z[\pi_0(Z-Z')]\otimes \G_m][2]\by{+1}\]
where $Z'$ is the union of the singular locus  $Z_\sing$ of $Z$ and its irreducible components of codimension $>1$. Hence a long exact sequence
\begin{multline*}
0\to\LA{2}(U)\to \LA{2}(X)\to [0\to
\Z[\pi_0(Z-Z')]\otimes\G_m]\\\to
\LA{1}(U)\to \LA{1}(X)\to 0
\end{multline*}
and an isomorphism $\LA{0}(U)\to \LA{0}(X)$.
\end{lemma}

\begin{proof} a) If $Z$ is smooth, this follows from \eqref{eq:gysin}. In general, it follows by Noetherian induction from the exact triangle
\begin{equation}\label{eq8.2}
 M^{Z-Z_\sing}(X-Z_\sing) \to M^Z(X) \to M^{Z_\sing}(X)\by{+1} 
 \end{equation}
 and the cancellation theorem (note that $Z_\sing\ne Z$ because $k$ is perfect).

b) follows from a) and Proposition \ref{ptate} (iii) for $c>1$. For $c=1$, let us apply $\LAlb$ to \eqref{eq8.2}, where we replace $Z_\sing$ by $Z'$. By a) and Proposition \ref{ptate}  (iii), we have $\LAlb(M^{Z'}(X))=0$, hence from Proposition \ref{ptate}  (i) and \eqref{eq:gysin}  we get an isomorphism
\[L\pi_0 M(Z-Z')(1)[2]\iso \LAlb(M^Z(X)). \]

The conclusion now follows from Proposition \ref{p5.4}.
\end{proof}

\subsubsection{Blow ups}\label{blowups} If $X$ is a scheme and $Z\subseteq X$ is a
closed subscheme, denote by $p :\tilde X\to X$ a proper surjective morphism such that
$p^{-1}(X-Z)\to X-Z$ is an isomorphism, \eg the blow up of $X$ at
$Z$. Then there is a distinguished triangle
\[\xymatrix{
\LAlb(\tilde Z) \ar[rr]&& \LAlb(\tilde X)\oplus \LAlb(Z)\ar[dl]\\
&\LAlb(X)\ar[ul]^{+1}
}\]
with $\tilde Z = p^{-1}(Z)$, yielding a long exact sequence 
\begin{equation}\label{exres} 
\cdots\to \LA{i}(\tilde Z) \to\LA{i}(\tilde X)\oplus
\LA{i}(Z)\to \LA{i}(X)\to \cdots\end{equation}

If $X$ and $Z$ are smooth, we get (using \cite[Prop. 3.5.3]{V}
and the above)
\[\LAlb(\tilde X)=\LAlb(X)\oplus [0\to \Z[\pi_0(Z)]\otimes \G_m][2]\]
and corresponding formulas for homology.

\subsubsection{Albanese map} If $X$ is a scheme we have
the natural map  \eqref{amap} in $\DM_{\gm,\et}^{\eff}$
\begin{equation}\label{amapX}
a_X :\alpha^sM(X) \to \Tot\LAlb(X)
\end{equation}
inducing homomorphisms on \'etale motivic cohomology
\[\Hom (M_\et(X), \Z_\et (1)[j]) \to\Hom (\LAlb(X),[0\to\G_m][j])\]
which are isomorphisms rationally by Theorem \ref{ladj}.

\subsection{The cohomological Picard complex} Dual to \ref{LAlb-} we set:

\begin{defn}\label{RPic+} Define the {\it cohomological Picard
complex} of $X$ by
$$\RPic (X)\df \RPic (M (X)).$$
Define, for $i\in\Z$
$$\RA{i}(X)\df {}^tH^i(\RPic (X))$$ 
the 1-motives with torsion
determined by the cohomology of the Picard complex (see Notation \ref{not}).
\end{defn}

The functor $\RPic$ has similar properties to $\LAlb$, deduced by
duality. Homotopy invariance, Mayer-Vietoris, Gysin and the exact triangle for abstract blow-ups are clear, and moreover we have
\[ \RPic (M (X)(n))=
\begin{cases}
0&\text{if $n>1$}\\
\relax [\Z[\pi_0(X)]\to 0]&\text{if $n=1$}.
\end{cases}
\]

We also have that $\RPic(X)=\LAlb(X)^\vee$, hence
$$\RA{i}(X) = \LA{i}(X)^\vee.$$

It is easy to compute the cohomology sheaves of $\Tot\RPic(X)$ with respect to the homotopy $t$-structure:

\begin{propose}\label{hopic} For $X\in \Sm(k)$, we have
\[\cH^i(\Tot\RPic(X)) =
\begin{cases}
\G_{X/k}[1/p]& \text{if $i=-1$}\\
\underline{\Pic}_{X/k}[1/p] &\text{if $i=0$}\\
0&\text{otherwise}
\end{cases}
\]
where these sheaves  were introduced in \S\ref{s:basic}. 
\end{propose} 

\begin{proof} We have
\[\Tot\RPic(X) =(D_{\le 1}^\et)^2\alpha^s D_{\le 1}^\Nis M(X) =\alpha^s D_{\le 1}^\Nis M(X).\]

Since $\alpha^s$ respects the homotopy $t$-structures of $\DM_-^\eff$ and $\DM_{-,\et}^\eff$ (Corollary \ref{cD.1}), the claim follows from the known Nisnevich cohomology of $\G_m$ for smooth schemes.
\end{proof}

We shall complete \S \ref{Gysin} by

\subsubsection{$\RPic$ with supports} \label{supports} With the notation of \S \ref{Gysin}, 
we have the cohomological complex with supports
\[\RPic_Z(X):=\RPic(M^Z(X))\] fitting in an exact triangle
\[\RPic_Z(X)\to \RPic(X)\to \RPic(U)\by{+1}.\]

We then have the following cohomological version of Lemma \ref{lsupports1}, with a different proof:

\begin{lemma} \label{lsupports} If $\dim X=d$ and   $Z\subset X$ is closed of codimension $\geq 1$, then $\RPic_Z(X)\simeq
[\underline{CH}_{d-1}(Z)\to 0][-2]$, where $\underline{CH}_{d-1}(Z)$ is
the lattice corresponding to the Galois module $CH_{d-1}(Z_{\bar k})$.
\end{lemma}

(Note that $CH_{d-1}(Z_{\bar k})$ is the free abelian group with
basis the irreducible components of $Z_{\bar k}$ which are of
codimension $1$ in $X_{\bar k}$, so that this computation is dual to that in Lemma \ref{lsupports1}.)

\begin{proof} This follows readily from Proposition \ref{hopic} and  the exact sequence
\[0\to \G_{X/k}\to \G_{U/k}\to \underline{CH}_{d-1}(Z) \to\underline{\Pic}_{X/k}\to \underline{\Pic}_{U/k}\to 0.
\]
\end{proof}

\subsection{Relative $\LAlb$ and $\RPic$}\label{relative}

For $f: Y\to X$ a map of schemes we let $M(X, Y)$ denote the cone of $C_*(Y)\to
C_*(X)$. Note that for a closed embedding $f:Y\into X$ in a proper scheme $X$, we have an isomorphism $M(X,Y)\iso M^c(X-Y)$.

We denote by $\LAlb (X, Y)$ and $\RPic (X, Y)$ the resulting complexes of
1-motives.


\subsection{The Borel-Moore Albanese complex}

\begin{defn}\label{LAlbc} We define the {\it Borel-Moore Albanese
complex} of $X$ by
$$\LAlb^c (X)\df \LAlb (M^c (X)).$$
Define, for $i\in\Z$
$$\LA{i}^c(X)\df {}_tH_i(\LAlb^c (X))$$ the 1-motivic homology of this complex.
\end{defn}\index{$\LAlb^c $}

We then have the following properties:

\subsubsection{Functoriality} The functor $X\mapsto \LAlb^c (X)$ is covariant
for proper maps and contravariant with respect to flat
morphisms of relative dimension zero, for example \'etale morphisms. We have a canonical, covariantly functorial map
$$\LAlb (X)\to \LAlb^c (X)$$
which is an isomorphism if $X$ is proper.

\subsubsection{Localisation triangle} For any closed subscheme
$Y$ of a scheme $X$ we have a triangle
\[\begin{CD}
\LAlb^c (Y)  @>>> \LAlb^c (X)\\
{\scriptstyle +1}\nwarrow &&\swarrow\\
&\LAlb^c(X-Y)&
\end{CD}\]
and therefore a long exact sequence of 1-motives
\begin{equation}\label{loc}
\dots\to \LA{i}^c(Y) \to\LA{i}^c(X)\to
\LA{i}^c(X-Y)\to \LA{i-1}^c(Y) \to\dots
\end{equation}

In particular, let $X$ be a scheme
obtained by removing a divisor $Y$ from a proper scheme $\bar X$,
\ie $X = \bar X -Y$. Then
\begin{multline*}
\cdots\to \LA{1}(Y) \to\LA{1}(\bar X)\to \LA{1}^c(X)\to
\LA{0}(Y)\\
\to\LA{0}(\bar X)\to \LA{0}^c(X)\to \dots.
\end{multline*}

\subsubsection{Albanese map} We have the following natural map \eqref{amap}
\[a_X^c :\alpha^sM^c (X) \to \Tot\LAlb^c(X)\]
which is an isomorphism if $\dim (X)\leq 1$. In general, for any
$X$, $a_X^c$ induces an isomorphism $H^j_c(X, \Q (1)) = \Hom (\LAlb^c
(X),[0\to \G_m][j])\otimes \Q.$

\subsection{Cohomological Albanese complex} 

\begin{lemma}\label{leffe} Suppose $p=1$ (\ie $\car k=0$), and let $n\ge 0$.  For any $X$ of
dimension $\le n$, the motive $M (X)^*(n)[2n]$ is effective. (Here, contrary to the rest of the
paper, $M(X)^*$ denotes the ``usual" dual $\Hom(M(X),\Z)$ in $\DM_\gm$.)
\end{lemma}

\begin{proof} First assume $X$ irreducible. Let $\tilde X\to X$ be a resolution of
singularities of $X$. With notation as in \S \ref{blowups}, we have an exact triangle
\[M(X)^*(n)\to M(\tilde X)^*(n)\oplus M(Z)^*(n)\oplus M(\tilde Z)^*(n)\by{+1}.\]

Since $\tilde X$ is smooth, $M(\tilde X)^*(n)\simeq M^c(X)[-2n]$ is effective by \cite[Th.
4.3.2]{V}; by induction on $n$, so are $M(Z)^*(n)$ and $M(\tilde Z)^*(n)$ and therefore
$M(X)^*(n)$ is effective.

In general, let $X_1,\dots,X_r$ be the irreducible components of $X$. Suppose $r\ge 2$ and
let $Y=X_2\cup\dots \cup X_r$: since $(X_1,Y)$ is a cdh cover of $X$, we have an exact triangle
\[M(X)^*(n)\to M(X_1)^*(n)\oplus M(Y)^*(n)\oplus M(X_1\cap Y)^*(n)\by{+1}.\]

The same argument then shows that $M(X)^*(n)$ is effective, by induction on $r$.
\end{proof}

We can therefore apply our functor $\LAlb$ and obtain another complex  $\LAlb (M (X)^*(n)[2n])$
of 1-motives. If $X$ is smooth this is just the Borel-Moore Albanese.

\begin{defn}\label{LAlb*} We define the {\it cohomological Albanese
complex} of a scheme $X$ of dimension $n$ by
$$\LAlb^* (X)\df \LAlb (M (X^{(n)})^*(n)[2n])$$
where $X^{(n)}$ is the union of the $n$-dimensional components of $X$. Define, for $i\in\Z$
$$\LA{i}^*(X)\df {}_tH_i(\LAlb^* (X))$$ the 1-motivic homology of this complex.
\end{defn}\index{$\LAlb^*$}

\begin{lemma}\label{l9.2} a) If $Z_1,\dots Z_n$ are the irreducible components of dimension $n$ of $X$, then the
cone of the natural map
\[\LAlb^*(X)\to \bigoplus \LAlb^*(Z_i)\]
is a complex of groups of multiplicative type.\\
b) If $X$ is integral and $\tilde X$ is a desingularisation of $X$, then the cone of the
natural map
\[\LAlb^*(X)\to \LAlb^*(\tilde X)\]
is a complex of groups of multiplicative type.
\end{lemma}

\begin{proof} a) and b) follow from dualising the abstract blow-up exact triangles of \cite[2.2]{V} and applying Proposition \ref{ptate}.
\end{proof}

\subsection{Compactly supported and homological Pic} We now
consider the dual complexes of the Borel-Moore and cohomological
Albanese.
\begin{defn} Define the {\it compactly supported Picard
complex} of any scheme $X$ by
$$\RPic^c (X)\df \RPic (M^c (X))$$
and the {\it homological Picard complex} of a
scheme $X$ of dimension $n$ by
$$\RPic^* (X)\df  \RPic (M (X^{(n)})^*(n)[2n]).$$
Denote $\RA{i}^c(X)\df {}^tH^i(\RPic^c (X))$ and $\RA{i}^*(X)\df
{}^tH^i(\RPic^* (X))$ the 1-motives with torsion determined by the
homology of these Picard complexes.
\end{defn}\index{$\RPic^c$, $\RPic^*$}

Recall that $\RPic^c (X) = \RPic (X)$ if $X$ is proper and
$\RPic^c (X) = \RPic^* (X)$ if $X$ is smooth and equidimensional.

\subsection{Topological invariance} To conclude this section, we note the
following useful

\begin{lemma} \label{l12.3} Suppose that $f:Y\to X$ is a universal topological homeomorphism,
in the sense that $1_U\times f:U\times Y\to U\times X$ is a homeomorphism of topological spaces
for any smooth $U$ (in particular $f$ is proper). Then $f$ induces isomorphisms $\LAlb(Y)\iso
\LAlb(X)$, $\RPic(X)\iso\RPic(Y)$,  $\LAlb^c(Y)\iso \LAlb^c(X)$ and $\RPic^c(X)\iso
\RPic^c(Y)$. Similarly, $\LAlb^*(X)\iso \LAlb^*(Y)$ and $\RPic^*(Y)\iso \RPic^*(X)$. This
applies in particular to
$Y=$ the semi-normalisation of
$X$.
\end{lemma}

\begin{proof} It suffices to notice that $f$ induces isomorphisms $L(Y)\iso L(X)$ and
$L^c(Y)\to L^c(X)$, since by definition these sheaves only depend on the underlying topological
structures.
\end{proof}

Lemma \ref{l12.3} implies that in order to compute $\LAlb(X)$, etc., we may always assume $X$
semi-normal if we wish so.

\newpage
\part{Some computations}

\section{Computing $\LAlb(X)$ and $\RPic(X)$ for smooth $X$}\label{comp}

 In this section,  we fully compute  the motives introduced in Section \ref{6} for smooth schemes; in Sections \ref{comps} -- \ref{12} we provide some interesting computations and comparisons  for singular schemes as well.

\subsection{The Albanese scheme} Let $X$ be a reduced $k$-scheme of finite type. ``Recall"
(\cite[Sect. 1]{ram}, \cite{spsz})  the Albanese scheme
$\cA_{X/k}$ \index{$\cA_{X/k}$} fitting in the following extension
\begin{equation}\label{unext}
0\to \cA_{X/k}^0\to \cA_{X/k}\to \Z[\pi_0(X)]\to 0
\end{equation}
where $\cA_{X/k}^0$ is Serre's generalised Albanese semi-abelian
variety, and $\pi_0(X)$ is the scheme of constants of $X$ viewed as
an \'etale sheaf on $\Sm(k)$.\footnote{In the said references,
$\cA_{X/k}^0$ is denoted by
$Alb_X$ and
$\cA_{X/k}$ is denoted by $\widetilde{Alb}_X$.} In particular,
\[\cA_{X/k}\in \AbS \text{ (see Definition \ref{dsabt}).}
\]

There is a
canonical morphism
\begin{equation}\label{alb}
\bar a_X: X
\to\cA_{X/k}
\end{equation}
which is  universal for morphisms from $X$ to group schemes of
the same type. 

For the existence of $\cA_{X/k}$, the reference 
\cite[Sect. 1]{ram} is sufficient if $X$ is a variety (integral separated $k$-scheme of finite
type), hence if $X$ is normal (for example smooth): this will be sufficient in this section.
For the general case, see \S \ref{albs}.

Recall the functor $\tot$ of \eqref{hts}: we shall denote the object 
\[\tot^{-1}(\underline{\cA}_{X/k})\in D^b(\M[1/p])\]
simply by $\cA_{X/k}$. As seen in Lemma \ref{l4.3.1}, we have 
\[{}^tH_i(\cA_{X/k})=
\begin{cases}
\relax [\Z[\pi_0(X)]\to 0]&\text{for $i=0$}\\
\relax [0\to \cA_{X/k}^0]&\text{for $i=1$}\\
0&\text{for $i\ne 0,1$.}
\end{cases}
\]

(Note that ${}^tH_i(\cA_{X/k})={}_tH_i(\cA_{X/k})$ here, because the involved $1$-motives are torsion-free.)

\subsection{The main theorem} Suppose $X$ smooth. Via \eqref{eq1.4}, \eqref{alb} induces  a
composite map 
\begin{equation}\label{precan}
M_\et(X) \to M_\et(\cA_{X/k})\to \cA_{X/k}.
\end{equation}

Theorem \ref{teq} gives:

\begin{lemma}\label{l7.2.1} We have an exact triangle
\[\Z[\pi_0(X)]^*[0]\to\ihom_\et(\cA_{X/k},\Z_\et(1))\to (\cA_{X/k}^0)^*[-2]\by{+1}\qed\] 
\end{lemma}

By Lemma \ref{l7.2.1}, the map
\[\ihom_\et (\cA_{X/k},\Z_\et(1))\to\ihom_\et (M_\et(X),\Z_\et(1))\]
deduced from \eqref{precan} factors into a map
\begin{equation}\label{eqdual}
\ihom_\et (\cA_{X/k},\Z_\et(1))\to\tau_{\le 2}\ihom_\et (M_\et(X),\Z_\et(1)).
\end{equation}

Applying Proposition \ref{cd} and Lemma \ref{l2.1}, we therefore get a canonical map in
$D^b(\M[1/p])$
\begin{equation}\label{can}
\LAlb (X) \to \cA_{X/k}.
\end{equation}

\begin{sloppypar}
\begin{thm} \label{trunc} Suppose $X$ smooth. Then the map \eqref{can} sits in an exact triangle
\[[0\to \NS_{X/k}^*][2]\to \LAlb (X) \to \cA_{X/k}\longby{+1}\]
where $\NS_{X/k}^*$ denotes the group of multiplicative type dual to $\NS_{X/k}$
(\cf Definition \ref{dNS} and Proposition \ref{belong}).
\end{thm}
\end{sloppypar}

This theorem says in particular that, on the object $\LAlb (X)$, the
motivic $t$-structure and the homotopy $t$-structure are compatible in a
strong sense.

\begin{cor}\label{HLAlb} For $X$ smooth over $k$ we have
$$\LA{i}(X) =
\begin{cases}
\relax [\Z[\pi_0(X)]\to 0]& \text{if $i = 0$}\\
\relax [0\to\cA_{X/k}^0] & \text{if $i= 1$}\\
\relax [0\to \NS_{X/k}^*] & \text{if $i= 2$}\\
0 & \text{otherwise.}
\end{cases}$$
\end{cor}

\begin{cor}\label{clalb} For $X$ smooth, $\LA{1}(X)$ is isomorphic to the homological Albanese
$1$-motive
$\Alb^- (X)$ of \cite{BSAP}.\qed
\end{cor}

\subsection{Reformulation of Theorem \ref{trunc}}   

It is sufficient to get an exact triangle after application of $D\1\circ\Tot$, so that we have
to compute the cone of the morphism \eqref{eqdual} in
$\DM_{\gm,\et}^\eff$. We shall use:

\begin{lemma}\label{l4.1} For $\cF\in \HI_\et^s$, the morphism $b$ of Proposition \ref{p1} is
induced by \eqref{eq1.4}.
\end{lemma}

\begin{proof} This is clear by construction, since
$\Hom_{\DM_\et}(M(G),\cF[1])=H^1_\et(G,\cF)$ \cite[Prop. 3.3.1]{V}.
\end{proof}

Taking the cohomology sheaves of \eqref{eqdual}, we get morphisms
\begin{gather}
\shom(\cA_{X/k},\G_m)\to p_*\G_{m,X}\label{homa}\\
f:\sext(\cA_{X/k},\G_m)\longby{b}\underline{\Pic}_{\cA_{X/k}/k}\longby{\bar a^*}
\underline{\Pic}_{X/k}\label{exta}
\end{gather}
where in \eqref{exta}, $b$ corresponds to the map of Proposition \ref{p1} thanks to Lemma
\ref{l4.1}. Thanks to Proposition \ref{p4.2}, Theorem \ref{trunc} is then equivalent to the
following

\begin{thm}\label{t6.3} Suppose $k$ algebraically closed. Then\\
a) \eqref{homa} yields an isomorphism $\Hom(\cA_{X/k},\G_m)\iso \Gamma(X,\G_m)$.\\
b) \eqref{exta} defines a short exact sequence
\begin{equation}\label{eq3}
0\to \Ext(\underline{\cA}_{X/k},\G_m)\longby{f} \Pic(X)\longby{e}\NS(X)\to 0
\end{equation}
where $e$ is the natural map.
\end{thm}

Before proving Theorem \ref{t6.3}, it is convenient to prove Lemma
\ref{l6.1} below. Let $A_{X/k}$  be the abelian part of $\cA_{X/k}^0$;
then the sheaf $\sext(\underline{A}_{X/k},\G_m)$ is represented by the dual
abelian variety $A^*_{X/k}$. Composing with the map $f$ of \eqref{exta},
we get a map of $1$-motivic sheaves
\begin{equation}\label{eqadj}
\underline{A}^*_{X/k}\to \underline{\Pic}_{X/k}.
\end{equation}

\begin{lemma}\label{l6.1} The map \eqref{eqadj}
induces an isogeny in $\SAb$
\[A^*_{X/k}\onto \gamma(\underline{\Pic}_{X/k})=\underline{\Pic}^0_{X/k}\]
where $\gamma$ is the adjoint functor appearing in Theorem \ref{t3.2.4}
a).
\end{lemma}

\begin{proof} We proceed in 3 steps:
\begin{enumerate}
\item The lemma is true if $X$ is smooth projective: this follows from
the representability of $\Pic^0_X$ and the duality between the Picard and the Albanese varieties.
\item Let $j:U\to X$ be a dense open immersion: then the lemma is true for $X$
if and only if it is true for $U$. This is clear since $\underline{\Pic}_{X/k}\to\underline{\Pic}_{U/k}$ is an epimorphism with discrete kernel.
\item Let $\phi:Y\to X$ be finite \'etale. If the lemma is true for
$Y$, then it is true for $X$. This follows from the  existence of
transfer maps $\phi_*:A^*_{Y/k}\to A^*_{X/k}$, $\underline{\Pic}_{Y/k}\to \underline{\Pic}_{X/k}$
commuting with the map of the lemma, plus the usual transfer argument. (The first transfer map may be obtained from the  composition $\cA_{X/k}^0\to \cA_{S^d(Y)/k}^0\longby{\text{sum}} \cA_{Y/k}^0$, where $d=\deg \phi$ and $X\to S^d(Y)$ is the standard ``multisection'' morphism.)
\end{enumerate}
We conclude by de Jong's theorem \cite[Th. 4.1]{DJ}. 
\end{proof}

\subsection{Proof of Theorem \ref{t6.3}}

We may obviously suppose that $X$ is irreducible.

a) is obvious from the universal property of $\cA_{X/k}$. For b) we
proceed in two steps:

\begin{enumerate}
\item Verification of $ef=0$.
\item Proof that the sequence is exact.
\end{enumerate}

(1) As above, let $A=A_{X/k}$ be the abelian part of $\cA_{X/k}^0$. In the diagram
\[\Ext(A,\G_m)\to \Ext(\cA_{X/k}^0,\G_m)\leftarrow\Ext(\cA_{X/k},\G_m)\]
the first map is surjective and the second map is an isomorphism, hence
we get a surjective map
\[v:\Ext(A,\G_m)\to \Ext(\cA_{X/k},\G_m).\]

Choose a rational point $x\in X(\bar k)$.  We have a diagram
\[\begin{CD}
\Ext(A,\G_m)@>v>> \Ext(\cA_{X/k},\G_m)\\
@V{a}VV @V{f}VV\\
\Pic(A)@>x^*>> \Pic(X)\\
@V{c}VV @V{e}VV\\
\NS(A)@>x^*>> \NS(X)
\end{CD}\] 
in which
\begin{thlist}
\item $a$ is given by \cite[p. 170, prop. 5 and 6]{gacl} (or by
Proposition \ref{p1}).
\item $ca=0$ (ibid., p. 184, th. 6).
\item $x^*$ is induced by the ``canonical" map $X\to A$ sending
$x$ to $0$. 
\end{thlist}

Lemma \ref{trans} applied to $G=\cA_{X/k}$ implies that the top square
commutes (the bottom one trivially commutes too). Moreover,  since $v$ is
surjective and $ca=0$, we get $ef=0$.\footnote{As suggested by the referee, we can use the functor $\pi_0$ of Theorem \ref{t3.2.4} for an alternative proof. Indeed, $e$ and $f$ are induced by maps of $1$-motivic sheaves (see Proposition \ref{belong} for $e$). Since $\sext( \cA_{X/k},\G_m)$ is a quotient of an abelian variety, $\pi_0(\sext( \cA_{X/k},\G_m))=0$ and on the other hand  $\pi_0(\NS_{X/k})=\NS_{X/k}$, hence $ef$ factors through the zero map $\pi_0(ef)$.}

(2) In the sequence \eqref{eq3}, the surjectivity of $e$ is clear. Let us prove the
injectivity of $f$: suppose that $f(\cE)$ is trivial. In the pull-back diagram
\[\begin{CD}
\bar a^*\cE@>\pi'>> X\\
@V{\bar a'}VV @V{\bar a}VV\\
\cE@>\pi>> \cA_{X/k}
\end{CD}\]
$\pi'$ has a section $\sigma'$. Observe that $\cE$ is a locally
semi-abelian scheme: by the universal property of $\cA_{X/k}$, the
morphism $\bar a'\sigma'$ factors canonically through $\bar a$. In other
words, there exists $\sigma:\cA_{X/k}\to \cE$ such that $\bar
a'\sigma'=\sigma\bar a$. Then
\[\pi\sigma\bar a=\pi \bar a'\sigma'=\bar a\pi'\sigma'=\bar a\]
hence $\pi\sigma=1$ by reapplying the universal property of $\cA_{X/k}$, and $\cE$ is trivial.
Finally, exactness in the middle follows immediately from Proposition \ref{belong} and Lemma
\ref{l6.1}.\qed

\begin{cor} The isogeny of Lemma \ref{l6.1} is an isomorphism.
\end{cor}

\begin{proof} This follows from the injectivity of $f$ in \eqref{eq3}.
\end{proof}

\subsection{An application}

\begin{cor}\label{divzero} Let $X$ be a smooth equidimensional $k$-variety of dimension $d$, $U$ a dense open
subset and $Z=X-U$ (reduced structure). Then the morphism $\cA_{U/k}\to \cA_{X/k}$ is epi; its
kernel $T_{X/U,k}$ is a torus whose character group fits into a short exact sequence
\[0\to T_{X/U,k}^*\to \underline{\CH}_{d-1}(Z)\to \NS_Z(X)\to 0\]
where $\underline{\CH}_{d-1}(Z)$ is as in Lemma \ref{lsupports} and $\NS_Z(X)=\ker(\NS(X)\to
\NS(U))$.
\end{cor}

\begin{proof} To see that $\cA_{U/k}\to \cA_{X/k}$ is epi with kernel of multiplicative type, it
is sufficient to see that $\pi_0(\cA_{U/k})\iso \pi_0(\cA_{X/k})$ and that $A_{U/k}\iso
A_{X/k}$. The first isomorphism is obvious and the second one follows from \cite[Th.
3.1]{milne}. The characterisation of $T_{X/U,k}$ is then an immediate consequence of Theorem
\ref{trunc} and Lemma \ref{lsupports}; in particular, it is a torus.
\end{proof}

\subsection{$\RPic(X)$} \label{rpic}

Recall that for $X$ smooth projective $\cA_{X/k}^0 = A_{X/k}$ is the classical Albanese
abelian variety $\Alb (X)$. In the case where $X$ is obtained by removing a divisor $Y$ from a
smooth proper scheme $\bar X$, $\cA_{X/k}^0$, can be described as follows (\cf \cite{BSAP}).
Consider the (cohomological Picard) 1-motive $\Pic^+ (X)\df [\Div_Y^0(\bar X)\to \Pic^0 (\bar
X)]$: its Cartier dual is $\cA_{X/k}^0$ which can be represented as a semi-abelian variety
\[0\to T_{\bar X/X,k} \to \cA_{X/k}^0\to \Alb (\bar X)\to 0\]
where $ T_{\bar X/X,k}$ has
character group $\Div_Y^0(\bar X)$ according to Corollary~\ref{divzero}.

From the previous remarks and Corollary~\ref{clalb}, we deduce:

\begin{cor}\label{HRPic} If $X$ is smooth,  $\RA{1}(X)$ is
isomorphic to the $1$-motive $\Pic^+(X)$ of \cite{BSAP} (the Cartier dual of $\Alb^-(X)$). If $\bar X$ is a smooth compactification of $X$, then
$$\RA{i}(X) =
\begin{cases}
\relax [0\to \Z[\pi_0(X)]^*]& \text{if $i = 0$}\\
\relax [\Div_Y^0(\bar X)\to\Pic^0 (\bar X)] & \text{if $i= 1$}\\
\relax [\NS (X)\to 0] & \text{if $i= 2$}\\
0 & \text{otherwise}
\end{cases}$$
where $Y=\bar X - X$. 
\end{cor}

\section{1-motivic homology and cohomology of singular schemes}\label{comps}

 In this section, we start to compute the motives introduced in Section \ref{6} for singular schemes. Recall the convention made at the end of \S \ref{s8.1}. We also refer to  \S \ref{s8.1} for a discussion on how one could extend these results to any characteristic.

\subsection{$\cA_{X/k}$ for $X\in \Sch(k)$}\label{albs} In this subsection, we extend the
construction of $\cA_{X/k}$ to arbitrary reduced $k$-schemes of finite type, starting
from the case where $X$ is integral (which is treated in \cite[Sect. 1]{ram}). So far, $k$ may
be of any characteristic.

To make the definition clear:

\begin{defn} Let $X\in \Sch(k)$. We say that \emph{$\cA_{X/k}$ exists} if the functor
\begin{align*}
\AbS&\to Ab\\
G&\mapsto G(X)
\end{align*}
is corepresentable.
\end{defn}

First note that $\cA_{X/k}$ does not exist (as a semi-abelian scheme, at least) if $X$ is not
reduced. For example, for $X=\Spec k[\epsilon]$ with $\epsilon^2=0$, we have an exact sequence
\[0\to \G_a(k)\to \Map_k(X,\G_m)\to \G_m(k)\to 0\]
which cannot be described by $\Hom(\cA,\G_m)$ for any semi-abelian scheme $\cA$.

On the other hand, $M(X)=M(X_\red)$ for any $X\in \Sch(k)$, where $X_\red$ is the reduced
subscheme of $X$ (see proof of Lemma \ref{l12.3}), so we are naturally led to
neglect nonreduced schemes.

\begin{lemma}\label{lred} Let $Z\in \Sch(k)$, $G\in \AbS$ and $f_1,f_2:Z\rightrightarrows G$ be
two morphisms which coincide on the underlying topological spaces (thus, $f_1 = f_2$ if $Z$ is
reduced). Then there exists a largest quotient $\bar G$ of $G$ such that $\pi_0(G)\iso
\pi_0(\bar G)$ and the two compositions
\[Z\rightrightarrows G\to \bar G\]
coincide.
\end{lemma}

\begin{proof} The set $S$ of such quotients $\bar G$ is in one-to-one correspondence with the
set of closed subgroups $H^0\subseteq G^0$. Clearly $\pi_0(G)\in S$, and if $\bar
G_1=G/H_1^0\in S$, $\bar G_2=G/H_2^0\in S$, then $\bar G_3 = G/(H_1^0\cap H_2^0)\in S$. Therefore
$S$ has  a smallest element, since it is Artinian (compare proof of Proposition \ref{pladj}).
\end{proof}

\begin{propose}\label{albexists} $\cA_{X/k}$ exists for any reduced $X\in \Sch(k)$.
\end{propose}

\begin{proof} When $X$ is integral, this is \cite[Sect. 1]{ram}. Starting from this case, we
argue by induction on $\dim X$. Let $Z_1,\dots, Z_n$ be the irreducible components of $X$ and
$Z_{ij} = Z_i\cap Z_j$. 

By induction, $\cA_{ij}\df\cA_{(Z_{ij})_\red/k}$ exists for any $(i,j)$. Consider
\[\cA=\coker\left(\bigoplus \cA_{ij}\to \bigoplus \cA_i\right)\]
with $\cA_i = \cA_{Z_i/k}$. Let $Z = \coprod Z_{ij}$ and $f_1,f_2:Z\rightrightarrows \coprod
Z_i$ be the two inclusions: the compositions $f_1,f_2:Z\rightrightarrows \cA$ verify the
hypothesis of Lemma \ref{lred}. Hence there is a largest quotient $\cA'$ of $\cA$ with
$\pi_0(\cA)\iso \pi_0(\cA')$, equalising $f_1$ and $f_2$. Then the composition
\[\coprod Z_i\to \bigoplus \cA_i\to \cA'\]
glues down to a morphism $X\to \cA'$. It is clear that $\cA'=\cA_{X/k}$ since,  for any
commutative group scheme $G$, the sequence
\[0\to \Map_k(X,G)\to \bigoplus \Map_k(Z_i,G)\to \bigoplus \Map_k(Z_{ij},G)\]
is exact. 
\end{proof}

Unfortunately this result is only useful to understand $\LA{1}(X)$ for $X$ normal
as we shall see below. In general, we shall have to consider Albanese schemes for the $\eh$
topology.

\subsection{The $\eh$ topology} In this subsection and the next ones, we assume that $k$ is of characteristic $0$. Recall that $\HI_\et=\HI_\et^s$ in this case by Proposition \ref{lD.1.3}.

The following \'etale analogue of the cdh topology was first considered by Thomas Geisser
\cite[Def. 2.1]{geisser} (it is the same as \cite[Def. 4.1.9]{V}, replacing the Nisnevich topology by the \'etale topology):

\begin{defn}\label{deh} The \emph{$\eh$-topology} on $\Sch(k)$ is the topology generated  by
the \'etale topology and coverings defined by abstract blow-ups (as in \S \ref{blowups}). 
We shall denote by $\EH$ the category of abelian $\eh$-sheaves.\index{$\EH$} 
\end{defn}

Let $\epsilon:\Sch(k)_\eh\to \Sch(k)_\et$ be the obvious morphism of sites. If
$\cF$ is an \'etale sheaf on $\Sch(k)$, we denote by $\cF_\eh$ its $\eh$ sheafification (that is,
$\cF_\eh=\epsilon^*\cF$). \index{$\cF_\eh$} We shall abbreviate $H^*_\eh(X,\cF_\eh)$ to
$H^*_\eh(X,\cF)$. Similarly, for $C\in D (\EST)$ we denote $C_\eh$ the $\eh$ sheafification
of the complex $C$ after forgetting transfers.

As in \cite[Prop. 13.27]{VL} we have:

 \begin{lemma}\label{ehs} a) Let $\cF\in \HI_\et$ and $X\in \Sm(k)$. Then $ \cF (X)\simeq  \cF_\eh(X) $ and (more generally) 
$H^q_\et(X,\cF)\simeq H^q_\eh(X,\cF)$ for all $q\geq 0$.\\
b) The isomorphism of a) remains true when replacing $\cF$ by any complex $C\in D (\EST)$, whose cohomology sheaves are in $\HI_\et$.
\end{lemma} 

\begin{proof} a) We adapt the proof in \cite[Prop. 13.27]{VL} to the $\eh$-topology. It is not difficult to adapt the facts 13.19--13.24 in \cite[\S13]{VL} once we have the analogue of 12.28 for the $\eh$-topolgy, \ie the fact that each $\eh$-cover of $X$ has a refinement $U \to X' \to X$ where $U$ is \'etale over $X'$ and $X'$ is a composition of blow ups along smooth centers. The latter is proven in \cite[Cor.2.6]{geisser}.

b) We proceed in two steps. If $C$ is bounded below, this follows from a) by comparing two (convergent) hypercohomology spectral sequences. In general, we may consider the good lower truncations of $C$, and the isomorphism follows from the bounded below case by comparing two Milnor exact sequences.
\end{proof}

As in \cite[Th. 4.1.10]{V} and \cite[Th. 14.20]{VL} for the cdh-topology we get:

\begin{propose}\label{peh} Let $C\in \DM_{-,\et}^\eff$. Then, for any $X\in \Sch(k)$
and any $q\in\Z$ we have
\[\Hom_{\DM_{-,\et}^\eff}(M_{\et}(X),C[q])\iso H^q_\eh(X,C).\]
In particular, if $X$ is smooth then $H^q_\et(X,C)\simeq  H^q_\eh(X,C)$.
\end{propose}

(See \cite[Th. 4.3]{geisser} for a different proof of the second statement for $C=\Z_\et (n)$.)

\begin{proof} We adapt the proof in \cite[Th. 14.20]{VL} to the $\eh$-topology. Denote by $\Z (X)$ the  \'etale sheaf [associated to] $U\mapsto  \Z[\Hom_{\Sm(k)}(U,X)]$. By considering the exact forgetful functor $\omega :\EST\to\ES$ as in \S {3.11}, we see that $\Z (X)\into \omega L(X)$ is a subsheaf. In $D^-(\EST)$ we have a canonical map $L(X)\to M_\et (X)\df RC_\et L(X)$ (by the adjunction in Theorem \ref{pD.2}). By forgetting transfers and composing we then get maps
 \[\Hom_{\DM_{-,\et}^\eff}(M_{\et}(X),C[q])\to\Hom_{D^-(\ES)}(\Z (X),\omega C[q])\iso\ H^q_\et (X,C) \]
 whose composition is an isomorphism for $X$ smooth by \cite[Ex. 6.25]{VL} (the second map is always an isomorphism, \cf \cite[Lemma 12.12]{VL}).
 By $\eh$ sheafification we also get maps
   \[ \Hom_{D^-(\ES)}(\Z (X),\omega C[q])\by{\epsilon^*}\Hom_{D^-(\EH)}(\Z (X)_{\eh},C_\eh[q])\iso H^q_\eh(X,C)\]
where the second map is clearly an isomorphism (see \cite[Lemma 3.1]{geisser}) and the first map is an isomorphism for $X$ smooth by 
Lemma \ref{ehs} b).

In general, for any $X\in \Sch(k)$ the claimed isomorphism follows  as in \cite[Prop. 3.1]{geisser} from blow-up induction (Lemma \ref{lbl} below).\end{proof}
 
\subsection{Blow-up induction} 
We just used the following lemma:

\begin{lemma} \label{lbl} Let $\cA$ be an abelian category.\\
a) Let $\cB\subseteq \cA$ be a thick subcategory and $H^*:\Sch(k)^\op\to \cA^{(\N)}$ a functor with the following property: 
\begin{quote} $(\dagger)$ given an abstract blow-up  as in \S \ref{blowups}, we have a long exact sequence
\[\dots\to H^i(X)\to H^i(\tilde X)\oplus H^i(Z)\to H^i(\tilde Z)\to H^{i+1}(X)\to\dots\]
\end{quote}
Let $n\ge 0$, and assume that $H^i(X)\in \cB$ for $i\le n$ and $X\in \Sm(k)$. Then $H^i(X)\in \cB$ for $i\le n$ and all $X\in \Sch(k)$.\\
b) Let $H_1^*,H_2^*$ be two functors as in a) and $\phi^*:H_1^*\to H_2^*$ be a natural transformation. Let $n\ge 0$, and suppose that $\phi^i_X$ is an isomorphism for all $X\in \Sm(k)$ and $i\le n$. Then $\phi^i_X$ is an isomorphism for all $X\in \Sch(k)$ and $i\le n$.\\
We get the same statements as a) and b) by replacing ``$i\le n$" by ``$i\ge n+\dim X$".
\end{lemma}

\begin{proof} Induction on $\dim X$ in two steps: 1) if $X$ is integral, choose a resolution of singularities $\tilde X\to X$; 2) in general, if $Z_1,\dots, Z_r$ are the irreducible components of $X$, choose $\tilde X=\coprod Z_i$ and $Z=\bigcup_{i\ne j} Z_i\cap Z_j$.
\end{proof}

\begin{examples}\label{ex11.1} 1) Thanks to \cite[Th. 4.1.10]{V} and Proposition \ref{peh}, cdh
and $\eh$ cohomology  have the property $(\dagger)$ (here $\cA=$ abelian groups).
(See also \cite[Prop. 3.2]{geisser}.)

2) \'Etale cohomology with torsion coefficients  has the property $(\dagger)$ by \cite[Prop.
2.1]{pw} (recall that the proof of \loccit relies on the proper base change theorem).
\end{examples}

Here is a variant of Lemma \ref{lbl}:

\begin{propose}\label{pdescent} a) Let $X\in \Sch(k)$ and $\Xs \to X$ be a hyperenvelope in the
sense of Gillet-Soul\'e \cite[1.4.1]{gs}. Let $\tau = \cdh$ or $\eh$. Then, for any (bounded
below) complex of sheaves $C$ over $\Sch(X)_\tau$, the augmentation map
\[H^*_\tau(X,C)\to H^*_\tau(\Xs,C)\]
is an isomorphism.\\
b) Suppose that $X_0$ and $X_1$ are smooth and $\cF$ is a homotopy invariant Nisnevich  (if
$\tau=\cdh$) or \'etale  (if $\tau=\eh$) sheaf with transfers. Then we have
\[\cF_\tau(X) = \ker(\cF(X_0)\to \cF(X_1)).\]
\end{propose}

\begin{proof} a) By \cite[Lemma 12.26]{VL}, $\Xs$ is a proper $\tau$-hypercovering (\cf
\cite[p. 46]{bondarko}). Therefore the proposition follows from the standard theory of
cohomological descent.

b) Let us take  $C=\cF_\tau[0]$.  By a) we have a descent spectral sequence which gives a short exact sequence
\[0\to\cF_\tau(X)\to \cF_\tau(X_0)\to \cF_\tau(X_1)\]
and the conclusion now follows from \ref{peh}.
\end{proof}

\subsection{$\LA{i}(X)$ for $X$ singular} 
The following is a general method for computing the $1$-motivic homology of $\LAlb^\Q (X)$:

\begin{propose} \label{cdhex}
For $X\in \Sch(k)$ consider cdh cohomology groups $\HH^i_\cdh(X,\pi^*(N))_{\sQ}$  with coefficients in a $1$-motive up to isogeny $N$ where $\pi:\Sch(k)_\cdh\to \Sch(k)_\Zar$ is the canonical map from the cdh site to the big Zariski site. Then we have short exact sequences,
for all $i\in \Z$
\begin {multline*}
0\to \Ext^1 (\LA{i-1}^\Q(X),N)\to \HH^i_\cdh(X,\pi^*(N))_{\sQ}\\
\to\Hom (\LA{i}^\Q(X),N)\to 0
\end{multline*}
\end{propose}

\begin{proof}
For any 1-motive $N$ we have a spectral sequence
\begin{equation}\label{eqss}
E^{p,q}_2 = \Ext^p (\LA{q}(X),N)\ \implies\  \EExt^{p+q}
(\LAlb(X), N)
\end{equation} 
yielding the following short exact sequence
\begin{multline*}
0\to \Ext^1 (\LA{i-1}^\Q(X),N)\to \EExt^{i} (\LAlb^\Q(X),N)\to \\
\Hom (\LA{i}^\Q(X),N)\to 0
\end{multline*}
because of Proposition~\ref{iso1}. By adjunction we
also obtain
$$\EExt^{i} (\LAlb^\Q(X),N) = \Hom (\LAlb^\Q(X),N[i])\cong
\Hom (M (X),\Tot N[i]).$$

Now from \cite[Thm. 3.2.6 and Cor. 3.2.7]{V}, for $X$ smooth
we have
\[\Hom (M (X),\Tot N[i])\cong \HH^i_\Zar(X, N)_{\sQ}.\]

As $k$ is of characteristic $0$, for $X$ arbitrary we get the same
isomorphism with cdh hypercohomology by \cite[Thm. 4.1.10]{V}.
\end{proof}

The following proposition folllows readily by blow-up induction (Lem\-ma \ref{lbl}) from Corollary \ref{HLAlb} and the exact sequences \eqref{exres}:

\begin{propose}\label{c3.1} For any $X\in \Sch(k)$ of dimension $d$, we
have\\
a) $\LA{i}(X) = 0$ if   $i <0$.\\
b) $\LA{0}(X) = [\Z[\pi_0(X)]\to 0]$.\\
c) $\LA{i}(X) = 0$ for $i >\max(2,d+1)$.\\
d) $\LA{d+1}(X)$ is a group of multiplicative type.
\qed
\end{propose}

\subsection{The cohomological 1-motives $\RA{i}(X)$}

If $X\in \Sch(k)$, we quote the following variant of Proposition \ref{cdhex}:

\begin{lemma} \label{cdhexp}
Let $N\in \M\otimes\Q$ and $X\in \Sch(k)$. We have a short exact
sequence, for all $i\in \Z$
\begin {multline*}
0\to \Ext (N, \RA{i-1}(X))\to \HH^i_\cdh(X,\pi^*(N^*))_{\sQ}\\\to \Hom (N, \RA{i}(X))\to 0
\end{multline*}
here $\pi : \Sch(k)_\cdh\to \Sch(k)_\Zar$ and $N^*$ is the
Cartier dual.
\end{lemma}
\begin{proof} The spectral sequence
$$E^{p,q}_2 = \Ext^p (N, \RA{q}(X))\ \implies\  \EExt^{p+q}
(N, \RPic (X))$$ yields the following short exact sequence
\begin{multline*}
0\to \Ext (N, \RA{i-1}(X))\to \EExt^{i} (N,\RPic (X))\to \\
\Hom (N, \RA{i}(X))\to 0
\end{multline*}
and by Cartier duality, the universal property and \cite[Thm.
4.1.10]{V} we obtain:
\begin{multline*}
\EExt^{i} (N,\RPic (X)) = \Hom (N,\RPic (X)[i])\cong \Hom (\LAlb
(X), N^*[i])=\\ \Hom (M (X),N^*[i])\cong \HH^i_\cdh(X,
\pi^*(N^*))_{\sQ}.
\end{multline*}
\end{proof}

On the other hand, here is a dual to Proposition \ref{c3.1}:

\begin{propose}\label{c3.1*} For any $X\in \Sch(k)$ of dimension $d$, we
have\\
a) $\RA{i}(X) = 0$ if   $i <0$.\\
b) $\RA{0}(X) = [0\to \Z[\pi_0(X)]^*]$.\\
c) $\RA{i}(X) = 0$ for $i >\max(2,d+1)$.\\
d) $\RA{d+1}(X)$ is discrete.
\qed
\end{propose}

\subsection{Borel-Moore variants}

\begin{defn}\label{d11.5} For $X\in \Sch(k)$, we denote by $\pi_0^c(X)$ the disjoint union of $\pi_0(Z_i)$
where $Z_i$ runs through the \emph{proper} connected components of $X$: this is the
\emph{scheme of constants with compact support}.
\end{defn}

\begin{propose}\label{c3.1c} Let $X\in \Sch(k)$ of dimension $d$. Then:\\
a) $\LA{i}^c(X) = 0$ if   $i <0$.\\
b) $\LA{0}^c(X) = [\Z[\pi_0^c(X)]\to 0]$. In particular, $\LA{0}^c(X)=0$ if no connected
component is proper.\\ 
c) $\LA{i}^c(X) = 0$ for $i >\max(2,d+1)$.\\
d) $\LA{d+1}^c(X)$ is a group of multiplicative type.
\end{propose}

\begin{proof} If $X$ is proper, this is Proposition \ref{c3.1}. In
general, we may choose a compactification $\bar X$ of $X$; if $Z = \bar X -X$, with $\dim
Z<\dim X$, the claim follows inductively by the long exact sequence \eqref{loc}.
\end{proof}

We leave it to the reader to formulate the dual of this proposition for $\RPic^c(X)$.

\section{1-motivic homology and cohomology of curves}\label{scurves}

We compute here the motives of Section \ref{6} for curves. Since we are dealing with curves, the following results hold in arbitrary characteristic.

\subsection{``Chow-K\"unneth" decomposition for a curve} Note that for any curve $C$, the map $a_C$ is an isomorphism by Proposition \ref{cd2}. Moreover, since the category of 1-motives up to isogeny is of cohomological dimension 1 (see Prop.~\ref{iso1}), the complex $\LAlb^\Q(C)$ can be represented by a complex with zero differentials. Using Proposition \ref{c3.1} c), we then have:

\begin{cor} If $C$ is a curve then the motive $M(C)$ decomposes in $\DM_\gm^\eff\otimes\Q$ as
\[M (C) \simeq M_0 (C) \oplus M_1 (C)\oplus M_2 (C)\] 
where $M_i (C)\df \Tot\LA{i}^\Q(C)[i]$.
\end{cor}

\subsection{$\LA{i}$ and $\RA{i}$ of curves} Here we shall complete the computation of
Proposition \ref{c3.1} in the case of a semi-normal curve $C$ (compare Lemma \ref{l12.3}).

Let $\tilde C$ denote the normalisation of
$C$. Let $\bar C$ be a smooth compactification of $\tilde C$ so
that $F = \bar C -\tilde C$ is a finite set of closed points.
Consider the following cartesian square
\[ \begin{CD}
\tilde S @>>> \tilde C\\
@V{}VV  @V{}VV  \\
S@>>> C
\end{CD}\]
where $S$ denotes the singular locus. Let $\bar S$ denote $\tilde
S$ regarded in $\bar C$. 

\begin{thm}\label{lasc}
 Let $C$, $\tilde C$, $\bar C$, $S$, $\tilde S$, $\bar S$ and
$F$ be as above. Then
$$\LA{i}(C) =
\begin{cases}
\relax [\Z[\pi_0(C)]\to 0]& \text{if $i = 0$}\\
\relax [\Div_{\bar S/S}^0(\bar C, F)\by{u}\Pic^0(\bar C, F)] & \text{if
$i= 1$}\\ 
\relax [0\to\NS_{\tilde C/k}^*] & \text{if $i= 2$}\\
0 & \text{otherwise}
\end{cases}$$
where: $\Div_{\bar S/S}^0(\bar C, F) =  \Div_{\tilde S/S}^0(\tilde C)$ here is the free group
of degree zero divisors generated by $\tilde S$ having trivial push-forward on $S$ and the map
$u$ is the canonical map (\cf \cite[Def. 2.2.1]{BSAP}); $\NS_{\tilde C/k}$ is the sheaf
associated to the free abelian group on the proper irreducible components of $\tilde C$.
In particular, $\LA{1}(C) = \Pic^{-}(C)$.
\end{thm}

\begin{proof} We use the long exact sequence \eqref{exres}
\[\dots\to\LA{i}(\tilde S)\to \LA{i}(\tilde C)\oplus \LA{i}(S)\to \LA{i}(C)\to\LA{i-1}(\tilde
S)\to\dots\]

Since $S$ and $\tilde S$ are $0$-dimensional we have $\LA{i}(\tilde S)=
\LA{i}(S)=0$ for $i>0$, therefore
\[\LA{i}(C)= \LA{i}(\tilde C) \text{ for } i\ge 2\]
and by~\ref{HLAlb} we get the claimed vanishing and
description of $\LA{2}(C)$. For $i=0$ see Corollary~\ref{c3.1}.
If $i=1$ then $\LA{1}(C)$ is here represented as an element of $\Ext ([\Lambda \to 0] ,
\LA{1}(\tilde C))$ where $\Lambda \df \ker (\Z[\pi_0(\tilde S)]\allowbreak\to
\Z[\pi_0(\tilde C)]\oplus \Z[\pi_0(S)]$.
Recall, see~\ref{HLAlb}, that $\LA{1}(\tilde C)=
[0\to \cA_{\tilde C/k}^0]$ thus $\Ext (\Lambda , \LA{1}(\tilde
C))= \Hom_k (\Lambda, \cA_{\tilde C/k}^0)$ and
$$\LA{1}(C) =[\Lambda \by{u} \cA_{\tilde C/k}^0].$$
Now $\Lambda = \Div_{\bar S/S}^0(\bar C, F)$,
$\cA_{\tilde C/k}^0 =\Pic^0(\bar C, F)$ and the map $u$ is induced
by the following canonical map.

Consider $\varphi_{\tilde C} : \tilde C \to \Pic (\bar C, F)$
where $\varphi_{\tilde C}(P)\df (\cO_{\bar C}(P), 1)$ yielding
$\cA_{\tilde C/k}= \Pic (\bar C, F)$ and such that
$$0\to \Div_F^0(\bar C)^*\to \Pic^0(\bar C, F)\to \Pic^0(\bar C)\to 0.$$
Thus $\LA{1}(\tilde C)= [0\to \Pic^0(\bar C, F)]$. Note that
$\Z[\pi_{0}(\tilde S)]=\Div_{\tilde S}(\tilde C) = \Div_{\bar
S}(\bar C, F)$, the map $\Z[\pi_{0}(\tilde S)]\to
\Z[\pi_{0}(\tilde C)]$ is the degree map and the following map
$\Z[\pi_{0}(\tilde S)]\to \Z[\pi_{0}(S)]$ is the proper
push-forward of Weil divisors, \ie $\Lambda = \Div_{\bar
S/S}^0(\bar C, F)$. The map $\varphi_{\tilde C}$ then induces the
mapping $u\in \Hom_k (\Lambda, \Pic^0(\bar C, F))$ which also is
the canonical lifting of the universal map $D\mapsto\cO_{\bar C}(D)$ as the
support of $D$ is disjoint from $F$ (\cf \cite[Lemma 3.1.3]{BSAP}).
\end{proof}

\begin{remark} Note that $\RA{1}(C)= \LA{1}(C)^*$ is Deligne's motivic cohomology
$H^1_m(C)(1)$ of the singular curve $C$ by \cite[Prop.
3.1.2]{BSAP}. In fact, $\Pic^{-}(C)^*=\Alb^{+}(C) \cong H^1_m(C)(1)\cong \Pic^{+}(C) =\Alb^{-}(C)^*$
for a curve $C$.
Thus $\LA{1}(C)$ also coincides with the homological
Albanese 1-motive $\Alb^-(C)$. The $\LA{i}(C)$ also coincide with Lichtenbaum-Deligne motivic
homology $h_i (C)$ of the curve $C$, \cf \cite{LI}.
\end{remark}

\begin{cor}\label{rasc}
 Let $C$ be a curve, $C'$ its seminormalisation, $\bar C'$  a compactification of $C'$, and $F = \bar C' - C'$. Let further $\tilde C$ denote the normalisation of $C$. Then
$$\RA{i}(C) =
\begin{cases}
\relax [0\to \G_m[\pi_0(C)]]& \text{if $i = 0$}\\
\relax [\Div_{F}^0(\bar C')\to\Pic^0(\bar C')] & \text{if $i= 1$}\\
\relax [\NS (\tilde C)\to 0] & \text{if $i= 2$}\\
0 & \text{otherwise}
\end{cases}$$
where $\NS (\tilde C)= \Z [\pi_0^c(\tilde C)]$ and $\pi_0^c(\tilde
C)$ is the scheme of proper constants. In particular, $\RA{1}(C)\cong \Pic^{+}(C)$.
\end{cor}

\subsection{Borel-Moore variants}

\begin{thm} \label{bmlac}
Let $C$ be a smooth curve, $\bar C$ a smooth compactification of
$C$ and $F =\bar C - C$ the finite set of closed points at
infinity. Then
$$\LA{i}^c(C) =
\begin{cases}
\relax [\Z[\pi_0^c(C)]\to 0]& \text{if $i = 0$}\\
\relax [\Div_F^0(\bar C)\to\Pic^0(\bar C)] & \text{if $i= 1$}\\
\relax [0\to\NS_{\bar C/k}^*] & \text{if $i= 2$}\\
0 & \text{otherwise}
\end{cases}$$
where $\NS (\bar C)=\Z [\pi_0 (\bar C)]$ and $\pi_0^c(C)$ is the
scheme of proper constants.
\end{thm}
\begin{proof} It follows from the distinguished triangle
\[\begin{CD}
\LAlb (F)  @>>> \LAlb (\bar C)\\
{\scriptstyle +1}\nwarrow &&\swarrow\\
&\LAlb^c(C)&
\end{CD}\]
and Corollary~\ref{HLAlb}, yielding the claimed description:
$\LA{0}^c(C)= \coker (\LA{0}(F) \to\LA{0}(\bar C))$ moreover we
have
$$[\Div_F^0(\bar C)\to 0]= \ker (\LA{0}(F) \to\LA{0}(\bar C))$$ and the
following extension
$$0\to \LA{1}(\bar C)\to \LA{1}^c(C)\to
[\Div_F^0(\bar C)\to0]\to 0.$$ Finally,  $\LA{i}(\bar C)=
\LA{i}^c(C)$ for $i\geq 2$.
\end{proof}

\begin{cor} Let $C$ be a smooth curve, $\bar C$ a smooth compactification of
$C$ and $F =\bar C - C$ the finite set of closed points at
infinity. Then
$$\RA{i}^c(C) =
\begin{cases}
\relax [0\to\G_m[\pi_0^c(C)]]& \text{if $i = 0$}\\
\relax [0\to\Pic^0(\bar C, F)] & \text{if $i= 1$}\\
\relax [\NS (\bar C)\to 0] & \text{if $i= 2$}\\
0 & \text{otherwise}
\end{cases}$$
where $\NS (\bar C)=\Z [\pi_0 (\bar C)]$ and $\pi_0^c(C)$ is the
scheme of proper constants.
\end{cor}

Here we have that $\RA{1}^c(C)=\RA{1}^*(C)$ is also the Albanese
variety of the smooth curve.

Note that $\LA{1}^c(C)$ (= $\Pic^+ (C) = \Alb^+(C)$ for curves,
see \cite{BSAP}) coincide with Deligne's motivic $H^1_m(C)(1)$ of
the smooth curve $C$. This is due to the Poincar\'e duality
isomorphism $M^c (C)= M (C)^*(1)[2]$ (\cf \cite[App. B]{hk}).

\section{Comparison with $\Pic^+,\Pic^-,\Alb^+$ and $\Alb^-$}\label{12}

 In this section, $k$ is of characteristic $0$ since we mainly deal with singular schemes (\cf \S \ref{s8.1}). We
want to study
$\LA{1}(X)$ and its variants in more detail. In particular, we show in Proposition \ref{p11.3a}
c) that it is always a Deligne
$1$-motive, and show in Corollaries \ref{c12.2.2} and \ref{c12.3} that, if $X$ is normal or
proper, it is canonically isomorphic to the
$1$-motive $\Alb^-(X)$ of \cite{BSAP}. Precise descriptions of $\LA{1}(X)$ are given in
Proposition \ref{pl1} and Corollary \ref{c12.2.1}.  

We also describe $\LA{1}^c(X)$ in Proposition \ref{c3.1d}; more precisely, we prove in Theorem
\ref{t12.9} that its dual $\RA{1}^c(X)$ is canonically isomorphic to $\Pic^0(\bar X,Z)/\cU$, where $\bar
X$ is a compactification of $X$ with complement $Z$ and $\cU$ is the unipotent radical of the
commutative algebraic group $\Pic^0(\bar X,Z)$. 
Finally, we prove in Theorem \ref{*=-}
that $\LA{1}^*(X)$ is abstractly isomorphic to the $1$-motive $\Alb^+(X)$ of \cite{BSAP}.

We also provide some comparison results between $\eh$ and \'etale cohomology for non smooth
schemes. 

\subsection{Torsion sheaves} The first basic result is a variant of \cite[Cor.
7.8 and Th. 10.7]{suvo}: it follows from Proposition \ref{peh} and Example
\ref{ex11.1} 2) via Lemma \ref{lbl} b).

\begin{propose}\label{p11.2} Let $C$ be a bounded below complex of torsion
\'etale sheaves on $\Sch(k)$. Then, for any $X\in \Sch$ and any $n\in\Z$,
$H^n_\et(X,C)\iso H^n_\eh(X,C)$.\qed
\end{propose}

(See \cite[Th. 3.6]{geisser} for a different proof.)

\subsection{Glueing lemmas} In this subsection, we recall some instances where a diagram of schemes is cocartesian, with proofs. We start with

\begin{lemma}\label{ldescent} Let $X\in \Sch(k)$ be normal connected, $p:X_0\to X$ a proper surjective map such that the restriction of $p$ to a suitable connected component $X'_0$ of $X_0$ is birational. Let $X_1\begin{smallmatrix}p_0\\\rightrightarrows\\p_1\end{smallmatrix} X_0$ be two morphisms such that $pp_0 = pp_1$ and that the induced map $\Psi:X_1\to X_0\times_X X_0$ is proper surjective. Let $Y\in \Sch(k)$ and let $f:X_0\to Y$ be such that $fp_0=fp_1$.
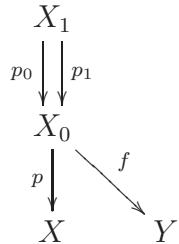
\[\xymatrix{
X_1
\ar@<-.7ex>[d]_{p_0} \ar@<.7ex>[d]^{p_1}\\
X_0
\ar[d]_p\ar[dr]^f\\
X&Y
}\]
Then there exists a unique morphism $\bar f:X\to Y$ such that $f=\bar f p$.
\end{lemma}

\begin{proof} Since $\Psi$ is proper surjective, the
hypothesis is true by replacing $X_1$ by $X_0\times_X X_0$, which we shall assume
henceforth. Let $x\in X$ and $K=k(x)$. Base-changing by the morphism $\Spec K\to
X$, we find (by faithful flatness) that $f$ is constant on $p^{-1}(x)$. Since $p$
is surjective, this defines $\bar f$ as a set-theoretic map, and this map is
continuous for the Zariski topology because $p$ is also proper.

It remains to show that $\bar f$ is a map of locally ringed spaces. Let $x\in X$,
$y=\bar f(x)$ and $x'\in p^{-1}(x)\cap X'_0$. Then $f^\sharp(\cO_{Y,y})\subseteq
\cO_{X'_0,x'}$. Note that $X$ and $X'_0$ have the same function field $L$, and
$\cO_{X,x}\subseteq \cO_{X'_0,x'}\subseteq L$. Now, since $X$ is normal,
$\cO_{X,x}$ is the intersection of the valuation rings containing it. 

Let $\cO$ be such a valuation ring, so that $x$ is the centre of $\cO$ on $X$. By
the valuative criterion of properness, we may find $x'\in p^{-1}(x)\cap X'_0$ such
that $\cO_{X_0',x'}\subseteq \cO$. This shows that
\[\cO_{X,x} =\bigcap_{x'\in p^{-1}(x)\cap X'_0}\cO_{X_0',x'}\]
and therefore that $f^\sharp(\cO_{Y,y})\subseteq \cO_{X,x}$. Moreover, the
corresponding map $\bar f^\sharp:\cO_{Y,y}\to \cO_{X,x}$ is local since $f^\sharp$
is. 

(Alternatively, observe that $f$ and its topological factorisation induce a map
\[f^\#: \cO_Y\to f_*\cO_{X_0}\simeq \bar f_* p_*\cO_{X_0}= \bar f_*\cO_X.)\]
\end{proof}

\begin{lemma}\label{ldescent1} Let 
\begin{equation}\label{eq13.3}
\begin{CD}
\tilde Z@>i'>> \tilde X\\
@Vp'VV @VpVV\\
Z@>i>> X
\end{CD}
\end{equation}
be a commutative square in $\Sch(k)$, with $X$ integral, $p$ a resolution of $X$, $Z$ the largest closed subset where $p$ is not an isomorphism and $\tilde Z=p^{-1}(Z)$. Then the diagram
\[\begin{CD}
\pi_0(\tilde Z)@>i'>> \pi_0(\tilde X)\\
@Vp'VV @VpVV\\
\pi_0(Z)@>i>> \pi_0(X)
\end{CD}\]
is cocartesian in the category of \'etale $k$-schemes.
\end{lemma}

Note that in this statement, $\pi_0(Z)$ and $\pi_0(\tilde Z)$ make sense thanks to Remark \ref{rinv}. 

\begin{proof} Let $p=rq$ be the Stein factorisation:
\[\tilde X\longby{q} X'\longby{r} X\]
so that $X'$ is the normalisation of $X$. It follows from Lemma \ref{normal} that $\pi_0(\tilde X)\iso \pi_0(X')$; on the other hand, if $Z'=r^{-1}(Z)$, the surjectivity of $\tilde Z\to Z'$ implies the surjectivity of $\pi_0(\tilde Z)\to \pi_0(Z')$.  This reduces us to the case where $\tilde X=X'$, \ie where $p$ defines the normalisation of $X$.

 Let $X=U\cup V$ be an open cover. The universal property of $\pi_0$ shows that the square
\[\begin{CD}
\pi_0(U\cap V)@>>> \pi_0(U)\\
@VVV @VVV\\
\pi_0(V)@>>> \pi_0(X)
\end{CD}\]
is cocartesian in the category of \'etale $k$-schemes. Since $X$ is separated, an induction plus an easy $4$-dimensional diagram chase reduce us to the case where $X$, hence also $Z$, $\tilde X$ and $\tilde Z$, are affine. (Alternately, we could do the following reasoning sheafwise, with more horrible notation.)

So assume $X=\Spec A$, $\tilde X=\Spec \tilde A$ and $Z$ is defined by the conductor ideal $I\subset A$ ($I$ is the annihilator of $\tilde A/A$). Let $\pi_0(X)=\Spec F$ and $\pi_0(\tilde X) = \Spec \tilde F$, so that $F\subset A$ and $\tilde F\subset \tilde A$. If $\tilde f\in \tilde F$ is such that its image in $\tilde A/I\tilde A$ comes from $A/I$, then
\[\tilde f = a +i\]
with $a\in A$ and $i\in I\tilde A\subset A$, so $\tilde f\in A\cap \tilde F=F$.
\end{proof}

\begin{lemma}\label{normal2} Let $f:X'\to X$ be a birational morphism, where $X\in\Sch(k)$ is normal and connected. Then, for any \'etale morphism $\phi:U\to X$ with $U$ connected, the base change $f_U : X'_U=U\times_X X'\to U$ is birational and  induces an isomorphism $\pi_0(X'_U)\iso\pi_0(U)$.
\end{lemma}

(This lemma is false if $X$ is not normal, \cf \cite[Exp. I, \S 11 a)]{SGA1}.)

\begin{proof}Since $X$ is normal, $U$ is normal \cite[Exp. I, 9.10]{SGA1}, hence irreducible.  By  \loccit, 9.2, $X'_U$ is reduced so  $f_U$ is  birational by \cite[6.15.4.1]{ega4}. We are thus reduced to $U=X$ for the  second statement, which follows from Lemma \ref{normal}. \end{proof}

\subsection{Discrete sheaves} Recall that a discrete sheaf is an \'etale sheaf associated to a discrete group scheme (see Definition \ref{d3.1}). 

\begin{lemma}\label{l11.2} If $\cF$ is discrete, then\\
a)  $\cF\iso \cF_\eh$. More precisely, for
any $X\in \Sch$, $\cF(\pi_0(X))\iso \cF(X)\iso \cF_\eh(X)$.\\
 b) If $f:Y\to X$ induces an isomorphism on $\pi_0$ after any \'etale base change, then $\cF_{|X}\iso f_*(\cF_{|Y})$. This applies, for example, to the situation of Lemma \ref{normal2}.\end{lemma}

\begin{proof} a) Clearly it suffices to prove that $\cF(\pi_0(X))\iso
\cF_\eh(X)$ for any $X\in \Sch$. We may assume $X$ reduced by Remark \ref{rinv}.   In the situation of \S \ref{blowups}, we have a commutative
diagram of exact sequences
\[\begin{CD}
0@>>> \cF_\eh(X)@>>> \cF_\eh(\tilde X)\oplus \cF_\eh(Z)@>>> \cF_\eh(\tilde Z)\\
&&@AAA @AAA @AAA\\
0@>>> \cF(\pi_0(X))@>>> \cF(\pi_0(\tilde X))\oplus \cF(\pi_0(Z))@>>> \cF(\pi_0(\tilde Z)).
\end{CD}\]
The lower sequence is exact by Lemma \ref{ldescent1}. 
The proof then goes exactly as the one of Proposition \ref{p11.2}. b) follows immediately from a).\end{proof}

It is well-known that $H^1_\et(X,\cF)=0$ for any normal scheme $X\in \Sch$ if
$\cF$ is constant and torsion-free (\cf \cite[IX, Prop. 3.6 (ii)]{sga4}). The following lemma
shows that this is also true for the $\eh$ topology.

\begin{lemma}[\cf \protect{\cite[Ex. 12.31 and 12.32]{VL}}]\label{l11.3} Let $\cF$ be a constant tor\-sion-free sheaf on $\Sch(k)$.\\
a) For any $X\in \Sch$, $H^1_\eh(X,\cF)$ is torsion-free. It is finitely generated if $\cF$ is a
lattice.\\ 
 b) Let $f:\tilde X\to X$ be a  finite flat morphism. Then $H^1_\eh(X,\cF)\to
H^1_\eh(\tilde X,\cF)$ is injective.\\
c) If $X$ is  normal, then $H^1_\eh(X,\cF)=0$.
\end{lemma}

\begin{proof} a) The first assertion follows immediately from Lemma \ref{l11.2} (consider the
exact sequence of multiplication by $n$ on $\cF$). The second assertion follows by blow-up induction from the fact that $H^1_\eh(X,\cF)=0$ if $X$ is
smooth, by Proposition \ref{peh}. 

b) The Leray spectral sequence yields an injection
\begin{equation}\label{eq13.2}H^1_\eh(X,f_*\cF)\into H^1_\eh(\tilde X,\cF)
\end{equation} 
The theory of trace \cite[XVII, Th.
6.2.3]{sga4} provides $\cF$, hence $\cF_\eh$, with a morphism
$Tr_f:f_*\cF\to \cF$ whose composition with the natural morphism $\cF\to f_*\cF$ is (on each connected component
of $X$) multiplication by some nonzero integer. This shows that the kernel of $H^1_\eh(X,\cF)\to
H^1_\eh(X,f_*\cF)$ is torsion, hence $0$ by a).

c) This is true if $X$ is smooth (see proof of a)). In general, if $f: \tilde X\to X$ is a desingularisation of $X$, we have $f_*\cF=\cF$ by Lemma \ref{l11.2} b). Then \eqref{eq13.2} becomes an injection
$H^1_\eh(X,\cF)\into H^1_\eh(\tilde X,\cF)$.
\end{proof}

The following is a version of \cite[Lemma 5.6]{weibel}:

\begin{lemma} \label{lweibel} Let $f:\tilde X\to X$ be a finite 
morphism. Denote by $i: Z\into X$ a closed subset and let $\tilde Z=f^{-1}(Z)$. Assume that $f$ induces
an isomorphism $\tilde X - \tilde Z \iso X-Z$. Then, for any discrete sheaf $\cF$, we have a long 
exact sequence:
\begin{multline*}
\dots\to H^i_\et(X,\cF)\to H^i_\et(\tilde X,\cF)\oplus H^i_\et(Z,\cF)\\
\to H^i_\et(\tilde Z,\cF)\to H^{i+1}_\et(X,\cF)\to \dots
\end{multline*}
\end{lemma}

\begin{proof} Let $g:\tilde Z\to Z$ be the induced map. Then $f_*,i_*$ and $g_*$ are exact for
the \'etale topology. Thus it suffices to show that the sequence of sheaves
\[0\to \cF\to f_*f^*\cF\oplus i_*i^*\cF\to (ig)_*(ig)^*\cF\to 0\]
is exact. The assertion is local for the \'etale topology, hence we may assume that $X$ is
strictly local, hence $\cF$ constant.  Then $Z$ is also strictly local while $\tilde X$  and $\tilde Z$ are finite disjoint unions of the same number of strictly local schemes,
thus the statement is obvious. 
\end{proof}

We can now prove:

\begin{propose}\label{l11.2.6} For any $X\in \Sch(k)$ and any discrete sheaf $\cF$, the map
$H^1_\et(X,\cF)\to H^1_\eh(X,\cF)$ is an isomorphism. If $X$ is normal, $H^2_\et(X,\cF)\to H^2_\eh(X,\cF)$ is injective.
\end{propose}

\begin{proof} Consider the exact sequence
\[0\to \cF_\tors\to \cF\to \cF/\cF_\tors\to 0\]
where $\cF_\tors$ is the torsion subsheaf of $\cF$. The 5 lemma, Lemma \ref{l11.2} a) and
Proposition
\ref{p11.2} reduce us to the case where $\cF$ is torsion-free. As a discrete \'etale sheaf over
$\Sch(k)$,
$\cF$ becomes constant over some finite Galois extension of $k$: a Hochschild-Serre argument
then reduces us to the case where $\cF$ is \emph{constant and torsion-free}.

Let $f:\tilde X\to X$ be the normalisation of $X$, and take for
$Z$ the non-normal locus of $X$ in Lemma \ref{lweibel}. The result now follows from comparing
the exact sequence of this lemma with the one for $\eh$ topology, and using Lemma \ref{l11.3}
c).

For the injectivity statement on $H^2$, we reduce to $\cF$ torsion-free by Proposition \ref{p11.2} (for $n=2$). Then $H^*_\et(X,\cF\otimes \Q)=0$ by (a small generalisation of) \cite[2.1]{deninger}: compare end of proof of Proposition \ref{ptorsion}. Then  the conclusion follows from a diagram chase involving the isomorphisms for $H^1$ and the long cohomology exact sequences associated to $0\to \cF\to \cF\otimes\Q\to \cF\otimes\Q/\Z\to 0$.
\end{proof}

\begin{cor} The exact sequence of Lemma \ref{lweibel} holds up to $i=1$ for a general abstract
blow-up.\qed
\end{cor}

\subsection{Normal schemes} 

If $G$ is a commutative $k$-group scheme, the associated presheaf $\uG$ is an \'etale sheaf on
reduced $k$-schemes of finite type. However, $\uG(X)\to \uG_\eh(X)$ is not an isomorphism in
general if $X$ is not smooth. Nevertheless we have some nice results in Lemma \ref{leh} and Theorem \ref{tnormal} below. 

\begin{lemma}\label{leh} a) If $X$ is reduced, then the map 
\begin{equation}\label{eq12.2}
\uG(X)\to \uG_\eh(X)
\end{equation}
 is injective for any semi-abelian $k$-scheme $G$.\\
b) If $X$ is proper and $G$ is a torus, the maps $\uG(\pi_0(X))\to\uG (X)\to \uG_\eh(X)$ are isomorphisms. 
If moreover $X$ is reduced, \eqref{eq12.2} is an isomorphism.
\end{lemma}

\begin{proof} a) Let $Z_i$ be the irreducible components of $X$ (with their reduced structure),
and for each
$i$ let
$p_i:\tilde Z_i\to Z_i$ be a resolution of singularities. We have a commutative diagram
\[\begin{CD}
\uG_\eh(X)@>>> \bigoplus \uG_\eh(Z_i)@>(p_i^*)>>\bigoplus \uG_\eh(\tilde Z_i)\\
@AAA @AAA @AAA\\
\uG(X)@>>> \bigoplus \uG(Z_i)@>(p_i^*)>>\bigoplus \uG(\tilde Z_i).
\end{CD}\]

The bottom horizontal maps are injective; the right vertical map is an isomorphism by
Proposition \ref{peh}. The claim follows.

For b), same proof as for Lemma \ref{l11.2} a). (The second statement of b) is true because $\uG(\pi_0(X))\iso \uG(X)$ if $X$ is proper and reduced.)
\end{proof}

As the $\eh$-topology is not subcanonical we need to be careful when considering a representable presheaf. Denote by $(\ \ )_\eh$ the functor that takes $X\in \Sch (k)$ to the representable $\eh$-sheaf $X_\eh$ \ie the $\eh$-sheaf associated to the presheaf $X(U)=\Hom_{\Sch (k)} (U,X)$.

We get here an analogue of Voevodsky's \cite[Prop. 3.2.10]{Vh}.
\begin{propose}\label{p13.4.2} Let $\Sch (k)^\eh$ be the category of representable $\eh$-sheaves. There is a functor
$(\ \ )_\sn : \Sch (k)^\eh\to \Sch (k)$ left adjoint to $(\ \ )_\eh : \Sch (k)\to \Sch (k)^\eh$  such that $(X_\eh)_\sn\in \Sch(k)$ is the semi-normalization of  $X\in \Sch (k)$. In particular, for any semi-normal scheme $X$ and any scheme $Y$ we have
\[\Hom_{\Sch (k)} (X, Y) \iso\Hom_{\Sch (k)^\eh} (X_\eh, Y_\eh) \]
\end{propose}
\begin{proof} Note that a universal homeomorphism is an $\eh$-covering. Therefore, adapting the proof of \cite[Th. 3.2.9]{Vh} to the $\eh$-topology we have that the category $\Sch(k)^\eh$ is the localization of $\Sch (k)$ at universal homeomorphisms. (Actually,  it is not difficult to see that the analogues of
3.2.7, 3.2.8 and 3.2.5 (3) in \cite{Vh} hold for the $\eh$-topology.) As a consequence we see that for any $\varphi : X_\eh \to Y_\eh$ there is a finite map $f: X'\to X$ which ``realises'' $\varphi$ (in the sense of Voevodsky \cite{Vh}) such that $f$ is a universal homeomorphism factorising the normalization of $X$ \ie the seminormalization.\end{proof}

\begin{lemma}\label{l11.2.1} Let $G$ be an affine group scheme and $f:Y\to X$ a proper surjective map with geometrically connected fibres. Then $\uG(X)\allowbreak\iso \uG(Y)$.  Moreover, for $X$ and $Y$ semi-normal we also have $\uG_\eh (X)\allowbreak \iso \uG_\eh (Y) $.
\end{lemma}

\begin{proof} Let $\phi:Y\to G$  be a morphism and $x$ a closed point of $X$. The restriction of $\phi$   to $f^{-1}(x)$ has a proper connected image; since $G$ is affine,  this image must be a closed point of $G$. Since closed points are dense in $X$, this shows that $\phi$ factors through $f$. The last claim follows from 
the Proposition \ref{p13.4.2} since $\Sch (X)^\eh $ is the category $\Sch (k)^\eh$ over $X_\eh$.  \end{proof}

The main result of this subsection is:

\begin{thm}\label{tnormal} Let $X$ be normal. Then, for any $\cF\in \Shv_1$, the map $\cF(X)\to \cF_\eh(X)$ is bijective and the map $H^1_\et(X,\cF)\to H^1_\eh(X,\cF)$ is injective (with torsion-free cokernel by Proposition \ref{p11.2}).
\end{thm}

\begin{proof} In several steps:

{\it Step 1.} The first result implies the second for a given sheaf $\cF$: let $\epsilon :\Sch(k)_\eh\to \Sch(k)_\et$ be the projection morphism. The associated Leray spectral sequence gives an injection
\[H^1_\et(X,\epsilon_*\cF_\eh)\into H^1_\eh(X,\cF_\eh).\]

But any scheme \'etale over $X$ is normal \cite[Exp. I, Cor. 9.10]{SGA1}, therefore $\cF\to \epsilon_*\cF_\eh$ is an isomorphism over the small \'etale site.

{\it Step 2.} Let $0\to \cF'\to \cF\to \cF''\to 0$ be a short exact sequence in $\Shv_1$. If the theorem is true for $\cF'$ and $\cF''$, it is true for $\cF$. This follows readily from  Step 1 and a diagram chase.

{\it Step 3.} In Step 2, if $\cF'$ is discrete and the theorem is true for $\cF$ and $\cF'$, then it is true for $\cF''$. This follows from a diagram chase involving the last part of Proposition \ref{l11.2.6}.

{\it Step 4.} Given the structure of $1$-motivic sheaves, Step 1 - 2 reduce us to prove that $\cF(X)\iso \cF_\eh(X)$ separately when $\cF$ is discrete, a torus or an abelian variety. The discrete case follows from Lemma \ref{l11.2} a).

{\it Step 5.} If $G$ is a torus, let $\pi:\tilde X\to X$ be a desingularisation of $X$. We have a commutative diagram
\[\begin{CD}
\uG_\eh(X)@>>>  \uG_\eh(\tilde X)\\
@AAA @AAA\\
\uG(X)@>>> \uG(\tilde X).
\end{CD}\]

Here the right vertical map is an isomorphism because $\tilde X$ is smooth and the bottom/up
horizontal maps are also isomorphisms by Lemma \ref{l11.2.1} applied to $\pi$ (Zariski's main
theorem). The result follows from the commutativity of the diagram.

{\it Step 6.} Let finally $G$ be an abelian variety. This time, it is not true in general that $\uG(X)\iso \uG(\tilde X)$ for a smooth desingularisation $\tilde X$ of $X$. However, we get the result from Proposition \ref{pdescent} b) and  Lemma \ref{ldescent}.
\end{proof}

\subsection{Some representability results}

\begin{propose}\label{p11.3} Let $\pi^X$ be the structural morphism of $X$ and $(\pi_*^X)^\eh$
the induced direct image  morphism on the $\eh$ sites. For any $\cF\in \HI_\et$, let
us denote the restriction of $R^q(\pi_*^X)^\eh \cF_\eh$ to $\Sm(k)$ by
$\uR^q\pi_*^X\cF$ (in other words, $R^q(\pi_*^X)^\eh \cF_\eh$ is the sheaf on
$\Sm(k)_\et$ associated to the presheaf $U\mapsto H^q_\eh(X\times U,\cF_\eh)$): it is an object
of $\HI_\et$. Then\\   
a) For any lattice $L$, $\uR^q\pi_*^X L$ is a
ind-discrete sheaf for all $q\ge 0$; it is a lattice for $q=0,1$.\\ 
b) For any torus $T$, $\uR^q\pi_*^X T$ is $1$-motivic for $q= 0,1$.
\end{propose}

\begin{proof} We apply Lemma \ref{lbl} a) in the following situation: $\cA=\HI_\et$,
$\cB=\Shv_0$, $H^i(X)=\uR^i\pi_*^XL$ in case a), $\cB=\Shv_1$, $H^i(X)=\uR^i\pi_*^X T$ in case b).The smooth case is trivial in a)  and the lattice
assertions follow from lemmas \ref{l11.2} and \ref{l11.3} a).  In b), the smooth case follows
from Proposition \ref{p3.3.1}.
\end{proof}

\subsection{$\LA{1}(X)$ and the Albanese schemes} We now compute the $1$-motive
$\LA{1}(X)=[L_1\to G_1]$ in important special cases. This is done in the following three
propositions; in particular, we shall show that it always ``is" a Deligne $1$-motive. Note
that, by definition of a $1$-motive with cotorsion, the pair $(L_1,G_1)$ is determined only up
to a \qi: the last sentence means that we may choose this pair such that $G_1$ is connected
(and then it is uniquely determined).

\begin{propose} \label{c3.1bis} Let $X\in \Sch(k)$. Then\\
a) $\cH_i(\LAlb(X))=0$ for $i<0$.\\
b) Let $\cF_X=\cH_0(\Tot\LAlb(X))$. Then $\cF_X$ corepresents the functor
\begin{align*}
\Shv_1&\to Ab\\
\cF&\mapsto \cF_\eh(X) \text{ (see Def. \ref{deh})}
\end{align*}
via the composition
\[\alpha^s M(X)\to \Tot\LAlb(X)\to \cF_X[0].\]
Moreover, we have an exact sequence, for any representative $[L_1\by{u_1}G_1]$ of $\LA{1}(X)$ :
\begin{equation}\label{eq11.1}
L_1\by{u_1} G_1\to \cF_X\to \Z[\pi_0(X)]\to 0.
\end{equation}
c) Let $\cA_{X/k}^\eh:=\Omega(\cF_X)$ (\cf Proposition \ref{pladj}). Then $\cA_{X/k}^\eh$ corepresents the functor \index{$\cA_{X/k}^\eh$}
\begin{align*}
{}^t\AbS&\to Ab\\
G&\mapsto \uG_\eh(X).
\end{align*}
Moreover we have an epimorphism
\begin{equation}\label{compalb}
 \cA_{X/k}^\eh\onto \cA_{X_\red/k}.
\end{equation}
d) If $X_\red$ is  normal, \eqref{compalb} is an isomorphism.
\end{propose}

\begin{proof} a) is proven as in Proposition \ref{c3.1} by blow-up induction (reduction to the
smooth case). If $\cF\in \Shv_1$, we have 
\[\Hom_{\DM_{-,\et}^\eff}(\alpha^sM(X),\cF)=\cF_\eh(X)\] 
by Propositions \ref{ptransf} and \ref{peh}. The latter group coincides with
$\Hom_{\Shv_1}(\cF_X,\cF)$ by \eqref{lalbuniv} and a), hence b); the exact sequence follows from Proposition \ref{p3.10}.  The sheaf $\cA_{X/k}^\eh$ clearly corepresents the said functor; the map then comes from the obvious
natural transformation in $G$: $G(X_\red)\to G_\eh(X)$ and its surjectivity follows from Lemma
\ref{leh} a), hence c). d) follows from Theorem \ref{tnormal} and the universal property of $\cA_{X/k}$.
\end{proof}

\begin{remark}\label{r11.1} One could christen
$\cF_X$ and
$\cA_{X/k}^\eh$ the
\emph{universal $1$-motivic sheaf} and the
\emph{$\eh$-Albanese scheme} of $X$.
\end{remark}

\begin{propose}\label{p11.3a}
a) The sheaves $\cF_X$ and $\cA_{X/k}^\eh$ have $\pi_0$ equal to $\Z[\pi_0(X)]$; in particular,
$\cA_{X/k}^\eh\in \AbS$.\\
b) In \eqref{eq11.1}, the composition $L_1\by{u_1}G_1\to \pi_0(G_1)$ is surjective.\\
c) One may choose $\LA{1}(X)\simeq [L_1\to G_1]$ with $G_1$ connected (in other words, $\LA{1}(X)$ is a Deligne $1$-motive).
\end{propose}

\begin{proof}
In a), it suffices to prove the first assertion for $\cF_X$: then it follows from
its universal property and Lemma \ref{l11.2} a). The second assertion of a) is obvious. 

b) Let $0\to L'_1\to G'_1\to \cF_X\to \Z[\pi_0(X)]\to 0$ be the normalised presentation of $\cF_X$ given by Proposition \ref{p3.1}. We have a commutative diagram
\[\begin{CD}
0@>>> L'_1@>>> G'_1@>>> \cF_X@>>> \Z[\pi_0(X)]@>>> 0\\
&&@VVV @VVV ||& &||&\\
0@>>> \overline{u_1(L_1)}@>\bar u_1>> \bar G_1@>>> \cF_X@>>> \Z[\pi_0(X)]@>>> 0\\
&&@AAA @AAA ||& &||&\\
&&L_1@>u_1>> G_1@>>> \cF_X@>>> \Z[\pi_0(X)]@>>> 0
\end{CD}\]
with $\overline{u_1(L_1)}=u_1(L_1)/F$ and $\bar G_1=G_1/F$, where $F$ is the torsion subgroup of $u_1(L_1)$. Indeed, $\Ext(G'_1,\overline{u_1(L_1)})=0$ so we get the downwards vertical maps as in the proof of Proposition \ref{p3.1}. By uniqueness of the normalised presentation, $G'_1$ maps onto $\bar G_1^0$. A diagram chase then shows that the composition
\[\overline{u_1(L_1)}\by{\bar u_1} \bar G_1\to \pi_0(\bar G_1)\]
is onto, and another diagram chase shows the same for $u_1$.

c) The pull-back diagram
\[\begin{CD}
L_1^0@>>> G_1^0\\
@VVV @VVV\\
L_1 @>>> G_1
\end{CD}\]
is a quasi-isomorphism in ${}_t\M^\eff$, thanks to b).
\end{proof}

\begin{lemma} \label{l1} Suppose that $k$ is algebraically closed. Let $[L_1\by{u_1} G_1]$ be the Deligne $1$-motive that lies in the \qi class of $\LA{1}(X)$, thanks to Proposition \ref{p11.3a} c), and let $L$ be a lattice.\\
a) We have an isomorphism
\[\Hom_{D^b(\M)}(\LAlb(X),L[1])\iso\Hom(L_1,L).\]
b) The map
\[
\Hom_{D^b(\M)}(\LAlb(X),L[j])\to  \Hom_{\DM_{\gm,\et}^\eff}(M_\et(X),L[j]) =H^j_\eh(X,L) 
\]
induced by $a_X$ \eqref{amapX} is an isomorphism for $j=0,1$ (see Prop. \ref{peh} for the last equality).
\end{lemma}

\begin{proof}
a) From the spectral sequence \eqref{eqss}, we get an exact sequence
\begin{multline*}
0\to \Ext^1(\LA{0}(X),L)\to \Hom_{D^b(\M)}(\LAlb(X),L[1])\\
\to\Hom(\LA{1}(X),L)\to
\Ext^2(\LA{0}(X),L).
\end{multline*}

Since the two Ext are $0$, the middle map is an isomorphism. Since $[L_1\to G_1]$ is a Deligne $1$-motive, $\Hom(\LA{1}(X),L)$ is isomorphic to $\Hom(L_1,L)$.

b) By blow-up induction (Lemma \ref{lbl}) we reduce to $X$ smooth. If $j=0$, the result is trivial; if $j=1$, it is also trivial because both sides are $0$ (by a) and Corollary \ref{HLAlb} for the left hand side).
\end{proof}

\begin{propose}\label{pl1}   Keep the above notation $\LA{1}(X)=[L_1\by{u_1} G_1]$.\\
a) We have an isomorphism 
\[L_1\simeq\shom(\uR^1\pi_*^X\Z,\Z)\]
(\cf Proposition \ref{p11.3}).\\ 
b) We have a canonical isomorphism
\[G_1/(L_1)_\Zar\iso (\cA_{X/k}^\eh)^0\]
where $(L_1)_\Zar$ is the Zariski closure of the image of $L_1$ in $G_1$ and
$\cA_{X/k}^\eh$ was defined in Proposition \ref{c3.1bis} ($(\cA_{X/k}^\eh)^0$ corepresents the functor $\SAb\ni G\mapsto G_\eh(X)$).
\end{propose}

\begin{proof} For the computations, we may assume $k$ algebraically closed.

Let $L$ be a lattice. By Lemma \ref{l1}, we have an isomorphism
\[
H^1_\eh(X,L) \simeq \Hom(L_1,L).
\]

 This gives a), since we obviously have  
$H^1_\eh(X,L)=H^0_\et(k,\uR^1\pi_*^X L)\allowbreak=\uR^1\pi_*^X\Z\otimes L$ by
Proposition \ref{p11.3} a).

b)  follows directly from the definition of $\cA_{X/k}^\eh$.
\end{proof}

\begin{cor}\label{c12.2.1} Let $\LA{1}(X)=[L_1\to G_1]$, as a Deligne $1$-motive.\\
a) If $X$ is proper, then $G_1$ is an abelian variety.\\
b)  If $X$ is normal, then $\LA{1}(X)=[0\to \cA_{X/k}^0]$.\\
c)  If $X$ is normal and proper then $\RA{1}(X)=[0\to \Pic^0_{X/k}]$ is an abelian
variety with dual the Serre Albanese $\LA{1}(X)\allowbreak=[0\to \cA_{X/k}^0]$.
\end{cor}

\begin{proof} a) is seen easily by blow-up induction, by reducing to the smooth
projective case (Corollary \ref{HLAlb}). b) follows from Proposition \ref{pl1} a),
b), Lemma \ref{l11.3} c) and Proposition \ref{c3.1bis} d). c) follows immediately
from a) and b).
\end{proof}

\subsection{$\LA{1}(X)$ and $\Alb^-(X)$ for $X$ normal}

In \ref{c12.2.2}, we prove that these two $1$-motives are isomorphic. We begin with a slight
improvement of Theorem \ref{tnormal} in the case of semi-abelian schemes:

\begin{lemma} \label{simpehalb} Let $X$ be normal, and let $\Xs$ be a smooth hyperenvelope (\cf Lemma \ref{ldescent}). Then we have
 $$\cA_{X/k} = \coker (\cA_{X_1}\longby{(p_0)_*-(p_1)_*}\cA_{X_0} )$$ and
 $$ (\cA_{X/k})^0 = \coker (\cA_{X_1}^0\longby{(p_0)_*-(p_1)_*}\cA_{X_0}^0 ).$$
\end{lemma}

\begin{proof} The first isomorphism follows from Lemma \ref{ldescent} applied with $Y$ running through torsors under semi-abelian varieties. To deduce the second isomorphism, consider the short  exact sequence of complexes
\[0\to \cA_{\Xs/k}^0\to \cA_{\Xs/k}\to \Z[\pi_0(\Xs)]\to 0\] 
and the resulting long exact sequence
\begin{equation}\label{eq13.7}
H_1(\Z[\pi_0(\Xs)])\to H_0(\cA_{\Xs/k}^0)\to H_0(\cA_{\Xs/k})\to H_0(\Z[\pi_0(\Xs)])\to 0.
\end{equation}

For any $i\ge 0$, $\Z[\pi_0(X_i)]$ is $\Z$-dual to the Galois module
$E_1^{i,0}$,
where $E_1^{p,q}= H^q_\et(X_p\times_k \bar k,\Z)$ is the $E_1$-term associated to the simplicial spectral sequence for $\Xs\times_k \bar k$. Since $H^1_\et(X_p\times_k \bar k,\Z)=0$ for all $p\ge 0$, we get
\[H_i(\Z[\pi_0(\Xs)])\simeq (H^i_\et(\Xs\times_k \bar k,\Z))^\vee\text{ for } i=0,1.\]

By Proposition \ref{peh}, these \'etale cohomology groups may be replaced by $\eh$ cohomology groups. By Proposition \ref{pdescent}, we then have
\[H^i_\et(\Xs\times_k \bar k,\Z)\simeq H^i_\eh(X\times_k \bar k,\Z).\]

Now, by Lemma \ref{l11.2} a), $H^0_\eh(X\times_k \bar k,\Z)$ is $\Z$-dual to $\Z[\pi_0(X)]$, and by Lemma \ref{l11.3} c), $H^1_\eh(X\times_k \bar k,\Z)=0$ because $X$ is normal. Hence \eqref{eq13.7} yields a short exact sequence
\[0\to  H_0(\cA_{\Xs/k}^0)\to \cA_{X/k}\to \Z[\pi_0(X)]\to 0\]
which identifies $H_0(\cA_{\Xs/k}^0)$ with $\cA_{X/k}^0$.
\end{proof}

\begin{propose} \label{c12.2.2} If $X$ is normal, $\RA{1}(X)$ and
$\LA{1}(X)$ are isomorphic,  respectively, to the $1$-motives $\Pic^+(X)$
and $\Alb^- (X)$ defined in \cite[Ch. 4-5]{BSAP}.
\end{propose}

\begin{proof} Let $\bar X$ be a normal compactification of $X$; choose a
smooth hyperenvelope $X_{\d}$ of $X$ along with $\bar X_{\d}$ a smooth
compactification with normal crossing boundary $Y_{\d}$ such that  $\bar X_{\d}\to
\bar X$ is an hyperenvelope. Now we have, in the notation of \cite[4.2]{BSAP}, a
commutative diagram with exact rows
\[
\begin{CD}
0@>{}>>\Pic^0_{\bar X/k} @>{}>>\Pic^0_{\bar X_0/k}@>{}>>  \Pic^0_{\bar
X_1/k}\\
& & @A{}AA @A{}AA @A{}AA  \\
0@>{}>>\Div_{Y_{\d}}^0(\bar X_{\d})@>{}>>\Div_{Y_{0}}^0(\bar X_{0})@>{}>>
\Div_{Y_{1}}^0(\bar X_{1})
\end{CD}
\]
where $\Pic^+(X) =[\Div_{Y_{\d}}^0(\bar X_{\d})\to \Pic^0_{\bar X/k}]$ since $\bar X$ is normal. Taking Cartier duals we get an exact sequence of $1$-motives
\[
[0\to \cA_{X_1/k}^0]\longby{(p_0)*-(p_1)_*}[0\to \cA_{X_0/k}^0]\to
\Alb^-(X)\to 0.
\]

Thus $\Alb^-(X)=[0\to \cA_{X/k}^0]$ by Lemma \ref{simpehalb}. We conclude by Corollary \ref{c12.2.1} b) since $X$ is normal and $\LA{1}(X)= [0\to \cA_{X/k}^0]$.
\end{proof}

\begin{remarks} 1) Note that, while $\LA{0}(X)$ and $\LA{1}(X)$ are Deligne $1$-motives, the same is not true of $\LA{2}(X)$ in general, already for $X$ smooth projective (see Corollary \ref{HLAlb}).\\
2) One could make use of Proposition~\ref{cdhex} to compute
$\LA{i} (X)$ for singular $X$ and $i>1$. However, $H^i_\eh(X, \G_m)_{\sQ}$
can be non-zero also for $i\geq 2$, therefore a precise computation for $X$ singular and higher
dimensional appears to be difficult. We did completely the case of curves in Sect. \ref{scurves}.
\end{remarks}

\subsection{$\RPic(X)$ and $H^*_\eh(X,\G_m)$}

By definition of $\RPic$, we have a morphism in $\DM_{-,\et}^\eff$
\begin{multline*}
\Tot\RPic(X)=\alpha^s\ihom_\Nis(M(X),\Z(1))\\
\to \ihom_\et(M_\et(X),\Z_\et(1))=\uR\pi_*^X \G_m[-1].
\end{multline*}

This gives homomorphisms
\begin{equation}\label{eq12.1}
\sH^i(\Tot\RPic(X))\to \uR^{i-1}\pi_*^X \G_m,\quad i\ge 0.
\end{equation}

\begin{propose} \label{p12.6}For $i\le 2$, \eqref{eq12.1} is an isomorphism.
\end{propose}

\begin{proof} By blow-up induction, we reduce to the smooth case, where it follows from Hilbert's theorem 90.
\end{proof}

\subsection{$H^1_\et(X,\G_m)$ and $H^1_\eh(X,\G_m)$}

In this subsection, we assume $\pi^X:X\to \Spec k$ \emph{proper}. Recall that, then, the \'etale sheaf associated to the presheaf 
\[U\mapsto \Pic(X\times U)\]
is representable by a $k$-group scheme $\Pic_{X/k}$ locally of finite type (Gro\-then\-dieck-Murre \cite{murre}). Its connected component $\Pic^0_{X/k}$ is an extension of a semi-abelian variety by a unipotent subgroup $\cU$. By homotopy invariance of $\uR^1\pi_*^X\G_m$, we get a map
\begin{equation}\label{eq12.5}
\Pic_{X/k}/\cU\to \uR^1\pi_*^X\G_m.
\end{equation}

Recall that the right hand side is a $1$-motivic sheaf by Proposition \ref{p11.3}. We have:

\begin{propose}\label{p12.7} This map is injective with lattice cokernel.
\end{propose}

\begin{proof} Consider multiplication by an integer $n>1$ on both sides. Using the Kummer
exact sequence, Proposition \ref{p11.2} and Lemma \ref{leh} b), 
we find that \eqref{eq12.5} is an isomorphism on $n$-torsion and injective on $n$-cotorsion. The conclusion then follows from
Proposition \ref{p3.6}.
\end{proof}

\subsection{$\RA{1}(X)$ and $\Pic^+(X)$ for $X$ proper}

\begin{thm}\label{t12.8} For $X$ proper, the composition
\[\Pic_{X/k}/\cU\to \uR^1\pi_*^X\G_m\to \sH^2(\Tot\RPic(X))\]
where the first map is \eqref{eq12.5} and the second one is the inverse of the isomorphism \eqref{eq12.1}, induces an isomorphism
\[\Pic^+(X)\iso \RA{1}(X)\]
where $\Pic^+(X)$ is the $1$-motive defined in \cite[Ch. 4]{BSAP}.
\end{thm}

\begin{proof}  Proposition \ref{p3.10}  yields an exact sequence
\[L^1\to G^1\to \sH^2(\Tot\RPic(X))\to L^2\]
where we write $\RA{i}(X)=[L^i\to G^i]$. Propositions \ref{p12.6} and \ref{p12.7} then imply that the map of Theorem \ref{t12.8} induces an isomorphism $\Pic^0_{X/k}/\cU\iso G^1$. The conclusion follows, since on the one hand $\Pic^+(X)\simeq [0\to \Pic^0_{X/k}/\cU]$ by \cite[Lemma 5.1.2 and Remark 5.1.3]{BSAP}, and on the other hand the dual of Corollary \ref{c12.2.1} a) says that $L^1=0$.
\end{proof}

\begin{cor}\label{c12.3} For $X$ proper there is a canonical isomorphism 
\[\LA{1}(X)\iso \Alb^-(X).\qed\]
\end{cor}

\subsection{The Borel-Moore variant}

Let $X\in \Sch$ be provided with a compactification $\bar X$ and closed complement $Z\by{i} \bar X$. The relative Picard functor is then representable by a $k$-group scheme locally of finite type $\Pic_{\bar X,Z}$, and we shall informally denote by $\cU$ its unipotent radical. Similarly to \eqref{eq12.1} and \eqref{eq12.5}, we have two canonical maps
\begin{equation}\label{eq12.9}
\sH^2(\Tot\RPic^c(X))\to \Pic_{\bar X,Z}^\eh\leftarrow \Pic_{\bar X,Z}/\cU
\end{equation}
where $\Pic_{\bar X,Z}^\eh$ is by definition the $1$-motivic sheaf associated to the presheaf $U\mapsto H^1_\eh(\bar X\times U,(\G_m)_{\bar X\times U}\to i_*(\G_m)_{Z\times U})$ (compare \cite[2.1]{BSAP}). Indeed, the latter group is canonically isomorphic to
\[\Hom_{\DM_{-,\et}^\eff}(M^c(X\times U),\Z(1)[2])\] 
via the localisation exact triangle. From Theorem \ref{t12.8} and Proposition \ref{c3.1*} b), we then deduce:

\begin{thm}\label{t12.9} The maps \eqref{eq12.9} induce an isomorphism
\[\RA{1}^c(X)\simeq [0\to\Pic^0(\bar X,Z)/\cU].\qed\]
\end{thm}

The following is a sequel of Proposition \ref{c3.1c}:

\begin{cor}\label{c3.1d} Let $X\in \Sch(k)$ of dimension $d$. Then:\\
a) $\LA{1}^c(X)=[L_1\to A_1]$, where $A_1$ is an abelian
variety. In particular, $\LA{1}^c(X)$ is a Deligne $1$-motive.\\
b) If $X$ is normal connected and not proper, let $\bar X$ be a normal
compactification of $X$. Then $\rank L_1=\#\pi_0(\bar X - X)-1$.
\end{cor}

\begin{proof} a) follows immediately from Theorem \ref{t12.9}. For b), consider
the complex of discrete parts associated to the exact sequence \eqref{loc}: we
get with obvious notation an almost exact sequence
\[L_1(\bar X)\to L_1(X)\to L_0(\bar X-X)\to L_0(\bar X)\to L_0(X)\]
where ``almost exact" means that its homology is finite. The last group is $0$
and $L_0(\bar X)=\Z$; on the other hand,
$L_1(\bar X)=0$ by Corollary \ref{c12.2.1} b). Hence the claim.
\end{proof}

\begin{remarks}[on Corollary \protect{\ref{c3.1d}}]\ \\
1) In fact, $A_1=0$ in a) if $X$ is smooth and quasi-affine of dimension $>1$: see Corollary \ref{c14.2}. This contrasts sharply with Theorem \ref{rasc} for smooth curves.\\
 2) As a consequence of the statement of b) we see that in b), the number of connected
components of $\bar X-X$ only depends on $X$. Here is an elementary proof of this
fact: let $\bar X'$ be another normal compactification and $\bar X''$ the closure
of $X$ in $\bar X\times \bar X'$. Then the two maps
$\bar X''\to \bar X$ and $\bar X''\to \bar X'$ have connected fibres by Zariski's main theorem,
thus
$\bar X-X$ and $\bar X'-X$ have the same number of connected components as $\bar X''-X$. (The
second author is indebted to Marc Hindry for a discussion leading to this proof.)
\end{remarks}

We shall also need the following computation in the next subsection.

\begin{thm}\label{t12.9.2} Let $\bar X$ be smooth and proper, $Z\subset \bar X$ a divisor with normal crossings and $X= \bar X-Z$. Let $Z_1,\dots,Z_r$ be the irreducible components of $Z$ and set
\[Z^{(p)} =
\begin{cases}
\bar X &\text{if $p=0$}\\
\coprod\limits_{i_1<\dots<i_p} Z_{i_1}\cap\dots \cap Z_{i_p}&\text{if $p>0$.}
\end{cases}\]
Let $\NS^{(p)}_c(X)$ (\resp $\Pic^{(p)}_c(X)$, $T^{(p)}_c(X)$) be the cohomology (\resp the connected component of the cohomology) in degree $p$ of the complex
\[\dots \to\NS(Z^{(p-1)})\to\NS(Z^{(p)})\to\NS(Z^{(p+1)})\to\dots\]
(\resp 
\[\dots \to\Pic^0(Z^{(p-1)})\to\Pic^0(Z^{(p)})\to\Pic^0(Z^{(p+1)})\to\dots\]
\[\dots \to R_{\pi_0(Z^{(p-1)})/k} \G_m\to R_{\pi_0(Z^{(p)})/k} \G_m\to R_{\pi_0(Z^{(p+1)})/k} \G_m\to\dots).\]
Then, for all $n\ge 0$, $\RA{n}^c(X)$ is of the form $[\NS^{(n-2)}_c(X)\by{u^n}G^{(n)}_c]$, where $G^{(n)}_c$ is an extension of $\Pic^{(n-1)}_c(X)$ by $T^{(n)}_c(X)$.
\end{thm}

\begin{proof} A standard argument (compare e.g. \cite[3.3]{eklv}) yields a spectral sequence of cohomological type in ${}^t\M$:
\[E_1^{p,q} = \RA{q}^c(Z^{(p)})\Rightarrow \RA{p+q}^c(X).\]

By Corollary \ref{HRPic}, we have $E_2^{p,2}=[\NS^{(p)}_c(X)\to 0]$, $E_2^{p,1}=[0\to\Pic^{(p)}_c(X)]$ and $E_2^{p,0}=[0\to T^{(p)}_c(X)]$. By Proposition \ref{hom} (b) and (c), all $d_2$ differentials are $0$, hence the theorem.
\end{proof}

\begin{cor}\label{c12.9} With notation as in Theorem \ref{t12.9.2}, the complex $\RPic(M^c(X)(1)[2])$ is \qi to
\[\dots\to [\Z^{\pi_0(Z^{(p-2)})}\to 0]\to \dots\]
In particular, $\RA{0}(M^c(X)(1)[2])=\RA{1}(M^c(X)(1)[2])=0$ and $\RA{2}(M^c(X)(1)[2])=[\Z^{\pi_0^c(X)}\to 0]$ (see Definition \ref{d11.5}).
\end{cor}

\begin{proof} This follows from Theorem \ref{t12.9.2} via the formula $M^c(X\times\P^1) =M^c(X)\oplus M^c(X)(1)[2]$, noting that $\bar X\times \P^1$ is a smooth compactification of $X\times\P^1$ with $\bar X\times \P^1-X\times \P^1$ a divisor with normal crossings with components $Z_i\times \P^1$, and
\begin{align*}
\NS(Z^{(p)}\times\P^1)&=\NS(Z^{(p)})\oplus \Z^{\pi_0(Z^{(p)})}\\
 \Pic^0(Z^{(p)}\times \P^1)&=\Pic^0(Z^{(p)})\\ 
 \pi_0(Z^{(p)}\times\P^1)&=\pi_0(Z^{(p)}).
\end{align*}
\end{proof}

\begin{remark} Let $X$ be arbitrary, and filter it by its successive singular loci, \ie
\[X=X^{(0)}\supset X^{(1)}\supset\dots\]
where $X^{(i+1)}=X^{(i)}_\sing$. Then we have a spectral sequence of cohomological type in ${}^t\M$
\[E_2^{p,q} =\RA{p+q}^c(X^{(q)}-X^{(q+1)})\Rightarrow \RA{p+q}^c(X) \]
in which the $E_2$-terms involve smooth varieties. This qualitatively reduces the computation of $\RA{*}^c(X)$ to the case of smooth varieties, but the actual computation may be complicated; we leave this to the interested reader.
\end{remark}

\subsection{$\LA{1}^*$ and $\Alb^+$}

\begin{lemma}\label{l11.6} Let $n>0$ and $Z\in \Sch$ of dimension $<n$; then 
\begin{gather*}
\RA{i}(M(Z)^*(n)[2n])=0 \text{ for } i\le 1\\
\RA{2}(M(Z)^*(n)[2n])=[\Z^{\pi_0^c(Z^{[n-1]})}\to 0]
\end{gather*}
where $Z^{[n-1]}$ is the disjoint union of the irreducible components of $Z$ of dimension $n-1$.
\end{lemma}

\begin{proof} Let 
\[\begin{CD}
\tilde T@>>> \tilde Z\\
@VVV @VVV\\
T@>>> Z
\end{CD}\]
be an abstract blow-up square, with $\tilde Z$ smooth and $\dim T,\dim \tilde T<\dim Z$. By
Lemma \ref{leffe}, $M(T)^*(n-2)[2n-4]$ and $M(\tilde T)^*(n-2)[2n-4]$ are effective, so by
Proposition \ref{ptate}, the exact triangle
\begin{multline*}
\RPic(M(\tilde T)^*(n)[2n])\to \RPic(M(\tilde Z)^*(n)[2n])\oplus \RPic(M(T)^*(n)[2n])\\
\to \RPic(M(Z)^*(n)[2n])\by{+1}
\end{multline*}
degenerates into an isomorphism
\[\RPic(M(\tilde Z)^*(n)[2n])\iso \RPic(M(Z)^*(n)[2n]).\]

The lemma now follows from Corollary \ref{c12.9} by taking for $\tilde Z$ a desingularisation
of $Z^{[n-1]}$ and for $T$ the union of the singular locus of $Z$ and its irreducible
components of dimension $<n-1$ (note that $M(\tilde Z)^*(n)[2n]\simeq M^c(\tilde Z)(1)[2]$).
\end{proof}

\begin{lemma} \label{relPicplus} Let $\bar X$ a proper smooth scheme with a
pair $Y$ and $Z $ of disjoint  closed (reduced) subschemes of pure
codimension $1$ in $\bar X$. We then have 
$$\RA{1}(\bar X- Z, Y) \cong\Pic^+(\bar X- Z, Y)$$ 
(see \cite[2.2.1]{BSAP} for the definition of
relative $\Pic^+$). 
\end{lemma}

\begin{proof} The following exact sequence provides the weight filtration
$$0\to\RA{1}(\bar X,Y)\to \RA{1}(\bar X - Z,Y)\to \RA{2}_Z(\bar X, Y)$$
where $\RA{1}(\bar X, Y) \cong \RA{1}^c(\bar X-  Y)\cong [0\to \Pic^0(\bar X, Y)]$ by
Theorem \ref{t12.9} (here $\cU=0$ since $\bar X$  is smooth). Also
$\RA{2}_Z(\bar X, Y)\cong \RA{2}_Z(\bar X)=[\Div_Z(\bar X)\to 0]$ from
\ref{lsupports}: thus the discrete part of $\RA{1}(\bar X- Z, Y)$ is given by a subgroup $D$ of
$\Div_Z(\bar X)=\Div_Z(\bar X,Y)$. 

It remains to identify the map $u:D\to \Pic^0(\bar X,Y)$. Using now the exact sequence
$$\RA{0}(Y)\to \RA{1}(\bar X - Z,Y)\to \RA{1}(\bar X - Z)\to \RA{1}(Y)$$
where $\RA{i}(Y)$ is of weight $<0$ for $i\leq 1$ (\ref{c3.1} and \ref{c12.2.1}), we get that
$u$ is the canonical lifting of the map of the $1$-motive $\RA{1}(\bar X
- Z)$ described in \ref{HRPic}. Thus  $D=\Div_Z^0(\bar X, Y)$ and the claimed isomorphism is
clear.
\end{proof}

This proof also gives:

\begin{cor}\label{algeqzero} We have
$$[\Div_Z^0(\bar X, Y)\to 0]=\ker (\RA{2}_Z(\bar X, Y)\to \RA{2}(\bar X, Y)).$$
\end{cor}

We shall need:

\begin{thm}[Relative duality]\label{treldual} Let $\bar X$,  $Y$ and $Z $ be
as above and further assume that $\bar X$ is $n$-dimensional. Then
$$M (\bar X- Z, Y)^*(n)[2n]\cong M(\bar X- Y, Z)$$
and therefore
$$\RPic^*(\bar X- Z, Y)\cong \RPic (\bar X- Y, Z)$$
and dually for $\LAlb$.
\end{thm}

\begin{proof} See \cite{rel_duality}.
\end{proof}

\begin{cor}\label{cycle} Let $Z$ be a divisor in $\bar X$ such that $Z\cap Y
=\emptyset$. There exists a  ``cycle class'' map $\eta$ fitting in the
following commutative diagram
$$
\begin{CD}
\RA{2}(M(Z)^*(n)[2n])@>{\eta}>>  \RA{2}^c(\bar X - Y)\\
 @V{||}VV @V{||}VV \\
 \RA{2}_Z(\bar X, Y)@>{}>> \RA{2}(\bar X, Y)
\end{CD}
$$
Writing $Z= \cup Z_i$ as union of its irreducible components  we have that
$\eta$ on $Z_i$  is the ``fundamental class'' of $Z_i$ in $\bar X$ modulo
algebraic equivalence.\end{cor}

\begin{proof} We have a map $M(Z)\to M(\bar X -Y)$, and the vertical isomorphisms
in the following commutative square
$$
\begin{CD}
M(\bar X- Y)^*(n)[2n]@>{}>>  M(Z)^*(n)[2n]\\
 @A{||}AA @A{||}AA \\
M(\bar X, Y)@>{}>> M^Z(\bar X, Y)
\end{CD}
$$ are given by relative duality.
\end{proof}

\begin{thm}\label{*=-} For $X\in \Sch$ we have
$$\RA{1}^*(X)\cong \Pic^-(X).$$
\end{thm}

\begin{proof} We are left to consider $X\in \Sch$ purely of dimension $n$
with the following associated set of data and conditions.

For the irreducible components $X_1,\dots,X_r$ of $X$ we let $\tilde X$ be a
desingularisation of $\coprod X_i$, $S\df X_\sing\cup \bigcup _{i\ne j}
S_i\cap S_j$ and $\tilde S$ the inverse image of $S$ in $\tilde X$. We let
$\bar X$ be a smooth proper compactification with normal crossing divisor
$Y$. Let $\bar S$ denote the Zariski closure of $\tilde S$ in $\bar X$.
Assume that $Y+\bar S$ is a reduced normal crossing divisor  in $\bar X$.
Finally denote by $Z$ the union of all compact components of divisors in
$\tilde S$ (\cf \cite[2.2]{BSAP}).

We have an exact sequence coming from the abstract blow-up square
associated to the above picture: 
\begin{multline*}
\dots\to\RA{1}^*(\tilde X)\oplus \RA{1}(M(S)^*(n)[2n])\to\RA{1}^*(X)\\
\to
\RA{2}(M(\tilde S)^*(n)[2n])\to
\RA{2}^*(\tilde X)\oplus\RA{2}(M(S)^*(n)[2n])
\end{multline*}

Now:
\begin{itemize}
\item the first map is injective (Lemma \ref{l11.6}),
\item $\RA{1}^*(\tilde X)=\RA{1}^c(\tilde
X)=[0\to\Pic^0(\bar X,Y)]$ since $\tilde X$ is smooth (Theorem \ref{t12.9}; note that $\cU=0$
by the smoothness of $\tilde X$),
\item $\RA{1}(M(S)^*(n)[2n])=0$ (Lemma
\ref{l11.6}),
\item $\RA{2}(M(\tilde S)^*(n)[2n])= [\Z^{\pi_0^c(\tilde S^{[n-1]})}\to
0]\df [\Div_{\bar S}(\bar X, Y)\to 0]$ is given by the free abelian group on compact
irreducible components of $\tilde S$ (Lemma \ref{l11.6}),
\item $\RA{2}^*(\tilde X)=\RA{2}^c(\tilde X)= \RA{2}(\bar X, Y)=[\NS^{(0)}_c(\tilde
X)\by{u^2}G^{(2)}_c]$ (Theorem \ref{t12.9.2}),
\item $\RA{2}(M(S)^*(n)[2n])=[\Z^{\pi_0^c( S^{[n-1]})}\to 0]=  [\Div_{S}(X)\to 0]$ (Lem\-ma
\ref{l11.6}).
\end{itemize}

We may therefore rewrite the above exact sequence as follows:
\begin{multline*}
0\to[0\to\Pic^0(\bar X,Y)]\to\RA{1}^*(X)\to [\Div_{\bar S}(\bar X, Y)\to 0]\\
\by{\alpha} [\NS^{(0)}_c(\tilde X)\by{u^2}G^{(2)}_c]\oplus [\Div_{S}(X)\to 0].
\end{multline*}

The map $\Div_{\bar S}(\bar X, Y)\to \Div_{ S}(X)$  induced from $M (\tilde S)\to M(S)$ is
clearly the proper push-forward of Weil divisors. The map 
\[[\Div_{\bar S}(\bar X, Y)\to 0]\to
[\NS^{(0)}_c(\tilde X)\by{u^2}G^{(2)}_c]\] 
is the cycle class map described in Corollary
\ref{cycle}. By Corollary \ref{algeqzero} we then get
\[
\ker  \alpha=[\Div_{\bar S/S}^0(\bar X, Y)\to 0] 
\]
where the lattice $\Div_{\bar S/S}^0(\bar X, Y) $ is from the definition of
$\Pic^-$ (see \cite[2.2.1]{BSAP}).
In other terms, we have
\[\RA{1}^*(X)=[\Div_{\bar S/S}^0(\bar X, Y) \by{u}\Pic^0(\bar X,Y)]\]
and we are left to check that the mapping $u$ is the one described in
\cite{BSAP}. Just observe that, by Lemma \ref{l11.6} and Theorem \ref{treldual},
\begin{multline*}\RA{1}^*(X) \into \RA{1}^*(X, S) \cong \RA{1}^*(\tilde X,
\tilde S)\cong \RA{1}^* (\tilde X, Z) \\= \RA{1}^*(\bar X -Y, Z)
\cong\RA{1}(\bar X -Z, Y)
\end{multline*} and the latter is isomorphic to $\Pic^+(\bar X-Z,Y)$ by
\ref{relPicplus}.
Since, by construction, $\Pic^-(X)$ is a sub-1-motive of $\Pic^+(\bar
X-Z,Y)$ the isomorphism of \ref{relPicplus} restricts to the claimed one.
\end{proof}

\section{Generalisations of Ro\v\i tman's theorem}\label{rovi}

In this section, we give a unified treatment of Ro\v\i tman's theorem on torsion $0$-cycles on a smooth projective variety and its various generalisations. We mainly  assume  $k= \bar k$.   \S \ref{s13.4} deals with smooth schemes, so is valid in any characteristic, but from \S \ref{s13.5} onwards we assume $\car k\allowbreak=0$ for singular schemes (see \S \ref{s8.1}).

\subsection{Motivic and classical Albanese} 

Let $X\in \Sch(k)$; we assume $X$ smooth if $p>1$ and $X$ semi-normal (in particular reduced)
if $p=1$, see Lemma \ref{l12.3}. Recall that Suslin's algebraic singular homology is
\[H_j(X)\df \Hom_{\DM_-^\eff}(\Z[j], M(X))= \HH^{-j}_\Nis(k,C_*(X)) \]
for any scheme $p: X\to k$. On the other hand, we may define
\[H^\et_j(X)\df \Hom_{\DM_{-,\et}^\eff}(\Z[j], M_\et(X))= \HH^{-j}_\et(k,\alpha^sC_*(X)).\]

We also have versions with coefficients in an abelian group $A$:
\[H_j^\et(X,A)=\Hom_{\DM_{-,\et}^\eff}(\Z[j], M_\et(X)\otimes A).\] 

We shall also use the following notation throughout:

\begin{notation} For any $M\in \DM_\gm^\eff$ and any abelian group $A$, we write 
$H_j^{(1)}(M,A)$ for the abelian group \index{$H_j^{(1)}$}
\[\Hom_{\DM_{-,\et}^\eff}(\Z[j],\Tot\LAlb(M)\otimes
A)\simeq \HH^{-j}_\et(k,\Tot\LAlb(M)\oo^L A).\] 
This is \emph{Suslin $1$-motivic homology
of $M$ with coefficients in $A$}. If $M=M(X)$, we write $H_j^{(1)}(X,A)$ for $H_j^{(1)}(M,A)$.
We drop $A$ in the case where $A=\Z$.\\
We also write $H_j^\et(M,A)=\Hom_{\DM_{-,\et}^\eff}(\Z[j], M\otimes A)$ and $H_j^\et(M) = H_j^\et(M,\Z)$.
\end{notation}

The motivic Albanese map
\eqref{amap} then gives  maps
\begin{equation}
H_j^\et(M,A) \to H_j^{(1)}(M,A)\label{eqalbt}
\end{equation}
for any abelian group $A$.

\begin{propose} \label{sus}
If $X$ is a smooth curve (or any curve in characteristic $0$), the map
\eqref{eqalbt} (for $M=M(X)$) is an isomorphism for any $A,j$. 
\end{propose}

\begin{proof} This follows immediately from Proposition \ref{cd2}.
\end{proof}

Note that if $X =\bar X -Y$ is a smooth curve obtained by removing
a finite set of closed points from a projective smooth curve $\bar
X $ then $\cA_{X/k}= \Pic_{(\bar X, Y)/k}$ is the relative Picard
scheme (see \cite{BSAP} for its representability) and the Albanese
map just sends a point $P\in X$ to $(\cO_{\bar X}(P), 1)$ where 1
is the tautological section, trivialising $\cO_{\bar X}(P)$ on
$Y$. We then have the following result  (\cf \cite[Lect. 7, Th.
7.16]{VL}).

\begin{cor}\label{relpic} If $X =\bar X -Y$ is a smooth curve,
$$H_0^\et(X)\to \Pic (\bar X, Y)[1/p]$$
is an isomorphism.
\end{cor}

Now let $\cA_{X/k}^\eh$ be as in Proposition \ref{c3.1bis} and Remark \ref{r11.1}. The map
$\Tot\LAlb(X)\to \cA_{X/k}^\eh$ of
\loccit induces a homomorphism
\begin{equation}\label{eqhom}
H_0^{(1)}(X)\to \cA_{X/k}^\eh(k)[1/p]
\end{equation}
which is not an isomorphism in general (but see Lemma \ref{l7.3.1}).
Composing \eqref{eqhom}, \eqref{eqalbt} (for $A=\Z$) and the obvious map $H_0(X)[1/p]\to H_0^\et(X)$,
we get a map
\begin{equation}\label{classalb}
H_0(X)[1/p]\to \cA_{X/k}^\eh(k)[1/p].
\end{equation}

We may further restrict to parts of degree $0$, getting a map
\[H_0(X)^0[1/p]\to (\cA_{X/k}^\eh)^0(k)[1/p].\]
Recall that
$\cA_{X/k}^\eh=\cA_{X/k}$ if $X$ is normal (Proposition \ref{c3.1bis} d)).
In this case, the above maps become
\begin{equation}\label{eqalb0}
H_0^{(1)}(X)\to \cA_{X/k}(k)[1/p], \quad H_0(X)^0[1/p]\to (\cA_{X/k})^0(k)[1/p].
\end{equation}

If $X$ is smooth, \eqref{eqalb0} is the $\Z[1/p]$-localisation of the generalised Albanese map of
Spie\ss-Szamuely \cite[(2)]{spsz}.

Dually to Lemma \ref{ltlexact}, the functor 
\begin{align*}
D^b(\M[1/p])&\to \DM_{-,\et}^\eff\\
C&\mapsto \Tot(C)\otimes (\Q/\Z)'
\end{align*}
is exact with respect to the ${}_t\M$ $t$-structure on the left and the homotopy $t$-structure on
the right; here, as usual, $(\Q/\Z)'\df \bigoplus_{l\ne p} \Q_\ell/\Z_\ell$.
In other words:

\begin{lemma}\label{l12.1} For any $C\in D^b(\M[1/p])$, there are canonical isomorphisms of
sheaves
\[\sH_j(\Tot(C)\otimes (\Q/\Z)')\simeq \Tot({}_t H_j(C))\otimes(\Q/\Z)'\]
(note that the right hand side is a single sheaf!) In particular, for $C=\LAlb(M)$ and $k$
algebraiclly closed:
\[H_j^{(1)}(M, (\Q/\Z)')\simeq \Gamma(k,\Tot(\LA{j}(M))\otimes(\Q/\Z)').\]
\end{lemma}

\subsection{A variant of the Suslin-Voevodsky theorem} We now assume that $k$ is algebraically closed until the end of this Section.

Let $n$ be invertible in $k$. For $M\in \DM_{\gm,\et}^{\eff}$, we have a pairing
\begin{equation}\label{eqSV}
\Hom(\Z,M\otimes \Z/n)\times \Hom(M,\Z/n)\to \Hom(\Z,\Z/n)=\Z/n
\end{equation}
constructed as follows: a morphism $\phi:M\to \Z/n$ yields a composite morphism $\tilde \phi:M\otimes \Z/n\longby{\phi\otimes 1} \Z/n\otimes \Z/n\to \Z/n$, where the second map is induced by the isomorphism $H^0(\Z/n\otimes \Z/n) = \Z/n$ in $\DM_{\gm,\et}^{\eff}$. Composing $\tilde \phi$ with a morphism $\psi:\Z\to M\otimes\Z/n$ yields the desired pairing $<\psi,\phi>\in \Z/n$.

The following theorem is parallel to Theorem \ref{t4.3.2}:

\begin{thm}\label{tSV} For any $M\in \DM_{\gm,\et}^{\eff}$, the pairing \eqref{eqSV} is a perfect duality of finite $\Z/n$-modules.
\end{thm}

\begin{proof} The statement is stable under exact triangles and direct summands, thus it is enough to check it for $M=M_\et(X)[-j]$, $X$ smooth, $j\in\Z$.
Then the statement amounts to the duality between \'etale (motivic) cohomology $H^j_{\et}(X, \Z/n) \cong\Hom(M_\et(X),\Z/n[j])$ and algebraic singular homology $H_j^{\et}(X,\Z/n)=\Hom(\Z[j],M_\et(X)\otimes \Z/n)$, which is the contents of \cite[Th. 10.9 \& Cor. 10.11]{VL}.
\end{proof}

\subsection{Change of topology and motivic Albanese map}  Recall the change of topology functor of Definition \ref{d2.1.1}
\[\alpha^s:\DM_{\gm}^\eff\to \DM_{\gm,\et}^{\eff}.\]

Recall the functor $d\1$  of \eqref{D2} and the motivic Albanese map $a_M$ of \eqref{amap}. Note that $a_{\Z/n(1)}:\alpha^s\Z/n(1)\to d\1\Z/n(1)$ is an isomorphism by Proposition \ref{cd2}. This gives a meaning to:

\begin{propose}\label{pchalpha} Let $M\in \DM_{\gm,\et}^{\eff}$. Then:\\
a) The diagram
\[\xymatrix{
\Hom_{\DM_\gm^\eff}(M,\Z/n(1)) \ar[r]^{d_{\le 1}\qquad}\ar[dr]_{\alpha^s}&\Hom_{\DM_{\gm,\et}^\eff}(d_{\le 1}M,\alpha^s\Z/n(1))\ar[d]^{(a_M)^*} \\
&\Hom_{\DM_{\gm,\et}^\eff}(\alpha^sM,\alpha^s\Z/n(1))
}\]
commutes.\\
b) In this diagram, $d_{\le 1}$ is an isomorphism.
\end{propose}

\begin{proof} Let $N \in \DM_{\gm,\et}^{\eff}$. By the naturality of the motivic Albanese map, the diagram
\[\begin{CD}
\Hom_\Nis(M,N) @>d\1>> \Hom_\et(d\1M,d\1N)\\
@V\alpha^sVV @V(a_M)^*VV\\
\Hom_\et(\alpha^sM,\alpha^sN) @>(a_N)_*>> \Hom_\et(\alpha^sM,d\1N)
\end{CD}\]
commutes. Taking $N=\Z/n(1)$, we get a).

For b), we write $d\1$ as a composition $D\1^\et\circ\alpha^s\circ D\1^\Nis$. We shall show that each of these three functors induces an isomorphism on the corresponding Hom groups.

For $D\1^\Nis$, this is because the composition
\begin{multline*}
\Hom_\Nis(M,\Z/n(1))\longby{D\1^\Nis} \Hom_\Nis(D\1(\Z/n(1)),D\1^\Nis(M))\\
=\Hom_\Nis(\Z/n[-1],D\1^\Nis(M))\simeq \Hom_\Nis(M\otimes \Z/n[-1],\Z(1))
\end{multline*}
(where the last map is the adjunction isomorphism for $\ihom_\Nis$) coincides with the trivial isomorphism
\begin{multline*}\Hom_\Nis(M,N\otimes \Z/n)\simeq \Hom_\Nis(M,\ihom_\Nis(\Z/n[-1],N))\\
\simeq \Hom_\Nis(M\otimes \Z/n[-1],N)
\end{multline*}
(here we take $N=\Z(1)$).

 For $\alpha^s$, this is because $k$ is algebraically closed, so that \'etale cohomology of $\Spec k$ coincides with Nisnevich cohomology. Finally, for $D\1^\et$, this is because $\Z/n$ and $\alpha^sD\1^\Nis(M)$ are in $d\1\DM_{\gm,\et}^\eff$ (Lemma \ref{l2.1}), and $D\1^\et$ restricts to a perfect duality on this subcategory (Prop. \ref{cd}). 
\end{proof}

\subsection{A proof of Ro\v\i tman's and Spie\ss-Szamuely's theorems}\label{s13.4}

 In this subsection, we only deal with smooth schemes and the characteristic is arbitrary: we shall
show how the results of Section \ref{comp} allow us to recover the
classical  theorem of Ro\v\i tman on torsion $0$-cycles up to $p$-torsion, as well as its generalisation to
smooth varieties by Spie\ss-Szamuely \cite{spsz}. The reader should compare our argument with
theirs (\loccit, \S 5).

Since $k$ is algebraically closed, Corollary \ref{cD.1} implies

\begin{lemma}\label{l7.3} For any $j\in\Z$, $H^\et_j(X)=H_j(X)[1/p]$; similarly with finite or divisible coefficients.\qed
\end{lemma}

Moreover, it is easy to evaluate
$H_j^{(1)}(X)=\sH_j(\Tot\LAlb(X))(k)$ out of Theorem \ref{trunc}: if
$\LA{n}(X)=[L_n\to G_n]$, we have a long exact sequence coming from Proposition \ref{p3.10}
\begin{multline}\label{eq13.1}
L_{j+1}(k)[1/p]\to G_{j+1}(k)[1/p]\to H_j^{(1)}(X)\\
\to L_j(k)[1/p]\to G_j(k)[1/p]\to \dots
\end{multline}

Thus:

\begin{lemma}\label{l7.3.1} For $X$ smooth, the maps \eqref{eqalb0} are
isomorphisms and we have
\begin{align}
H_1^{(1)}(X)&\simeq \NS_{X/k}^*(k)[1/p]\label{eqhom1}\\
H_j^{(1)}(X)&=0 \text{ if } j\ne 0,1.\notag
\end{align}
\end{lemma}

Here is now the main lemma:

\begin{lemma}\label{l7.2} Let $M=M(X)$ with $X$ smooth, and let $A=\Z/n$  with $(n,p)=1$. Then the map \eqref{eqalbt} is an isomorphism for $j=0,1$ and surjective for $j=2$.
\end{lemma}

\begin{proof} By Theorem \ref{tSV},  the statement is equivalent to the following: the motivic Albanese map 
\[(a_X)^*:\Hom_\et(d\1M(X),\alpha^s\Z/n(1)[j])\to \Hom_\et(\alpha^sM(X),\alpha^s\Z/n(1)[j]) \]
induced by  \eqref{amapX} is bijective for $j=0,1$ and injective for $j=2$. (Here we also use the fact that $\alpha^s\Z/n(1)=\mu_n$ and that $k$ is algebraically closed.)

By Proposition \ref{pchalpha}, we may replace the above map by the change of topology map
\begin{multline*}
\alpha^s:H^j_\Nis(X,\Z/n(1))=\Hom_\Nis(M(X),\Z/n(1)[j])\\
\to \Hom(\alpha^sM(X),\alpha^s\Z/n(1)[j])=H^j_\et(X,\Z/n(1)). 
\end{multline*}

Then the result follows from Hilbert's theorem 90 (aka Beilinson-Lichtenbaum in weight $1$).
\end{proof}

From this and Lemma \ref{l12.1} we deduce:
\begin{cor}\label{c12.2} The homomorphism \eqref{eqalbt}
\[H_j(X,(\Q/\Z)')\to H_j^{(1)}(X,(\Q/\Z)')\]
(see Lemma \ref{l7.3}) is bijective for $j=0,1$ and surjective for $j=2$.
\end{cor}

The following theorem extends in particular \cite[Th. 1.1]{spsz} to all smooth
varieties\footnote{In \loccit, $X$ is supposed to admit an open embedding into a smooth
projective variety.}.

\begin{thm}\label{troitclas} Let $X$ be smooth.\\
a) The maps  \eqref{eqalb0} are isomorphisms on torsion.\\ 
b) $H_1(X)\otimes (\Q/\Z)'=0$.\\
c) The map \eqref{eqalbt} for $A=\Z$ yields a surjection
\[H_1(X)\{p'\}\onto \NS_{X/k}^*(k)\{p'\}\]
where $M\{p'\}$ denotes the torsion prime to $p$ in an abelian group $M$.
\end{thm}

\begin{proof} Lemmas \ref{l7.3} and \ref{l7.3.1} reduce us to show that \eqref{eqalbt} is an
isomorphism on torsion for $A=\Z$, $j=0$. We have commutative diagrams with exact
rows:
\begin{equation}\label{eq7.7}
\begin{CD}
0&\to&H_j(X)\otimes(\Q/\Z)'&\to& H_j(X,(\Q/\Z)')&\to& H_{j-1}(X)\{p'\}&\to& 0\\
&&@VVV @VVV @VVV\\
0&\to&H_j^{(1)}(X)\otimes (\Q/\Z)'&\to&H_j^{(1)}(X,(\Q/\Z)') &\to&
H_{j-1}^{(1)}(X)\{p'\}&\to& 0 
\end{CD}
\end{equation}

For $j=1$, the middle vertical map is an isomorphism by Lemma \ref{l7.2} or Corollary
\ref{c12.2} and
$H_1^{(1)}(X)\otimes
\Q/\Z=0$ by \eqref{eqhom1}, which gives a) and b). For $j=2$, the middle map is surjective by
the same lemma and corollary, which gives c). The proof is complete.
\end{proof}

\begin{remark} If $X$ is smooth projective of dimension $n$, $H_j(X)$ is isomorphic to the
higher Chow group $CH^n(X,j)$. In \eqref{eq7.7} for $j=2$, the lower left term is $0$ by Lemma
\ref{l7.3.1}. The composite map
\[
H_2(X,(\Q/\Z)')\to H_1(X)\{p'\}\to H_1^{(1)}(X)\{p'\}=\NS_{X/k}^*(k)\{p'\}
\] 
is ``dual" to the map
\[\NS(X)\otimes (\Q/\Z)'\to H^2_\et(X,\Q/\Z(1))\]
whose cokernel is $Br(X)\{p'\}$. Let
\[Br(X)^D=\varinjlim_{(n,p)=1} \Hom({}_n
Br(X),\mu_n):\] 
a diagram chase in \eqref{eq7.7} for $j=2$ then yields an exact sequence
\begin{multline*}
0\to CH^n(X,2)\otimes (\Q/\Z)'\to Br(X)^D\\
\to CH^n(X,1)\{p'\}\to
\NS_{X/k}^*(k)\{p'\}\to 0.
\end{multline*}
Together with $CH^n(X,1)\otimes (\Q/\Z)'=0$, this should be considered as a natural complement
to Ro\v\i tman's theorem.
\end{remark}

\subsection{Generalisation to singular schemes}\label{s13.5} We now assume  $k= \bar k$ and $\car k\allowbreak=0$ (but see \S \ref{s8.1}), and show
how the results of Section \ref{comps} allow us to extend the results of the previous
subsection to singular schemes. By blow-up induction and the 5 lemma, we get:

\begin{propose} \label{p12.3} The isomorphisms and surjection of Lemma \allowbreak\ref{l7.2} and
Corollary
\ref{c12.2} extend to all $X\in \Sch$.\qed
\end{propose}

Let $\LA{1}(X)=[L_1\by{u_1} G_1]$. Proposition \ref{p12.3}, the exact sequence \eqref{eq13.1}
and the snake chase in the proof of Theorem \ref{troitclas} give:

\begin{cor} For $X\in \Sch$, we have exact sequences
\begin{gather*}
0\to H_1(X)\otimes \Q/\Z\to \ker(u_1)\otimes \Q/\Z\to H_0(X)_\tors\to \coker(u_1)_\tors\to 0\\
0\to H_1(X)\otimes \Q/\Z\to H_1^{(1)}(X,\Q/\Z)\to H_0(X)_\tors\to 0.
\end{gather*}
\end{cor}

The second exact sequence is more intrinsic than the first, but note that it does not give
information on $H_1(X)\otimes \Q/\Z$.

\begin{cor}\label{c13.3} If $X$ is normal, $H_1(X)\otimes \Q/\Z=0$ and there is an isomorphism
\[H_0(X)_\tors\iso \cA_{X/k}(k)_\tors.\]
\end{cor}

\begin{proof} This follows from the previous corollary and Corollary \ref{c12.2.1} c).
\end{proof}

\begin{remark} Theorem \ref{t12.8} shows that the second isomorphism of Proposition
\ref{p12.3} coincides with the one of Geisser in \cite[Th. 6.2]{geisser2} when $X$ is proper.
When $X$ is further normal, the isomorphism of Corollary \ref{c13.3} also coincides with the
one of his Theorem 6.1.
\end{remark}

\begin{remarks} Note that the reformulation of ``Roitman's theorem'' involving $\ker u_1$ is
the best possible!\\
1) Let X be a proper scheme such that $\Pic^0(X)/\cU =
\G_m^r$ is a torus (more likely such that $\Alb (X_0) = 0$ where $X_0\to X$
is a resolution, according with the description in \cite[p. 68]{BSAP}).
Then $\RA{1} (X)^*=\LA{1} (X) = [\Z^r\to 0]$ is the character group (\cf
\cite[5.1.4]{BSAP}).
For example, take a nodal projective curve $X$ with resolution $X_0=\P^1$.
In this case the map \eqref{eqalbt} for $A=\Z$ is an isomorphism for all $j$ and  thus
$\ker (u_1)\otimes \Q/\Z = H_1(X)\otimes  \Q/\Z=(\Q/\Z)^r$.\\
2) For Borel-Moore and $\LA{1}^c (X) = \LA{1}^*(X)$ for $X$ smooth open is
Cartier dual of $\Pic^0 (\bar X, Y)$ then (\cf \cite[p. 47]{BSAP}) 
$\ker u_1^c$ can be non-zero: take $\bar X = \P^1$ and $Y =$ a finite number of
points.
\end{remarks} 

\subsection{Borel-Moore Ro\v\i tman} Recall that the
Borel-Moore motivic homology group 
\[H_j^c(X,\Z)\df \Hom(\Z[j],M^c(X))\] 
is canonically isomorphic to Bloch's higher Chow group $CH_0(X,j)$.  Similarly to the previous
sections, we have maps
\begin{gather*}
H_j^c(X,\Z)\to H_j^{(1)}(M^c(X))=:H_j^{c,(1)}(X)\\
H_j^c(X,\Q/\Z)\to H_j^{(1)}(M^c(X),\Q/\Z)=:H_j^{c,(1)}(X,\Q/\Z)
\end{gather*}
and

\begin{propose}\label{p13.4} The second map is an isomorphism for $j=0,1$ and surjective for
$j=2$.
\end{propose}

\begin{proof} By localisation induction, reduce to $X$ proper and use Proposition \ref{p12.3}.
\end{proof}

\begin{cor}\label{c14.2} For $X\in \Sch$, we have exact sequences
\begin{multline*}
0\to CH_0(X,1)\otimes \Q/\Z\to \ker(u_1^c)\otimes \Q/\Z\\
\to CH_0(X)_\tors\to \coker(u_1^c)_\tors\to 0
\end{multline*}
\[
0\to CH_0(X,1)\otimes \Q/\Z\to H_j^{c,(1)}(X,\Q/\Z)\to CH_0(X)_\tors\to 0
\]
where we write $\LA{1}^c(X)=[L_1^c\by{u_1^c} G_1^c]$. In particular, if $X$ is smooth quasi-affine of dimension $>1$, $G_1^c=0$.
\end{cor}

\begin{proof} Only the last assertion needs a proof: if $X$ is smooth affine of dimension $>1$ then $CH_0(X)_\tors=0$ \cite[Th. 4.1 (iii)]{ct}, hence $\coker(u_1^c)_\tors=0$; this forces the semi-abelian variety $G_1^c$ to be $0$. We may then pass from affine to quasi-affine by using the localisation exact sequence and the description of $\LA{0}^c$ in Proposition \ref{c3.1c} b).
\end{proof}

\subsection{``Cohomological" Ro\v\i tman}

\begin{lemma}\label{lvan} Let $0< r\le n$. Then for any $Z\in \Sch$ of dimension $\le n-r$ and
any $i> 2(n-r)$, we have $H^{i}_\cdh(Z,\Q/\Z(n))=0$.
\end{lemma}

\begin{proof} By blow-up induction we reduce to the case where $Z$ is smooth of pure dimension
$n-r$; then $H^i_\cdh(Z,\Q/\Z(n))=H^i_\Nis(Z,\Q/\Z(n))$. Since $k$ is algebraically closed, and
$n\ge \dim Z$, $H^i_\Nis(Z,\Q/\Z(n))\simeq H^i_\et(Z,\Q/\Z(n))$ by Suslin's theorem
\cite{suslin} and the vanishing follows from the known bound for \'etale cohomological
dimension.
\end{proof}

Now consider the $1$-motive $\LA{1}^*(X)$ for $X$ of
dimension $n$. This time, we have  maps
\begin{gather}
H^{2n-j}_\cdh(X,\Z(n))\to H_j^{(1)}(M(X)^*(n)[2n])=:H^{2n-j}_{(1)}(X,\Z(n))\notag\\
H^{2n-j}_\cdh(X,\Q/\Z(n))\to
H_j^{(1)}(M(X)^*(n)[2n],\Q/\Z)=:H^{2n-j}_{(1)}(X,\Q/\Z(n)).\label{corot}
\end{gather}

\begin{lemma} \label{lsmallis} Let $Z\in \Sch$ be of dimension $<n$. Then the map
\[H^{2n-2}_\cdh(Z,\Q/\Z(n))\to H_2^{(1)}(M(Z)^*(n)[2n],\Q/\Z)\]
is an isomorphism.
\end{lemma}

\begin{proof} For notational simplicity, write $H^*(Y,n)$ for $H^*_\cdh(Y,\Q/\Z(n))$ and
$F_j(Y)$ for $H_j^{(1)}(M(Y)^*(n)[2n],\Q/\Z))$, where $Y$ is a scheme of dimension $\le
n$. Let $\tilde Z,T,\tilde T$ be as in the proof of Lemma \ref{l11.6}. Then Lemma
\ref{lvan} and proposition \ref{ptate} yield a commutative diagram
\[\begin{CD}
H^{2n-2}(Z,n)@>>> H^{2n-2}(\tilde Z,n)\\
@VVV @VVV\\
F_2(Z)@>>> F_2(\tilde Z)
\end{CD}\]
in which both horizontal maps are isomorphisms. Therefore, it suffices to prove the lemma when
$Z$ is smooth quasiprojective of dimension $n-1$. 

The motive $\RA{2}(M(Z)^*(n)[2n])\simeq\RA{2}(M^c(Z)(1)[2])$ was computed in Corollary
\ref{c12.9}: it is $[\Z^{\pi_0^c(Z)}\to 0]$. Therefore, we get
\[F_2(Z)\simeq \Q/\Z(1)[\pi_0^c(Z)].\]

On the other hand, the trace map defines an isomorphism 
\[H^{2n-2}(Z,n)\iso \Q/\Z(1)[\pi_0^c(Z)]\]
and the issue is to prove that the vertical map in the diagram is this isomorphism. For this,
we first may reduce to $Z$ projective and connected. Now we propose the following argument: take
a chain of smooth closed subvarieties $Z\supset Z_2\supset\dots \supset Z_{n}$, with
$Z_i$ of dimension $n-i$ and connected (take multiple hyperplane sections up to $Z_{n-1}$ and
then a single point of $Z_{n-1}$ for $Z_n$. The Gysin exact triangles give commutative diagrams
\[\begin{CD}
H^{2n-2i-2}(Z_{i+1},n-i)@>>> H^{2n-2i}(Z_i,n-i+1)\\
@VVV @VVV\\
F_2(Z_{i+1})@>>> F_2(Z_i)
\end{CD}\]
in which both horizontal maps are isomorphisms: thus we are reduced to the case $\dim Z=0$,
where it follows from Proposition \ref{sus} applied to $X=\P^1$.
\end{proof}

\begin{thm} The map \eqref{corot} is an isomorphism for $j=0,1$.
\end{thm}

\begin{proof} This is easy and left to the reader for $j=0$. For $j=1$, we argue as usual by
blowup induction.  In the situation of
\ref{blowups}, we then have a commutative diagram of exact sequences
\[\begin{CD}
\scriptstyle H^{2n-2}(\tilde X,n)\oplus H^{2n-2}(Z,n)& \to& \scriptstyle
H^{2n-2}(\tilde Z,n)&\to&\scriptstyle  H^{2n-1}(X,n)& \to&
\scriptstyle H^{2n-1}(\tilde X,n)\oplus H^{2n-1}(Z,n)\\ 
@VVV @VVV @VVV @VVV \\
\scriptstyle F_2(\tilde X)\oplus F_2(Z)& \to&\scriptstyle   F_2(\tilde
Z)& \to&\scriptstyle  F_1(X)& \to&\scriptstyle  F_1(\tilde X)\oplus
F_1(Z).
\end{CD}\]

In this diagram, we have $F_1(Z)=0$ by Lemma \ref{l11.6} and $H^{2n-1}(Z,n)\allowbreak=0$ by
Lemma
\ref{lvan}, and the same lemmas imply that both rightmost horizontal maps are surjective. The
rightmost vertical map is now an isomorphism by Proposition \ref{p13.4}, which also gives the
surjectivity of $H^{2n-2}(\tilde X,n)\to F_2(\tilde X)$. Finally, Lemma \ref{lsmallis} implies
that $H^{2n-2}(\tilde Z,n)\to F_2(\tilde Z)$ and $H^{2n-2}(Z,n)\to F_2(Z)$ are isomorphisms,
and the conclusion follows from the 5 lemma.
\end{proof}

\begin{cor} \label{c14.4} For $X\in \Sch$ of dimension $n$, we have exact sequences
\begin{multline*}
0\to H^{2n-1}_\cdh(X,\Z(n))\otimes \Q/\Z\to \ker(u^*_1)\otimes \Q/\Z\\
\to H^{2n}_\cdh(X,\Z(n))_\tors\to\coker(u^*_1)_\tors\to 0
\end{multline*}
\begin{multline*}
0\to H^{2n-1}_\cdh(X,\Z(n))\otimes \Q/\Z\to H^{2n-1}_{(1)}(X,\Q/\Z(n))\\
\to H^{2n}_\cdh(X,\Z(n))_\tors\to 0
\end{multline*}
where $u^*_1$ is the map involved in the $1$-motive $\LA{1}^*(X)$ (which is isomorphic to $\Alb^+(X)$ by the dual of Theorem \ref{*=-}).
\end{cor}

\begin{cor}\label{c14.5} If $X$ is a \emph{proper}\/ scheme of dimension $n$ we then get $H^{2n-1}(X,\Z(n))\otimes \Q/\Z=0$ and an isomorphism
\[\Alb^+(X)(k)_\tors\iso H^{2n}_\cdh(X,\Z(n))_\tors.\]
\end{cor}

\begin{proof} If $X$ is proper then $\LA{1}^*(X)\cong \Alb^+(X)$ is semi-abelian and the claim follows from the previous corollary.\end{proof}

\begin{remark}\label{r14.5} Marc Levine outlined us how to construct a ``cycle map" $ c\ell^\cdh$ from $CH^n_{LW}(X)$ to
$H^{2n}_\cdh(X,\Z(n))$, where $CH^n_{LW}(X)$ is the Levine-Weibel cohomological Chow group of zero cycles (see \cite[\S 6.2]{KVS} for the definition). This gives a map
$$ c\ell^\cdh_\tors :CH^n_{LW}(X)_\tors\to H^{2n}_\cdh(X,\Z(n))_\tors $$
which most likely fits in a commutative diagram (for $X$ projective)
\[\begin{CD}
CH^n_{LW}(X)_\tors@>c\ell^\cdh_\tors>> H^{2n}_\cdh(X,\Z(n))_\tors\\
@V{a^+_\tors}VV @V{\wr}VV\\
\Alb^+(X)(k)_\tors@>\sim>> \LA{1}^*(X)(k)_\tors
\end{CD}\]
where: the horizontal bottom isomorphism is that induced by Theorem \ref{*=-} and the right vertical one comes from the previous Corollary \ref{c14.4}; the left vertical map is the one induced, on torsion, by the universal regular homomorphism $a^+:CH^n_{LW}(X)_{\deg 0}\to \Alb^+(X)(k) $ constructed in \cite[6.4.1]{BSAP}. This would imply that 

\begin{center}
$c\ell^\cdh_\tors$  is an isomorphism $\iff$ $a^+_\tors$  is an isomorphism.
\end{center}

If $X$ is normal and $k = \bar k$ or for any $X$ projective if $k= \C$ then
$a^+_\tors$ is known to be an isomorphism, see \cite{KVS}. For $X$ projective over
any algebraically closed field, see Mallick \cite{mallick}.

We expect that Levine's ``cycle map" $ c\ell^\cdh$ is surjective with uniquely divisible kernel (probably representable by a unipotent group).
\end{remark}

\newpage
\part{Realisations}

\section{An axiomatic version of Deligne's conjecture}\label{axDel}

Let $k$ be a perfect field. As usual we drop the reference to $k$ from the notation for categories
of motives associated to $k$.

\subsection{A review of base change}\label{sbase}

Suppose given a diagram of categories and functors
\begin{equation}\label{15.diag}
\xymatrix{
\cT_1\ar[r]^T&\cT\ar@<.7ex>[l]^A\\
\cD_1\ar[r]^S\ar[u]^{R_1}&\cD\ar@<.7ex>[l]^B\ar[u]_R
}\end{equation}
where $A$ is left adjoint to $T$ and $B$ is left adjoint to $S$, plus
a natural transformation $\phi:RS\Rightarrow TR_1$. Then we get a natural transformation
\[\psi:AR\Rightarrow R_1B\]
as the adjoint of the composition
\[R\Rightarrow RSB\Rightarrow TR_1B 
\]
where the first natural transformation is given by the unit $Id_\cD\Rightarrow
SB$ and the second one is induced by $\phi$. 

We are interested in proving that $\psi$ is an isomorphism of functors under certain hypotheses. Suppose that all categories and functors are triangulated: then it suffices  to check this on generators of $\cD$. 

\subsection{A weight filtration on $\M\otimes\Q$}\label{s15.2} In this subsection, we prove that
the weight filtration on $1$-motives defines a weight filtration on $\M\otimes\Q$ in the sense
of Definition \ref{dE.6}.

For $w\in \Z$, let $(\M\otimes\Q)_{\le w}$ be the full subcategory consisting of $1$-motives of
weight $\le w$ (\S\ref{s.weights}, \cf \cite[(10.1.4)]{D}). Thus $(\M\otimes\Q)_{\le w}=0$ for $w<-2$ and $(\M\otimes\Q)_{\le
w}=\M\otimes\Q$ for $w\ge 0$. 

\begin{propose}\label{pwM} The inclusion functors
\[\iota_w:(\M\otimes\Q)_{\le w}\into (\M\otimes\Q)_{\le w+1}\]
define a weight filtration on $\M\otimes\Q$.
\end{propose}

\begin{proof} By Remark \ref{lF.3.6}, it suffices to check that the weight filtration on
$1$-motives verifies the conditions in \cite[p. 83, Def. 6.3 a)]{jannsen}. The only point is
its exactness, which is clear.
\end{proof}

\subsection{A left adjoint in the category of realisations} \label{realalb}
Let $K$ be a field and $\cT$ be a triangulated $K$-linear  category. 
We assume given a  $t$-structure on $\cT$ whose heart $\cB$ is provided with a
weight filtration $\cB_{\le n}$ in the sense of Definition \ref{dE.6}. For convenience, we
assume that $\cB_{\le 0}=\cB$.

We give ourselves a Serre subcategory $\cB_\L$ of $\cB_{-2}$; we call the objects of $\cB_\L$ the \emph{Lefschetz objects}.

\begin{hyp}\label{ss} There is a full abelian subcategory $\cB_{-2}^\tr$ of $\cB_{-2}$ such that 
\begin{thlist}
\item $\Hom(\cB_\L,\cB_{-2}^\tr)=\Hom(\cB_{-2}^\tr,\cB_\L)=0$.
\item Any object of $\cB_{-2}$ is the direct sum of an object of $\cB_\L$ and an object of $\cB_{-2}^\tr$.
\end{thlist}
\end{hyp}

Note that under this hypothesis, $\cB_{-2}^\tr$ is a Serre subcategory of $\cB_{-2}$;  
every object $H\in \cB_{-2}$ has a unique decomposition 
\begin{equation}\label{dec}
H=H_\L\oplus H^\tr
\end{equation}
 with $H_\L\in
\cB_\L$ and $H^\tr\in \cB_{-2}^\tr$. 

Also, if $\cB_{-2}$ is semi-simple, we have $\cB_{-2}^\tr=\cB_\L^\perp$, \cf \cite[Lemme 2.1.3]{AK}.

\begin{defn}\label{dlev1} An object $H\in \cB$ is \emph{of level $\le 1$} if
\begin{thlist}
\item The weights of $H$ belong to $\{-2,-1,0\}$;
\item $H_{-2}$ is a Lefschetz object.
\end{thlist}
We write $\cB_{(1)}$ for the full subcategory of $\cB$ consisting of objects of level $\le 1$.
\end{defn}

\begin{propose}\label{pAlbB} \
\begin{enumerate}
\item $\cB_{(1)}$ is a Serre subcategory of
$\cB$.
\item The inclusion functor $\cB_{(1)}\into \cB$ has an exact
left adjoint $H\mapsto \Alb^\cB(H)$.
\end{enumerate}
\end{propose}

\begin{proof} (1) Given a short exact sequence in $\cB$, its middle term verifies (i) and (ii) of Definition \ref{dlev1} if and only if its extreme terms do: this is clear for (ii) by the thickness (Serreness) of $\cB_\L$ and for (i) by the properties of a weight filtration (see esp. Proposition \ref{proadj} (2)).

(2) We shall construct $\Alb^\cB$ as the
composition of two functors: 
\begin{itemize}
\item The first functor sends an object $H$ to $H_{>-3}$.
\item Suppose that $H_{\le -3}=0$, and consider $H_{-2}$.  
Let $H_{-2}=(H_{-2})_\L\oplus H_{-2}^\tr$ be the canonical decomposition \eqref{dec}, so that 
\[\Hom_\cB(L,H_{-2}^\tr) =\Hom_\cB(H_{-2}^\tr,L)=0\] 
for any Lefschetz object $L$.   Then the second functor sends $H$ to $H/H_{-2}^\tr$.
 \end{itemize}

The fact that this indeed defines an  exact left adjoint is readily checked (use again Proposition \ref{proadj} (2)). 
\end{proof}

\begin{remark}\label{r15.1}
  Let 
\[\cB^\tr_{\le -2} = \{H\in \cB_{\le -2}\mid (H_{-2})_\L=0\}.\]

This is a Serre subcategory of $\cB_{\le -2}$. Then Proposition \ref{pAlbB} and Proposition   \ref{psplit} (1) (ii) yield a split exact sequence of abelian categories, in the sense of Definition \ref{dsplit}:
\begin{equation}\label{eq15.2}
0\to \cB^\tr_{\le -2}\by{i} \cB\longby{\Alb^\cB} \cB_{(1)}\to 0
\end{equation}
in which the right adjoint of $\Alb^\cB$ is the natural inclusion.
\end{remark}

\begin{propose}\label{pAlbT} Let $\cT_{(1)}$
be the full subcategory of
$\cT$ consisting of objects $T$ such that $H^i(T)\in \cB_{(1)}$
for all
$i\in\Z$. Then:
\begin{enumerate}
\item $\cT_{(1)}$ is a thick subcategory of $\cT$, and the $t$-structure of
$\cT$ induces a $t$-structure on $\cT_{(1)}$.
\item If the $t$-structure is bounded,  the inclusion functor
$\cT_{(1)}\into \cT$ has a $t$-exact left adjoint
$\LAlb^\cT$.
\end{enumerate}
\end{propose}

\begin{proof} For (1), just note that $\cB_{(1)}$ is thick
in $\cB$ by Proposition \ref{pAlbB} (1) (\cf Proposition \ref{psplitt1} (1) (2)). For (2), the exact sequence \eqref{eq15.2} of Remark
\ref{r15.1} yields an exact sequence of triangulated categories
\[0\to \cT^\tr_{\le -2}\by{i} \cT\by{\pi} \cT/\cT^\tr_{\le -2}\to 0\]
where $\cT^\tr_{\le -2}=\{C\in \cT\mid H^*(C)\in \cB^\tr_{\le -2}\}$. 
The claim now follows from  Proposition \ref{psplitt1} (3). More precisely, this proposition shows that $\pi$ has a right adjoint $j$
such that $j(\cT/\cT^\tr_{\le 2}) = \cT_{(1)}$; then $\pi$ gets identified with the desired functor
$\LAlb^\cT$.
\end{proof}

\subsection{Realisation functor} \label{sextend0} Let $K,\cT,\cB$ be as in
\ref{realalb}. For $T\in \cT$, we denote by $H_*^\cT(T)$ the homology objects of $T$ (with
values in $\cB$) with respect to the $t$-structure of $\cT$.   We give ourselves a (covariant)
triangulated functor
$$R:\DM_\gm^\eff\otimes \Q\to \cT.$$
 (Note: if $R\ne 0$,  this implies that $K$ is of characteristic zero.)  For $X\in \Sm(k)$ and $i\in\Z$ we write $$H_i^R(X):=H_i^\cT(R(M(X))).$$ We assume:

\begin{hyp}\label{h15.1} If $X$ is smooth projective connected of dimension $d\le 1$, we have
\begin{enumerate}
\item $RM(X)\in \cT_{[-2d,0]}$ (in particular, $H_i^R(X)=0$ for $i\notin [-2d,0]$).
\item If $d=1$ and $E$ is the field of constants of $X$, the map $H_0^R(X)\to H_0^R(\Spec E)$
is an isomorphism.
\item If $d=1$, $E$ is the field of constants of $X$ and $f:X\to \P^1_E$ is a nonconstant
rational function, the map $f_*:H_2^R(X)\to H_2^R(\P^1_E)$ is an isomorphism.
\end{enumerate}
\end{hyp}

\begin{propose}\label{p15.1} Under Hypothesis \ref{h15.1}, the composition
\[\M\otimes\Q\to D^b(\M)\otimes\Q\longby{\Tot}\DM_\gm^\eff\otimes\Q\by{R} \cT\]
has image in the heart $\cB$ of $\cT$. The corresponding functor
\[R_1:\M\otimes\Q\to \cB\]
is exact. If moreover
\begin{enumerate}
\item[(W)]  $H_i^R(X)\in \cB_{-i}$ for all $i\in\N$
\end{enumerate}
for $X$ smooth projective with $\dim X\le 1$, then $R_1$ respects the splittings of
$\M\otimes\Q$ and $\cB$  in the sense of Definition \ref{dF.4.4} (2).
\end{propose}

\begin{proof}
By definition of a $t$-structure, the first assertion will hold provided $R\circ
\Tot(N)\in\cB$ for any 1-motive $N$ of pure weight. We are then left with lattices ($[L\to
0]$), tori ($[0\to T]$) and abelian varieties ($[0\to A]$). Moreover, since we work with
rational coefficients, we may assume that $L=R_{E/k}\Z$ and $T=R_{E/k} \G_m$ for a finite
separable extension $E/k$.

We have $\Tot([L\to 0])=M(\Spec E)$ and $\Tot([0\to T])=M(\Spec E)(1)$; hence
$R\Tot([L\to 0]) =RM(\Spec E) =H_0^R(\Spec E)[0]\in \cB$ by Hypothesis \ref{h15.1} (1).
On the other hand, since the only idempotents in $\End(M(\P^1_E))$ are
the K\"unneth idempotents, yielding the decomposition
$M(\P^1_E)=M(\Spec E)\oplus M(\Spec E)(1)[2]$, we get $R\Tot([0\to T])=H_2^R(\P^1_E)\in
\cB$, and also $H_1^R(\P^1_E)=0$, by Hypothesis \ref{h15.1} (2). 

For abelian varieties, we may reduce to the case of Jacobians of curves. Recall Voevodsky's functor $\Phi_\Q$ obtained by tensoring \eqref{eqvf} with $\Q$.
Let $C$ be a smooth projective geometrically connected $k$-curve: the choice of a closed point $c\in C$
determines a Chow-K\"unneth decomposition. Also, $\Tot([0\to J(C)])\simeq 
\Phi_\Q(h_1(C))[-1] $ as a Chow-K\"unneth direct summand of $M(C)[-1]$ (see \ref{s2.4.1}  and  Remark \ref{r16.3}). This already
shows that $R\Tot([0\to J(C)])\in \cT_{[-1,1]}$, by Hypothesis \ref{h15.1} (1).

Choose a
nonconstant rational fonction $f:C\to \P^1$, and let $x=f(c)\in \P^1$. If $\pi_i^c$ and
$\pi_i^x$ are the corresponding Chow-K\"unneth projectors, we clearly have
$f_*\pi_0^c=\pi_0^xf_*$. Hence the matrix of $f_*$ on the decompositions $h(C)=h_0(C)\oplus
h_1(C)\oplus h_2(C)$ and
$h(\P^1)=h_0(\P^1)\oplus h_2(\P^1)$ is of the form
\[f_* = \begin{pmatrix}
1&0&*\\
0&0&1
\end{pmatrix}
\]
if we identify $h_0(C),h_0(\P^1)$ with $\un$ and $h_2(C),h_2(\P^1)$ with $\L$. (One can compute
that $*$ equals $\pm\left(f^{-1}(f(c))-\deg f \cdot c\right)$, as an element of
$J(C)(k)\otimes \Q$.) Thus we have an exact triangle in $\DM_\gm^\eff\otimes\Q$:
\[\Tot([0\to J(C)][1])\to M(C)\longby{f_*}M(\P^1)\by{+1}\]
hence a long exact sequence in $\cB$:
\[\dots\to H_{i+1}^R(\P^1)\to H_i^\cT(R\Tot([0\to J(C)]))\to H_i^R(C)\longby{f_*} H_i^R(\P^1)\to
\dots\]

Using now the computation of $H_*^R(\P^1)$ and Hypothesis \ref{h15.1} (3), we find:
\[H_i^\cT(R\Tot([0\to J(C)][1]))=
\begin{cases}
0&\text{$i\ne 1$}\\
H_1^R(C)&\text{$i=1$.}
\end{cases}
\]

This shows that $R\Tot([0\to J(C)])\in \cB$, hence the functor $R_1:\M\otimes\Q\to \cB$. By Lemma \ref{lJ.1}, the composition
\[D^b(\M)\otimes\Q\longby{\Tot}\DM_\gm^\eff\otimes\Q\longby{R}\cT\]
is $t$-exact relatively the canonical $t$-structure on $D^b(\M)\otimes\Q$ (the motivic one,
with heart $\M\otimes\Q$), and its restriction $R_1$ to the hearts is exact.

If Condition (W) is verified, the proof shows that $R_1$ respects the splittings of
$\M\otimes\Q$ and $\cB$ in the sense of Definition \ref{dF.4.4} (2).
\end{proof}

\subsection{The base change theorem}\label{bch} Let $K, \cB,\cT,R$ be as in \S\ref{sextend0}. If
$X$ is a smooth projective $k$-variety, we set as before
$H^R_i(X):=H_i^\cT(R(M(X))$
where $H_*^\cT$ denotes the homology functors with values in $\cB$ defined by the $t$-structure
on $\cT$.

We have a commutative diagram 
\[\begin{CD}
&&&& X@>>>\pi_0(X)\\
&&&& @V{a_X}VV @V{\nu}VV\\
0@>>> \cA^0_{X/k}@>>>  \cA_{X/k}@>>> \Z[\pi_0(X)]@>>> 0
\end{CD}\]
 where $a_X$ is the  Albanese map and $\nu$ is the natural map;
hence a refined map
\begin{equation}\label{r16.1}
a_X^1:X\to\cA^1_X\df \cA_X\times_{\Z[\pi_0(X)]}\pi_0(X)
\end{equation}
 where $\cA^1_X$ is an $\cA^0_X$-torsor over $\pi_0(X)$. (Note that $a_X^1$ is a $\pi_0(X)$-morphism.)

\begin{hyp}\label{h16.1} 
We assume Hypothesis \ref{h15.1}, and moreover:
\begin{enumerate}
\item The $t$-structure on $\cT$ is bounded.
\item The restriction of $R_1$ to $(\M\otimes\Q)_{-2}$ induces a full embedding $(\M\otimes K)_{-2}\into \cB_\L$, with essential image the full subcategory of semi-simple objects.
\item For any $X$ smooth projective:
\begin{thlist} 
\item $H^R_i(X):=H_i(R(M(X))=0$ for $i<0$.
\item $H^R_i(X)\in \cB_{-i}$ for $i\ge 0$.
\item $H^R_0(X)\iso H^R_0(\pi_0(X))$.
\item  The map   \eqref{r16.1} induces an isomorphism 
\[(a_X^1)_*:H^R_1(X)\iso H^R_1(\cA^1_{X/k}).\]
\end{thlist}
\end{enumerate}
\end{hyp}

Let $R_1:=R\Tot$. Using Proposition \ref{pAlbT} and Section
\ref{sbase}, we get from this equality a base change morphism 
\begin{equation}\label{eq15.1}
v:\LAlb^\cT R\Rightarrow R_1\LAlb^\Q.
\end{equation}

\begin{lemma}\label{comreal2} Under \ref{h16.1}, we have a
commutative diagram
\[\begin{CD}
\Hom_{\DM_\gm^\eff\otimes\Q}(M(X),\Z(1)[2])@>R>> \Hom_\cT(R(X),R(\Z(1))[2])\\
@V{\alpha}VV @V{\beta}VV\\
\Hom_\cB(R_1\LA{2}^\Q(X),\Lambda)@>v^*>> \Hom_\cB(\Alb^\cB
H^R_2(X),\Lambda)
\end{CD}\]
for any smooth projective variety $X$, where $\Lambda\df R_1([0\to\G_m])$.
\end{lemma}

\begin{proof} Playing with the adjunctions and $t$-structures, we 
have chains of maps
\begin{multline}\label{eq16.1}
(\DM_\gm^\eff\otimes\Q)(M(X),\Z(1)[2])\simeq (D^b(\M)\otimes\Q)(\LAlb^\Q(X),[0\to
\G_m][2])\\
\longby{{}^tH_2}(\M\otimes\Q)(\LA{2}^\Q(X),[0\to \G_m])
\longby{R_1}\cB(R_1\LA{2}^\Q(X),\Lambda)
\end{multline}
and
\begin{multline}\label{eq16.2}
\cT(R(X),R(\Z(1))[2])\simeq\cT(R(X),\Lambda[2])\\
\longby{H_2^R}\cB(H^R_2(X),\Lambda)\simeq\cB(\Alb^\cB H^R_2(X),\Lambda).
\end{multline}

This defines respectively $\alpha$ and $\beta$. By following the various adjunction
isomorphisms, the diagram is commutative as claimed.
\end{proof}

\begin{remark}\label{r16.2}
Note that $\Hom_\cB(R_1\LA{2}^\Q(X),\Lambda)=\NS(X)\otimes K$ by Corollary 10.2.3
and Hypothesis \ref{h16.1} (2). Also, we have an injection $\Hom_\cB(\Alb^\cB
H^R_2(X),\Lambda)\into \Hom_\cB(H^R_2(X),\Lambda)$ by (the proof of) Proposition \ref{pAlbT}
(2). If
$\cB$ sits in a larger ``non-effective" category $\cB'$ which carries a duality, the latter
group may be interpreted as
$\Hom_\cB(K_{\cB'},H^2_R(X)(1))$, with $K_\cB:=R_1(\Z)$. Finally, the composition $\beta R$ in
the diagram of Lemma
\ref{comreal2}, followed by the latter inclusion, is easily checked to be the cycle class map
in the classical cases. So  the bottom row of this diagram contains an abstract argument that
algebraic equivalence is weaker than homological equivalence in codimension $1$.

On the other hand, it is known that algebraic and numerical equivalences coincide rationally
in codimension $1$. Thus, if $H_*^R$ defines an adequate equivalence relation on algebraic
cycles, we automatically get that $v^*$ is injective in the diagram of  Lemma \ref{comreal2}.
This will be the case if $H^*_R$ defines a Weil cohomology theory, which will follow if $\cB$
and $\cT$ can be extended to categories satisfying natural extra axioms (tensor structure,
duality),
\cf Cisinski-D\'eglise
\cite{cis-deg}.
\end{remark}

\begin{thm}\label{abstractdelconj} Under \ref{h16.1},
\begin{enumerate}
\item The base change morphism  $v$ is an isomorphism in weights $0$ and $-1$. 
\item Let $(\DM^\eff_\gm)^\cB$ be the thick  triangulated subcategory of $\DM_\gm^\eff\boxtimes\Q$ generated
by the $M(X)$ where $X$ is a smooth projective variety such that
\begin{thlist}
\item $H_2^R(X)$ is a semi-simple object of $\cB_{-2}$.
\item For any finite extension
$E/k$, the map
$v^*$ of Lemma \ref{comreal2} is injective for $X_E$ (see Remark \ref{r16.2}) and the
``geometric cycle class map"
\begin{multline}\label{eq16.3}
\Pic(X_E)\otimes K=\Hom_{\DM_\gm^\eff\otimes\Q}(M(X_E),\Z(1)[2])\otimes K\\
\longby{R} \Hom_\cT(R(X_E),R(\Z(1))[2])\longby{H_2^R} \Hom_\cB(H^R_2(X_E),R(\Z(1)))
\end{multline}
is surjective.
\end{thlist}
 Then the
restriction of
$v$ to
$(\DM^\eff_\gm)^\cB$ is an isomorphism. 
\end{enumerate}
\end{thm}

\begin{proof} (1) By de Jong's theorem, it suffices to prove the statement for $M=M(X)$, $X$
smooth projective. Thus we have to prove that, for all $i\in\Z$, the map
\[\Alb^\cT H^R_i(X)\simeq H_i^\cT(\LAlb^\cT R(X))\to R_1\LA{i}^\Q(X)\]
defined by $v_{M(X)}$ is an isomorphism in weights $0$ and $-1$. Here we used the $t$-exactness of $\LAlb^\cT$, \cf Proposition \ref{pAlbT} (2).

For $i<0$ (resp. $i> 2$), both sides are $0$ by Corollary 10.2.3 and \ref{h16.1} (3) (i) (\resp
(ii)). For
$i=0,1$, (iii) and (iv) imply that the map is an isomorphism. Finally, \ref{h16.1} (3) (ii)
implies that
$H^R_2(X)$ is pure of weight $-2$ and so is $\LA{2}^\Q(X)$, hence the statement is still true in
this case.

(2) It is sufficient to show that, for any smooth projective $X$ verifying the condition of (2),
the morphism
$\Alb^\cT H^R_2(X)\to R_1\LA{2}^\Q(X)$ is an isomorphism. Note that both sides are semi-simple Lefschetz
objects: semi-simplicity is clear for the right hand side, while for the left hand side it follows from the assumed semi-simplicity of $H_2^R(X)$ since $\Alb^\cT H^R_2(X)$ is a direct summand of this object (see proof of Proposition \ref{pAlbB}). Thus, by Yoneda's lemma, it suffices to show that for any semi-simple Lefschetz object $\Theta$,
the map
\begin{equation}\label{eq16.0}
\Hom(R_1\LA{2}^\Q(X),\Theta)\to \Hom(\Alb^\cT
H^R_2(X),\Theta)
\end{equation}
is an isomorphism.

By Hypothesis \ref{h16.1} (2), we have  $\Theta\simeq R_1([0\to T])$ for some torus $T$.   Without loss of generality, we may assume that $T=R_{E/k}\G_m=M(\Spec E)\otimes\Z(1)[1]$. In the commutative diagram of Lemma \ref{comreal2}
for $X_E$,  the composition of $\beta$ and $R$ is surjective by the surjectivity of
\eqref{eq16.3}. Therefore
\eqref{eq16.0} is surjective. By assumption, \eqref{eq16.0} is also injective. This concludes
the proof.
\end{proof}

\section{The Hodge realisation} \label{Hodge} Let $\MHS$ denote the abelian
category  of (graded polarizable, $\Q$-linear) mixed Hodge structures.  Recall that, for a mixed Hodge structure  $(H, W, F)$ we have $(i,j)$-Hodge components $ H^{i,j}\df (\gr^W_{i+j} H_{\sC})^{i,j}$ of the associated pure Hodge structure of weight $i+j$.  We write as usual $h^{ij}=\dim H^{i,j}$ for the $(i,j)$-{th} Hodge number. We say that  $(H, W, F)$ is \emph{effective} if $h^{ij}=0$ unless $i\leq 0$ and $j\leq 0$. 

We take for $\cB$ the full subcategory $\MHS^\eff$ of $\MHS$ given by effective mixed Hodge structures. For
$\cT$ we take $D^b(\cB)$. The weight filtration provides $\cB$ with a weight filtration in the
sense of Definition \ref{dE.6}. We take for
$\cB_\L$ the Hodge structures purely of type $ (-1,-1)$.  With the notation of \S
\ref{realalb}, $\cB_{(1)}=\MHS_{(1)}$ is the full subcategory of $\MHS^\eff$ given by mixed
Hodge structures of type $\{(0,0),
(0, -1), (-1,0), (-1,-1)\}$.

\subsection{$\LAlb^{\cT}$ for mixed Hodge structures} Note that $\MHS_{-2}$ is semi-simple
since pure polarizable Hodge structures are.  As a special case of Proposition
\ref{pAlbT}, we therefore have:

\begin{propose} The full embeddings $\iota: \MHS_{(1)}\into \MHS^{\eff}$ and
$\iota:D^b(\MHS_{(1)})\into D^b(\MHS^\eff)$ have left adjoints $\Alb^\MHS$ and
$\LAlb^\MHS$.\qed 
\end{propose}

\begin{notation} \label{not1}
For $H\in \MHS^{\eff}$ we shall write $H_{\leq 1}$ for $\Alb^\MHS(H)$.  If $h^{pq}=0$ unless $p,q\geq 0$, then  
the dual $H^\vee=\ihom (H, \Q)\in \MHS^{\eff}$ is effective and we denote
$$H^{\leq 1}\df \ihom (\ihom (H, \Q)_{\leq 1}, \Q(1)).$$

Denoting by $\Pic^\MHS$ the Cartier dual of $\Alb^\MHS$, the latter $H^{\leq 1}$ translates in $\Pic^\MHS(\ihom (H, \Q))$.
\end{notation}

\begin{remark}\label{delnot} In Deligne's notation \cite[10.4.1]{D}, for $H\in \MHS$ we
have:
\begin{itemize}
\item ${\rm II}_n(H)_\Q = H (n)_{\leq 1}$ if  $H(n)$ is effective
\item ${\rm I} (H)_\Q = H^{\leq 1}$ if   $H^\vee$ is effective. 
\end{itemize}
Note that $H \mapsto H^{\leq 1}$ is actually right adjoint to the (fulll embedding) functor $H\mapsto H (-1)$ from $\MHS_{(1)}$ to the full subcategory $\MHS_{\eff}$ of $\MHS$ consisting of objects whose dual is effective.
\end{remark}

\subsection{Huber's Hodge realisation functor} We have A. Huber's  realisation functor
\cite{HuberH} (see also \cite{Lecomte-Wach} for an integral refinement)
\[R_\Hodge:\DM_\gm^\eff(\C)\otimes\Q\to D^b(\MHS_{\eff}).\]

For $X$ a smooth variety, we have $R_\Hodge(M(X))=R\Gamma(X,\Q)$, and in
particular $R_\Hodge(M(X))^\vee$ is effective. This functor is contravariant. To get a covariant functor, we compose it with the (exact)
duality of $D^b(\MHS)$ sending $H$ to $\ihom(H,\Q)$.  Thus we get a functor
\[R^\Hodge:\DM_\gm^\eff(\C)\otimes\Q\to D^b(\MHS^{\eff}).\]

Consider the Voevodsky functor $\Phi_\Q:\Chow^\eff\boxtimes\Q\to \DM_\gm^\eff\boxtimes\Q$ of \eqref{eqvf}.
Since $R_\Hodge\Phi_\Q:\Chow^\eff\boxtimes\Q\to D^b(\MHS_{\eff})$ is by construction isomorphic to the
functor $X\mapsto R\Gamma(X,\Q)$, the conditions of Hypothesis \ref{h15.1} are verified, as
well as Condition (W) of Proposition \ref{p15.1}. This proposition then shows that
$R^\Hodge\Tot$ defines an exact functor 
\[R_1^\Hodge:\M(\C)\otimes\Q\to \MHS_{(1)}.\]

Although this is irrelevant for our purpose, it is nice to know:

\begin{thm}[\cf \protect{\S \ref{16.1}}]\label{t16.1} The functor $R_1^\Hodge$ is an equivalence of categories.
\end{thm}

\begin{proof} Note that $R_1^\Hodge$ respects the splittings of $\M\otimes\Q$ and $\MHS_{(1)}$,
by Proposition \ref{p15.1}. Therefore we are in a position to apply Theorem \ref{tglueII}.

Let $W_i\M(\C)\otimes\Q$ and $W_i(\MHS_{(1)})$ denote the full subcategories of objects pure of weight $i$ ($i=0,-1,-2$). We first check that $R_1$ induces equivalences of categories 
$W_i\M(\C)\otimes\Q\iso W_i(\MHS_{(1)})$. Note that these three categories are semi-simple.
The cases $i=0$ and $i=-2$ are obvious. For $i=-1$, it is known (\eg from Deligne's
equivalence of categories \cite[10.1.3]{D}) that any $H\in W_{-1}(\MHS_{(1)})$ is a direct
summand of some
$H_1(C)$: this proves essential surjectivity. For the full faithfulness, we reduce to proving
that, given two connected smooth projective curves, the map
\[\Hom(J(C),J(C'))\to \Hom(H_1(C),H_1(C'))\]
given by $R_1^\Hodge$ is the usual action of divisorial correspondences, which follows from the
construction of $R^\Hodge$. 

(Note that these arguments provide natural isomorphisms of the restrictions of $R_1^\Hodge$ to
$W_i\M(\C)\otimes\Q$ with Deligne's realisation functor.)

We are left to check the conditions  of Theorem \ref{tglueII} on isomorphisms of Hom and Ext
groups. Note that the condition on Hom groups in (2) is empty because they are identically $0$.
For the Ext groups, since
$\Ext^2_{\MHS}$ is identically $0$, we reduce by (4)  to prove that, for $N_m,N_n\in
\M\otimes\Q$ of pure weights $m<n$, the map 
\[\Ext^1_{\M\otimes\Q}(N_n,N_m)\to \Ext^1_{\MHS}(R_1^\Hodge(N_n),R_1^\Hodge(N_m))\]
is bijective.

Since $\Tot:D^b(\M\otimes \Q)\to \DM_\gm^\eff\boxtimes\Q$ is fully faithful, it suffices to
prove that the map
\begin{multline*}
\Hom_{\DM_\gm^\eff\boxtimes\Q}(\Tot(N_n),\Tot(N_m)[1])\\
\to \Hom_{D^b(\MHS)}(R^\Hodge\Tot(N_n),R^\Hodge\Tot(N_m)[1])
\end{multline*}
is bijective. We distinguish $3$ cases:

\begin{thlist}
\item $(m,n)=(-1,0)$. Without loss of generality, we may assume $N_0=[\Z\to 0]$, $N_{-1}=[0\to
J_C]$ for a smooth projective curve $C$. Then $\Tot(N_{-1})=\Phi_\Q h_1(C)[-1]$,
$R^\Hodge\Tot(N_{-1})=H_1(C)$  and we are looking at the map
\[\Hom_{\DM\boxtimes\Q}(\Z,\Phi_\Q h_1(C))\to \Hom_{D^b(\MHS)}(\Z,H_1(C)[1]).
\]

By Poincar\'e duality, this map is equivalent to the map
\begin{multline*}
J_C(\C)\otimes\Q=\Hom_{\DM\boxtimes\Q}(\Phi_\Q h^1(C),\Z(1)[2])\\
\to\Hom_{D^b(\MHS)}(H^1(C),\Z(1)[1])=\Ext^1_{\MHS}(H^1(C),\Z(1))
\end{multline*}
which coincides with the Abel-Jacobi map\footnote{because, by construction, the restriction of
Huber's functor to pure motives is the ``usual" Hodge realisation functor.}, hence is bijective.
\item $(m,n)=(-2,-1)$. Without loss of generality, we may assume $N_{-2}=[0\to \G_m]$,
$N_{-1}=[0\to J_C]$ for a smooth projective curve $C$. Then $\Tot(N_{-1})=\Phi_\Q h^1(C)[-1]$,
$R\Tot(N_{-1})=H^1(C)$  and we are looking at the map
\[\Hom_{\DM\boxtimes\Q}(\Phi_\Q h^1(C),\Z(1)[2])\to \Hom_{D^b(\MHS)}(H^1(C),\Z(1)[1])
\]
which is the Abel-Jacobi map as in (ii).
\item $(m,n)=(-2,0)$. Without loss of generality, we may assume $N_{-2}=[0\to \G_m]$,
$N_0=[\Z\to 0]$. We are now looking at the map
\begin{multline*}
\C^*\otimes \Q=\Hom_{\DM\boxtimes\Q}(\Z,\Z(1)[1])\\
\to\Hom_{D^b(\MHS)}(\Z,\Z(1)[1])=\Ext^1_\MHS(\Z,\Z(1))
\end{multline*}
which is again the usual isomorphism, by definition of Huber's realisation functor.
\end{thlist}
\end{proof}

\subsection{Deligne's conjecture}
For any $M\in \DM_\gm^\eff(\C)\otimes\Q$,  from  \eqref{eq15.1} we get a
natural map
\begin{equation}\label{eqhodge}
R^\Hodge(M)_{\leq 1}\to R_1^\Hodge\LAlb(M).
\end{equation}
Taking homology of both sides, we get comparison maps
\begin{equation}\label{eqhomHodge}
H_i(R^\Hodge(M))_{\leq 1}\to R_1^\Hodge\LA{i}(M)
\end{equation}
for all $i\in\Z$. From Theorem \ref{abstractdelconj} and the Lefschetz 1-1 theorem, we now get:

\begin{thm} \label{isodel}
The maps \eqref{eqhodge} and \eqref{eqhomHodge} are isomorphisms for  all motives $M$.\qed
\end{thm}

This theorem recovers the results of \cite{BRS}, with rational coefficients.

The isomorphisms \eqref{eqhomHodge} may be applied to geometric motives like $M(X)$, $X$ any
$\C$-scheme of finite type (yielding $\LAlb(X)$), but also
$M(X)^*(n)[2n]$ and $M^c(X)$ yielding
$\LAlb^*(X)$ and $\LAlb^c(X)$ respectively. We thus get the following corollary, overlapping
Deligne's conjecture (see \ref{not1} and \ref{delnot} for the notation):

\begin{cor} \label{isodelcor}
Let $X$ be an $n$-dimensional complex algebraic variety. The mixed Hodge structures
$H^i(X,\Q)^{\leq 1}$, $H^{2n-i}(X,\Q(n))_{\leq 1}$ and $H_i^c(X,\Q)_{\leq 1}$ induced by the
mixed Hodge structures on the Betti cohomology  and the Borel-Moore homology admit
a purely algebraic construction provided by the previously explained isomorphisms
\begin{align*}
R_1^\Hodge(\RA{i}(X))&\iso H^i(X,\Q)^{\leq 1}\\
H^{2n-i}(X,\Q(n))_{\leq 1}&\iso R_1^\Hodge(\LA{i}^*(X))\\
H_i^c(X,\Q)_{\leq 1}&\iso R_1^\Hodge(\LA{i}^c(X))
\end{align*}
of mixed Hodge structures.
\end{cor}

\begin{remark}\label{notall} Deligne's conjecture \cite[(10.4.1)]{D} concerns three types of
Hodge structures of level $\le 1$ for $X$ of dimension $\le N$: $I(H^n(X,\Z))$,
$II_n(H^n(X,\Z))$ ($n\le N$) and $II_N(H^n(X,\Z))$ ($n\ge N$).
Corollary
\ref{isodelcor} covers the first and last (compare Remark \ref{delnot}), but not the second in
general. The issue for $II_n$ and $II_N$ is that the motive $M(X)^*(n)[2n]$ is effective for
$n\ge\dim X$ by Lemma
\ref{leffe}, but not for
$n<\dim X$ in general. Indeed, if $M(X)^*(n)[2n]$ is effective, then it is isomorphic  in
$\DM_-^\eff\otimes\Q$ to $\ihom_\eff(M(X),\Z(n)[2n])$ for formal reasons, and therefore the
latter is a geometric motive. But this is false \eg for $n=2$ and $X$ a suitable smooth
projective $3$-fold, see \cite{Ay-Hub}.

Suppose that the motivic $t$-structure exists on $\DM_\gm^\eff\otimes\Q$. By a recent result
of Beilinson \cite{BG}, this implies Grothendieck's standard conjecture B. For any $X$, let us then write
$M_i(X)$ for the $i$-th $t$-homology of $M(X)$. If $X$ is smooth projective, by Poincar\'e
duality and Conjecture B we find that
\[\ihom_\eff(M_n(X),\Z(n))\simeq M_n(X)\]
is geometric. By blow-up induction, this then implies that the motive
$\ihom_\eff(M_n(X),\Z(n))$ is geometric for any $X$ of finite type, and we obtain the remaining
part of Deligne's conjecture.
\end{remark}

\subsection{Deligne's Hodge realisation functor} \label{16.1} In Deligne's
construction, the {\it integrally defined}\, Hodge realization $$T^\Hodge(M)\df (T_{\sZ}(M), W_*, F^*)$$ of a 1-motive
(with torsion) $M$ over $k= \C$ (see \cite[\S 1]{BRS} and
\cite[10.1.3]{D}) is obtained as follows. The finitely generated abelian group $T_{\sZ}(M)$ is given by the pull-back of  $u : L\to G$ along $\exp : {\rm Lie} (G) \to G$, $W_*$ is the weight filtration 
$$W_iT(M) \df\left \{\begin {array}{rr} T_{\sZ}(M)\otimes \Q & i\geq 0\\
    H_1(G,\Q) & i= -1\\ H_1(T,\Q) & i= -2\\ 0 & i \leq -3 
\end{array} \right.$$ 
and $F^*$ is defined by
$F^0(T_{\sZ}(M)\otimes\C)\df \ker (T_{\sZ}(M)\otimes\C\to {\rm Lie} (G)).$ We have
that $T_{\sZ}(M)$, $W_*$ and $F^0$ are independent of the representative of $M$. Thus $T^\Hodge(M)$ is a mixed $\Z$-Hodge structure such that $\gr_{i}^W$ is polarizable.  We have $\gr_{0}^WT^\Hodge(M)\cong L\otimes\Q$, 
$\gr_{-1}^WT^\Hodge(M)\cong H_1(A,\Q)$ and $\gr_{-2}^WT^\Hodge(M)\cong
H_1(T,\Q)$ as pure polarizable $\Q$-Hodge structures. 

Let $\MHS^{\Z}_{1}$ be the category of mixed $\Z$-Hodge structure of type $\{(0,0), (0,-1), (-1,0), (-1,-1)\}$ such that $\gr_{i}^W$ is polarizable.  We have $\MHS^{\Z}_{1}\otimes \Q = \MHS_{(1)}$.

The functor $T^\Hodge$ is the {\em covariant}\,  Deligne Hodge realization 
$$T^\Hodge: {}^t\M(\C)\longby{\simeq} \MHS^{\Z}_{1}$$ which is an equivalence of
abelian categories by \cite[Prop. 1.5]{BRS}. It induces an equivalence  
\[T^\Hodge_{\Q}:D^b(\M(\C)\otimes\Q)\longby{\simeq} D^b(\MHS_{(1)}).\]

\begin{thm}[\protect{\cite{vadim}, \cite{ABN}}]\label{cfh} The functor $R_1^\Hodge$ of Theorem \ref{t16.1} is
isomorphic to $T^\Hodge_{\Q}$. 
\end{thm} 
Vologodsky \cite{vadim} gives an explicit construction of this isomorphism.    
Actually, such an isomorphism along with its uniqueness is a simple consequence of the fact that $1$-motives with torsion is a universal abelian category in the sense of Nori for an explicit diagram of curves: see \cite{ABN} for details in the more natural and general framework that applies to mixed realisations.

Note that Theorem \ref{cfh} is reproving Theorem \ref{t16.1}
(using \cite[10.1.3]{D}) and yield a naturally commutative diagram
\[\xymatrix{
D^b(\MHS_{(1)})\ar[r]^\iota&D^b(\MHS^\eff)\ar@<.7ex>[l]^{(-)_{\leq 1}}\\
D^b(\M(\C)\otimes\Q)\ar[r]^\Tot\ar[u]^{T^\Hodge_{\Q}}&\DM_\gm^\eff(\C)\otimes\Q\ar@<.7ex>[l]^{\LAlb^\Q}\ar[u]_{R^\Hodge}.
}\]

\section{The mixed realisation}\label{shuber}

We consider here the other part of Deligne's conjecture:

\begin{quote}
\it Les morphismes
\begin{align*}
T_\ell(I(H^n(X,\Z))) &\to (H^n(X, \Z_\ell)/\text{torsion})(1)\\
T_\ell(II_n(H^n(X,\Z))) &\leftarrow H^n(X, \Z_\ell)(n) \text{ (pour  $n\le N$)}\\
T_\ell(II_N(H^n(X,\Z))) &\leftarrow H^n(X, \Z_\ell)(N) \text{ (pour  $n \ge N$)}
\end{align*}
et leurs analogues en cohomologie de De Rham devraient aussi admettre une d\'efinition
purement alg\'ebrique.
\end{quote}

Huber's Hodge realisation is only a part of her construction of a functor from $\DM_\gm(k,\Q)$
($k$ a finitely generated field over $\Q$) to her category of mixed realisations. We exploit
this much richer structure to give a proof of the above conjecture, rationally and excluding
the second case (see Remark \ref{notall} for the latter).

\subsection{$\LAlb^\cT$ for mixed realisations} Let $k$ be a finitely generated extension of
$\Q$. Recall from \cite[Def. 11.1.1]{HuberLN} Huber's category
$\cMR$ of mixed realisations. An object $A\in \cMR$ is a collection
\[(A_\DR, A_\ell, A_{\sigma,\ell}, A_\sigma,A_{\sigma,\C}; I_{\DR,\sigma}, I_{\sigma,\C},
I_{\bar\sigma,\ell}, I_{\ell,\sigma})_{\ell\in \cP,\sigma\in S} Ê\]
where $\cP$ is the set of prime numbers, $S$ is the set of embeddings $\sigma: k\into \C$ and
\begin{itemize}
\item $A_\DR$ is a (finite dimensional) bifiltered $k$-vector space.
\item $A_\ell$ is a filtered $\Q_\ell$-adic representation of $G_k \df Gal(\bar k/k)$; it is
assumed to be constructible (\ie unramified over some model of finite type of $k/\Z$) and that
the filtration is a weight filtration.
\item $A_{\sigma,\ell}$ is a filtered $\Q_\ell$-vector space.
\item $A_\sigma$ is a filtered $\Q$-vector space.
\item $A_{\sigma,\C}$ is a filtered $\C$-vector space.
\item $I_{\DR,\sigma}: A_\DR\otimes_\sigma \C\to A_{\sigma,\C}$ is a filtered isomorphism.
\item $I_{\sigma,\C}: A_\sigma\otimes_\Q \C\to A_{\sigma,\C}$ is a filtered isomorphism.
\item $I_{\bar\sigma,\ell}: A_\sigma\otimes_\Q \Q_\ell\to A_{\sigma,\ell}$ is a filtered
isomorphism.
\item $I_{\ell,\sigma}: A_\ell\to A_{\sigma,\ell}$ is a filtered isomorphism.
\end{itemize}

Additionally, it is required that the
systems $(A_\sigma,A_\DR,A_{\sigma,\C},I_{\DR,\sigma},I_{\sigma,\C})$ define mixed Hodge
structures.

Morphisms in $\cMR$ are defined in the obvious way.

This is a refinement of the category
defined by Jannsen \cite[Ch. 1]{jannsen} and Deligne \cite{3points}: the refinement is that on
the
$\ell$-adic components of an object, the filtration is required to be a weight filtration.
Huber also defined a triangulated
$t$-category $D_\cMR$, with heart $\cMR$ (\loccit Def. 11.1.3); it comes with a canonical
functor (\loccit p. 94)
\[D^b(\cMR)\to D_\cMR\]
which is the identity on the hearts.

These categories are not sufficient for our purposes: instead, we shall use the categories
$\cMR^P$ and $D_{\cMR^P}$ of \emph{polarizable} mixed realisations, see \cite[Def. 21.1.1 and 21.1.3]{HuberLN}.
Recall that an object $A\in \cMR$ of pure weight $n$ is called polarizable if there is a morphism 
$A\otimes A\to \Q(-n)$ 
which induces a polarization on the associated pure Hodge structure. Let $\cMR^P_n$ be the abelian 
category of polarizable mixed realisations of weight $n$. The category $\cMR^P$ is the subcategory of $\cMR$ whose objects  $H$ satisfy $\gr_n^W H\in\cMR^P_n$ for all $n$. There is the same formalism as above \cite[Prop. 21.1.4]{HuberLN}, which maps to the previous one.

We need an ``effective" full subcategory $\cMR^{P,\eff}\subset \cMR^P$, and a corresponding
effective triangulated category $D_{\cMR^P}^\eff\subset D_{\cMR^P}$:

\begin{defn}  a) An object
$A\in \cMR^P$ is in $\cMR^{P,\eff}$ if
\begin{enumerate}
\item the associated Hodge structure is effective;
\item the eigenvalues of (arithmetic) Frobenius acting on the $\ell$-adic components $A_\ell$
are algebraic integers.
\end{enumerate}
b) An object $C\in D_{\cMR^P}$ is in $D_{\cMR^P}^\eff$ if $H_i^t(C)\in \cMR^{P,\eff}$ for all
$i\in \Z$, where $H_*^t$ is the homology respective to the canonical $t$-structure of
$D_{\cMR^P}$.
\end{defn}

In this way, the above functor refines to
\[D^b(\cMR^{P,\eff})\to D_{\cMR^P}^\eff.\]

By Remark \ref{rE.21} and \cite[Prop. 11.1.5]{HuberLN} (adapted to $\cMR^P$), $\cMR^P$ and
$D_{\cMR^P}$ enjoy weight filtrations in the sense of Definitions \ref{dE.6} and
\ref{d16.3tri}, which are compatible in the sense of Proposition \ref{psplitt}. Moreover, the
categories of pure weights $\cMR^P_n$ are semi-simple \cite[Prop. 21.1.2]{HuberLN}, which is
the main point of passing from $\cMR$ to $\cMR^P$. These weight filtrations induce compatible
weight filtrations on $\cB=\cMR^{P,\eff}$ and $\cT=D_{\cMR^P}^\eff$. We take for $\cB_\L$ the
full subcategory of $\cMR^P$ consisting of objects of the form $R_{\cMR^P}(A\otimes \bT)$, where
$A$ is a (pure) Artin motive and $\bT$ is the Tate motive. Here $R_{\cMR^P}$ is the
functor defined in \cite[Def. 21.2.4]{HuberLN}  (see also \cite[Prop. 21.2.5 \& Th. 20.2.3]{HuberLN}).\footnote{This construction involves fixing an algebraic closure $\bar k$ of $k$ and a section  $S\to \bar S$ of the projection $\bar S\to S$, where $\bar S$ is the set of 
embeddings $\bar\sigma: \bar k\into \C$. }

With the notation of \S \ref{realalb}, $\cB_{(1)}=\cMR^P_{(1)}$ is the full subcategory of
$\cMR^{P,\eff}$ of objects $A$ such that  $A_\DR\in \MHS_{(1)}$ (see beginning of \S \ref{Hodge}). 

We have seen that the full subcategory $\cMR^P_{-2}$ of pure objects of weight $-2$ is
semi-simple. As a special case of Proposition
\ref{pAlbT}, we therefore have:

\begin{propose} The full embeddings $\iota: \cMR^P_{(1)}\into \cMR^{P,\eff}$ and
$\iota:(D^\eff_{\cMR^P})_{(1)}\into D_{\cMR^P}^\eff$ have left adjoints $\Alb^{\cMR^P}$ and
$\LAlb^{\cMR^P}$.\qed 
\end{propose}

\begin{notation} For simplicity we shall write $(-)_{\leq 1}$ for $\Alb^{\cMR^P}$.
\end{notation} 

We shall need the following lemma:

\begin{lemma}\label{lisoMR} Let $X$ be smooth projective over $k$. Then the map
\[\Pic(X)\otimes \Q\to \Hom_{\cMR^P}(\Q,H^2_{\cMR^P}(X)(1))\]
given by the realisation functor $R_{\cMR^P}$ induces an isomorphism
\[\NS(X)\otimes \Q \iso \Hom_{\cMR^P}(\Q,H^2_{\cMR^P}(X)(1)).\]
\end{lemma}

\begin{proof} Clearly the given map factors through $\NS(X)\otimes \Q$, since all its
components do, and the resulting map is injective because its Betti components are. To prove
surjectivity, choose an embedding $\bar \sigma:\bar k\into \C$. For any finite extension $E/k$
we have a commutative diagram
\[\begin{CD}
\NS(X_E)\otimes \Q @>>> \Hom_{\cMR^P(E)}(\Q,H^2_{\cMR^P(E)}(X_E)(1))\\
@VVV @V{\bar\sigma_{|E}}VV\\
\NS(\bar X)\otimes \Q@>>> H^2_\Hodge(X_\C,\Q(1))^{(1,1)}
\end{CD}\]
in which all maps are injective ($\bar X = X\otimes_k \bar k$). The bottom row is an
isomorphism by the (1,1) theorem and the isomorphism $\NS(\bar X)\otimes \Q\iso\
\NS(X_\C)\otimes \Q$. Passing to the limit, we get a commutative diagram
\[\begin{CD}
\NS(\bar X)\otimes \Q @>>> \colim_E \Hom_{\cMR^P(E)}(\Q,H^2_{\cMR^P(E)}(X_E)(1))\\
||&& @V{\bar\sigma}VV\\
\NS(\bar X)\otimes \Q@>\sim>> H^2_\Hodge(X_\C,\Q(1))^{(1,1)}
\end{CD}\]
of injections, which forces the top horizontal map to be bijective. To conclude, we use the
commutative diagram
\[\begin{CD}
\NS(X)\otimes \Q @>>> \Hom_{\cMR^P}(\Q,H^2_{\cMR^P}(X)(1))\\
@V{\wr}VV @V{\wr}VV\\
(\NS(\bar X)\otimes \Q)^{G_k} @>>> \colim_E \Hom_{\cMR^P(E)}(\Q,H^2_{\cMR^P(E)}(X_E)(1))^{G_k}
\end{CD}\]
where the vertical maps are isomorphisms by the usual transfer argument. 

(Let us give some details on the last point. If $E/k$ is Galois of group $G$, then $G$ acts on $\Hom_{\cMR^P(E)}(\Q,H^2_{\cMR^P(E)}(X_E)(1))$ via its action on $X_E$, hence the Galois action on the colimit. To define transfers on the right hand side, note that for each $E/k$ the natural functor $\cMR^P\to \cMR^P(E))$ has a left adjoint sending $\Q$ to $\Q[\Spec E]$; we get the transfer from the dual $\Q\to \Q[\Spec E]$ of the adjunction morphism $\Q[\Spec E]\to \Q$.)
\end{proof}

\subsection{Huber's mixed realisation functor} In \cite[Th. 2.3.3 and following remark]{HuberH},
Annette Huber defines a (contravariant) realisation functor:
\[R_{\cMR^P}: \DM_\gm(k,\Q)^\op\to D_{\cMR^P}.\]

The categories $\cMR^P$ and $D_{\cMR^P}$ carry a duality. As in the previous section, we define
a covariant realisation functor
\[R^{\cMR^P}: \DM_\gm(k,\Q)\to D_{\cMR^P}\]
as the composite of $R_{\cMR^P}$ with this duality. This induces a functor
\[R^{\cMR^P}: \DM_\gm^\eff(k,\Q)\to D_{\cMR^P}^\eff.\]

Since $R_{\cMR^P}\Phi_\Q:\Chow^\eff\otimes\Q\to D_{\cMR^P}$ is by construction isomorphic to the
mixed realisation fonctor $X\mapsto R_{\cMR^P}(X)$ collecting all individual realisations of
$X$ with their comparison isomorphisms \cite[Th. 20.2.3]{HuberLN}, the conditions of Hypothesis
\ref{h15.1} are verified, as well as Condition (W) of Proposition \ref{p15.1}. This proposition
then shows that
$R^{\cMR^P}\Tot$ defines an exact functor 
\[R_1^{\cMR^P}:\M(k)\otimes\Q\to \cMR_{(1)}^P.\]

\subsection{Deligne's conjecture}
For any $M\in \DM_\gm^\eff(k)\otimes\Q$,  from  \eqref{eq15.1} we get a
natural map
\begin{equation}\label{eqhodgeMR}
R^{\cMR^P}(M)_{\leq 1}\to R_1^{\cMR^P}\LAlb(M).
\end{equation}
Taking homology of both sides, we get comparison maps
\begin{equation}\label{eqhomHodgeMR}
H_i(R^{\cMR^P}(M))_{\leq 1}\to R_1^{\cMR^P}\LA{i}(M)
\end{equation}
for all $i\in\Z$. From Theorem \ref{abstractdelconj} and Lemma \ref{lisoMR}, we now get:

\begin{thm} \label{isodelMR}
The maps \eqref{eqhodgeMR} and \eqref{eqhomHodgeMR} are isomorphisms for  all motives $M$.\qed
\end{thm}

Let us develop this theorem in view of the definition of $\cMR^P$. For this, we shall write
$R^\ell$, $R^\DR$,  $R^\sigma$ for the $\ell$-adic, de Rham and (relative to $\sigma$) Betti
components of $R^{\cMR^P}$, and similarly $R^\%_1$, $H_i^\%=H_i\circ R^\%$ ($\%\in
\{\DR,\ell,\sigma\}$).

\begin{cor}\label{ceqhodgeMR} Let $M\in \DM_\gm^\eff(k,\Q)$. Then:\\
a) One has isomorphisms
\begin{align*}
H_i^\ell(M)_{\leq 1}&\iso R_1^\ell\LA{i}(M)\\
H_i^\DR(M)_{\leq 1}&\iso R_1^\DR\LA{i}(M)\\
H_i^\sigma(M)_{\leq 1}&\iso R_1^\sigma\LA{i}(M).
\end{align*}
The last two isomorphisms yield the isomorphism of Hodge structures from Theorem \ref{isodel},
and the first one is an isomorphism of constructible Galois representations with weight
filtrations.\\ 
b) The diagrams
\[\begin{CD}
H_i^\sigma(M)_{\leq 1}\otimes_\Q\Q_\ell@>\sim>> R_1^\sigma\LA{i}(M)\otimes_\Q\Q_\ell\\
@V{I_{\ell,\sigma}^{-1}\circ I_{\bar\sigma,\ell}}VV @V{I_{\ell,\sigma}^{-1}\circ
I_{\bar\sigma,\ell}}VV\\  
H_i^\ell(M)_{\leq 1}@>\sim>> R_1^\ell\LA{i}(M)
\end{CD}
\]
\[\begin{CD}
H_i^\sigma(M)_{\leq 1}\otimes_\Q\C@>\sim>> R_1^\sigma\LA{i}(M)\otimes_\Q\C\\
@V{I_{\DR,\sigma}^{-1}\circ I_{\sigma,\C}}VV @V{I_{\DR,\sigma}^{-1}\circ I_{\sigma,\C}}VV\\  
H_i^\DR(M)_{\leq 1}\otimes_k \C@>\sim>> R_1^\DR\LA{i}(M)\otimes_k \C
\end{CD}
\]
commute.
\end{cor}

Corollary \ref{ceqhodgeMR} partly extends to $\DM_\gm^\eff(\C,\Q)$. On the one hand, one has
\[2-\colim_{k \text{ f.g.}} \DM_\gm^\eff(k,\Q)\iso \DM_\gm^\eff(\C,\Q)\]
by \cite[Prop. 4.16]{ivorra}. On the other hand, de Rham cohomology commutes with base change
and $\ell$-adic cohomology is invariant under algebraically closed extensions. This yields:

\begin{cor}\label{ceqhodgeMR2} Let $M\in \DM_\gm^\eff(\C,\Q)$. Then:\\
a) The comparison isomorphisms extend to comparison isomorphisms of realisations:
\begin{align*}
R^B(M)\otimes_\Q\Q_\ell&\longby{I_\ell} R^\ell(M)\\
R^B(M)\otimes_\Q\C&\longby{I_\DR} R^\DR(M).
\end{align*}
b) There exist vector spaces
\[H_i^\ell(M)_{\leq 1},H_i^\DR(M)_{\leq 1},H_i^B(M)_{\leq 1}\]
quotients of the $\ell$-adic, de Rham and Betti realisations of $M$, and functorially attached
to $M$.\\ 
c) One has isomorphisms
\begin{align*}
H_i^\ell(M)_{\leq 1}&\iso R_1^\ell\LA{i}(M)\\
H_i^\DR(M)_{\leq 1}&\iso R_1^\DR\LA{i}(M)\\
H_i^B(M)_{\leq 1}&\iso R_1^\sigma\LA{i}(M).
\end{align*}
The last two isomorphisms yield the isomorphism of Hodge structures from Theorem
\ref{isodel}.\\  
d) The diagrams
\[\begin{CD}
H_i^B(M)_{\leq 1}\otimes_\Q\Q_\ell@>\sim>> R_1^B\LA{i}(M)\otimes_\Q\Q_\ell\\
@V{I_\ell}VV @V{I_\ell}VV\\  
H_i^\ell(M)_{\leq 1}@>\sim>> R_1^\ell\LA{i}(M)
\end{CD}
\]
\[\begin{CD}
H_i^B(M)_{\leq 1}\otimes_\Q\C@>\sim>> R_1^B\LA{i}(M)\otimes_\Q\C\\
@V{I_\DR}VV @V{I_\DR}VV\\  
H_i^\DR(M)_{\leq 1}@>\sim>> R_1^\DR\LA{i}(M)
\end{CD}
\]
commute.
\end{cor}

Applied to geometric motives like $M(X)$ or $M(X)^*(n)[2n]$, this gives a reasonable
interpretation of the second part of Deligne's conjecture.

\section{The $\ell$-adic realisation in positive characteristic}\label{elladic}

In this section, $k$ is a finitely generated field and $\ell$ is a
prime number different from
$\car k$. From \S \ref{elladic2} onwards we shall assume $\car k>0$.

\subsection{The tame derived category of $\ell$-adic sheaves} We shall use an analogue of the
categories considered by Huber \cite{Huber}. Namely,

\begin{defn} Let
\[\tilde D(k,\Z_\ell)\df  2-\colim D(S,\Z_\ell)\]
where $S$ runs through the regular models of $k$, of finite type over $\Spec \Z$, and for such
$S$, $D(S,\Z_\ell)$ is the category defined by Ekedahl in \cite{Ekedahl}, for the \'etale topos of $S$.
Let \[\tilde
D^b_c(k,\Z_\ell)\df 2-\colim D^b_c(S,\Z_\ell)\] where $D^b_c(S,\Z_\ell)$ is the thick
subcategory of $D(S,\Z_\ell)$ consisting of those complexes whose cohomology sheaves are
constructible. Finally,
\[\tilde D^b_m(k,\Q_\ell)\df 2-\colim D^b_m(S,\Q_\ell)\]
where $D^b_m(S,\Q_\ell)$ is the thick subcategory of $D^b_c(S,\Q_\ell)$ consisting of those
complexes whose cohomology sheaves are mixed in the sense of \cite[(1.2.2)]{weilII}.
\end{defn}

Recall \cite[Thm. 4.7]{Ekedahl} where the following pairings are constructed
\begin{align}\oo^L: D^-(S,\Z_\ell)\times D(S,\Z_\ell)&\to D(S,\Z_\ell)\label{ek}\\
\RHom: D(S,\Z_\ell)^\op\times D^+(S,\Z_\ell)&\to D(S,\Z_\ell)\notag\\
\intertext{From these we get similar pairings}
\oo^L: \tilde D^-(k,\Z_\ell)\times \tilde D(k,\Z_\ell)&\to \tilde D(k,\Z_\ell)\label{ek2}\\
\RHom: \tilde D(k,\Z_\ell)^\op\times \tilde D^+(k,\Z_\ell)&\to \tilde D(k,\Z_\ell)\notag
\end{align}

\begin{remark}\label{rek=bbd} By Deligne's finiteness theorems in \'etale cohomology \cite[Th.
finitude]{sga4 1/2}, the categories $D^b_c(S,\Z_\ell)$ and $D^b_m(S,\Z_\ell)$ are equivalent to
those considered in \cite{BBD}.
\end{remark}

The category $\tilde D^b_m(k,\Q_\ell)$ enjoys a $t$-structure which is the 2-colimit of the
natural $t$-structures of  $D^b_m(S,\Q_\ell)$, with hearts the categories $\Mix(S,\Q_\ell)$ of
mixed $\Q_\ell$-sheaves on $S$.

\begin{defn}\label{dmix}
We write $\Mix(k,\Q_\ell)$ for the 2-colimit of the $\Mix(S,\Q_\ell)$: this is the heart of
$\tilde D^b_m(k,\Q_\ell)$.
\end{defn}

\begin{remarks}\label{rpi1}  1) For any mixed sheaf $\cF$ on $S$  normal, the restriction of $\cF$ to
a suitable open set is lisse, and if $\cF$ and $\cG$ are lisse over $S$, the map
\[\Hom_S(\cF,\cG)\to \Hom_U(\cF,\cG)\]
is bijective for any open subset $U\subseteq S$, as one sees by interpreting $\cF$ and $\cG$
as $\ell$-adic representations of $\pi_1(S)$ \cite[VI]{sga5}. Hence the  objects of $\Mix(k,\Q_\ell)$ may be identified with ``generically
unramified" $\ell$-adic representations of $G_k$, the absolute Galois group of $k$.\\ 2) Let
$\Perv(S,\Q_\ell)$ be the category of perverse sheaves over $S$, for the middle perversity. By
definition of perverse shaves, we also have an equivalence
\[2-\colim \Perv(S,\Q_\ell) \iso \Mix(k,\Q_\ell)[\dim S]\]
inside $\tilde D^b_m(k,\Q_\ell)$.\\
3) The definition of $D^-(S,\Z_\ell)$ yields canonical $t$-exact functors
\[D^b(\Mix(S,\Q_\ell))\to D^b_m(S,\Q_\ell)\]
hence a canonical $t$-exact functor
\[D^b(\Mix(k,\Q_\ell))\to \tilde D^b_m(k,\Q_\ell)\]
which is the identity on the hearts.
\end{remarks}

\subsection{$\DM$ and $\DA$ over a base}

Let $S$ be Noetherian and separated. For $R$ a commutative ring, let $\NST(S,R)$ (\resp 
$\EST(S,R)$) be the category of Nisnevich (\resp \'etale) sheaves of
$R$-modules with transfers over $S$  (using the category $Cor_S$ of \cite[App. 1A]{VL}  restricted to smooth $S$-schemes). If
$R=\Z$, we drop it from the notation. Then $\EST(S,R)$ is a full subcategory of $\NST(S,R)$
which contains the representable sheaves $L(X)$ for $X$ smooth over $S$. A tensor
structure on $\EST(S,R)$ and $\NST(S,R)$ is induced by the rule $L (X)\otimes L (Y) =
L (X\times_S Y)$.

Recall that over the base $S$ the category $\DM^{\eff}(S,R)$ has the same descriptions as done
in \cite{V} or \cite[Lect. 14]{VL} over a field: see \cite{V2}. Namely, it may be defined as
the localisation of the (unbounded) derived category of Nisnevich sheaves with transfers with
respect to $\Aff^1$-equivalences. Moreover, this localisation functor has a right adjoint so
that it may be regarded as the full subcategory of $\Aff^1$-local objects, \cf
\cite[2.2.6]{ABV}.

The same picture works for \'etale sheaves with transfers, at least if $cd_R(S)<\infty$: let
\[L_{\Aff^1}:D(\EST(S,R))\to \DM^\eff_\et(S,R)\]
be the localisation of $D(\EST(S,R))$ with respect to $\Aff^1$-equivalences. Then $L_{\Aff^1}$
has a fully faithful right adjoint $i$, identifying $\DM^\eff_\et(S,R)$ with the full
subcategory of $\Aff^1$-local objects in $D(\EST(S,R))$.

Moreover, there are non-effective versions of these categories
\[\DM^\eff(S,R)\to \DM(S,R), \quad \DM^\eff_\et(S,R)\to \DM_\et(S,R)\]
obtained by stabilising with respect to tensor product with $L (\G_m,1)$, \cf \cite[\S
10]{cis-deg2}.

In the next subsection we shall need a version without transfers. Let $\NS(S,R)$ and
$\ES(S,R)$ be the categories of Nisnevich (\resp \'etale) sheaves of $R$-modules over
$\Sm(S)$. One constructs triangulated categories $\DA(S,R)$ and $\DA_\et(S,R)$, based on
$\NS(S,R)$ and $\ES(S,R)$, in the same fashion as above, see \cite[\S 3]{real.etale}.

The forgetful functors $\NST(S,R)\to \NS(S,R)$ and $\EST(S,R)\to \ES(S,R)$ have
left adjoints which respect the tensor structures. This yields a naturally commutative diagram
of triangulated $\otimes$-functors:
\begin{equation}\label{eq.dadm}
\begin{CD}
\DA(S,R)@>>> \DM(S,R)\\
@VVV @VVV\\
\DA_\et(S,R)@>>> \DM_\et(S,R).
\end{CD}
\end{equation}

Ayoub proves: 

\begin{thm}[\protect{\cite[Th. B.1]{real.etale}}]\label{t.ay1} The functor $\DA_\et(S,R)\to \DM_\et(S,R)$ in \eqref{eq.dadm} is an equivalence of
categories if $S$ is normal and universally Japanese and any prime number is invertible either
in
$R$ or in $\cO_S$.
\end{thm}

\subsection{$\ell$-adic realisations for $\DA$ and $\DM$}\label{elladic2} 

Let $S$ be separated of finite type over $\Spec \Z$. For $\ell$ invertible on $S$,  Ayoub
constructs a triangulated
$\otimes$-functor 
\begin{equation}\label{eqay1}
\DA_\et(S,\Z_{(\ell)})\to D(S,\Z_\ell)
\end{equation}
which commutes with the six operations \cite[Th. 7.9]{real.etale}. 

If $S$ is a smooth
$\F_p$-scheme, $p\ne \ell$,  we can take $R = \Z_{(\ell)}$ in Theorem \ref{t.ay1}. In this way
we get a $\otimes$-functor
\[\DM_\et(S,\Z_{(\ell)})\to D(S,\Z_\ell)\]
hence a composite realisation functor
\[R^\ell:\DM(S)\to \DM_\et(S)\to \DM_\et(S,\Z_{(\ell)})\to D(S,\Z_\ell)\]

The construction of \eqref{eqay1} relies on a generalisation of the Suslin-Voevodsky rigidity
theorem of \cite{suvo}; using this, one sees that given a smooth $S$-scheme $f:X\to S$, one has
a canonical isomorphism
\begin{equation}\label{eqivorra}
R^\ell M_S(X)\simeq (Rf_*\Z_\ell)^*
\end{equation}
where $M_S(X)$ is the motive of $X$, defined as the image in $\DM(S)$ of $L(X)$. Another
way to get this is to use \cite[Prop. 5.9]{real.etale}, which relates $R^\ell$ with Ivorra's
(contravariant) realisation functor \cite{ivorra}.

Recall the category $\DM_\gm^\eff(S)$, constructed over a base exactly as over a field except
that the Mayer-Vietoris relations for the Zariski topology are replaced by Nisnevich excision
relations \cite[Def. 1.14]{ivorra}. In \cite{ivorra2}, Ivorra also constructs a functor
$\DM_\gm^\eff(S)\to \DM^\eff_\Nis(S)$ extending Voevodsky's functor for base fields. Using it, 
we can restrict $R^{\ell}$ to $\DM_\gm^\eff(S)$. 

By construction, $R^\ell$ commutes with pull-backs for arbitrary morphisms.  By \cite[Prop.
4.16]{ivorra}, if $S$ runs through the smooth models of $k$ over
$\F_p$, the natural functor
\begin{equation}\label{isoDM}
2-\colim \DM_\gm^\eff(S)\to \DM_\gm^\eff(k)
\end{equation}
is an equivalence of categories; hence the $R^\ell_S$ induce  a
 triangulated functor
\[
\tilde R^{\ell}:\DM^\eff_\gm(k) \to \tilde D(k,\Z_{\ell}).
\]

Note that we also have an equivalence of categories
\begin{equation}\label{isoM1}
2-\colim \M(S)\to \M(k)
\end{equation}
where the left hand side is Deligne's category of $1$-motives over a base $S$ (see \S \ref{1motbase} and \cite[(10.1)]{D}): it is obtained by similar, but easier, arguments as for \eqref{isoDM}.

\subsection{Weights and effectivity}\label{19.2}

By definition of a mixed sheaf, we can define a filtration $(\Mix(k,\Q_\ell)_{\le n})_{n\in\Z}$
on $\Mix(k,\Q_\ell)$ in the sense of Definition \ref{d16.3} as follows: $\cF\in \Mix(k,\Q_\ell)$ is
in $\Mix(k,\Q_\ell)_{\le n}$ if and only if its punctual weights (on a suitable model of $k$) are
all $\le n$. This filtration is clearly exhaustive and separated.

By Remarks \ref{rek=bbd} and \ref{rpi1} and by
\cite[(3.4.1)]{weilII}, the category
$\Mix(k,\Q_\ell)$ of Definition \ref{dmix} then enjoys a good theory of weights. Namely:

\begin{propose}\label{p18.1} Suppose $\car k >0$. The filtration $\Mix(k,\Q_\ell)_{\le n}$ provides
$\Mix(k,\Q_\ell)$ with a weight filtration in the sense of Definition \ref{dE.6}.\footnote{This is not a weight filtration if 
$\car k=0$: this fact was pointed out by J.
Wildeshaus. More precisely, the inclusions
\[\iota_n:\Mix(k,\Q_\ell)_{\le n}\into \Mix(k,\Q_\ell)_{\le n+1}\] 
do not have right adjoints: see \cite[p. 90, Remark 6.8.4 i)]{jannsen}.}\end{propose}

\begin{proof} a) Consider the filtration $W_n$ defined on mixed sheaves by \cite[(3.4.1)
(ii)]{weilII}. By Remark \ref{lF.3.6}, it suffices to know that the $W_n$ are exact, exhaustive
and separated. The last two facts are clear, and the first follows from the strict
compatibility with morphisms. 
\end{proof}

The categories $\Mix(k,\Q_\ell)_n$ are not semi-simple.
However, it follows from \cite[(3.4.1) (iii)]{weilII} that if $\cF\in \Mix(k,\Q_\ell)_n$ is viewed as a Galois representation as in Remark \ref{rpi1} 1), then the restriction of $\cF$ to $G_{k\bar\F_p}$ is semi-simple. We now use this fact to get to the situation of \S \ref{realalb}.

\begin{lemma}\label{isotyp} Let $\cF\in \Mix(k,\Q_\ell)_n$, and write $\cF=\bigoplus_{\alpha\in A} \cF_\alpha$ for the decomposition of $\cF_{|G_{k\bar\F_p}}$ into its isotypic components. Then
\begin{enumerate}
\item $G_k$ permutes the $\cF_\alpha$.
\item For  $n$ sufficiently large, $G_{k\F_{p^n}}$ leaves each $\cF_\alpha$ invariant.
\end{enumerate}
\end{lemma}

\begin{proof} We have an exact sequence of profinite groups
\[1\to G_{k\bar\F_p}\to G_k\to G_{\F_p}\to 1.\]

(1) is a general fact with a standard proof.
 Since $G_{k\bar\F_p}$ leaves each $\cF_\alpha$ invariant, the action of $G_k$ on the index set $A$ factors through $G_{\F_p}$. But $A$ is finite since $\cF$ has finite rank, so a suitable open subgroup of $G_{\F_p}$ acts trivially on $A$, hence (2).
\end{proof}

This leads us to replace all categories with base $k$ by the corresponding categories with base $k\bar \F_p$, namely:
\begin{align*}
\tilde D^b_m(k\bar \F_p,\Q_\ell) &= 2-\colim \tilde D^b_m(k\F_{p^n},\Q_\ell)\\
\Mix(k\bar \F_p,\Q_\ell) &= 2-\colim \Mix(k\F_{p^n},\Q_\ell) \\
\Mix(k\bar \F_p,\Q_\ell)_n &= 2-\colim \Mix(k\F_{p^n},\Q_\ell)_n .
\end{align*}

The $t$-structures on the $\tilde D^b_m(k\F_{p^n},\Q_\ell)$ induce a $t$-structure on the category $\tilde D^b_m(k\bar \F_p,\Q_\ell)$ with heart $\Mix(k\bar \F_p,\Q_\ell)$; the weight filtrations on the $\Mix(k\F_{p^n},\Q_\ell)$ induce a weight filtration on $\Mix(k\bar \F_p,\Q_\ell) $, with ``associated graded'' the $\Mix(k\bar \F_p,\Q_\ell)_n$.

The weight filtration we shall need on $\tilde D^b_m(k\bar \F_p,\Q_\ell)$ is not the one considered in
\cite{BBD} and \cite{Huber}, but rather the one in Definition \ref{dJ.1}, which coincides with
the one introduced by S. Morel in \cite[\S 3.1]{morel}:

\begin{defn} \label{dweight} An object $C\in \tilde D^b_m(k\bar \F_p,\Q_\ell)$ is \emph{of weight $\leq
w$} if the weights of $H^i(C)$ are $\leq w$ for all $i\in \Z$ (compare Remark \ref{rpi1} 2)).
\end{defn}

We also need a notion of effectivity:

\begin{defn}\label{deff} a) An object $\cF\in\Mix(k\bar \F_p,\Q_\ell)$
is \emph{effective} if the eigenvalues of  \emph{arithmetic} Frobenius elements acting on the
stalk(s) of (a finite level representative of) $\cF$ are algebraic integers. We denote by
$\Mix(k\bar \F_p,\Q_\ell)^\eff$ their full subcategory.\\ 
b) We denote by $\tilde
D^b_m(k\bar \F_p,\Q_\ell)^\eff$ the full subcategory of $\tilde D^b_m(k\bar \F_p,\Q_\ell)$ consisting of
objects with cohomology in $\Mix(k\bar \F_p,\Q_\ell)^\eff$.
\end{defn}

Clearly, the abelian category $\Mix(k\bar \F_p,\Q_\ell)^\eff$ is a Serre subcategory  of $\Mix(k\bar \F_p,\Q_\ell)$, hence
$\tilde D^b_m(k\bar \F_p,\Q_\ell)^\eff$ is a thick triangulated subcategory of $\tilde
D^b_m(k\bar \F_p,\Q_\ell)$ with heart $\Mix(k\bar \F_p,\Q_\ell)^\eff$. The following lemma is probably hidden somewhere in the literature:

\begin{lemma}\label{leff} Let $\cF \in \Mix(k\bar \F_p,\Q_\ell)_0$. The following are equivalent:
\begin{thlist}
\item $\cF$ is effective.
\item Let  $X$ be a smooth scheme of finite type over $\F_q$, with function field $k\F_q$, such that $\cF$ comes from an $\ell$-adic sheaf on $X$ still denoted by $\cF$. Then, for any $x\in X_{(0)}$, the Frobenius $F_x$ acting on $\cF$ has roots of unity for eigenvalues. 
\item Let $X$ be as in (ii). Then for one $x_0\in X_{(0)}$, the Frobenius $F_{x_0}$ acting on $\cF$ has roots of unity for eigenvalues, and the action of the geometric fundamental group factors through a finite quotient.
\end{thlist}
If $\cF$ is semi-simple, then $\cF$ is effective if and only if a corresponding Galois representation $Gal(\bar k/k\F_q)\to \Aut(\cF)$ factors through a finite quotient of $Gal(\bar k/k\F_q)$.
\end{lemma}

\begin{proof} (i) $\iff$ (ii): obvious by Kronecker's theorem (an algebraic integer with complex absolute values $1$ is a root of unity).

(ii) $\Rightarrow$ (iii): this is a variant of the proof of the $\ell$-adic monodromy theorem. Let  $\bar x\to X$ be a geometric point, $G=\pi_1(X,\bar x)$ and
\[\rho:G\to \Aut(\cF) = {\rm GL}_N(\Q_\ell) \quad (N= \rk \cF)\]
be the corresponding representation. Since $G$ is compact, its image lands into some ${\rm GL}_N(\Z_\ell)$ (given by a $G$-invariant $\Z_\ell$-lattice in $\cF$). The subgroup $1 + \ell^2 M_N(\Z_\ell)$ is a pro-$\ell$-group of finite index, so its inverse image $H$ is of finite index in $G$. If $h\in H$ and $\lambda$ is an eigenvalue of $\rho(h)$, then $\lambda\in 1+\ell^2 O_K$ for some finite extension $K/\Q_\ell$. 

Let $X'\to X$ be the \'etale covering corresponding to $H$. If $x'$ is a closed point of $X'$ and $x$ is its image in $X$, then $F_{x'}\in H$ is a power of $F_x\in G$. So if $x'\in X'_{(0)}$, the eigenvalues of $\rho(F_{x'})$ are roots of unity in $1+\ell^2 O_K$, so they must be $1$. (We use $\ell^2M_N(\Z_\ell)$ only for $\ell = 2$; for $\ell> 2$, $\ell M_N(Z_\ell)$ is sufficient.) This implies that $\Tr(\rho(F_{x'})) = N$ for any such $x'$.

By a version of \v Cebotarev's density theorem \cite[Th. 7]{serrezetaL}, the  $F_{x'}$ are dense in $H$. Since $h\mapsto \Tr(\rho(h))$ is continuous, it has the constant value $n$ on $H$. So for $h\in H$, $\Tr(\rho(h^r)) = n$ for any $r\ge 1$, which implies that the eigenvalues of $\rho(h)$ are all equal to $1$.

Now a theorem of Lie-Kolchin implies that $\rho_{|H}$ is unipotent, i.e. $\cF_{|H}$ is a successive extension of trivial representations. Write $\bar X=X\otimes_{k\F_q} k\bar \F_q$, $\bar X'=X'\otimes_{k\F_q} k\bar \F_q$. By  \cite[(3.4.1) (iii)]{weilII}, the restriction of $\rho_{|H}$ to  $\pi_1(\bar X',\bar x)$ is semi-simple; since it is unipotent, it is trivial. Therefore, $\rho_{|H}$ factors through a unipotent representation of $Gal(\bar \F_{q'}/\F_{q'})$, where $\F_{q'}$ is the field of constants of $X'$, and 
the restriction of $\rho$ to $\pi_1(\bar X,\bar x)$ factors through the finite quotient $\pi_1(\bar X,\bar x)/\pi_1(\bar X',\bar x)$.

(iii) $\Rightarrow$ (ii): If $\rho_{|\pi_1(\bar X)}$ factors through a finite quotient, this quotient determines an \'etale covering $Y\to \bar X$, and $Y$ is defined over a finite extension $\F_{q'}$ of $\F_q$, say $Y=\bar X'$. Let $x\in X_{(0)}$: to show that the eigenvalues of $\rho(F_x)$ are roots of unity, it suffices to see that those of $\rho(F_{x'})$ are $1$ for an $x'\in X'$ above $x$. With the same notation as in the proof of (ii) $\Rightarrow$ (iii), $\rho_{|H}$ factors through the projection $\pi:H\to Gal(\bar \F_{q'}/\F_{q'})$. If $\phi$ is the ``absolute Frobenius'' of $\F_{q'}$, generating $Gal(\bar \F_{q'}/\F_{q'})$, then $\pi(F_{x'})=\phi^{n_{x'}}$ with $n_{x'}=[\F_{q'}(x'):\F_{q'}]$. Applying this with $x=x_0$, we get that the eigenvalues of $\rho(\phi)$ are roots of unity, and then the same for those of $\rho(F_{x'})$.

Suppose $\cF$ semi-simple. In the above proof of (ii) $\Rightarrow$ (iii), $\rho_{|H}$ is semi-simple by Clifford's theorem \cite{clifford} (note that $H$ is normal in $G$). Since it is unipotent, it is trivial hence $\rho$ factors through $G/H$. Conversely, if $\rho$ factors through a finite quotient, $\cF$ is clearly effective.
\end{proof}

\subsection{Sheaves of level $\leq 1$}

\begin{defn} Let $\cF\in \Mix(k\bar \F_p,\Q_\ell)$. We say that $\cF$ is \emph{of level $\leq 1$} if
\begin{itemize}
\item $\cF$ is effective (Def. \ref{deff}) and its weights are in $\{0,-1,-2\}$ (Def.
\ref{dweight});
\item $W_{-2}\cF(-1)$ is effective.
\end{itemize}
We write $\Mix(k\bar \F_p,\Q_\ell)_{(1)}$ for the full subcategory of $\Mix(k\bar \F_p,\Q_\ell)^\eff$ consisting
of objects of level $\leq 1$, and $\tilde D^{b}_m(k\bar \F_p,\Q_\ell)_{(1)}$ the full subcategory of
$\tilde D^b_m(k\bar \F_p,\Q_\ell)$ consisting of objects with cohomology in $\Mix(k\bar \F_p,\Q_\ell)_{(1)}$.
\end{defn}

Note that this coincides with the definitions in \S\ref{realalb}, with
\[\Mix(k\bar \F_p,\Q_\ell)_\L=\{\cF(1)\mid \cF \text{ effective of weight } 0\}.\]

In particular, $\tilde D^{b}_m(k\bar \F_p,\Q_\ell)_{(1)}$ is a thick triangulated $t$-subcategory of
$\tilde D^b_{m}(k\bar \F_p,\Q_\ell)^\eff$ with heart $\Mix(k\bar \F_p,\Q_\ell)_{(1)}$. 

\begin{propose}\label{ssl} The full subcategory $\Mix(k\bar \F_p,\Q_\ell)_\L\subset \Mix(k\bar \F_p,\Q_\ell)_{-2}$ verifies Hypothesis \ref{ss}.
\end{propose}

\begin{proof} Let $\cF\in  \Mix(k\bar \F_p,\Q_\ell)_{-2}$. Using Lemma \ref{isotyp}, write $\cF=\bigoplus_{\alpha\in A} \cF_\alpha$, where $(\cF_\alpha)_{|k\bar \F_p}$ is isotypic of type $\alpha$. Then $\Hom(\cF_\alpha,\cF_\beta)=0$ for $\alpha\ne \beta$. Let 
\[A_1=\{\alpha\in A\mid \cF_\alpha(-1)\text{ is effective}\}\]
and $A_2=A-A_1$. Set
\[\cF_\L = \bigoplus_{\alpha\in A_1}\cF_\alpha, \quad\cF^\tr=\bigoplus_{\alpha\in A_2}\cF_\alpha\]
so that $\cF_\L\in \Mix(k\bar \F_p,\Q_\ell)_\L$, $\cF=\cF_\L\oplus \cF^\tr$ and $\Hom(\cF_\L,\cG^\tr)=\Hom(\cG^\tr,\cF_\L)=0$ for $\cF,\cG\in  \Mix(k\bar \F_p,\Q_\ell)_{-2}$. 
\end{proof}

\begin{cor} The full embeddings $\Mix(k\bar \F_p,\Q_\ell)_{(1)}\into\allowbreak 
\Mix(k\bar \F_p,\Q_\ell)^{\eff}$ and
$D^b(k\bar \F_p,\Q_\ell)_{(1)}\into\allowbreak D^b(k\bar \F_p,\Q_\ell)^\eff$ have left adjoints $\Alb^\ell$
and
$\LAlb^\ell$.
\end{cor}

\begin{proof} This follows from 
Propositions \ref{pAlbT} and \ref{ssl}.
\end{proof}

\begin{propose} $\tilde R^\ell$ verifies Hypothesis \ref{h15.1} and \ref{h16.1}.
\end{propose}

\begin{proof} We first prove that $\tilde R^\ell$ maps $\DM_\gm^\eff(k)\otimes\Q$ into $\tilde
D^b_m(k\bar \F_p,\Q_\ell)$ (this much does not need characteristic $p$). It suffices to prove that for
any smooth model $S$ of $k$, $R^\ell$ maps
$\DM_\gm^\eff(S\otimes \F_{p^n})\otimes\Q$ into $D^b_m(S\otimes \F_{p^n},\Q_\ell)$ for each $n$: this follows from \eqref{eqivorra}, Deligne's finiteness theorem \cite[Th. finitude]{sga4 1/2} and the main result of Weil II \cite[Th. 3.3.1]{weilII}.

The fact that $\tilde R^\ell(\DM_\gm^\eff(k)\otimes\Q)\subset \tilde
D^b_m(k\bar \F_p,\Q_\ell)^\eff$ now follows from
\cite[Exp. XXI, (5.2.2)]{sga7} which says that the eigenvalues of the geometric Frobenius
acting on $\ell$-adic cohomology with compact supports are algebraic integers.

The rest of the properties follow from the standard properties of $\ell$-adic cohomology, plus
the Riemann hypothesis \cite{weilI}. Hypothesis \ref{h16.1} (2) follows from Tate-twisting Lemma \ref{leff}, which implies that the restriction of $\tilde R_1^\ell$ to  $(\M\otimes\Q)_0$ induces a full embedding  $(\M\otimes\Q_\ell)_0\into \Mix(k\bar \F_p,\Q_\ell)_0^\eff$, with essential image the semi-simple objects.
\end{proof}

\subsection{Derived realisation for $1$-motives}

To get the correct derived realisation for $1$-motives, we start from Deligne's $\ell$-adic
realisation for smooth $1$-motives over a base $S$ \cite[(10.1.10)]{D} (see also \S \ref{1motbase} and \cf Definition \ref{ladic}):
\[T_\ell(S):\M(S)\otimes\Q\to \Mix(S,\Q_\ell).\]

For $S$ running through smooth models of $k\F_{p^n}$ over $\F_{p^n}$, the $T_\ell(S)$ induce an exact
functor 
\[T_\ell:\M(k\bar\F_p)\otimes\Q\to \Mix(k\bar\F_p,\Q_\ell)\]
(see \eqref{isoM1} and \cf Lemma \ref{ltlexact}). Deligne's description of $T_\ell(S)$ shows that, in fact, $T_\ell$ takes its values in
$\Mix(k\bar\F_p,\Q_\ell)_{(1)}$. From the definition of $T_\ell$ we get immediately:

\begin{thm} Let $T_\ell$ still denote the trivial extension of $T_\ell$ to a $t$-exact
functor $D^b(\M(k\bar\F_p)\otimes \Q)\to \tilde D^b_m(k\bar\F_p,\Q_\ell)$ (via Remark \ref{rpi1} 3)). Then there
is a canonical isomorphism $u:T_\ell\iso \tilde R^\ell\circ \Tot$.
\end{thm}
\begin{proof}   In fact, the essential image of $\Tot$ is  the thick subcategory generated by motives of points and smooth curves (Theorem \ref{t1.2.1}). We are then left to check that $u$ is an isomorphism when evaluated at motives of  points and smooth curves: by construction of \eqref{eqay1} and the formula \eqref{eqivorra} this is clear (by reduction to the case of lattices, Jacobians and tori). 
\end{proof}

\subsection{An $\ell$-adic version of Deligne's conjecture}

From the above we get, for any $M\in \DM_\gm^\eff(k\bar\F_p,\Q)$, the base change morphism \eqref{eq15.1}:
\begin{equation}\label{eqadic}
\tilde R^\ell (M)_{\leq 1}\to T_\ell\LAlb(M).
\end{equation}

Taking homology of both sides, we get comparison maps
\begin{equation}\label{eqadich}
H_i(\tilde R^\ell (M))_{\leq 1}\to T_\ell\LA{i}(M)
\end{equation}
for all $i\in\Z$.

Applying Theorem \ref{abstractdelconj} and the Tate-Zarhin-Mori theorem for endomorphisms of abelian
varieties, we get:

\begin{thm} \label{telladic}\
\begin{thlist}
\item If $X$ is smooth and projective, then \eqref{eqadich} is an isomorphism for $M=M(X)$ and  all $i\neq 2$; for $i=2$, it is an isomorphism if and only if the Galois action on $H^2(\bar X,\Q_\ell)$ is semi-simple and the Tate conjecture holds in codimension $1$ for $X$.
\item The map \eqref{eqadic} is an isomorphism for motives $M$ of abelian type.
\item For a general $M$, \eqref{eqadic} induces an isomorphism after applying the functor
$Q\mapsto \text{\rm cone}(W_{-2}Q\to Q)$.\qed
\end{thlist}
\end{thm}

Let us specify what Galois action is considered in (i). Since $X$ is a smooth projective variety over $k\bar \F_p$, it is defined over $k\F_{p^n}$ for some large enough $n$; for such an $n$, the Galois group $Gal(\bar k/k\F_{p^n})$ acts on $H^2(\bar X,\Q_\ell)$. We require this action to be semi-simple: this does not depend on the choice of $n$. (Note that $Gal(\bar k/k\bar \F_{p})$ acts semi-simply on $H^2(\bar X,\Q_\ell)$ by \cite[(3.4.1) (iii)]{weilII}.)

Using the $4$ operations, one can define the motive and motive with compact support of schemes of finite type over $k\bar\F_p$, \cf \cite[\S 6.7]{fullfaith}. Modulo the Tate conjecture in codimension $1$, we then get the same
corollaries as in Corollary \ref{isodelcor}.

\newpage
\part*{Appendices}

\appendix

\section{Homological algebra}\label{appendixA}

\subsection{Some comparison lemmas}

The following lemma is probably well-known:

\begin{lemma}\label{lA.2} Let $T:\cT\to\cT'$ be a full triangulated
functor between two triangulated categories. Then $T$ is conservative if
and only if it is faithful.
\end{lemma}

\begin{proof} ``If" is obvious. For ``only if", let $f:X\to X'$ be a
morphism of
$\cT$ such that
$T(f)=0$. Let $g:X'\to X''$ denote a cone of $f$. Then $T(g)$ has a
retraction
$\rho$. Applying fullness, we get an equality $\rho=T(r)$. Applying
conservativity, $u=rg$ is an isomorphism. Then $r'=u^{-1}r$ is a
retraction of $g$, which implies that $f=0$.
\end{proof}

\begin{propose}\label{derxact} a) Let $i: \E \into \A$ be an exact full 
subcategory of an abelian category $\A$, closed under kernels.
Assume further that for each
$A^{\d}\in C^b(\A)$ there exists
$E^{\d}\in  C^b(\E)$ and a quasi-isomorphism  $i (E^{\d})\to A^{\d}$
in $K^b(\A)$.  Then  $i: D^b(\E)\to D^b(\A)$ is an equivalence
of categories.\\
b) The hypothesis of a) is granted when every object in $\A$ has a
finite left resolution by objects in $\E$.
\end{propose}

\begin{proof} a) Clearly, the functor $D^b(\E)\to D^b(\A)$ is
conservative. The assumption implies that
$i$ is essentially surjective: thanks to Lemma \ref{lA.2}, in order to
conclude it remains to see that $i$ is full.

Let $f\in D^b(\A)(i (D^{\d}), i (E^{\d}))$. Since
$D^b(\A)$ has left calculus of fractions there  exists a
quasi-isomorphism $s$ such that $f = f^{\prime} s^{-1}$ where 
$f^{\prime} : A^{\d}\to i (E^{\d})$ is a map in $K^b(\A)$, which
then  lifts to a map in $C^b(\A)$. By hypothesis there exists $F^{\d}\in
C^b(\E)$ and a  quasi-isomorphism  $s^{\prime} :i (F^{\d})\to
A^{\d}$. Set 
$f^{\prime\prime}\df f^{\prime} s^{\prime}: i (F^{\d})\to i (E^{\d})$. 
Then $f = f^{\prime\prime} (s s^{\prime})^{-1}$ where $s s^{\prime} : i 
(F^{\d})\to i (D^{\d})$ is a quasi-isomorphism. By conservativity of $i$,
we are reduced to check fullness for effective maps, \ie arising from true
maps in $C^b(\A)$: this easily follows from the fullness of the
functor $C^b(\E)\into C^b(\A)$. 

b) This follows  by adapting the argument in \cite[I, Lemma~4.6]{HA}.
\end{proof}

\begin{propose}\label{triexact} Let $\E \into \A$ be an exact category. Let
$D$ be a triangulated category and let
$T : D^b(\E)\to D$ be a triangulated functor such that
$$\Hom_{D^b(\E)}(E', E[i]) \longby{\cong} \Hom_{D}(T(E'), T(E[i])) $$
for all $E', E\in \E$ and $i\in \Z$. Then $T$ is fully faithful.
\end{propose}

\begin{proof} Let $C,C'\in C^b(\E)$: we want to show that the map
\[\Hom_{D^b(\E)}(C, C') \to \Hom_{D}(T(C), T(C'))
\]
is bijective. We argue by induction on the lengths of $C$ and $C'$.
\end{proof}

Finally, we have the following very useful criterion for a full embedding of derived
categories, that we learned from Pierre Schapira.

\begin{propose}[\protect{\cite[p. 329, Th. 13.2.8]{KS}}]\label{pschapira} Let $\cA\into \cB$
be an exact full embedding of abelian categories. Assume that, given any mo\-no\-mor\-ph\-ism
$X'\into X$ in $\cB$, with $X'\in \cA$, there exists a morphism $X\to X''$, with $X''\in \cA$,
such that the composite morphism $X\to X''$ is a monomorphism. Then the functor
\[D^*(\cA)\to D^*(\cB)\]
is fully faithful for $*=+,b$.
\end{propose}

\subsection{The $\Tot$ construction}

\begin{lemma}\label{ltot} Let $\cA$ be an abelian category and let
$\cA^{[0,1]}$ be the (abel\-ian) category of complexes of length $1$ of
objects of $\cA$. Then the ``total complex" functor induces a
triangulated functor
\[D^*(\cA^{[0,1]})\to D^*(\cA)\]
for any decoration $^*$.
\end{lemma}

\begin{proof}
We may consider a complex of objects of $\cA^{[0,1]}$ as a double complex
of objects of $\cA$ and take the associated
total complex. This yields a functor
\[\Tot:C^*(\cA^{[0,1]})\to D^*(\cA).\]

(Note that if we consider a complex of objects of $\cA^{[0,1]}$
$$M^{\d} = [L^{\d}\by{u{\d}}G^{\d}]$$
as a map $u\d: L^{\d}\to G^{\d}$
of complexes of $\cA$, then
$\Tot (M^{\d})$ coincides with the shifted cone of $u\d$.)

This functor factors through a triangulated functor from
$D^*(\cA^{[0,1]})$: indeed it is easily checked that a) $\Tot$ preserves
homotopies, b) the induced functor on $K^*(\cA^{[0,1]})$ is
triangulated; c) $\Tot$ of an acyclic complex is
$0$ (which follows from a spectral sequence argument). 
\end{proof}

\begin{lemma}\label{lloc} With notation as in Lemma \ref{ltot}, the set of \qi of $\cA^{[0,1]}$
enjoys a calculus of left and right fractions within the homotopy category
$K(\cA^{[0,1]})$ (same objects, morphisms modulo the homotopy relation).
\end{lemma}

\begin{proof} It is enough to show the calculus of right fractions (for left fractions, replace
$\cA$ by $\cA^\op$).

a) Let 
\[\begin{CD}
&& \tilde C\\
&&@VuVV\\
D@>f>> C
\end{CD}\]
be a diagram in $\cA^{[0,1]}$, with $u$ a \qi. Consider the mapping fibre (= shifted mapping
cone)
$F$ of the map $\tilde C\oplus D\to C$. The complex $F$ is in general concentrated in degrees
$\{0,1,2\}$; however, since $F\to D$ is a \qi, the truncation $\tau_{< 2} F$ is \qi to $F$; then
$\tilde D=\tau_{<2} F$ fills in the square in  $K(\cA^{[0,1]})$.

b) Let 
\[[C^0\to C^1]\by{f} [D^0\to D^1]\by{u} [\tilde D^0\to \tilde D^1]\]
be a chain of maps of $\cA^{[0,1]}$ such that $u$ is a \qi and $uf$ is homotopic to $0$. Let
$s:C^1\to \tilde D^0$ be a corresponding homotopy. Define $\tilde C^1$ as the fibre product of
$C^1$ and $D^0$ over $\tilde D^0$ (via $s$ and $u^0$), $\tilde s:\tilde C^1\to D^0$ the
corresponding map and $\tilde C^0$ the fibre product of $C^0$ and $\tilde C^1$ over $C^1$. One
then checks that $v:[\tilde C^0\to \tilde C^1]\to [C^0\to C^1]$ is a \qi and that $\tilde s$
defines a homotopy from $fv$ to $0$.
\end{proof}

\begin{remark} One can probably extend these two lemmas to complexes of a fixed length $n$ by
the same arguments: we leave this to the interested reader.
\end{remark}

\section{Torsion objects in additive categories} 

\subsection{Additive categories}
\begin{defn}\label{dA.2} Let $\cA$ be an additive category, and let $A$ be
a subring of $\Q$.\\
a)  As in Definition \ref{1mot}, we write $\cA\otimes A$ \index{$\cA\otimes R$} for
the category with the same objects as $\cA$ but morphisms
\[(\cA\otimes A)(X,Y)\df\cA(X,Y)\otimes A.\]
b) We denote by $\cA\{A\}$ the full subcategory of $\cA$:
\[\{X\in \cA\mid \exists n>0 \text{ invertible in } A, n 1_X=0\}.\]
For $A=\Q$, we write $\cA\{A\} =\cA_\tors$. We say that $X\in \cA\{A\}$ is an
\emph{$A$-torsion object} (a torsion object if $A=\Q$).\\ 
c) A morphism $f:X\to Y$ in $\cA$ is an
\emph{$A$-isogeny} (an \emph{isogeny} if $A=\Q$) if there exists a morphism $g:Y\to X$ and an integer $n$ invertible
in $A$ such that $fg= n 1_Y$ and $gf=n1_X$. We denote by $\Sigma_A(\cA)$ the collection of
$A$-isogenies of $\cA$.\\ d) We say that two objects $X,Y\in \cA$ are \emph{$A$-isogenous} if they can be linked by a chain of $A$-isogenies (not necessarily pointing in the same direction).
\end{defn}

\begin{lemma}\label{lB.1} Let $F:\cA\to \cB$ be an additive functor. Then $F$ extends canonically to an $A$-linear functor $F\otimes A:\cA\otimes A\to \cB\otimes A$. If $F$ is faithful (resp. fully faithful), so is $F\otimes A$. If $F$ has the (left or right) adjoint $G$, $F\otimes A$ has the (left or right) adjoint $G\otimes A$.
\end{lemma}

\begin{proof} This is clear; faithfulness is preserved because $A$ is flat over $\Z$.
\end{proof}

\begin{propose}\label{pB.1.2} a) The subcategory $\cA\{A\}$ is additive and closed under direct
summands.\\ 
b) Consider the obvious functor $P:\cA\to \cA\otimes A$. Then
\[\Sigma_A(\cA)=\{f\mid P(f)\text{ is invertible.}\}\]
c) The $A$-isogenies $\Sigma_A(\cA)$ form a multiplicative system of morphisms in
$\cA$, enjoying calculi of left and right fractions. The corresponding localisation of $\cA$ is
isomorphic to
$\cA\otimes A$.
\end{propose}

\begin{proof} a) is clear. In b), one inclusion is clear. Conversely, let $f:X\to Y$ be such that $P(f)$ is invertible. This
means that there exists $\gamma\in (\cA\otimes A)(Y,X)$ such that $P(f)\gamma =1_Y$ and $\gamma
P(f) = 1_X$. Choose an integer $m\in A^*$ such that $m\gamma = P(g_1)$ for some $g_1$.
Then there is another integer $n\in A^*$ such that
\[n(fg_1 - m1_Y) = 0\text{ and } n(g_1 f - m1_X) = 0.\]

Taking $g = ng_1$ shows that $f\in \Sigma_A(\cA)$.

For c),  it is clear that homotheties by nonzero integers of $A$ form a cofinal system in
$\Sigma_A(\cA)$. Together with b), this shows immediately that we have calculi of left and right fractions. It remains to show that the induced functor
\[\Sigma_A(\cA)^{-1}\cA\to \cA\otimes A\]
is an isomorphism of categories; but this is immediate from the well-known formula, in the
presence of calculus of fractions:
\[\Sigma_A(\cA)^{-1}\cA(X,Y) = \varinjlim_{X'\by{f} X\in \Sigma} \cA(X,Y) =\varinjlim_{X\by{n}
X, n\in A^*} \cA(X,Y).\]
\end{proof}

The following lemma is clear.

\begin{lemma}\label{lB.1.3} Let $\cB$ be a full additive subcategory of
$\cA$, and suppose that every object of $\cA$ is $A$-isogenous to an object of $\cB$. Then
$\cB\otimes A\iso \cA\otimes A$.\qed
\end{lemma}

\subsection{Triangulated categories} (See \cite[A.2.1]{riou} for a different treatment.)

\begin{propose}\label{p1.1} Let $\cT$ be a triangulated category. Then\\
a) The subcategory $\cT\{A\}$ is triangulated and thick.\\ 
b) Let $X\in \cT$ and $n\in A^*$. Then ``the" cone $X/n$ of multiplication by $n$
on
$X$ belongs to $\cT\{A\}$.\\   
c) The localised category
$\cT/\cT\{A\}$ is canonically isomorphic to
$\cT\otimes A$. In particular, $\cT\otimes A$ is triangulated.\\
d) A morphism $f$ of $\cT$ belongs to $\Sigma_A(\cT)$ if and only if $\cone(f)\in
\cT\{A\}$.
\end{propose}

\begin{proof} a) It is clear that $\cT\{A\}$ is stable under direct
summands; it remains to see that it is triangulated. Let $X,Y\in
\cT\{A\}$, $f:X\to Y$ a morphism and $Z$ a cone of $f$. We may assume
that $n1_X=n1_Y=0$. The commutative diagram
\[\begin{CD}
Y@>>> Z@>>> X[1]\\
@V{n=0}VV @V{n}VV @V{n=0}VV \\
Y@>>> Z@>>> X[1]
\end{CD}\]
show that multiplication by $n$ on $Z$ factors through $Y$; this implies
that $n^2 1_Z=0$.

b) Exactly the same argument as in a) shows that multiplication by $n$ on
$X/n$ factors through $X$, hence that $n^2 1_{X/n}=0$.

c) Let $f\in \cT$ be such that $C:=\cone(f)\in \cT\{A\}$, and let $n>0$
be such that $n 1_C=0$. The same trick as in a) and b) shows that there
exist factorisations $n=ff'=f''f$, hence that $f\in \Sigma_A(\cT)$. In particular, $f$ becomes
invertible under the canonical (additive) functor $\cT\to \cT\otimes A$. Hence an
induced (additive) functor
\[\cT/ \cT\{A\}\to \cT\otimes A\]
which is evidently bijective on objects; b) shows immediately that it is
fully faithful.

d) One implication has been seen in the proof of c). For the other, if $f\in \Sigma_A(\cT)$,
then $f$ becomes invertible in $\cA\otimes A$, hence $\cone(f)\in \cT\{A\}$ by c).
\end{proof}

\begin{remark} As is well-known, the stable homotopy category gives a
counterexample to the expectation that in fact $n 1_{X/n}=0$ in b)
($X=S^0$,
$n=2$).
\end{remark}

We now show that $\otimes A$ is a ``flat" operation on triangulated categories.

\begin{lemma}\label{lB.2.3} Let $0\to \cT'\to \cT\to \cT''\to 0$ be a short exact sequence of
triangulated categories (by definition, this means that $\cT'$ is strictly full in $\cT$,
stable under cones and shifts, and that
$\cT''$ is equivalent to $\cT/\cT'$). Then the sequence
\[0\to \cT'\otimes A\to \cT\otimes A\to \cT''\otimes A\to 0\]
is exact.
\end{lemma}

\begin{proof} By Proposition \ref{p1.1} b), all categories remain triangulated after $\otimes
A$ and the induced functors are clearly triangulated functors. Since $\cT'\to\cT$ is
strictly full, so is $\cT'\otimes A\to \cT\otimes A$. If $f$ is a morphism in $\cT'\otimes A$,
then $f$ is the composition of an isomorphism with a morphism coming from $\cT'$, whose cone in
$\cT$ lies in $\cT'$, thus the cone of $f$ in $\cT\otimes A$ lies in $\cT'\otimes A$.

We now show that the functor
\[a:\frac{\cT\otimes A}{\cT'\otimes A}\to \cT''\otimes A\]
is an equivalence of categories. Since the left hand side is $A$-linear, the natural functor
\[\cT/\cT'\to \cT\otimes A/\cT'\otimes A\]
canonically extends to a functor
\[b:(\cT/\cT')\otimes A\to \cT\otimes A/\cT'\otimes A.\]

It is clear that $a$ and $b$ are inverse to each other. 
\end{proof}

\begin{propose}\label{p1.2} a) Let $T:\cS\to \cT$ be a triangulated functor
between triangulated categories. Then $T$ is fully faithful if and only
if the induced functors $T\{A\}:\cS\{A\}\to\cT\{A\}$ and
$T\otimes A:\cS\otimes A\to\cT\otimes A$ are fully faithful. \\
b) Assuming a) holds, $T$ is an equivalence of categories if and only if $T\{A\}$ and
$T\otimes A$ are.
\end{propose}

\begin{proof} a) ``Only if" is obvious; let us prove ``if'. Let $X,Y\in
\cS$: we have to prove that $T:\cS(X,Y)\to \cT(T(X),T(Y))$ is bijective.
We do it in two steps:

1) $Y$ is torsion, say $n 1_Y=0$. The claim follows from the commutative
diagram with exact rows
\[\begin{CD}
\scriptstyle\cS(X/n[1],Y)&\to&\scriptstyle \cS(X,Y)@>n=0>>\scriptstyle
\cS(X,Y)&\to&
\scriptstyle\cS(X/n,Y)\\ 
 @V{T}V{\wr}V @V{T}VV @V{T}VV @V{T}V{\wr}V \\
\scriptstyle\cT(T(X)/n[1],T(Y))&\to&
\scriptstyle\cT(T(X),T(Y))@>n=0>>\scriptstyle\cT(T(X),T(Y))&\to&
\scriptstyle\cT(T(X)/n,T(Y))
\end{CD}\]
and the assumption (see Proposition \ref{p1.1} b)).

2) The general case. Let $n>0$. We have a commutative diagram with exact
rows
\[\begin{CD}
\scriptstyle 0 &\to&\scriptstyle \cS(X,Y)/n&\to&
\scriptstyle\cS(X,Y/n)&\to& \scriptstyle{}_n\cS(X,Y[1])&\to& \scriptstyle
0 \\  && @V{T}VV @V{T}V{\wr}V @V{T}VV \\
\scriptstyle 0 &\to& \scriptstyle\cT(T(X),T(Y))/n&\to&
\scriptstyle\cT(T(X),T(Y)/n)&\to& \scriptstyle {}_n\cT(T(X),T(Y)[1])&\to&
\scriptstyle 0 
\end{CD}\]
where the middle isomorphism follows from 1). The snake lemma yields an
exact sequence
\begin{multline*}
0\to \cS(X,Y)/n\longby{T} \cT(T(X),T(Y))/n\\
\to
{}_n\cS(T(X),T(Y)[1])\longby{T}{}_n\cT(T(X),T(Y)[1])\to 0.
\end{multline*}

Passing to the limit over $n$, we get another exact sequence
\begin{multline}\label{eqB.1}
0\to \cS(X,Y)\otimes A/\Z\longby{T} \cT(T(X),T(Y))\otimes A/\Z\\
\to
\cS(T(X),T(Y)[1])\{A\}\longby{T}\cT(T(X),T(Y)[1])\{A\}\to 0.
\end{multline}

Consider now the commutative diagram with exact rows
\[\begin{CD}
\scriptstyle 0&\to& \scriptstyle \cS(X,Y)\{A\}&\to&
\scriptstyle \cS(X,Y)&\to&\scriptstyle
\cS(X,Y)\otimes A&\to&\scriptstyle \cS(X,Y)\otimes A/\Z&\to& 0\\
&& @V{T}V\underline{1}V @V{T}V\underline{2}V @V{T}V{\wr}V @V{T}V\underline{4}V\\
\scriptstyle 0&\to& \scriptstyle \cT(T(X),T(Y))\{A\}&\to&
\scriptstyle \cT(T(X),T(Y))&\to&\scriptstyle
\cT(T(X),T(Y))\otimes A&\to&\scriptstyle \cT(T(X),T(Y))\otimes A/\Z&\to&
0
\end{CD}\]
where the isomorphism is by assumption. By this diagram and \eqref{eqB.1}, $\underline{4}$ is
an isomorphism. Using this fact in \eqref{eqB.1} applied with $Y[-1]$, we get that
$\underline{1}$ is an isomorphism; then $\underline{2}$ is an isomorphism by the 5
lemma, as desired.

b) If $T$ is essentially surjective, so is $T\otimes A$, as well as $T\{A\}$ as long as $T$ is
faithful, which is implied by a). Conversely, let $X\in \cT$. Using only the essential
surjectivity of $T\otimes A$, we find an $A$-isogeny
\[\phi:X\to T(Y)\]
with $Y\in \cS$. A cone of $\phi$ is a torsion object, hence, if $T\{A\}$ is essentially
surjective, it is isomorphic to $T(C)$ for $C\in \cS\{A\}$. Thus $X$ sits in an exact triangle
\[X\by{\phi} T(Y)\by{\psi} T(C)\by{+1}.\]

If $T$ is full, then $\psi=T(\psi')$ and $X \simeq T(X')$, where $X'$ is a fibre of
$\psi'$.
\end{proof}

\subsection{Torsion objects in an abelian category} The proof of the following proposition is
similar to that of Proposition \ref{p1.1} and is left to the reader.

\begin{propose}\label{p1.1ab} Let $\cA$ be an abelian category. Then\\
a) The full subcategory $\cA\{A\}$
is a Serre subcategory.\\ 
b) Let $X\in \cA$ and $n>0$
invertible in $A$. Then the kernel and cokernel of multiplication by
$n$ on $X$ belong to $\cA\{A\}$.\\  
c) The localised category
$\cA/\cA\{A\}$ is canonically isomorphic to
$\cA\otimes A$. In particular, $\cA\otimes A$ is abelian.\\
d) A morphism $f\in \cA$ is in $\Sigma_A(\cA)$ if and only if $\ker f\in \cA\{A\}$ and
$\coker f\in \cA\{A\}$.\qed
\end{propose}

The following corollary is a direct consequence of Proposition \ref{p1.1ab} and Lemma
\ref{lB.1.3}:

\begin{cor}\label{cB.1.3} Let $\cA$ be an abelian category. Let $\cB$ be a full additive
subcategory of
$\cA$, and suppose that every object of $\cA$ is $A$-isogenous to an object of $\cB$ (see
Definition
\ref{dA.2}). Then $\cB\otimes A$ is abelian, and in particular idempotent-complete.\qed
\end{cor}

\subsection{Abelian and derived categories}

\begin{propose}\label{pB.4.1} Let $\cA$ be an abelian category. Then the natural functor
$D^b(\cA)\to D^b(\cA\otimes A)$ induces an equivalence of categories
\[D^b(\cA)\otimes A\iso D^b(\cA\otimes A).\]
In particular, $D^b(\cA)\otimes A$ is idempotent-complete.
\end{propose}

\begin{proof} In 3 steps:

1) The natural functor $C^b(\cA)\otimes A\to C^b(\cA\otimes A)$ is an equivalence of
categories. Full faithfulness is clear. For essential surjectivity, take a bounded complex $C$
of objects of $\cA\otimes A$. Find a common denominator to all differentials involved in $C$.
Then the corresponding morphisms of $\cA$ have torsion composition; since they are finitely
many, we may multiply by a common bigger integer so that they compose to $0$. The resulting
complex of $C^b(\cA)$ then becomes isomorphic to $C$ in $C^b(\cA\otimes A)$.

2) The functor of 1) induces an equivalence of categories $K^b(\cA)\otimes A\iso
K^b(\cA\otimes A)$. Fullness is clear, and faithfulness is obtained by the same technique as
in 1).

3) The functor of 2) induces the desired equivalence of categories. Write 	as usual $D^b_{\cA\{A\}}(\cA)$ for the thick subcategory of $D^b(\cA)$ whose objects have cohomology in the thick subcategory $\cA\{A\}\subseteq \cA$. First, the functor
\[D^b(\cA)/D^b_{\cA\{A\}}(\cA)\to D^b(\cA/\cA\{A\})\]
is obviously conservative. But clearly $D^b_{\cA\{A\}}(\cA)=D^b(\cA)\{A\}$ (by induction of the lengths of complexes). Hence,
by Propositions \ref{p1.1} and \ref{p1.1ab}, this functor translates as
\[D^b(\cA)\otimes A\to D^b(\cA\otimes A).\]

Let $A^b(\cA)$ denote the thick subcategory of $K^b(\cA)$ consiting of acyclic complexes. By
Lemma \ref{lB.2.3} we have a commutative diagram of exact sequences of triangulated categories
\[\begin{CD}
0@>>> A^b(\cA)\otimes A@>>> K^b(\cA)\otimes A@>>> D^b(\cA)\otimes A@>>> 0\\
&&@VVV @VVV @VVV\\
0@>>> A^b(\cA\otimes A)@>>> K^b(\cA\otimes A)@>>> D^b(\cA\otimes A)@>>> 0.
\end{CD}\]

We have just seen that the right vertical functor is conservative, and by 2), the middle one
is an equivalence. Hence the left one is essentially surjective. By the same argument as in the
proof of Proposition \ref{p1.2} b), we get that the right functor is full, and the result
follows from Lemma \ref{lA.2}. 
\end{proof}

\section{$1$-motives with torsion}\label{AppendixB}

Effective $1$-motives which admit torsion are introduced in \cite[\S
1]{BRS} (in characteristic $0$). We investigate some properties (over a perfect field of  exponential
characteristic $p\ge 1$) which are not included in op. cit. as a supplement  to our Sect. 1.

Throughout this appendix, we put $1$-motives in degrees $-1$ and $0$ (as we did in the subsection \ref{sbiext}) in order to follow Deligne's notation in \cite{D}.

\subsection{Effective $1$-motives}
An \emph{effective $1$-motive with torsion} over $k$ is a complex of group 
schemes $M= [ L \by{u} G]$ where $L$ is finitely generated locally constant for the
{\'e}tale  topology (\ie discrete in the sense of Def.~\ref{d1.1.1}) and
$G$ is a  semi-abelian $k$-scheme. Therefore
$L$ can be represented by an extension 
$$0\to L_{\tor}\to L \to L_{\fr}\to 0$$ where $L_{\tor}$ is a finite \'etale $k$-group scheme and $L_{\fr}$ is free, \ie a lattice. Also $G$ can be represented 
by an extension of an abelian $k$-scheme $A$ by a $k$-torus $T$.

\begin{defn}\label{eff1mot}
 An {\it effective} map from $M = [ L \by{u} G]$ to $M'= [ L' \by{u'} 
G']$ is a commutative square
\[ \begin{CD}
 L @>u>>  G\\
@V{f}VV  @V{g}VV  \\
L'@>u'>> G'
\end{CD}\]
in the category of commutative group schemes. Denote by $(f, g):M\to M'$ such a
map. The natural composition of squares makes  up a category, denoted by
${}^t\M^{\eff}$.  We will denote by $\Hom_{\eff}(M, M')$ the abelian group
of effective  morphisms.\index{${}^t\M$, ${}^t\M^\eff$}
\end{defn}

For a given $1$-motive $M = [L \by{u} G]$  we have a commutative diagram 
\begin{equation}\label{basic}
\begin{CD}
&&&&0&  &0 \\
&&&&@V{}VV @V{}VV  \\
0@>>> \ker (u)\cap L_{\tor}@>>>  L_{\tor}@>{u}>> u(L_{\tor})@>>> 0\\
&&@V{}VV  @V{}VV @V{}VV  \\
0@>>> \ker (u)@>>> L @>{u}>> G\\
&&&&@V{}VV  @V{}VV  \\
&&&&L_{\fr} @>{\bar u}>> G/u(L_{\tor})\\
&&&&@V{}VV  @V{}VV  \\
&&&&0&  &0 
\end{CD}
\end{equation}
with exact rows and columns. We set
\begin{itemize}
\item $M_{\fr}\df [L_{\fr}  \by{\bar u}  G/u(L_{\tor})]$
\item $M_{\tor}\df [\ker (u)\cap L_{\tor}\to 0]$
\item $M_{\tf}\df [L/\ker (u)\cap L_{\tor}  \by{u} G]$
\end{itemize}
considered as effective $1$-motives. From Diagram \eqref{basic}
there  are canonical effective maps $M \to M_{\tf}$, $M_{\tor}\to
M$ and
$M_{\tf} 
\to M_{\fr}$.
\index{$M_{\fr}$, $M_{\tor}$, $M_{\tf}$}
\begin{defn}\label{frtortf}
A $1$-motive $M = [L \by{u} G]$ is {\it free}\, if $L$ is free, \ie if 
$M = M_{\fr}$. $M$ is  {\it torsion}\, if $L$ is torsion and $G =0$, \ie 
if $M=M_{\tor}$, and 
{\it torsion-free}\, if $\ker (u)\cap L_{\tor} =0$, \ie if $M=M_{\tf}$. 

Denote by ${}^t\M^{\eff ,\fr}$, ${}^t\M^{\eff ,\tor}$ and 
${}^t\M^{\eff ,\tf}$, the full sub-categories of 
${}^t\M^{\eff}$ given by free, torsion and torsion-free $1$-motives
respectively.
\end{defn}

The category ${}^t\M^{\eff ,\fr}$ is nothing else than the category $\M$ of Deligne $1$-motives
and we shall henceforth use this notation. It is clear that ${}^t\M^{\eff,\tor}$  is equivalent
to the category of finite \'etale group schemes. If $M$ is torsion-free then the lower square in diagram \eqref{basic} is a
pull-back, \ie $L$  is the pull-back of $L_{\fr}$ along the isogeny $G\to G/L_{\tor}$.

\begin{propose} \label{lim} The categories ${}^t\M^{\eff}$ and $\M$ have all finite limits and
colimits. 
\end{propose}

\begin{proof} Since these are additive categories (with biproducts), it is
enough to show that they have kernels, dually cokernels. Now let
$\phi=(f,g):M=[L \by{u} G]\to M'= [L' \by{u'} G']$ be an effective map. We claim that 
$$\ker \phi = [\ker^0(f)  \by{u}  \ker^0(g)]$$
where: (i) $ \ker^0(g)$ is the (reduced) connected component of the identity of the kernel of $g : G\to G'$ and (ii)
 $\ker^0(f) \subseteq \ker(f)$ is the pull-back of $\ker^0(g)$ along $u\mid_{\ker f}$.
We have to show that the following diagram of effective
$1$-motives
\[
\begin{CD}
0&& 0 \\
@V{}VV @V{}VV  \\
\ker^0(f) @>{u}>>  \ker^0(g)\\
@V{}VV @V{}VV  \\
L @>{u}>> G\\
@V{f}VV  @V{g}VV  \\
L' @>{u'}>> G'\\
\end{CD}
\]
satisfies the universal property for kernels. Suppose that $M''= [L'' 
\by{u''}  G'']$ is mapping to $M$ in such a way that the composition 
$M'' \to M \to M'$ is the zero map. Then $L''$ maps to $\ker (f)$ and
$G''$ maps to $\ker (g)$. Since $G''$ is connected, it actually maps
to $\ker^0(g)$ and, by the universal property of pull-backs in the
category of group schemes, $L''$ then maps to $\ker^0(f)$. Finally note
that if $L$ is free then also $\ker^0(f)$ is free. 

For cokernels, we see that 
$$[\coker (f)  \by{\bar u'}  \coker(g)]$$
is an effective $1$-motive which is clearly a cokernel of $\phi$.

For $\M$, it is enough to take the free part of the
cokernel, \ie  given $(f, g) : M \to M'$ then $[\coker (f) \to \coker
(g)]_{\fr}$ meets the universal property
for coker of free $1$-motives.
\end{proof}

\subsection{Quasi-isomorphisms} (\cf \cite[\S 1]{BRS}). 

\begin{defn} An effective morphism of $1$-motives $M \to M'$, here $M = [L \by{u} G]$ and $M' = [L'
\by{u} G']$, is a  {\it quasi-isomorphism}\, (\qi for short) of $1$-motives if it yields a
pull-back diagram
\begin{equation}\label{qi1mot}
\begin{CD}
0&  &0 \\
@V{}VV @V{}VV  \\
  F @= F\\
@V{}VV @V{}VV  \\
 L @>{u}>>  G\\
@V{}VV @V{}VV  \\
L' @>{u'}>> G'\\
@V{}VV@V{}VV  \\
0&  &0 
\end{CD}
\end{equation}
where $F$ is a finite \'etale group.
\end{defn}

\begin{remarks}\label{qi=qi}
1) Note that kernel and cokernel of a quasi-iso\-mor\-phism of
$1$-motives are $0$ but, in general, a quasi-isomorphism is not an
isomorphism in
${}^t\M^{\eff}$. Hence the category ${}^t\M^{\eff}$ is not abelian.\\
2) A \qi of $1$-motives $M \to M'$ is actually a \qi of complexes of group
schemes. In fact, an effective map of $1$-motives $M \to M'$ is a \qi of
complexes if and only if we have the following diagram
\[ \begin{CD}
0&\to& \ker (u)@>>> L @>{u}>> G@>>> \coker (u)&\to& 0\\
&&\veq &&@V{}VV@V{}VV\veq  \\
0&\to & \ker (u')@>>> L' @>{u'}>> G'@>>> \coker (u')&\to& 0.\\
\end{CD}\]

Therefore $\ker$ and $\coker$ of $L\to L'$ and $G\to G'$ are equal. Then 
$\coker (G\to G')= 0$, since it is connected and discrete,
and $\ker (G\to G')$ is a finite group. Conversely, Diagram
\eqref{qi1mot} clearly yields a \qi of complexes. In particular, it
easily follows that the class of \qi of $1$-motives is closed under
composition of effective morphisms.
\end{remarks}

\begin{propose}\label{B1.1} Quasi-isomorphisms are simplifiable on the
left and on the right.
\end{propose}

\begin{proof} The assertion ``on the right" is obvious since the two
components of a \qi are epimorphisms. For the left, let $\phi=(f,g): M\to
M'$ and  $\sigma=(s,t):M' \to\tilde M$ a \qi such that $\sigma\phi=0$.
In the diagram
\[\begin{CD}
 L  @>{u}>>  G\\
@V{f}VV  @V{g}VV  \\
L' @>{u'}>> G'\\
@V{s}VV   @V{s}VV  \\
 \tilde L @>{\tilde u}>> \tilde G
\end{CD}\]
we have $ \tilde L = L'/F$, $\tilde G = G'/F$, for some finite group $F$,
$\im (f)\subseteq F$ and $\im (g) \subseteq F$. Now $u'$ restricts to the
identity on $F$ thus  $\im (f)\subseteq \im (g)$  
and $\im (g) =0$, since $\im (g)$ is connected, hence $\phi=0$.  
\end{proof}

\begin{propose}\label{calfrac} The class of \qi admits a calculus of right
fractions in the sense of (the dual of) \cite[Ch. I, \S 2.3]{GZ}.
\end{propose}

\begin{proof} By \cite[Lemma 1.2]{BRS}, the first condition of calculus of right fractions
is verified, and Proposition \ref{B1.1} shows that the second one is
verified as well.  (Note that we only consider isogenies with \'etale kernel here.)
\end{proof}

\begin{remark} The example of the diagram
\[\begin{CD}
[L\to G]@>\sigma>> [L'\to G']\\
@V{(1,0)}VV \\
[L\to 0]
\end{CD}\]
where $\sigma$ is a nontrivial \qi shows that calculus of left fractions
fails in general.
\end{remark}

\begin{lemma}\label{fact} Let $s,t,u$ be three maps in ${}^t\M^{\eff}$, with
$su=t$. If two out of them are $\qi$, then so is the third.
\end{lemma}

\begin{proof} The case of $(s,u)$ is explained in Remark \ref{qi=qi} 2) and the case of $(u,t)$ is not hard; let us explain the case of $(s,t)$. Consider the exact sequence of complexes of sheaves
\[0\to\ker u\to\ker t\to \ker s\to \coker u\to\coker t\to\coker s\to 0.\]

Since $s$ and $t$ are $\qi$, the last two terms are $0$. Hence $\coker
(u) =[L\to G]$ is a quotient of $\ker (s)$; since $G$ is connected, we must
have $G=0$. On the other hand, as a cokernel of a map of acyclic
complexes of length $1$, $\coker (u)$ is acyclic, hence $L=0$. Similarly,
$\ker (u)$ is acyclic.
\end{proof}

\subsection{$1$-motives} We now
define the category of $1$-motives with torsion from ${}^t\M^{\eff}$
by formally inverting quasi-isomorphisms.

\begin{defn}\label{1tors} The category ${}^t\M$ of  \emph{$1$-motives with torsion} is the
localisation of ${}^t\M^{\eff}$ with respect to the multiplicative class $\{\qi\}$ of
quasi-isomor\-phisms. \index{${}^t\M$, ${}^t\M^\eff$}
\end{defn}

\begin{remark}
Note that there are no nontrivial \qi between free (or torsion)
$1$-motives. However, the canonical map  $M_{\tf}  \to M_{\fr}$ is
a quasi-isomorphism (it is an effective isomorphism when
$u(L_{\tor})=0$). 
\end{remark}

It follows from Proposition \ref{calfrac} and \cite[Ch. I, Prop. 2-4]{GZ} that the Hom sets in ${}^t\M$
are given by the formula
$$\Hom (M, M') = \limdir{\rm q.i.} \Hom_{\eff}  (\tilde M, M')$$
where the limit is taken over the filtering set of all quasi-isomorphisms
$\tilde M\to M$. A morphism of $1$-motives $M\to M'$ can be represented as
a diagram 
$$ \begin{array}{ccc}
 M\ &  &\ M'\\
{\scriptstyle {\rm q.i.}}\nwarrow & &  \nearrow {\scriptstyle {\rm eff}}
\\ & \tilde M & 
\end{array}$$
and the composition is given by the existence of an $\hat M $
making the following diagram
$$ \begin{array}{ccccc}
 M\ &  &\ M'\ &&\ M''\\
{\scriptstyle {\rm q.i.}}\nwarrow & &  \nearrow \ \ 
\nwarrow & &  \nearrow {\scriptstyle {\rm eff}}\\
& \tilde M & & \tilde M'&\\
&{\scriptstyle {\rm q.i.}}\nwarrow & &  \nearrow {\scriptstyle {\rm
eff}}&\\ &  &\hat M & &
\end{array}$$
commutative. (This $\hat M$ is uniquely determined by the other morphisms and the commutativity,
see \cite[Lemma 1.2]{BRS}.)

\subsection{Strict morphisms}
The notion of strict morphism is essential in order to show that the
$\Z[1/p]$-linear category of $1$-motives with torsion is abelian (\cf \cite[\S 1]{BRS}).

\begin{defn}\label{dstrict}  We say that an effective morphism $(f,g): M\to M'$ is
\emph{strict} if we have 
\[\ker (f, g)= [\ker (f)\to\ker (g)]\]
\ie if $\ker (g)$ is (connected) semi-abelian.
\end{defn}

To get a feeling on the notion of strict morphism, note:

\begin{lemma} Let $\phi=(f,g):M\to N$ be a strict morphism, with $g$ onto.
Suppose that $\phi=\sigma\tilde\phi$, where $\sigma$ is a \qi Then
$\sigma$ is an isomorphism.
\end{lemma}

Conversely, we obtain: 

\begin{propose}[\protect{\cite[Prop. 1.3]{BRS}}]\label{pstrict}
Any effective morphism
$\phi\in $ \goodbreak $\Hom_{\eff}(M,M')$ can be factored as follows
\begin{equation}\label{strict}
\begin{array}{c}
M  \longby{\phi} M'\\
\scriptstyle \tilde\phi\displaystyle
\searrow\hspace*{0.5cm} \nearrow\\
\tilde M
\end{array}
\end{equation}
where  $\tilde
\phi$ is a strict morphism and
$\tilde M \to M'$ is a composition of a \qi and a $p$-power isogeny.
\end{propose}
\begin{proof} {\it (Sketch)}\,
Note that if $\phi  = (f,g)$ we always have the following natural
factorisation of the map $g$ between semi-abelian schemes
$$\begin{array}{c}
G  \longby{g} G'\\
\searrow\hspace*{0.5cm} \nearrow\\
G/\ker^0(g)
\end{array}$$
If $g$ is a surjection we get the claimed factorisation
by taking $\tilde M = [ \tilde L \to G/\ker^0(g)]$ where $\tilde L$ is the
pull-back of $L'$, the lifting of $f$ is granted by the universal
property of pull-backs. In general, we can extend the so obtained isogeny  
on the image of $g$ to an isogeny of  $G'$ (see the proof of Prop.
1.3 in \cite{BRS} for details).
\end{proof}

\begin{lemma}\label{strictqi}
Let
\[\begin{CD}
M'@>f>> M\\
@A{u}AA @A{t}AA \\
N' @>h>> N 
\end{CD}\]
be a commutative diagram in ${}^t\M^{\eff}$, where $f$ is strict and $u,t$ are
\qi Then the induced map $v:\coker (h)\to \coker (f)$ is a \qi
\end{lemma}

\begin{proof} In all this proof, the term ``kernel" is taken in the
sense of kernel of complexes of sheaves. Let $K$ and
$K'$ be the kernels of $f$ and
$h$ respectively:
\[\begin{CD}
0@>>> K@>>>M'@>f>> M@>>> \coker (f)\\
&&@A{w}AA @A{u}AA @A{t}AA @A{v}AA \\
0@>>> K'@>>> N' @>h>> N@>>> \coker (h) 
\end{CD}\]

By a diagram chase, we see that $v$ and $\ker t\to \ker v$
are onto. To conclude, it will be sufficient to show that the
sequence of complexes 
\begin{equation}\label{eq1}
\ker (u)\to \ker (t)\to\ker (v)
\end{equation}
is exact termwise. For this, note that the second component of $w$ is onto
because $f$ is strict and by dimension reasons. This implies by a diagram
chase that the second component of \eqref{eq1} is exact. But then the
first component has to be exact too.
\end{proof}

\subsection{Exact sequences of $1$-motives}
We have the following basic properties of $1$-motives. 

\begin{propose} \label{faith} The canonical functor $${}^t\M^{\eff}\to {}^t\M$$
is left exact and faithful. 
\end{propose}

\begin{proof} Faithfulness immediately follows from Proposition
\ref{B1.1}, while left exactness follows from Proposition
\ref{calfrac} and (the dual of) \cite[Ch. I, Prop. 3.1]{GZ}.
\end{proof}

\begin{lemma}\label{coker} Let $f:M'\to M$ be an effective map. 
\begin{enumerate}
\item The canonical projection $\pi : M \to \coker (f)$ remains an
epimorphism  in ${}^t\M[1/p]$.
\item If $f$ is strict then $\pi$  remains a cokernel in ${}^t\M[1/p]$.
\item  Cokernels exist in ${}^t\M[1/p]$. 
\end{enumerate}
\end{lemma}

\begin{proof} To show (1), let $\pi: M \to  N$ be an effective
map.  One sees immediately that $\pi$ is epi in ${}^t\M[1/p]$ if and only if for
any commutative diagram
\[\begin{CD}
M@>\pi>> N\\
@A{s'}AA@A{s}AA \\
Q'@>\pi'>> Q  
\end{CD}\]
with $s,s'$ \qi, the map $\pi'$ is an epi in the effective category. Now
specialise to the case $N = \coker (f)$ and remark that (up to modding
out by $\ker f$) we may assume $f$ to be a monomorphism as a map of
complexes, thus strict. Take $\pi',s,s'$ as above. We have a commutative
diagram of effective maps
\[\begin{CD}
0@>>> M'@>f>> M&@>\pi>> &\coker (f)\\
&&&&&&&\scriptstyle t\displaystyle\nearrow\\
&& @A{s''}AA @A{s'}AA\coker(f')&&@A{s}AA \\
&&&&&\displaystyle\nearrow&&\scriptstyle u\displaystyle\searrow\\
&& Q'' @>f'>> Q'&@>\pi'>> &Q.
\end{CD}\]

\begin{sloppypar}
Here $s''$ is a \qi and $Q'',f',s''$ are obtained by calculus of right
fractions (Proposition \ref{calfrac}). By Lemma \ref{strictqi}, the
induced map  $t:\coker (f') \to \coker (f)$ is a \qi. By Proposition
\ref{B1.1}, $\pi'f'=0$, hence the existence of $u$. By Lemma \ref{fact},
$u$ is a \qi. Hence $\pi'$ is a composition of two epimorphisms and
(1) is proven.
\end{sloppypar}

To show (2), let $gt^{-1}:M
\to M''$ be such that the composition  $M' \to M''$ is zero. By
calculus of right fractions we have a commutative diagram
\[\begin{CD}
M'@>f>> M@>\pi>> \coker (f)\\
@A{u}AA @A{t}AA @A{v}AA\\
N''' @>h>> N'' @>>> \coker (h)\\ 
&&&\scriptstyle{g}\searrow& @V{}VV\\ 
&&&&M''
\end{CD}\]
where all maps are effective and $u$ is a \qi. As above
we have $gh=0$, hence the factorisation of $g$ through $\coker (h)$.
Moreover $\coker (h)$ maps canonically to $\coker (f)$ via a map $v$ (say),
which is a
\qi by Lemma \ref{strictqi}. This shows that $gt^{-1}$ factors
through $\coker (f)$ in ${}^t\M[1/p]$.
Uniqueness of the factorisation is then granted by Part (1).

In a category, the existence of cokernels is
invariant by left or right composition by isomorphisms, hence (3) is a
consequence of Parts (1) and (2) via Proposition \ref{pstrict}.
\end{proof}

Now we can show the following key result (\cf \cite[Prop.~1.3]{BRS}).

\begin{thm} \label{1mtora}
The category ${}^t\M[1/p]$ is abelian.
\end{thm}

\begin{proof} The existence and description of kernels follow
from Propositions \ref{lim}, \ref{calfrac}, \ref{faith} and (the dual of)
\cite[Ch. I, Cor. 3.2]{GZ}, while the existence of cokernels has been proven
in Lemma \ref{coker}.
We are then left to show that, for any (effective) strict map $\phi: M\to
M'$, the canonical effective morphism  from the coimage of $\phi$ to the
image of $\phi$ is a \qi of $1$-motives, \ie the canonical
morphism
\begin{equation}\label{coimage}
\coker (\ker \phi \to M)\to \ker (M'\to \coker \phi) 
\end{equation}
is a quasi-isomorphism. Since we can split $\phi$ in two short
exact sequences of complexes in which each term is an effective
$1$-motive we see that \eqref{coimage} is even a isomorphism in ${}^t\M^{\eff}[1/p]$.
\end{proof}

\begin{remark} Note that (even in characteristic zero) for a given non-strict effective map
$(f,g): M\to M'$ the effective morphism \eqref{coimage} is not a \qi of
$1$-motives. In fact, the following diagram
\[
\begin{CD}
0&  &0 \\
@V{}VV @V{}VV  \\
  \ker (f)/\ker^0(f) & {\subseteq} & \ker (g)/\ker^0(g)\\
@V{}VV @V{}VV  \\
 L/\ker^0(f) @>>> G/\ker^0(g)\\
@V{}VV @V{}VV  \\
\im (f) @>>> \im (g)\\
@V{}VV @V{}VV  \\
0&  &0 
\end{CD}
\]
is not a pull-back, in general. For example, let $g:G\to G'$ be with finite
kernel and a proper sub-group $F\subsetneq\ker (g)$, and consider 
$$ (0,g): [F\to G] \to [0\to G'].$$
\end{remark} 

\begin{cor} \label{corexseq} A short exact sequence of $1$-motives in ${}^t\M[1/p]$ 
\begin{equation} \label{exseq}
0\to M'\to M \to M''\to 0
\end{equation}
can be represented up to isomorphisms by a strict effective epimorphism
$(f,g) : M\to M''$ with kernel $M'$, \ie by an exact sequence of complexes.
\end{cor}
\begin{example} 
Let $M$ be a $1$-motive with torsion. We then always have a canonical
exact sequence in ${}^t\M[1/p]$
\begin{equation}\label{shortfree}
0\to M_{\tor}\to M \to M_{\fr}\to 0
\end{equation}
induced by \eqref{basic}, according to Definition~\ref{frtortf}. 
Note that in the following canonical factorisation
$$\begin{array}{c}
M  \longby{} M_{\fr}\\
\searrow\hspace*{0.4cm} \nearrow \\
M_{\tf}
\end{array}$$
the effective map $M \to M_{\tf}$ is a strict epimorphism with kernel
$M_{\tor}$ and $M_{\tf}\to M_{\fr}$ is a \qi (providing an example of
Proposition~\ref{pstrict}).
\end{example}

\subsection{$\ell$-adic realisation} \label{s.ladic}
Let $n: M\to M$ be the (effective) multiplication by $n$ on a $1$-motive
$M =[L\by{u} G]$ over a field $k$ where $n$ is prime to the characteristic of
$k$. It is then easy to see, \eg by the description of kernels in
Proposition~\ref{lim}, that 
$${}_nM \df \ker (M\longby{n} M) = [\ker (u)\cap {}_nL \to 0].$$
Thus ${}_nM = 0$ (all $n$ in characteristic zero)  if and only if $M$ is
torsion-free, \ie
$M_{\tor} =0$. Moreover, by Proposition~\ref{pstrict} and 
Lemma~\ref{coker} we see that
$$M/n\df \coker (M\longby{n} M)$$ is always a torsion $1$-motive. If $L=0$,
let simply $G$ denote, as usual, the $1$-motive $[0\to G]$. Then we
get an extension in $\M[1/p]$
\begin{equation}\label{ngexseq}
0\to G \longby{n} G \to {}_nG[1]\to 0
\end{equation}
where ${}_nG[1]$ is the torsion $1$-motive $[{}_nG\to 0]$. In general, $M/n$ can be regarded as an extension of $L/n$ by $\coker ({}_nL\to
{}_nG)$, \eg also by applying the snake lemma to the multiplication by $n$
on the following canonical short exact sequence (here $L[1] = [L\to 0]$ as
usual)
\begin{equation}\label{stexseq}
0\to G \to M \to L[1]\to 0
\end{equation}
of effective $1$-motives (which is also exact in $\M[1/p]$ by
Corollary~\ref{corexseq}). Summarizing, we get a long exact sequence in ${}^t\M^{\tor}[1/p]$
\begin{equation}\label{nexseq}
0 \to {}_nM \to {}_nL[1]\to {}_nG[1] \to M/n \to L/n[1] \to 0.
\end{equation}

Let now be $n = \ell^{\nu}$ where $\ell \neq {\rm char} (k)$. Set:

\begin{defn}\label{ladic} The {\it $\ell$-adic realisation}\, of a
$1$-motive
$M$ is $$T_{\ell}(M)\df ``\liminv{\nu}" L_\nu$$
in the category of $l$-adic sheaves, where $M/\ell^\nu=[L_\nu\to 0]$.
\end{defn}

Since the inverse system $``\varprojlim_{\nu}" {}_{l^\nu} L$ is Mittag-Leffler
trivial, we obtain a short exact sequence
$$0 \to T_{\ell}(G) \to T_{\ell}(M) \to L\otimes \Z_{\ell}
\to 0$$ 
where $T_{\ell}(G)$ is the Tate module of the semia-belian
variety $G$. More generally, using Corollary \ref{corexseq}, we have:

\begin{lemma}\label{ltlexact} The functor $T_\ell$ is exact on ${}^t\M[1/p]$,
and extends can\-on\-ically to ${}^t\M\otimes\Z_\ell$.\qed
\end{lemma}

\subsection{Deligne $1$-motives} Let ${}^t\M^{\fr}[1/p]$, ${}^t\M^{\tor}[1/p]$ and 
${}^t\M^{\tf}[1/p]$ denote the corresponding full subcategories of 
${}^t\M[1/p]$ given by free, torsion and torsion-free effective $1$-motives
respectively. The following $M\mapsto M_{\fr}$ (resp. $M\mapsto M_{\tor}$) define functors
from ${}^t\M[1/p]$ to ${}^t\M^{\fr}[1/p]$ (resp. from ${}^t\M[1/p]$ to ${}^t\M^{\tor}[1/p]$). We have (\cf \cite[(1.1.3)]{BRS}):

\begin{propose}\label{free} The natural functor 
\[\M[1/p]\to {}^t\M[1/p]\] 
from Deligne
$1$-motives to $1$-motives with torsion has a left adjoint/left inverse given by $M\mapsto
M_{\fr}$. In particular, it is fully faithful and makes $\M[1/p]$ an exact sub-category of ${}^t\M[1/p]$. The above left adjoint defines equivalences 
$$\M[1/p]\cong{}^t\M^{\fr}[1/p]\cong {}^t\M^{\tf}[1/p].$$
\end{propose}

\begin{proof} Consider an effective map $(f, g):\tilde M
\to M'$, to a free $1$-motive $M'$, and a \qi $\tilde M \to M$, \ie $M
= [\tilde L/F \to \tilde G/F]$ for a finite group $F$. Since $M'$ is free
then $F$ is contained in the kernel of $f$ and the same holds for $g$. Thus $(f, g)$ induces an effective map $ M\to M'$. Let $M = [L\by{u}
G]$. Then $L_{\tor}\subseteq \ker (f)$ and also $u(L_{\tor})\subseteq \ker
(g)$ yielding an effective map $(f, g): M_{\fr}\to M'$. This proves the first assertion.

Since ${}^t\M^{\eff,\fr}\into {}^t\M^{\eff,\tf}$, the claimed equivalence is
obtained from the canonical \qi $M\to M_{\fr}$ for $M\in {}^t\M^{\eff,\tf}$, see
\eqref{basic}. 
Finally, consider the exact sequence \eqref{exseq} of $1$-motives with torsion such that
$M'_{\tor}=M''_{\tor}=0$. Since $M_{\tor}$ is mapped to zero in $M''$, it
injects in $M'$. Thus also $M$ is torsion-free, \ie $M=M_{\tf}$, and
quasi-isomorphic to $M_{\fr}$.
\end{proof}

\begin{remark} We also clearly have that the functor $M\mapsto
M_{\tor}$ is a right adjoint to the embedding ${}^t\M^{\tor}[1/p]\into {}^t\M[1/p]$ , \ie 
$$\Hom_{\eff}(M, M'_{\tor}) \cong \Hom (M, M')$$
for $M\in{}^t\M^{\tor}[1/p]$ and $M'\in {}^t\M[1/p]$.
\end{remark}

\begin{cor}\label{isofree} We have ${}^t\M^{\tor}\otimes \Q =0$ and 
the full embedding $\M[1/p]\to {}^t\M[1/p]$ induces an equivalence 
$$\M\otimes \Q\iso {}^t\M\otimes \Q.$$
\end{cor}

\subsection{Homs and Extensions} We will provide a characterisation
of the Yoneda $\Ext$ in the abelian category ${}^t\M[1/p]$. 

\begin{propose}\label{hom} We have
\begin{itemize}
\item[(a)] $\Hom_{{}^t\M} (L[1],L'[1]) = \Hom_k (L,L')$,
\item[(b)] $\Hom_{{}^t\M} (L[1],G') =0$,
\item[(c)] $\Hom_{{}^t\M} (G,G') \subseteq \Hom_k (A,A')\times
\Hom_k(T,T')$ if $G$ (resp. $G'$) is an extension of an abelian
variety $A$ by a torus $T$ (resp. of $A'$ by $T'$),
\item[(d)] $\Hom_{{}^t\M} (G, L'[1])=\Hom_k({}_nG,L'_{\tor})$ if
$nL'_{\tor}=0$.
\end{itemize}
In particular, the group $\Hom_{{}^t\M}(M, M')$ is finitely generated
for all $1$-motives
$M, M'\in {}^t\M[1/p]$. 
\end{propose}

\begin{sloppypar}
\begin{proof} Since there are no \qi to $L[1]$, we have
$\Hom_{{}^t\M}(L[1],L'[1])=\Hom_{\eff}(L[1],L'[1])$ and the latter is
clearly isomorphic to $\Hom_k(L,L')$. By Proposition \ref{free}, we
have
$\Hom_{{}^t\M} (L[1],G') = \Hom_{\eff} (L[1],G')$ and 
$\Hom_{{}^t\M} (G, G') =  \Hom_{\eff}  (G, G')$. The former is clearly
$0$ while $\Hom_{\eff}(G,G')=\Hom_k(G,G')\subseteq
\Hom_k (A,A')\times
\Hom_k (T,T')$ since $\Hom_k(T,A')=\Hom_k(A,T')=0$. For (d),
let $[F \to \tilde G] \to [0 \to G]$ be a \qi and $[F \to \tilde G]
\to [L'\to 0]$ be an effective map providing an element of $\Hom
(G, L'[1])$. If
$L'$ is free then it yields the zero map, as $F$ is torsion. Thus
$\Hom_{{}^t\M}(G, L'[1]) = \Hom_{{}^t\M}(G, L'_{\tor}[1])$. For $n\in \N$
consider the short exact sequence \eqref{ngexseq} in ${}^t\M$. If $n$
is such that
$nL'_{\tor} =0$ taking $\Hom_{{}^t\M}(-,L'_{\tor}[1])$ we further
obtain $\Hom_{{}^t\M}(G, L'[1]) = \Hom_k ({}_nG, L'_{\tor})$.

The last statement follows from these computations and
an easy d\'evissage from \eqref{stexseq}.
\end{proof}
\end{sloppypar}

\begin{remark} If we want to get rid of the integer $n$ in (d), we
may equally write
\[\Hom_{{}^t\M}(G,
L'[1])=\Hom_{\text{cont}}(\hat{T}(G),L'_{\tor})=
\Hom_{\text{cont}}(\hat{T}(G),L')\]
where $\hat{T}(G)=\prod_{\ell\ne p} T_\ell(G)$ is the complete Tate module of
$G$.
\end{remark}

\begin{propose}\label{ext} We have isomorphisms (for $\Ext$ in ${}^t\M[1/p]$):
\begin{itemize}
\item[(a)] $\Ext^1_k (L,L')\iso\Ext^1_{{}^t\M} (L[1],L'[1])$,
\item[(b)] $\Hom_k (L, G')\iso\Ext^1_{{}^t\M} (L[1],G')$,
\item[(c)] $ \Ext^1_k (G,G')\iso\Ext^1_{{}^t\M} (G,G')$ and 
\item[(d)] $\displaystyle\limdir{n}
\Ext^1_k ({}_nG,L')\iso\Ext^1_{{}^t\M} (G, L'[1])$; these two groups
are $0$ if $L'$ is torsion and $k$ algebraically closed. 
\end{itemize}
\end{propose}

\begin{proof} By Corollary \ref{corexseq}, any short exact sequence
of $1$-motives can be represented up to isomorphism by a short exact
sequence of complexes in which each term is an effective $1$-motive.

For (a), just observe that there are no nontrivial \qi of
$1$-motives  with zero semi-abelian part. For (b), note that an
extension of $L[1]$ by $G'$ is given by a diagram 
\begin{equation}\label{ex1} \begin{CD}
0@>>> F' @>{}>> L''@>>> L@>>>  0\\
&&@V{}VV@V{v}VV@V{}VV  \\
0@>>> \tilde G' @>{}>> G''@>>> 0@>>>  0
\end{CD}
\end{equation}
where $\tilde M' = [F'\to \tilde G']$ is \qi to $[0\to G']$. When
$F'=0$, this diagram is equivalent to the datum of $v$: this
provides a linear map $\Hom_k (L, G')\to \Ext^1_{{}^t\M} (L[1],G')$.
This map is surjective since we may always mod out by $F'$ in
\eqref{ex1} and get an quasi-isomorphic exact sequence with $F'=0$.
It is also injective: if \eqref{ex1} (with $F'=0$) splits in ${}^t\M[1/p]$,
it already splits in ${}^t\M^{\eff}[1/p]$ and then $v =0$. 

For (c) we see that an extension of $G$ by $G'$ in ${}^t\M[1/p]$ can be
represented by a diagram 
\[ \begin{CD}
0@>>> F' @>{}>> L''@>>> F@>>>  0\\
&&@V{}VV@V{}VV@V{}VV  \\
0@>>> \tilde G' @>{}>> G''@>>> \tilde G@>>>  0
\end{CD}\]
with $\tilde M'$ as in (b) and $\tilde M = [F\to \tilde
G]$ \qi to $[0\to G]$. Since the top line is exact, $L''$ is
finite. For $F=F'=0$ we just get a group scheme extension of $G$ by
$G'$, hence a homomorphism $\Ext^1_k(G,G')\to \Ext^1_{{}^t\M}(G,G')$.
This homomorphism is surjective:  dividing by
$F'$ we get a quasi-isomorphic exact sequence
\[ \begin{CD}
0@>>> 0 @>{}>> L''/F'@>{\sim}>> F@>>>  0\\
&&@V{}VV@V{}VV@V{}VV  \\
0@>>> G' @>{}>> G''/F'@>>> \tilde G@>>>  0
\end{CD}\]
and further dividing by $F$ we then obtain 
\[ \begin{CD}
0@>>> 0 @>{}>> 0@>>> 0@>>>  0\\
&&@V{}VV@V{}VV@V{}VV  \\
0@>>> G' @>{}>> G''/L''@>>> G@>>>  0.
\end{CD}\]

Injectivity is seen as in (b).

For (d) we first construct a map $\Phi_n:\Ext^1_k({}_nG,L') \to
\Ext^1_{{}^t\M}(G,L'[1])$ for all $n$.  Let $[L'']\in \Ext_k
({}_nG,L')$ and consider the following diagram
\begin{equation}\label{ex2} 
\begin{CD}
0@>>> L' @>{}>> L''@>>> {}_nG@>>>  0\\
&&@V{}VV@V{}VV @V{}VV  \\
0@>>> 0 @>{}>> G@= G@>>>  0.
\end{CD}\end{equation}

Since $[{}_nG \to G]$ is \qi to $[0\to G]$, this provides an
extension of $G$ by $L'[1]$ in ${}^t\M[1/p]$. For $n$ variable $\{\Ext^1_k
({}_nG,L')\}_n$ is a direct system and one checks easily that the
maps $\Phi_n$ are compatible (by pull-back), yielding a
well-defined linear map
\[\Phi:\varinjlim \Ext^1_k ({}_nG,L')\to\Ext^1_{{}^t\M} (G,
L'[1]).\] 

This map is surjective since any extension of $G$ by $L'[1]$ can be
represented by a diagram \eqref{ex2} for some $n$ (as
multiplication by $n$ is cofinal in the direct sytem of isogenies).
We now show that $\Phi$ is also injective. 

In fact, we have a short exact sequence \eqref{ngexseq} of 1-motives
yielding the following short exact sequence
$$0\to \Hom_{{}^t\M} (G,L'[1])\otimes \Z/n\to \Ext^1_{{}^t\M} ({}_nG[1],L'[1])\to {}_n\Ext^1_{{}^t\M} (G, L'[1])\to 0$$
Passing to the limit we obtain
$$
\begin{CD}
\ker \Phi@>>>\varinjlim \Ext^1_k ({}_nG,L') @>\Phi>> \Ext^1_{{}^t\M} (G,
L'[1])\\
||&&||&&||\\
 0@>>>\varinjlim\Ext^1_{{}^t\M} ({}_nG[1],L'[1]) @>>>\varinjlim {}_n\Ext^1_{{}^t\M} (G, L'[1])
\end{CD}
$$
In fact, $\varinjlim \Hom_{{}^t\M} (G,L'[1])\otimes \Z/n=0$ because $\Hom_{{}^t\M} (G,L'[1])$ is a finite group by Proposition~\ref{hom} (d), $\Ext^1_{{}^t\M} (G, L'[1])$ is torsion as $\Phi$ is surjective and we then just apply (a).

Finally, let $n\mid m$, \eg  $rn =m$, so that the following sequence is exact
$$0\to {}_rG\to {}_mG\longby{r} {}_nG\to 0$$
and yields a long exact sequence
$$\Hom_k ({}_mG,L')\to \Hom_k ({}_rG,L')\to \Ext^1_k
({}_nG,L')\longby{r}
\Ext^1_k ({}_mG,L').$$

If $L'$ is torsion, we have $rL'=0$ for some $r$, hence  
$\varinjlim \Ext^1_k ({}_nG,L')=0$ if $k$ is algebraically closed.
In particular, it shows that  $\Ext^1_{{}^t\M}(G,L'[1])\allowbreak=0$ in this case. 
\end{proof}

\subsection{Projective objects in ${}^t\M[1/p]$}
We show that there are not enough projective objects in ${}^t\M[1/p]$, at least when $k$ is
algebraically closed:

\begin{propose}\label{proj} Suppose that $k=\bar k$.  Then the only projective object of
${}^t\M[1/p]$  is $0$.
\end{propose}

\begin{proof} Suppose that $M =[L\to G]\in {}^t\M[1/p]$ is such that $\Ext (M, N)\allowbreak
=0$ for any $N\in {}^t\M[1/p]$. From \eqref{stexseq}
we then get a long exact sequence
\begin{multline*}
\Hom (G,\G_m)\to\Ext (L[1],\G_m)\to\Ext (M,\G_m)\\
\to\Ext (G,\G_m)\to
\Ext^2 (L[1],\G_m)
\end{multline*}
where $\Ext (M,\G_m)=0$, thus {\it i)}\,  $\Ext (G,\G_m)$ is finite, and 
{\it ii)}\,  $\Hom (L,\G_m)$ is finitely generated.
We also have an exact sequence
$$\Hom (T,\G_m)\to\Ext (A,\G_m)\to\Ext (G,\G_m)$$
where $\Hom (T,\G_m)$ is the character group of the torus $T$ and $\Ext
(A,\G_m)$ is the group of $k$-points of the dual abelian variety $A$, the
abelian quotient of $G$. From {\it i)}\, we get $A=0$. Since  $\Hom
(L,\G_m)$ is an extension of a finite group by a divisible group, from {\it
ii)}\, we get that $L$ is a finite group.
Now consider the exact sequence, for $l\ne p$
\begin{multline*}
0\to\Hom (L[1],\Z/l[1])\to\Hom (M,\Z/l[1])\\
\to\Hom (T,\Z/l[1])\to 
\Ext (L[1],\Z/l[1])\to 0
\end{multline*}
where the right-end vanishing is $\Ext (M,\Z/l[1])=0$ by assumption. Now 
$\Hom (T,\Z/l[1])= \Ext (T,\Z/l)$ and any extension of the torus $T$ is
lifted to an extension of $M$ by $\Z/l$, therefore to an element of 
$\Hom (M,\Z/l[1])\allowbreak = \Ext (M,\Z/l)$. This yields $\Ext
(L,\Z/l)=0$ for any prime $l\ne p$, thus we see that $L=0$.

Finally, $[0\to \G_m]$ is not projective since $\Ext^1(\G_m,\Z[1])\ne 0$ by Proposition \ref{ext} (d).
\end{proof}

\subsection{Weights}\label{sweights}
If $M = [L\by{u}G]\in {}^t\M[1/p]$ is free then Deligne \cite{D} equipped $M = M_{\fr}$ with an
increasing filtration by sub-$1$-motives as follows:
$$W_{-2}(M) \df [0\to T]\subseteq W_{-1}(M)\df [0\to G] \subseteq W_{0}(M) \df M$$

If $M$ is torsion-free we then pull-back the weight filtration along the 
effective map $M \to M_{\fr}$ as follows:
$$W_i(M) \df \left \{\begin {array}{cl} M & i\geq 0\\{} 
[L_{\tor}\into G]& i= -1\\{}
[L_{\tor}\cap T\into T] & i= -2\\ 0 & i \leq -3 \end{array} \right. $$
Note that $W_i(M) $ is \qi to $W_i(M_{\fr})$.

If $M$ has torsion we then further pull-back the weight filtration along the 
effective map $M \to M_{\tf}$. 
\begin{defn}
{\rm Let $M = [L\by{u}G]$ be an effective $1$-motive. Let $u_A : L \to A$
denote the induced map where $A = G/T$. Define
$$W_i(M) \df \left \{\begin {array}{cl} M & i\geq 0\\{} 
[L_{\tor}\to G]& i= -1\\{}
[L_{\tor}\cap \ker (u_A)\to T] & i= -2\\{}
M_{\tor} = L_{\tor}\cap \ker (u)[1] & i = -3\\
0 & i \leq -4 \end{array} \right.$$}
\end{defn}

\begin{remark}It is easy to see that $M\mapsto W_i(M)$ yields a functor
from ${}^t\M[1/p]$ to ${}^t\M[1/p]$. However, this does not define a weight filtration on ${}^t\M[1/p]$ in the sense of Definition \ref{dE.6} or Remark \ref{lF.3.6}.
\end{remark}

\subsection{$1$-motives over a base} \label{1motbase}
Let $S$ be a scheme. According to \cite[(10.1.10)]{D}, a
smooth $1$-motive over $S$ is a complex $[L\to G]$ of $S$-group schemes where $L$ is a lattice
(corresponding to a locally constant $\Z$-constructible free \'etale sheaf) and $G$ is an
extension of an abelian scheme $A$ by a torus $T$. By \cite[Cor. 2.11]{fa-ch}, it is sufficient
to have this condition fibre by fibre as long as the rank of $T_s$ is locally constant.  Smooth
$1$-motives form an additive category denoted (here) by $\M(S)$. A smooth $1$-motive is
provided with a weight filtration as over a field.

We don't know when $\M(S)\otimes \Q$ is abelian, but we only use its existence and functoriality in $S$ in Section \ref{elladic}, along with that of $\ell$-adic realisation functors extending the one of \S \ref{s.ladic}. See \cite{spl} for some results on relative $1$-motives.

\section{Weight filtrations}\label{AppendixD}

In this appendix, we propose a theory of weight filtrations adapted to our needs. This is
closely related to U. Jannsen's and A. Huber's setting in \cite{jannsen} and
\cite[1.2]{HuberLN}, and fits perfectly with S. Morel's viewpoint in \cite{morel}. We relate it
precisely to Jannsen's approach in Remark \ref{lF.3.6}. After this theory was written up, Bondarko proved in \cite{bondarko4} that it is ``compatible'' (via the realisation functors
considered in Section \ref{axDel}) with the weight structure he constructed in \cite{bondarko2} on
$\DM_\gm^\eff$, by introducing a notion of transversality between a weight structure and a $t$-structure.

\subsection{Filtrations of abelian categories}

\subsubsection{A glueing lemma} Let $\cA$ be an abelian category. Consider two exact sequences
\[\begin{CD}
0@>>> A'@>>> A@>>> A''@>>> 0\\
0@>>> B'@>>> B@>>> B''@>>> 0
\end{CD}\]
and two morphisms $f':A'\to B'$, $f'':A''\to B''$. We say that $f:A\to B$ is a \emph{glueing}
of $f'$ and $f''$ if it yields a commutative diagram of exact sequences.

\begin{lemma}\label{lglue0} Let $[A]\in \Ext^1_\cA(A'',A')$ and $[B]\in \Ext^1_\cA(B'',B')$ be
the extension classes of
$A$ and $B$.
\begin{enumerate}
\item For a glueing to exist, it is necessary and sufficient that $f'_*[A]={f''}^*[B]\in
\Ext^1_\cA(A'',B')$.
\item For a glueing to be unique, it is necessary and sufficient that $\Hom_\cA(A'',B')=0$.
\item Suppose that $\cA$ is the heart of a triangulated category $\cT$ provided with a
$t$-structure. Then Condition (1) is equivalent to the following: the diagram
\[\begin{CD}
A''@>[A]>> A'[1]\\
@V{f''}VV @V{f'}VV\\
B''@>[B]>> B'[1]
\end{CD}\] 
commutes in $\cT$.
\end{enumerate}
\end{lemma}

\begin{proof} (1). The condition is clearly necessary, and it is sufficient since by
definition, two extension classes are equal if there exists an isomorphism between the
corresponding extensions.

(2) is obvious since a difference of two glueings is given by a map from $A''$ to $B'$.

(3) is clear by Axiom TR3 of triangulated categories.
\end{proof}

\subsubsection{Extensions panach\'ees} Let $\cA$ still be an abelian category, and
consider now a filtered object
\[
0\subseteq A_{-2}\subseteq A_{\le -1}\subseteq A.
\]

We write $A_{-1}=A_{\le -1}/A_{-2}$, $A_0=A/A_{\le -1}$, $A_{>-2}=A/A_{-2}$, so that we have a
commutative diagram of exact sequences
\begin{equation}\label{eqpan}
\begin{CD}
&&&&0 &&0\\
&&&&@VVV @VVV \\
0@>>> A_{-2}@>j>> A_{\le -1}@>\pi>> A_{-1}@>>> 0\\
&& ||&& @V{\tilde i}VV @V{i}VV\\
0@>>> A_{-2}@>\tilde j>> A@>\tilde\pi>> A_{>-2}@>>> 0\\
&&&&@V{\tilde\varpi}VV @V{\varpi}VV\\
&&&& A_0@= A_0\\
&&&& @VVV @VVV\\
&&&& 0 && 0.
\end{CD}
\end{equation}

From the viewpoint of \cite[Exp. IX, \S 9.3]{sga7}, this displays $A$ as
an \emph{extension panach\'ee}\footnote{L. Breen suggested
\emph{blended extension} for an English translation: we learned this from D. Bertrand.} of
$A_{>-2}$ by $A_{\le -1}$. We shall call the datum of this diagram minus
$A$ a
\emph{panachage datum} (donn\'ee de panachage). There is an obvious notion of morphism of
panachage data.

The following is a variant of \cite[\S 1]{bertrand}.

\begin{lemma} Given a panachage datum as above,
\begin{enumerate}
\item An extension panach\'ee exists if and only if the Yoneda product of the extensions
$A_{\le -1}$ and $A_{>-2}$ is $0$.
\item  An extension panach\'ee of $A_{>-2}$ by
$A_{\le -1}$ corresponds up to isomorphism to:
\begin{thlist}
\item an extension class $\alpha\in \Ext^1_\cA(A_0,A_{\le -1})$;
\item an extension class $\tilde\alpha\in \Ext^1_\cA(A_{>-2},A_{-2})$;\\
such that
\item $\pi_*\alpha = [A_{>-2}]$, $i^*\tilde\alpha=[A_{\le -1}]$.
\end{thlist}
\item Let $A,A'$ be two extensions panach\'ees, with classes $(\alpha,\tilde\alpha)$,
$(\alpha',\tilde\alpha')$ as in (1). Then $\alpha-\alpha'$ and $\tilde\alpha-\tilde\alpha'$
come from classes
$\gamma,\delta\in \Ext^1_\cA(A_0,A_{-2})$, well-defined modulo the images of
$\Hom_\cA(A_0,A_{-1})$ and $\Hom_\cA(A_{-1},A_{-2})$ respectively. Moreover, 
\[\gamma=\delta\in\coker\left(\Hom_\cA(A_0,A_{-1})\oplus \Hom_\cA(A_{-1},A_{-2})\to
\Ext^1_\cA(A_0,A_{-2})\right).\]
\end{enumerate}
\end{lemma}

\begin{proof} (1) is \cite[Exp. IX, 9.3.8 c)]{sga7}. (2) is obvious as well as the existence of
$\gamma$ and
$\delta$ in (3). The equality can be proven as in \cite[lemme 2]{bertrand}. Another way to
prove it is to consider the diagram
\[\begin{CD}
&& \Hom_\cA(A_{-1}, A_{-2})\\
&&@VVV\\
\Hom_\cA(A_0,A_{-1})@>>> \Ext^1_\cA(A_0,A_{-2})@>j_*>> \Ext^1_\cA(A_0,A_{\le -1})\\
&&@V{\varpi^*}VV @V{\varpi^*}VV\\
&& \Ext^1_\cA(A_{>-2},A_{-2})@>j_*>> \Ext^1_\cA(A_{>-2},A_{\le -1})
\end{CD}\]
and to apply a variant of \cite[Lemma 2.8]{bs}.
\end{proof}

\subsubsection{The case of $3$-step filtrations} Let $\cA$ still be an abelian category, and
consider now two filtered objects
\begin{gather*}
0\subseteq A_{-2}\subseteq A_{\le -1}\subseteq A\\
0\subseteq B_{-2}\subseteq B_{\le -1}\subseteq B.
\end{gather*}

\begin{defn}
Suppose given morphisms $f_i:A_i\to B_i$. A \emph{glueing} of the $f_i$ is a
filtered morphism
$f:A\to B$ inducing the $f_i$ on the associated graded. A \emph{partial glueing} is a morphism
of panachage data inducing the $f_i$.
\end{defn}

We want to find a condition for the $f_i$ to glue. We shall make the following simplifying
hypothesis:

\begin{hyp}\label{qvan}
$\Hom_\cA(A_i,B_j)=0$ for $i>j$.
\end{hyp}

\begin{propose}\label{pglue1}\
\begin{enumerate}
\item If a glueing exists, it is unique.
\item A necessary condition is that Condition (1) 
of Lemma \ref{lglue0} is satisfied for
the pairs $(f_{-2},f_{-1})$ and $(f_{-1},f_0)$, relative to the extensions $A_{\le -1},B_{\le
-1}$ and
$A_{>-2},B_{>-2}$.
\item Suppose the conditions of (2) are satisfied. Then we get a (unique) partial glueing.
\item Given a morphism of panachage data $(f_{-2},f_{-1},f_0,f_{\le -1}, f_{>-2})$, we have two
obstructions in
$\Ext^1_\cA(A_0,B_{-2})$ to the existence of a glueing, given respectively (with obvious
notation) by
$(f_{\le -1})_*[A]-f_0^*[B]$ and $(f_{-2})_*[A]-f_{>-2}^*[B]$. 
\item The glueing exists if and only if either of the two obstructions of (4) vanishes.
\item If all $f_n$ are isomorphisms, so is $f$. 
\end{enumerate}
\end{propose}

\begin{proof} (1) is easy, (2)  and (3) are clear.

For (4), we have to check that these extension classes,
which are a priori respectively in $\Ext^1_\cA(A_0,B_{\le -1})$ and in
$\Ext^1_\cA(A_{>-2},B_{-2})$, define unique elements of $\Ext^1_\cA(A_0,B_{-2})$. For the
existence, it suffices to see that their respective images in $\Ext^1_\cA(A_{0},B_{-1})$ and $\Ext^1_\cA(A_{-1},B_{-2})$ are $0$, which
follows from the conditions of (2), while the uniqueness follows from
\ref{qvan}. 

(5) follows from Lemma \ref{lglue0}. (6) is clear.
\end{proof}

\begin{remark} It is likely that the two obstructions of Proposition \ref{pglue1} (4) are
opposite, but we don't have a good proof.
\end{remark}

\subsubsection{Right adjoints}\label{sD.4}
We let $\cA$ be an abelian category. Consider the situation 
\begin{equation}\label{eqabsetup}
0\to \cA'\by{i} \cA\by{\pi} \cA''\to 0
\end{equation}
where $\cA'\subseteq \cA$ is a  Serre subcategory and $\pi$ is the corresponding
localisation functor. Thus $\cA''$ is the Serre quotient of $\cA$ by
$\cA'$. Note that $i$ and $\pi$ are exact. 

Let $A\in \cA$. Recall that, by definition, the right adjoint $p$ of $i$ is \emph{defined at
$A$} if the functor
\[\cA'\ni B\mapsto \Hom_\cA(iB,A)\]
is representable. A representative object is then unique up to unique isomorphism: we write it
$pA$ and call it the \emph{value} of $p$ at $A$. We then have a canonical ``counit" map
\[ipA\to A\]
given by the universal property of $pA$.

Define the following full subcategories of $\cA$:
\begin{align*}
\cA_1&=\{A\in \cA\mid p \text{ is defined at } A\}\\
\cB&=\{B\in \cA_1\mid pB=0\}=\ker p\\
\cA_2&=\{A\in \cA\mid \exists \text{ exact sequence } 0\to iA'\to A\to A''\to 0\\
& \quad \text{ with }
A'\in \cA', A''\in \cB\}.
\end{align*}

\begin{lemma} \label{psplit1}\ \begin{enumerate}
\item For $A\in \cA_1$, the counit map $f:ipA\to A$ is a monomorphism.
\item $\cB=\{B\in \cA\mid \Hom_\cA(i\cA',B)=0\}$. In particular, $i\cA'\cap \cB=0$.
\item $\cA_2\subseteq \cA_1$, with $A'=pA$ for $A\in \cA_2$. In particular, 
$A'$ and $A''$ are unique and functorial. Moreover, the functor $A\mapsto A''$  from $\cA_2$ to $\cB$ is left adjoint to the inclusion $\cB\into \cA_2$.
\end{enumerate}
\end{lemma}

\begin{proof} (1) Since $\cA'$ is  Serre in $\cA$, $\ker f=i(K)$ for some $K\in \cA'$. Thus
we have an exact sequence 
\[0\to i K\by{i(\alpha)} ip A\to A\]
for some $\alpha:K\to pA$ (full faithfulness of $i$) But $\alpha$ is adjoint to the $0$ map $iK\to A$, hence is $0$, showing that  $K=0$. Thus $f$ is a monomorphism. 

(2) is obvious by adjunction. For ``in particular", if $C\in i\cA'\cap \cB$, then $1_C\in
\Hom_\cA(C,C)=0$.

(3) Let $A_1\in \cA'$. Then we have a short exact sequence
\[0\to \Hom_\cA(iA'_1,iA')\to \Hom_\cA(iA'_1,A)\to \Hom_\cA(iA'_1,A'')\]
in which the last term is $0$ by (2). This shows that $A'=pA$ and the uniqueness of $A'$, hence of $A''$; the last claim is obvious from the definition of $\cB$.
\end{proof}

Now let us introduce further full subcategories of $\cA$:
\begin{align*}
\cB'&=\{B\in \cB\mid \Hom_\cA(B,i\cA')=0\}\\
&=\{B\in \cA\mid \Hom_\cA(i\cA',B)=\Hom_\cA(B,i\cA')=0\}\\
\cA_3&=\{A\in \cA_2\mid A''\in \cB'\}.
\end{align*}

\begin{propose}\label{lsplit}  In the situation of \eqref{eqabsetup}, suppose that  $i$ has an
everywhere defined right adjoint $p$.
 Then the following conditions are equivalent:
\begin{thlist}
\item $p$ is exact.
\item The full subcategory $\ker p=\{A\mid pA=0\}$ is stable under quotients and contains
$A/ipA$ for any $A\in
\cA$.
\item  $\cA_3=\cA$.
\end{thlist}
When this happens, $\cB'$ is a Serre subcategory of $\cA$, the inclusion $\iota:\cB'\into \cA$ has an exact left adjoint $q$, any object $A\in \cA$ sits in a functorial exact sequence
\[0\to ip A\to A \to \iota q A\to 0\]
 and morphisms in $\cA$ are strictly compatible with this filtration in the sense of \cite[(1.1.5)]{D2}. Moreover, the composition
 \[r:\cB'\longby{\iota}\cA \longby{\pi} \cA''\]
 is an equivalence of categories.
\end{propose}

\begin{proof}  (i) $\Rightarrow$ (ii): Applying $p$ to the exact sequence
\[0\to ip A\to A\to A/ipA\to 0\]
we get an exact sequence
\[0\to p A\by{=} pA\to p(A/ipA)\to 0\]
which shows that $p(A/ipA)= 0$. Moreover, $\ker p$ is clearly stable under quotients.

(ii) $\Rightarrow$ (iii): by hypothesis we have $\cA_2=\cA$,
so we are left to show that
$\Hom_\cA(\ker p,i\cA')=0$. Let $B\in \ker p$, $A'\in \cA'$, $f:B\to A'$ and $C=\im f$. By (ii),
$C\in\ker p$; since $i(\cA')$ is  Serre, we also have $C\in \cA'$; hence $C=0$ thanks to
Lemma \ref{psplit1} (2).

(iii) $\Rightarrow$ (i): we first show that $\cB'$ is closed under subobjects. Let $B\in \cB'$: if $B_1\subseteq B$, write $B_1$ as an extension of $B'_1$ by $ip B_1$, with $B'_1\in \cB'$. Then the injection $ipB_1\into B_1\into B$ is $0$ by adjunction, hence $ipB_1=0$ and $B_1\in \cB'$.

Now we show that $\cB'$ is closed under quotients. Let $f:B\onto B_2$ be an epimorphism; write $B_2$ as an extension of $B'_2\in \cB'$ by $ipB_2$. Let $\tilde B=f^{-1}(ip B_2)\subseteq B$; then $\tilde B\in \cB'$, hence the epimorphism $\tilde B\to ip B_2$ is $0$ and $ipB_2=0$, so that $B_2\in \cB'$.

 We now prove the exactness of $p$. We already know that it is left exact, as a right adjoint. Let $0\to A_1\to A_2\to A_3\to 0$ be an exact sequence in $\cA$. Consider the commutative diagram with exact rows
\begin{equation}\label{bigdiagram}
\begin{CD}
&&0&&0 &&0\\
&& @VVV @VVV @VVV\\
0@>>> ipA_1@>>> A_1@>>>  A'_1@>>> 0\\
&& @VVV @VVV @VVV\\
0@>>> ipA_2@>>> A_2@>>> A'_2@>>> 0\\
&& @VVV @VVV @VVV\\
0@>>> ipA_3@>>> A_3@>>>  A'_3@>>> 0\\
&& @VVV @VVV @VVV\\
&&0&& 0 && 0
\end{CD}
\end{equation}
with $A'_1,A'_2,A'_3\in \cB'$.  Viewing this as a short exact sequence of vertical complexes $C'\to C \to C''$ with $C$ acyclic, the long homology exact sequence yields an isomorphism
\[H^1(C'')\iso H^2(C').\]

But $H^1(C'')\in \cB'$, hence $H^2(C')=0$ by Lemma \ref{psplit1} (2).

We have seen while proving (iii) $\Rightarrow$ (i) that $\cB'$ is stable under subobjects and quotients; clearly, it is also stable under extensions and thus is a Serre subcategory of $\cA$. Using  (iii), 
we get the functor $q$, given by $qA=A/ipA$. The exactness of $i$ and $p$ plus a chase in \eqref{bigdiagram} show that $q$ is exact, and $q$ is easily seen to be left adjoint to $\iota$. The strict compatibility assertion is now easily verified.

 Since $q_{|\cA'}=0$, $q$ induces a functor $\bar q:\cA''\to \cB'$. 
Applying $\pi$ to the adjunction between $\iota$ and $q$ then yields two natural isomorphisms $Id\overset{\sim}{\Rightarrow} r\bar q$, $\bar q r \overset{\sim}{\Rightarrow} Id$ (note that the unit map $A\to \iota q A$ is the projection $A\to A/ipA$), which shows the last claim of the proposition.\end{proof}

\subsubsection{Split exact sequences} We are still in the situation of \eqref{eqabsetup}.

\begin{propose}\label{psplit}\
\begin{enumerate}
\item The following conditions are equivalent:
\begin{thlist}
\item $i$ has an exact right adjoint $p$.
\item $\pi$ has an exact right adjoint $j$ and $j\cA''$ is   a Serre subcategory of $\cA$.
\item For any $A\in \cA$ there exists $A'\in \cA'$ and a monomorphism
  $iA'\to A$ such that 
\[\Hom_\cA (i\cA', A/i A')=\Ext^1_\cA(i\cA', A/i A') =0.\]
\item Same as {\rm (iii)}, replacing the condition  $\Ext^1_\cA(i\cA', A/i A')
\allowbreak=0$ by
$\Hom_\cA(A/iA',i\cA')=0$.
\end{thlist}
\item If these conditions are verified, then for any $A\in\cA$
there exists a unique exact sequence 
\[0\to iA'\to A\to jA''\to 0\]
such that $A'\in \cA'$ and $A''\in \cA''$; this sequence is functorial in
$A$. Moreover, we have $\Hom_\cA(j\cA'',i\cA')=0$ and 
$\Ext^r_\cA(i\cA',j\cA'')=0$ for all $r\ge 0$. 
\item If $\cA$ is semi-simple, the conditions of {\rm (1)} are verified.
\end{enumerate}
\end{propose} 

\begin{proof} (1) We shall prove (i) $\Rightarrow$ (ii)  $\Rightarrow$ (iii) $\Rightarrow$ (i) and (i) $\Rightarrow$ (iv) $\Rightarrow$ (i).

(i) $\Rightarrow$ (ii): this follows from Proposition \ref{lsplit}. (With the notation in its proof, we have $j=\iota \bar q$.)

(ii) $\implies$ (iii): the functor $\pi$ kills the kernel
and cokernel of the unit map $f:A\to j\pi A$. By the thickness assumption,
$\coker f\in j\cA''$ which implies that $f$ is epi; on the other hand,
we get $\ker f\in i\cA'$. Thus we have an exact sequence 
\[0\to iA'\to A\to j\pi A\to 0.\]

If $B\in \cA'$, then 
\[\Hom(iB,j\pi A)= \Hom(\pi i B,\pi A)=\Hom(0,\pi A)=0.\]

On the other hand, let $0\to j\pi A\by{\alpha} E\to i B\to 0$ be an
extension. Since $\pi i=0$ and $\pi$ is exact, $\pi(\alpha)$ is an
isomorphism. Its inverse $\pi(\alpha)^{-1}:\pi E\to \pi A$ yields by adjunction a
map $\beta:E\to j\pi A$ and one verifies that $\beta$ is a retraction of $\alpha$. 

(iii) $\implies$ (i): let $A,B\in \cA$, with $iA'\to A$, $iB'\to B$ as in (iii),
and let $f:A\to B$ be a morphism. By assumption, the composition  
\[i A'\to A\by{f} B\to B/i B'\]
is $0$, hence $f$ induces a map $A'\to B'$. This shows that $A'$ is
unique up to unique isomorphism (take $f=1_A$) and is functorial in
$A$. Thus we have a functor $p:\cA\to \cA'$ mapping $A$ to $A'$ plus a natural monomorphism $\eta_A:ipA\into A$. If $A=iA_1$, then $A/ipA=0$ because $1_{A/ipA}=0$. Since $i$ is fully faithful, this provides a natural isomorphism $\epsilon_{A_1}:piA_1\iso A_1$, and one checks that $(\eta,\epsilon)$ makes $p$ a right adjoint to $i$. 

It remains to see that $p$ is right exact. For this, it suffices to show the criterion (ii) of
Proposition 
\ref{lsplit}. Let $A\in \cA$: then $ip(A/ipA)$ is a subobject of $A/ipA$, hence is $0$ by the
hypothesis $\Hom(i\cA',A/ipA)=0$. Thus $A/ipA\in\ker p$. Let now $B\in \ker p$ and $C$ be a
quotient of $B$ so that we have an exact sequence
\begin{equation}\label{eqsplit}
0\to A\to B\to C\to 0.
\end{equation}

Note that $A\in \ker p$ by left exactness of $p$, hence $A\iso A/ipA$. Pulling back the extension \eqref{eqsplit} via the monomorphism
$ipC\to C$, we get an extension of $ipC$ by $A$, which is split by hypothesis. A given
splitting yields a monomorphism $f:ipC\into B$, which is $0$ again by hypothesis, and finally
$C\in \ker p$.

(i) $\implies$ (iv): this follows from (i) $\implies$ (iii) in Proposition \ref{lsplit}.

(iv) $\implies$ (i):  the condition $\Hom_\cA (i\cA', A/i (A'))=0$ implies
that $p$ is defined at $A$ with value $A'$, as in the proof of (iii) $\Rightarrow$ (i). Then (iv) is equivalent to Condition (iii) in Proposition \ref{lsplit}.

(2) Everything follows from (1), except the
vanishing of the higher Ext's; the claimed exact sequence is given by
\begin{equation}\label{eqglue}
0\to ip A\to A\to j\pi A\to 0.
\end{equation}

For $\Ext^r$ with $r>0$, we argue by induction on $r$. This
reduces us to show that, for all
$r>0$ and all $A'\in \cA'$, the functor
\[\cA''\ni A''\mapsto \Ext^r_\cA(iA',jA'')
\]
is effaceable.

Let $\cE\in \Ext^r_\cA(iA',jA'')$, represented by the exact sequence 
\[0\to jA''\by{f} E_1\to\dots\to E_r\to iA'\to 0.\]

Clearly, $f_*\cE=0\in \Ext^r_\cA(iA',E_1)$\footnote{To see this, let $F=\coker f$, and note that $\cE$
is the image of $[0\to F\to E_2\to\dots\to E_r\to iA'\to 0] \in
\Ext^{r-1}_\cA(iA',F)$ by the boundary map associated to the exact sequence $0\to jA''\to E_1\to
F\to 0$.}. A fortiori, $\cE$ maps to $0$ under the composition
\[\Ext^r_\cA(iA',jA'')\by{f_*} \Ext^r_\cA(iA',E_1)\to \Ext^r_\cA(iA',j\pi E_1).\]

It remains to observe that the composition $jA''\to E_1\to j\pi E_1$ is a monomorphism: indeed,
it is the image of the monomorphism $jA''\to E_1$ under the exact functor $j\pi$.

(3) Let $I$ be the set of isomorphism classes of simple objects of $\cA$, and $J$ the subset of $I$ defined by simple objects of $\cA$ belonging to $\cA'$. For any $A\in \cA$, we have a direct sum decomposition 
\[A=A_J\oplus A_{I-J}\]
where $A_J$ (\resp $A_{I-J}$) is the sum of simple subobjects of $A$ whose class belongs to $J$ (\resp to $I-J$). Clearly, $A_J\in \cA'$, and $\Hom(A_J,A_{I-J})=0$. Also, $\Ext^1(A_J,A_{I-J})=0$ since $\cA$ is semi-simple. Thus Condition (iii) of (1) is satisfied.
\end{proof}

\begin{defn}\label{dsplit} In the situation of Proposition
  \ref{psplit}, we say that the exact sequence \eqref{eqabsetup} is
  \emph{split}. 
\end{defn}

We shall need the following lemma:

\begin{lemma}\label{lnat} In the situation of Proposition \ref{psplit}, let $\cB$ be an additive
category and let $E:\cA''\to \cB$, $F:\cA'\to \cB$ be two additive functors. Then a natural
transformation $E\pi \Rightarrow Fp$ on $\cA$ is equivalent to a bivariant transformation
\[\Ext^1_\cA(jA'',iA')\to \Hom_\cB(E(A''),F(A'))
\]
on $\cA''\times \cA'$.
\end{lemma}

\begin{proof} Let $u_A:E\pi A\to FpA$ be a natural transformation. For $(A',A'')\allowbreak\in
\cA'\times
\cA''$ and two extensions $A_1,A_2$ of $jA''$ by $iA'$ in $\cA$, the Baer sum $A_1\boxplus A_2$
may be obtained as $\Delta^*\Sigma_*(A_1\oplus A_2)$, where $\Delta:jA''\to jA''\oplus jA''$ is
the diagonal map and $\Sigma:iA'\oplus iA'\to iA'$ is the sum map. Since $u_{A_1\oplus
A_2}=u_{A_1}\oplus u_{A_2}$, this implies that
\[u_{A_1\boxplus A_2} = u_{A_1}+u_{A_2}\]
so that $u$ induces a homomorphism as in the lemma. The converse will not be used and is left
to the interested reader.
\end{proof}

\subsubsection{More adjoints} Assume that \eqref{eqabsetup} is
split. Let us now consider the conditions

\begin{quote} ($\flat$) {\it $j$ has a right adjoint $\pi'$.} \\
($\sharp$) {\it $i$ has a left adjoint $p'$.}
\end{quote}

Recall that, under ($\flat$), we get a canonical natural transformation $\alpha:\pi'\Rightarrow \pi$
from the composition $j\pi'A\to A\to j\pi A$ for any $A\in \cA$, and the full
faithfulness of $j$. Similarly, under ($\sharp$), we get a canonical natural
transformation $\beta:p\Rightarrow p'$.

\begin{propose}\label{pstreng} 
Under ($\flat$), $\pi'i=0$ and $\alpha$ is a monomorphism.\\
Under ($\sharp$),  $p'j=0$ and $\beta$ is an  epimorphism.\\
If  $\cA''$ is Noetherian (\resp $\cA'$
is Artinian), then {\rm ($\flat$)} (\resp {\rm ($\sharp$)}) holds.
\end{propose}

\begin{proof} Suppose that $\pi'$ exists. Since $\Hom_\cA(j\cA'',i\cA')=0$ (Proposition \ref{psplit} (2)), we have $\Hom_{\cA''}(\cA'',\pi' i\cA')=0$, \ie $\pi'i=0$. Note that $\pi'$ is left exact, as
  a right adjoint. Let now  $A\in\cA$. Applying $\pi'$ to the exact sequence \eqref{eqglue}, we get
  an exact sequence
\[0\to 0=\pi'ipA\to \pi' A\by{\alpha_A}\pi A.\]

This shows that $\alpha_A$ is injective. Dually, if $p'$ exists then  $p'j=0$ and $\beta$ is an epimorphism. 

It remains to prove the last assertions. Let $A\in \cA$. For a subobject
$B$ of $\pi A$, consider the pull-back $A'\by{\phi} jB$ of the map $A\to j\pi
A$. We define $\pi'A$ as the largest $B$ such that $\phi$ is split:
clearly, $\pi' A$ is functorial in $A$. Moreover, a splitting of $\phi$ is
unique (still by Proposition \ref{psplit} (2)), hence defines a natural transformation
\[\epsilon_A:j\pi' A\to A'\to A.\]

On the other hand, $\pi'j$ is clearly the identity functor; define then for
$A''\in\cA''$ a unit map $\eta_{A''}:A''\to \pi'j A''$  as the identity
map. The adjunction identities are readily checked. The case of ($\sharp$)
is dealt with dually.
\end{proof}

\subsubsection{Filtrations}

\begin{defn}\label{d16.3} Let $\cA$ be an abelian category.\\
a) A \emph{filtration} on $\cA$ is a sequence of  Serre subcategories
\begin{equation}\label{eqE.6.1}
\dots \longby{\iota_{n-1}}\cA_{\le n}\longby{\iota_n}\cA_{\le
n+1}\longby{\iota_{n+1}}\dots\to \cA.
\end{equation} 
It is \emph{separated} if $\bigcap \cA_{\le n}=\{0\}$ and \emph{exhaustive} if $\cA=\bigcup
\cA_{\le n}$.\\  
b) A filtration on $\cA$ is \emph{split} if all $\iota_n$ have exact right adjoints.\\  
\end{defn}

\begin{defn} \label{dE.6} A filtration on $\cA$ is a \emph{weight filtration} if it is
separated, exhaustive and split.
\end{defn}

Let $\cA$ be provided with a split filtration (Definition \ref{d16.3}). Define
\[\cA_n = \cA_{\le n}/\cA_{\le n-1},\qquad \cA_{n,n+1} = \cA_{\le n+1}/\cA_{\le n-1}\]
so that we have exact sequences
\begin{equation}\label{eqE.6.2}
0\to \cA_n\longby{i_n} \cA_{n,n+1}\longby{\pi_n} \cA_{n+1}\to 0
\end{equation}
where $i_n$ is induced by $\iota_n$. Let us write $j_n$ for the right adjoint of the
localisation functor $\cA_{\le n}\to \cA_{\le n}/\cA_{\le n-1}$ (Proposition \ref{psplit}). By
abuse of notation, we shall identify $\cA_n$ with its thick  image in $\cA$ via $j_n$ 
\text{(\ibid).}

\begin{lemma} 
If $\iota_n$ \eqref{eqE.6.1} has a right (\resp left) adjoint, so does $i_n$ \eqref{eqE.6.2}. If
one is exact, so is the other.
\end{lemma}

\begin{proof} We have a diagram of exact sequences of abelian categories:
\[\begin{CD}
&& \cA_{\le n-1} @= \cA_{\le n-1}\\
&& @V{\iota_{n-1}}VV @V{\iota_n\iota_{n-1}}VV\\
0@>>> \cA_{\le n} @>\iota_n>> \cA_{\le n+1}@>\tilde\pi_n>> \cA_{n+1}@>>> 0\\
&& @V{\tilde \pi_{n-1}}VV @V{q_{n-1}}VV @V{||}VV\\
0@>>> \cA_n @>i_n>> \cA_{n, n+1}@>\pi_n>> \cA_{n+1}@>>> 0.
\end{CD}\]

Let $\varpi_n$ be a right adjoint to $\iota_n$. Since $\iota_n$ is fully faithful, the unit
map $\eta_n:Id_{\cA_{\le n}}\Rightarrow \varpi_n\iota_n$ is an isomorphism. Hence
$\varpi_n(\iota_{n-1}\cA_{\le n-1})\subseteq \iota_{n-1}\cA_{\le n-1}$ and $\varpi_n$ induces a
functor $p_n:\cA_{n,n+1}\to \cA_n$. Let $\epsilon_n:\iota_n\varpi_n\Rightarrow Id_{\cA_{\le
n+1}}$ be the counit map of the adjunction: then $\eta_n$ and $\epsilon_n$ induce natural
transformations $p_{n-1}*\eta_n:Id_{\cA_n}\Rightarrow p_ni_n$ and
$q_{n-1}*\epsilon_n:i_np_n\Rightarrow Id_{\cA_{n,n+1}}$. Since $\eta_n$ and $\epsilon_n$ verify
the identities of \cite[p. 82, Th. 1 (8)]{mcl2}, these identities are preserved when applying
$p_{n-1}$ and $q_{n-1}$, hence $p_{n-1}*\eta_n$ and $q_{n-1}*\epsilon_n$ define an adjunction
between $i_n$ and $p_n$ by {\it ibid.}, p. 83, Th. 2 (v). (Alternately, one can check directly
that $p_n$ is right adjoint to $i_n$.) The assertion on exactness is true because the functor
$q_{n-1}$ is exact. The reasoning is the same for a left adjoint.
\end{proof}

\begin{propose} \label{proadj}
Suppose that $\cA$ is provided with an exhaustive split filtration as in Definition
\ref{d16.3} a), b).  Then:
\begin{enumerate}
\item Every object $A\in \cA$ is provided with a unique filtration
\[\dots\subseteq A_{\le n}\subseteq A_{\le n+1}\subseteq\dots \subseteq A\]
with $A_{\le n}\in \cA_{\le n}$ and $A_n:=A_{\le n}/A_{\le n-1}\in \cA_n$. One has  $A_{\le
n}=A$ for $n$ large enough.
\item The functors $A\mapsto A_{\le n}$ and $A\mapsto A_n$ are exact, as well as $A\mapsto
A_{\ge n}:=A/A_{\le n-1}$. 
\item If moreover the filtration is separated, the $A\mapsto A_n$ form a set of conservative
functors.  
\item We have $\Hom_\cA(\cA_n,\cA_{\le n-1})=0$  and $\Ext_\cA^i(\cA_{\le n-1},\cA_n)=0$ for all $i\ge
0$ for all $n\in\Z$ and $i\ge 0$.
\end{enumerate}
\end{propose}

\begin{proof} (1) Since $\cA=\bigcup \cA_{\le n}$, there exists $n_0$ such that $A\in \cA_{\le
n_0}$. For $n<n_0$, write $I_n$ for the inclusion $\cA_{\le n}\to \cA_{\le n_0}$ and $P_n$ for
its right adjoint, which exists and is exact as a composition of the right adjoints of the
$\iota_r$ for $n\le r<n_0$. Define
\[A_{\le n} = I_nP_n A.\]

Since $A=\iota_{n_0}\varpi_{n_0} A$, $A_{\le n}$ does not depend on the choice of $n_0$, and
we have a filtration of $A$ as in Proposition \ref{proadj}. Clearly, $A_{\le n}\in \cA_{\le n}$
and $A_{\le n_0} =A$. The fact that $A_n\in \cA_n$ follows from Proposition \ref{psplit} (2),
which also shows the uniqueness of the filtration. 

(2) This still follows from Proposition \ref{psplit} (2).

(3) Let $f:A\to B$ be such that $f_n:A_n\to B_n$ is an isomorphism for all $n\in\Z$. Let
$K=\ker f$ and $C=\coker f$. By exactness, $K_n = C_n = 0$ for all $n$. Thus $K,C\in \bigcap
\cA_{\le n} = 0$ and $f$ is an isomorphism.

(4) This follows from Proposition \ref{psplit} (2) again. 
\end{proof}

\begin{lemma}\label{lF.3.4} Assume given a weight filtration $\cA_{\le n}$ on $\cA$ and assume 
that the categories $\cA_{n}$
are semi-simple. Let $A\in \cA_m$ and $B\in \cA_n$. Then, 
$\Ext^i_\cA(A,B) = 0$ if $i > m-n$.
\end{lemma}

\begin{proof} By induction on $i$. For $i=0$, it means that 
$\Hom_\cA(A,B) = 0$ if $m<n$, which
is true by \ref{proadj} (4). For $i=1$, it means that $\Ext^1_\cA(A,B) = 0$ 
if $m\le n$: for $m<n$
this is still \ref{proadj} (4) and for $m=n$ it comes from the 
semi-simplicity of $\cA_n$ and its
thickness (in Serre's sense) in $\cA$ (\ref{psplit} (1) (ii)).

Suppose now $i>1$ and $m-n<i$. Let $\alpha\in
\Ext^i_\cA(A,B)$. We may write
$\alpha$ as a Yoneda product
\[\alpha = \gamma\beta\]
with $\beta\in \Ext^1_\cA(A,C)$ and $\gamma\in \Ext^{i-1}_\cA(C,B)$ 
for some $C\in \cA$. By
induction, the map
\[\Ext^{i-1}_\cA(C_{\ge n+i-1},B)\to \Ext^{i-1}_\cA(C,B)\]
is surjective, hence we may assume $C\in \cA_{\ge n+i-1}$. But then
$\Ext^1_\cA(A,C)\allowbreak=0$ and we are done.
\end{proof}

\begin{defn}\label{dflength} Let $\cA$ be provided with an exhaustive split filtration. Let
$A\in \cA$. We say that $A$ \emph{has finite length} if $A_{\le n}=0$ for $n$ small enough.
The length of $A$ is then the integer
\[\ell(A)=n-m\]
where $n$ is the smallest integer such that $A_{\le n}=A$ and $m$ is the largest integer such
that $A_{\le m} = 0$.
\end{defn}

\begin{lemma}\label{lF.3.5} Let $\cA$ be provided with a weight filtration
$(\cA_{\le n},\iota_n)$. Let $A,B\in
\cA$, with $B$ of finite length. Let $f:A\to B$ be a morphism such that $f_n:A_n\to B_n$
is $0$ for all $n$. Then $f=0$.
\end{lemma}

\begin{proof} The assumption implies that $f(A_{\le n})\subseteq B_{\le n-1}$ for
all
$n\in
\Z$. Hence $f$ induces morphisms $f_n^{(1)}:A_n\to B_{n-1}$, which are $0$ by Proposition
\ref{proadj} (4). Inductively on $k$, we get $f(A_{\le n})\subseteq B_{\le n-k}$ for all
$n\in\Z$ and all $k\ge 0$. For $n$ large enough we have $A_{\le n}=A$ and for $k$ large enough
we have $B_{\le n-k}=0$, hence $f=0$.
\end{proof}

\begin{remark}\label{lF.3.6} If $\cA$ is provided with a weight filtration such that every
object has finite length, then the functors $A\mapsto A_{\le n}$ of Proposition \ref{proadj}
(1) verify Jannsen's conditions in  \cite[p. 83, Def. 6.3 a)]{jannsen}. Conversely, let as in
\loccit $(W_n)_{n\in\Z}$ be an increasing sequence of exact subfunctors of $Id_\cA$ such that,
for all $A\in \cA$, one has $W_n A=0$ for $n \ll 0$ and $W_nA=A$ for $n \gg 0$. Define
$\cA_{\le n}$ as the full subcategory of $\cA$ consisting of objects $A$ such that $W_n A=A$.
This filtration is clearly separated and exhaustive in the sense of Definition \ref{d16.3}.
Moreover, for all $n\in \Z$, $\iota_n: \cA_{\le n}\into \cA_{\le n+1}$ has the exact right
adjoint $W_n$, so our filtration is also split. Summarising, the datum of a weight filtration
as in \cite[p. 83, Def. 6.3 a)]{jannsen} is equivalent to that of a weight filtration in the
sense of Definition \ref{dE.6} for which every object has finite length.
\end{remark}

In view of this remark, the following is an abstract version of \cite[p. 87, Ex. 6.8]{jannsen}: 

\begin{propose} Let $\cA$ be provided with a filtration $(\cA_{\le n})_{n\in\Z}$. For each
$n\in\Z$, let $\cA'_n$ be a full subcategory of $\cA_{\le n}$ such that
\begin{thlist}
\item For all $n\in\Z$, $\cA'_n$ is abelian.
\item For $m\ne n$, $\Hom_\cA(\cA'_m,\cA'_n)=0$.
\end{thlist}
Let $\cA'$ be the full subcategory of $\cA$ consisting of objects $A$ admitting a finite 
increasing filtration $(A_{\le n})$ with 
\begin{thlist}
\item $A_{\le n}=0$ for $n\ll 0$.
\item $A_{\le n}=A$ for $n\gg 0$.
\item $A_{\le n}/A_{\le n-1}\in \cA'_n$ for all $n\in\Z$.
\end{thlist}
Then $\cA'$ is an abelian subcategory of $\cA$, closed under subobjects and quotients. The
above filtration is unique for any $A\in \cA'$, and any morphism is strict. If $\cA'_{\le
n}:=\cA'\cap \cA_{\le n}$, then the filtration $\cA'_{\le n}$ is a weight filtration and any
object of $\cA'$ has finite length.
\end{propose} 

\begin{proof} In principle this follows from  Proposition \ref{lsplit} by induction,
but it would be tedious to write up (the induction would have to bear on the length of the
filtration of an object). Instead, we observe that the proof in
\cite[p. 88, Lemma 6.8.1]{jannsen} shows that
$\cA'$ is provided with a weight filtration in the sense of \cite[p. 83, Def. 6.3 a)]{jannsen},
and we apply Remark
\ref{lF.3.6}.
\end{proof}

\subsection{Morphisms of filtered categories}

\subsubsection{The case of a $2$-step filtration}

\begin{propose}\label{pglue} Let $0\to \cA'\by{i} \cA\by{\pi} \cA''\to 0$ and
  $0\to \cB'\by{i'} \cB\by{\pi'} \cB''\to 0$ be as in \eqref{eqabsetup}, and
  consider a naturally commutative diagram of exact functors 
\[\begin{CD}
0@>>> \cA'@>{i}>> \cA@>{\pi}>> \cA''@>>> 0\\
&&@V{R'}VV @V{R}VV @V{R''}VV\\
0@>>> \cB'@>{i'}>> \cB@>{\pi'}>> \cB''@>>> 0.
\end{CD}\]
Suppose that the two rows are split in the sense of Definition
  \ref{dsplit}; we use the notation $(p,j)$ (\resp $(p',j')$) for the
  corresponding adjoints. Then the following are equivalent: 
\begin{thlist}
\item The natural ``base change" transformation $R'p\Rightarrow p'R$ of
  \S \ref{sbase} is a natural isomorphism. 
\item   $Rj\Rightarrow j'R''$ is a natural isomorphism. 
\item  $R(j\cA'')\subseteq j'\cB''$.
\end{thlist}
\end{propose}

\begin{proof} Let $A\in \cA$. We have a commutative diagram of exact sequences
(using \eqref{eqglue} and the base change morphisms): 
\[\begin{CD}
0@>>> R(ipA)@>>> R(A)@>>> R(j\pi A)@>>> 0\\
&& @VVV @V{||}VV @VVV \\
0@>>> i'p'R(A)@>>> R(A)@>>> j'\pi'R(A)@>>> 0.
\end{CD}\]

Thus the left vertical map is an isomorphism if and only if the right
vertical map is one, which shows that (i) $\iff$ (ii). If this is the
case, then $R(j\pi A)\in j'\cB''$, hence (iii). Conversely, if (iii)
holds, then all vertical maps must be isomorphisms by the uniqueness
of \eqref{eqglue}. 
\end{proof}

\subsubsection{The general case}

\begin{defn} \label{dF.4.4}
Let $\cA$ and $\cB$ two filtered abelian categories $(\cA_{\le n},\iota_n)$, $(\cB_{\le
n},\iota'_n)$ as in Definition \ref{d16.3}. Let  $R: \cA\to \cB$ be an additive functor.
\begin{enumerate}
\item  We say that $R$ \emph{respects the filtration} if it is exact and $R(\cA_{\le
n})\subseteq \cB_{\le n}$.
\item Suppose that the two filtrations are split. We say that $R$ \emph{respects the
splittings} if, moreover, $R(\cA_n)\subseteq \cB_n$ for all $n\in \Z$.
\item We denote by $R_{\le n} :\cA_{\le n}\to \cB_{\le n}$ the restriction of $R$ to $\cA_{\le
n}$ and by $R_n:\cA_n\to \cB_n$ its restriction to $\cA_n$, if applicable. 
 \end{enumerate}
 \end{defn}
 
\begin{lemma}\label{lF.4.5} Let $\cA$, $\cB$, $R$ be as in Definition \ref{dF.4.4}. Suppose
that the filtrations of $\cA$ and $\cB$ are exhaustive and split.  Then:
\begin{enumerate}
\item If every object of $\cA$ has  finite length in the sense of Definition \ref{dflength},
then $R$ respects the splittings as soon as $R(\cA_n)\subseteq \cB_n$ for all $n$.
\item If $R$ respects the splittings, then for all $n\in\Z$, $R_{\le n}$ and $R_n$ are exact. 
\item Let $A\in \cA$. Then $R(A_{\le n})= R(A)_{\le n}$ and $R(A_n)\iso R(A)_n$ for any
$n\in\Z$.
\end{enumerate}
\end{lemma}

\begin{proof} (1)  Let $A\in \cA$. We have to show that, if $A\in \cA_{\le n}$, then $R(A)\in
\cB_{\le n}$. We argue by induction on the length $\ell(A)$ of the weight filtration on $A$. If
$\ell(A)=0$, then $A=0$ this is clear. If $\ell(A)>0$, we may assume $n$ minimal.  Then
$\ell(A_{\le n-1})=\ell(A)-1$, hence $R(A_{\le n-1})\in \cB_{\le n-1}$, and the exact sequence
\[0\to R(A_{\le n-1})\to R(A)\to R(A_n)\to 0\]
with $R(A_n)\in \cB_n$ shows that $R(A)\in \cB_{\le n}$.

(2) This follows from the exactness of the inclusions $\iota_n$ and $j_n$, and from the
faithful exactness of the $\iota'_n$ and $j'_n$. 

(3) Let $A\in \cA_{\le n_0}$; then $R(A)\in \cB_{\le n_0}$. For $n<n_0$, let $I'_n$ be the
inclusion $\cB_{\le n}\to \cB_{n_0}$ and $P'_n$ its exact right adjoint. We have
\[R(A_{\le n}) = I'_n P'_n R(A_{\le n})\subseteq I'_n P'_n R(A)= R(A)_{\le n}.\]

For equality, consider the commutative diagram of exact sequences
\[\begin{CD}
0@>>> R(A_{\le n})@>>> R(A)@>>> R(A_{\ge n+1})@>>> 0\\
&& @V{a}VV @V{||}VV @V{b}VV\\
0@>>> R(A)_{\le n}@>>> R(A)@>>> R(A)_{\ge n+1}@>>> 0.
\end{CD}\]
(see Proposition \ref{proadj} (2) for $A_{\ge n+1}$). Let $\cA_{]n,n_0]}$ be the image of
$\cA_{\le n_0}/\cA_{\le n}$ by the exact right adjoint of the localisation functor, and
similarly for
$\cB$ (see Proposition \ref{psplit} (1)). Note that $R(\cA_{]n,n_0]})\subseteq \cB_{]n,n_0]}$
because $R(\cA_i)\subseteq \cB_i$ for all $i>n$. Hence $R(A_{>n})\in \cB_{]n,n_0]}$. The snake
lemma shows that $\ker b\iso \coker a\in \cB_{]n,n_0]}\cap \cB_{\le n}=0$. 

The last claim follows.
\end{proof}

\subsection{Glueing natural transformations}

\subsubsection{The case of a $2$-step filtration}\label{sgluenat}

Let $0\to \cA'\by{i} \cA\by{\pi} \cA''\to 0$ and
  $0\to \cB'\by{i'} \cB\by{\pi'} \cB''\to 0$ be as in \eqref{eqabsetup}, and
  consider two naturally commutative diagrams of exact functors 
\[\begin{CD}
0@>>> \cA'@>{i}>> \cA@>{\pi}>> \cA''@>>> 0\\
&&@V{R'_n}VV @V{R_n}VV @V{R''_n}VV\\
0@>>> \cB'@>{i'}>> \cB@>{\pi'}>> \cB''@>>> 0
\end{CD}\]
for $n=1,2$. We assume that the conditions of Proposition \ref{pglue}
are satisfied for $n=1,2$. We also assume given two natural
transformations $u':R'_1\Rightarrow R'_2$ and $u'':R''_1\Rightarrow R''_2$.

\begin{defn}\label{dgluenat} A \emph{glueing} of $u'$ and $u''$ is a natural transformation
$u:R_1\Rightarrow R_2$ such that, for any $A\in \cA$, the diagram 
\begin{equation}\label{eqcomm} \begin{CD}
0@>>> i'R'_1(pA)@>>> R_1(A)@>>> j'R''_1(\pi A)@>>> 0\\
&& @V{i'(u'_{pA})}VV @V{u_A}VV @V{j'(u''_{\pi A})}VV \\
0@>>> i'R'_2(pA)@>>> R_2(A)@>>> j'R''_2(\pi A)@>>> 0
\end{CD}\end{equation}
commutes.
\end{defn}

\begin{thm}\label{tgluenat}\ 
\begin{enumerate}
\item For the existence of $u$, the following condition is necessary: for any
$(A',A'')\in
\cA'\times
\cA''$, the diagram 
\[\begin{CD}
\Ext^i_\cA(jA'',iA')@>\bar R_1>> \Ext^i_\cB(j'R''_1A'',i'R'_1A')\\
@V{\bar R_2}VV @V{i'(u'_{A'})_*}VV\\
\Ext^i_\cB(j'R''_2A'',i'R'_2A')@>j'(u''_{A''})^*>>
\Ext^i_\cB(j'R''_1A'',i'R'_2A') 
\end{CD}\]
commutes for $i=0,1$, where $\bar R_n$ denotes the composition of $R_n$ with
suitable natural isomorphisms.\\
We may view this obstruction as a bivariant natural transformation
\[i'(u'_{A'})_*\bar R_1- j'(u''_{A''})^*\bar R_2:\Ext^i_\cA(jA'',iA')\to
\Ext^i_\cB(j'R''_1A'',i'R'_2A').\]
\item Suppose that $\cB$ is the heart of a $t$-structure on a triangulated category $\cT$. Then
the condition of (1) for $i=1$ is equivalent to the following: for any $(A',A'')\in
\cA'\times
\cA''$, the diagram 
\[\begin{CD}
j'R''_1(\pi A)@>[R_1(A)]>> i'R'_1(pA)[1]\\
@V{j'(u''_{\pi A})}VV @V{i'(u'_{pA})}VV\\
j'R''_2(\pi A)@>[R_2(A)]>> i'R'_2(pA)[1]
\end{CD}\] 
commutes in $\cT$.
\item Suppose that $\Hom_\cB(j'\cB'',i'\cB')=0$.  Then $u$
exists and is unique if and only if the Condition in (1) is satisfied for $i=1$.
\item $u$ is a natural isomorphism if and only if $u'$ and $u''$ are.
\end{enumerate} 
\end{thm}

\begin{proof} 
(1) This is clear for $n=0$, and for $n=1$ it follows by applying $u$ to exact
sequences of type \eqref{eqglue}. 

(2) and (3) follow from Lemma \ref{lglue0}. 

(4) ``Only if" is obvious since $u'$ and $u''$ are restrictions of $u$ to $\cA'$ and $\cA''$,
and ``if" is obvious by the snake lemma.
\end{proof}

\subsubsection{The case of filtered categories}\label{sgluenatII}
We now generalise Theorem
\ref{tgluenat} to abelian categories provided with weight filtrations.

Let $\cA,\cB$ be two abelian categories, filtered  in the sense of Definition \ref{d16.3}. We
assume that the filtrations are \emph{weight filtrations} (Definition \ref{dE.6}).

We slightly change notation and consider two exact functors $\uT,\uT':\cA\to \cB$ which respect
the splittings in the sense of Definition \ref{dF.4.4}. By Lemma \ref{lF.4.5} (2),  we then have
for every object
$A\in
\cA$:
\[\uT(A_{\le n})=\uT(A)_{\le n},\qquad \uT'(A_{\le n}) = \uT'(A)_{\le n}.\]

If $u:\uT'\Rightarrow \uT$ is a natural transformation, then $u_A$ maps $\uT'(A_{\le n})$ to
$\uT(A_{\le n})$, hence $u$ \emph{induces} natural transformations $u_{\le n}:\uT'_{\le
n}\Rightarrow \uT_{\le n}$ and $u_n:\uT'_n\Rightarrow \uT_n$.

\begin{defn}\label{dgluenatII} For each $n\in\Z$, let $u_n:\uT'_n\Rightarrow \uT_n$ be a
natural transformation. A \emph{glueing} of the $u_n$ is a natural transformation
$u:\uT'\Rightarrow \uT$ which induces the $u_n$.
\end{defn}

\begin{thm}\label{tgluenatII} 
\begin{enumerate}
\item There exists at most one glueing.
\item Suppose that the glueing $u$ exists on $\cA_{\le n-1}$. Then $u$ extends to $\cA_{\le
n}$ if and only if, for any
$(A,B)\in\cA_{\le n-1}\times\cA_n$, the diagram 
\[\begin{CD}
\Ext^1_\cA(B,A)@>\uT'>> \Ext^1_\cB(\uT'(B),\uT'(A))\\
@V{\uT}VV @V{u(A)_*}VV\\
\Ext^1_\cB(\uT(B),\uT(A))@>u_n(B)^*>>
\Ext^1_\cB(\uT'(B),\uT(A)) 
\end{CD}\]
commutes.
\item $u$ is a natural isomorphism on objects of finite lengths if and only if all the $u_n$
are.
\end{enumerate} 
\end{thm}

\begin{proof} 
(1) follows from Lemma \ref{lF.3.5}, while (2) follows from Theorem
\ref{tglueII} (2).

(3) ``Only if" is obvious, and ``if" follows inductively from Theorem
\ref{tglueII} (3).
\end{proof}

\subsubsection{The case of a $3$-step filtration}\label{sgluenatIII} We now restrict to the case where the filtration on $\cA$ has
only $3$ steps, and will try and get a condition involving only the $u_n$. For convenience, we
assume that $\cA_n\ne 0$
$\Rightarrow n\in \{-2,-1,0\}$.

\begin{thm}\label{tgluenatIII} \
\begin{enumerate}
\item Suppose that, for $(m,n)\in \{(-2,-1),(-1,0)\}$ and any
$(A_m,A_n)\allowbreak\in
\cA_m\times
\cA_n$, the diagram
\[\begin{CD}
\Ext^1_\cA(A_n,A_m)@>\uT'>> \Ext^1_\cB(\uT'(A_n),\uT'(A_m))\\
@V{\uT}VV @V{u(A_m)_*}VV\\
\Ext^1_\cB(\uT(A_n),\uT(A_m))@>u_n(A_n)^*>>
\Ext^1_\cB(\uT'(A_n),\uT(A_m)) 
\end{CD}\]
commutes or, equivalently, the obstruction in Theorem \ref{tgluenat} (1) vanishes. Let $u_{\le
-1}$ be the resulting natural transformation on $\cA_{\le -1}$ (Theorem \ref{tgluenat} (3)).
Then, for any $(A_{\le -1},A_0)\in
\cA_{\le -1}\times \cA_0$, the glueing obstruction
\[\Ext^1_\cA(A_0,A_{\le -1})\to \Ext^1_\cB(\uT' A_0,\uT A_{\le -1})\]
refines to an obstruction
\[\Ext^1_\cA(A_0,A_{\le -1})\to \Ext^1_\cB(\uT' A_0,\uT A_{-2}).\]
\item Suppose moreover that the diagram of (1) commutes for $(m,n)= (-2,0)$ and that
the map $\Ext^1_\cA(A_0,A_{\le -1})\to \Ext^1_\cA(A_0,A_{ -1})$ is surjective for any $A_0\in
\cA_0$ and $A_{\le -1}\in \cA_{\le -1}$ (for example, that $\cA$ is of cohomological dimension
$\le 1$). Then the obstruction of (1) refines to a bilinear obstruction
\[\Ext^1_\cA(A_0,A_{ -1})\times\Ext^1_\cA(A_{ -1},A_{ -2})\to \Ext^1_\cB(\uT' A_0,\uT A_{-2})
\]
which is covariant in $A_2$, contravariant in $A_0$ and ``dinatural" in $A_{-1}$ in the sense
that the corresponding map
\[\Ext^1_\cA(A_{ -1},A_{ -2})\to \Hom(\Ext^1_\cA(A_0,A_{ -1}),\Ext^1_\cB(\uT' A_0,\uT
A_{-2}))
\]
is contravariant in $A_{-1}$.
\item If $u_{-2}$ is a natural isomorphism, the obstruction of (2) may be reformulated as an
obstruction with values in $\Ext^1_\cB(\uT' A_0,\uT' A_{-2})$.
\item The glueing exists if and only if the obstruction of (2) vanishes.
\end{enumerate}
\end{thm}

\begin{proof} (1) follows from Proposition \ref{pglue1} (4). For (2), we have a commutative
diagram
\[\begin{CD}
\Ext^1_\cA(A_0,A_{ -2})@>>> \Ext^1_\cB(\uT' A_0,\uT A_{-2})\\
@VVV @V||VV\\
\Ext^1_\cA(A_0,A_{\le -1})@>>> \Ext^1_\cB(\uT' A_0,\uT A_{-2})
\end{CD}\]
and the top horizontal map is $0$ by hypothesis. Hence (by the surjectivity assumption) it
induces a map $\Ext^1_\cA(A_0,A_{ -1})\to \Ext^1_\cB(\uT' A_0,\uT A_{-2})$, which is covariant
in the second variable viewed as a functor of $A_{\le -1}$. The result then follows from Lemma
\ref{lnat}.

(3) is obvious and (4) follows from Theorem \ref{tgluenatII}.
\end{proof}

\subsection{Glueing equivalences of abelian categories}

\subsubsection{The case of a $2$-step filtration}

\begin{thm}\label{tglue} Let $0\to \cA'\by{i} \cA\by{\pi} \cA''\to 0$ and
  $0\to \cB'\by{i'} \cB\by{\pi'} \cB''\to 0$ be as in \eqref{eqabsetup}, and
  consider a naturally commutative diagram of exact functors 
\[\begin{CD}
0@>>> \cA'@>{i}>> \cA@>{\pi}>> \cA''@>>> 0\\
&&@V{R'}VV @V{R}VV @V{R''}VV\\
0@>>> \cB'@>{i'}>> \cB@>{\pi'}>> \cB''@>>> 0.
\end{CD}\]
Then:
\begin{enumerate}
\item If $R'$ and $R''$ are faithful, $R$ is faithful.
\item Assume the conditions of Proposition \ref{pglue} are satisfied and suppose further that,
for any two objects $A'\in \cA'$, $A''\in \cA''$, the map  
\[\Hom_\cA(j A'',i A')\to \Hom_\cB(R(jA''),R(iA'))\]
is surjective and the map
\begin{equation} \label{eqExt}
\Ext^1_\cA(j A'',i A')\to \Ext^1_\cB(R(jA''),R(iA'))
\end{equation}
is injective. Then, if $R'$ and $R''$ are full, $R$ is full.
\item Assume the conditions of {\rm (2)} and suppose further that, for
  any two objects $A'\in \cA'$, $A''\in \cA''$, the map \eqref{eqExt} is
  surjective. Then, if $R'$ and $R''$ are essentially surjective, $R$
  is essentially surjective. 
\end{enumerate}
\end{thm}

\begin{proof} (1) Let $f:A_1\to A_2$ in $\cA$ be such that
  $R(f)=0$. Since $R''\pi\simeq \pi'R$, $R''\pi(f)=0$, hence $\pi(f)=0$. By
  calculus of fractions, this implies that one can find $s:A_2\to B$
  with $\ker(s),\coker(s)\in i(\cA')$ such that $sf=0$. This means that
  $f$ factors as a composition 
\[A_1\by{g} iK \by{h} A_2\]
with $iK=\ker(s)$ and $h$ is a mono. Then $R(h)R(g)=0$. Since $R$ is
exact, $R(h)$ is a mono, hence $R(h)=0$ which implies that
$R(iK)=i'R'(K)=0$, hence $K=0$ since $R'$ and $i'$ are faithful.

(2) Let $A_1,A_2\in\cA$ and let $g\in\cB(R(A_1),R(A_2))$. By the
functoriality of \eqref{eqglue}, we get a commutative diagram 
\[\begin{CD}
0@>>> i'p'R(A_1)@>>> R(A_1)@>>> j'\pi'R( A_1)@>>> 0\\
&& @V{i'p'(g)}VV @V{g}VV @V{j'\pi'(g)}VV \\
0@>>> i'p'R(A_2)@>>> R(A_2)@>>> j'\pi'R(A_2)@>>> 0.
\end{CD}\]

Using the equivalent conditions of (2), $i'p'(g)$ and $j'\pi'(g)$
respectively give maps 
\[g':i'R'(p A_1)\to i'R'(p A_2),\qquad g'':j' R''(\pi A_1)\to j' R''(\pi A_2).\]

By fullness, $g'$ and $g''$ are induced by maps
\[f':p A_1\to p A_2,\qquad f'':\pi A_1\to \pi A_2.\]

Now consider the diagram
\[\begin{CD}
0@>>> ipA_1@>>> A_1@>>> j\pi A_1@>>> 0\\
&& @V{i(f')}VV&& @V{j(f'')}VV \\
0@>>> ipA_2@>>> A_2@>>> j\pi A_2@>>> 0
\end{CD}\]

Consider the extension classes $[A_r]\in \Ext^1_\cA( j\pi A_r,ipA_r)$
($r=1,2$). Then a map $f$ filling in this diagram exists if and only
if 
\[i(f')_*[A_1]= j(f'')^*[A_2]\in \Ext^1_\cA(j\pi A_1,ip A_2).\]

Hence, by the existence of $g$, this equality is true after applying
the functor $R$, and therefore it holds by the injectivity assumption.  

Now the map $R(f)-g$ induces a map $h:Rj\pi(A_1)\to Rip(A_2)$. By the
surjectivity assumption, $h$ is of the form $R(h')$, and then
$R(f-h')=g$.   

(3) Let $B\in \cB$. Then $p' B$ is in the essential image of $R''$ and
$\pi' B$ is in the essential image of $R'$. Using the exact sequence
\eqref{eqglue}, Property (ii) of (2) and the surjectivity assumption,
we get $B\simeq R(A)$ for some $A\in \cA$. 
\end{proof}

Here is a converse to Theorem \ref{tglue}:

\begin{thm}\label{tglueopp}
Keep the notation of Theorem \ref{tglue}. Then:
\begin{enumerate}
\item If $R$ is faithful (\resp full), so is $R'$.
\item If $R$ is faithful (\resp full) and the equivalent conditions of Proposition
\ref{pglue} hold, so is $R''$.
\item If $R$ is essentially surjective, $R''$ is essentially surjective, and so is $R'$
provided $R''$ is faithful or conservative.
\end{enumerate}
\end{thm}

\begin{proof} (1) and (2) are obvious, considering commutative squares of the type
\[\begin{CD}
\cA'(A'_1,A'_2)@>\sim>> \cA(iA'_1,iA'_2)\\
@V{R'}VV @V{R}VV\\ 
\cA'(R'A'_1,R'A'_2)@>\sim>> \cA(RiA'_1,RiA'_2)
\end{CD}\]
for $R'$, and similarly for $R''$. The first part of (3) is obvious. For the second one, let
$B'\in \cB'$. Write $i'B'\simeq RA$ for some $A\in \cA$. Then $R''\pi A\simeq \pi'RA=0$. The
hypothesis implies
$\pi A=0$, hence $A\simeq iA'$ for some $A'\in \cA'$. Now $i'B'\simeq RiA'\simeq i'R'A'$, hence
$B'\simeq R'A'$.
\end{proof}

\begin{cor}\label{cglue} With the notation of Theorem \ref{tglue}, suppose that
\begin{thlist}
\item $R'$ and $R''$ are equivalences of categories;
\item the conditions of Proposition \ref{pglue} are verified;
\item for any two objects $A'\in \cA'$, $A''\in \cA''$, the map
\[
\Ext^i_\cA(j A'',i A')\to \Ext^i_\cB(R(jA''),R(iA'))
\]
is an isomorphism for $i=0,1$.
\end{thlist}
Then $R$ is an equivalence of categories.\\
Conversely, if $R$ is an equivalence of categories and the conditions of Proposition \ref{pglue}
are verified, $R'$ and $R''$ are equivalences of categories.
\qed
\end{cor}

\subsubsection{The general case}
Here is now a version of Theorem \ref{tglue} for filtered abelian categories. 

\begin{thm}\label{tglueII} Let $\cA$ and $\cB$ be two abelian categories provided with
exhaustive separated split filtrations $(\cA_{\le n},\iota_n)$ and $(\cB_{\le n},\iota'_n)$.
Let $R:\cA\to \cB$ be an exact functor and assume that $R$ respects the splittings in the sense
of Definition \ref{dF.4.4}. 
\begin{enumerate}
\item If $R_n$ is faithful for every $n$, then $R$ is faithful.
\item Assume that for any $m<n$ and
  any two objects $A_m\in \cA_m$, $A_n\in \cA_n$, the map  
\[
\Hom_\cA(A_n,A_m)\to \Hom_\cB(R(A_n),R(A_m))
\]
is surjective and the map
\[
\Ext^1_\cA(A_n,A_m)\to \Ext^1_\cB(R(A_n),R(A_m))
\]
is injective. Assume also that every object of $\cA$ is of finite length in the sense of Definition \ref{dflength}. Then, if $R_n$ is full for every $n$, $R$ is full.
\item Assume  that, for any $n\in\Z$, the map
\[
\Ext^1_\cA(A,B)\to \Ext^1_\cB(R(A),R(B))
\]
is surjective for $A\in \cA_n$ and $B\in \cA_{\le n-1}$. Assume also that every object of $\cB$
is of finite length. Then, if
$R_n$ are essentially surjective for every $n$, $R$
  is essentially surjective. 
\item If $\Ext^2_\cB(A,B)=0$ for all $A,B\in \cB$, we may weaken Condition (3) by requesting
surjectivity for all $(A,B)\in \cA_n\times \cA_m$, $m<n$.
\end{enumerate}
\end{thm}

\begin{proof} 
(1)  Let $f:A\to A'$ be such that $R(f)=0$. There exists $n_0$ such that $A,A'\in \cA_{n_0}$.
By induction, Theorem  \ref{tglue} implies that $f_{\ge n}=0$ for all $n$ (see Proposition
\ref{proadj} (2) for $f_{\ge n}$). Thus, $\im f\in \bigcap \cA_{\le n}=0$.

(2) Since the filtrations are exhaustive, it suffices to prove that $R_{\le n}$ is full for
all $n$. Thus we may assume that $\cA_{\le n} = \cA$ for $n$ large enough. Similarly, since
every object is of finite length, it suffices to prove that the restriction of $R$ to the  Serre
subcategory consisting of objects of length $\le \ell$ is full for all $\ell \ge 0$. Thus, we
may also assume that $\cA_{\le n'}=0$ for $n'$ small enough. By induction on $n-n'$, this
reduces us to showing that the assumption of Theorem \ref{tglue} (3) are verified for
$\cA'=\cA_{\le n-1}$.

Let $A'\in \cA_{\le n-1}$ and $A''\in \cA_n$. We want to show that the map
\[
\Hom_\cA(A'',A')\to \Hom_\cB(R(A''),R(A'))
\]
is surjective and the map
\[
\Ext^1_\cA(A'',A')\to \Ext^1_\cB(R(A''),R(A'))
\]
is injective. Consider the short exact sequence $0\to A'_{\le n-2}\to A'\to A'_{n-1}\to 0$.
The statement is true by replacing $A'$ by $A'_{\le n-2}$ (by induction) or by $A'_{n-1}$ (by
hypothesis). Hence it is true for $A'$, by the five lemma.

(3) We argue as in (2): since the filtration of $\cB$ is exhaustive, it suffices to prove that
$R_{\le n}$ is essentially surjective for all $n$, hence we may assume $\cB_{\le n} = \cB$ for
$n$ large enough. Since every object of $\cB$ has finite length, it suffices to prove that
every object of $\cB$ of length $\le \ell$ is isomorphic to $R(A)$, where $A\in \cA$ is of
length $\le \ell$. Thus we may also assume that
$\cA_{\le n'}=\cB_{\le n'}=0$ for $n'$ small enough. We argue by induction on $n-n'$: by
induction,
$R_{\le n-1}$ is essentially surjective, and so is $R_n$ by hypothesis. On the other hand, the
assumptions of Theorem \ref{tglue} (4) are verified for $\cA'=\cA_{\le n-1}$ and $\cB'=\cB_{\le
n-1}$.

(4) By the 5 lemma, we deduce the condition of (3) from this weaker condition.
\end{proof}

\subsection{The case of triangulated categories}
\subsubsection{Split exact sequences}
We let $\cT$ be  a triangulated category. Consider the situation 
\begin{equation}\label{eqabsetupbis}
0\to \cT'\by{i} \cT\by{\pi} \cT''\to 0
\end{equation}
where $\cT'\subseteq \cT$ is a thick subcategory and $\pi$ is the corresponding
localisation functor. Thus, $\cT''$ is the Verdier quotient of $\cT$
by $\cT'$. 

\begin{propose}[Verdier]\label{psplitbis} The following conditions are
 equivalent:
\begin{thlist}
\item $i$ has a right adjoint $p$.
\item $\pi$ has a right adjoint $j$ and $j(\cT'')$ is thick in $\cT$.
\item for any $A\in \cT$ there exists $A'\in \cT'$ and a map $f:i(A')\to A$ such that
$\Hom (i(\cT'), \cone(f))=0$.
\end{thlist}
If these conditions are verified, then for any $A\in\cT$ there exists a
unique exact triangle  
\[A'\to A\to A''\by{+1}\]
such that $A'\in i(\cT')$ and $A''\in j(\cT'')$; this triangle is
functorial in $A$. Moreover, we have
$\Hom_\cT(i\cT',j\cT'')=0$.
\end{propose} 

\begin{proof} (i) $\Rightarrow$ (ii): let $A\in \cT$, and choose a cone $C_A$ of
  the counit $\epsilon_A:ip A\to A$.
Applying $p$ to the exact triangle
\[ip A\by{\epsilon_A} A\to C_A\by{+1}\]
we get an exact triangle
\[p A\by{=} pA\to pC_A\by{+1}\]
which shows that 
\begin{equation}\label{eqpC}
pC_A= 0. 
\end{equation}

Let $f:A\to B$ be a morphism, and choose a
cone $C_B$ analogously. There
exists a morphism $g$ such that the diagram of exact triangles
\[\begin{CD}
ip A@>\epsilon_A>> A@>>> C_A@>+1>>\\
@V{ipf}VV @V{f}VV @V{g}VV\\
ip B@>\epsilon_B>> B@>>> C_B@>+1>>
\end{CD}\]
commutes. 

Applying \eqref{eqpC} to $C_B$, we get by adjunction
$\Hom_\cT(ipA[1],C_B) = 0$, which shows that $g$ is \emph{unique}; in
particular, taking $f=1_A$, we get that $C_A$ is defined up to unique
isomorphism.

Letting $\cS\df \ker p$, this thus defines
a functor $q:\cT\to \cS$ ($qA=C_A$), and $q$ is easily seen to be
left adjoint to the inclusion $\cS\into \cT$. Moreover $q_{|\cT'}=0$,
hence $q$ induces a functor $\bar q:\cT''\to \cS$. 

On the other hand, we have the obvious functor
\[r:\cS\to\cT\to \cT''.\]

Let $B\in \cS$. By definition, $\bar q rB$ is represented by
$C_B=B$. This provides a natural isomorphism $B\iso \bar q r B$. 

Let $A\in \cT$. The map $A\to C_A$ induces another natural
isomorphism $A\iso r\bar q A$, where $A$ is now viewed in $\cT''$. So
$\bar q$ and $r$ are quasi-inverse equivalences of categories. Under
this equivalence, $q$ becomes $\pi$, and the inclusion $\cS\into \cT$
defines a functor $j:\cT''\to\cT$, which is right adjoint to
$\pi$. Finally, since $\cS=\ker p$ is thick, $j(\cT'')$ is thick. 

(ii) $\implies$ (iii): let $A\in \cT$. The functor $\pi$ kills the cone
$D_A$ of the unit map $\eta_A:A\to j\pi A$. Hence $D_A\in i(\cT')$. If $A'_1\in
\cT'$, then 
\[\Hom(iA'_1,j\pi A)= \Hom(\pi i A'_1,\pi A)=\Hom(0,\pi A)=0\]
hence we may take $A'=D_A[-1]$ in (iii).

(iii) $\implies$ (i): one checks that $A\mapsto A'$ yields the desired
adjoint, by the same kind of arguments as in the proof of (i)
$\implies$ (ii).

Finally, the claimed exact triangle is given by
\begin{equation}\label{eqgluebis}
ip A\by{\epsilon_A} A\by{\eta_A} j\pi A\by{\delta} ipA[1]
\end{equation}
and its properties follow by construction. The vanishing of
$\Hom_\cT(i\cT',j\cT'')$ follows from the adjunction. 
\end{proof}

\begin{defn}\label{dsplitbis} In the situation of Proposition
  \ref{psplitbis}, we say that the exact sequence \eqref{eqabsetupbis} is
  \emph{split}. 
\end{defn}

\subsubsection{More adjoints} Assume that \eqref{eqabsetupbis} is
split. Let us now consider the conditions

\begin{quote} ($\flat$) {\it $j$ has a right adjoint $\pi'$.} \\
($\sharp$) {\it $i$ has a left adjoint $p'$.}
\end{quote}

Recall that, under ($\flat$), we get a canonical natural transformation
$\alpha:\pi'\Rightarrow \pi$ from the composition $j\pi'A\to A\to j\pi A$ for any $A\in \cT$,
and the full faithfulness of $j$. Similarly, under (\#), we get a canonical natural
transformation $\beta:p\Rightarrow p'$.

\begin{propose}\label{pstrengbis} Consider the
  following conditions: 
\begin{thlist}
\item $\Hom_\cT(j\cT'',i\cT')=0$.
\item {\rm ($\flat$)} holds and $\pi'=\pi$.
\item {\rm ($\flat$)} holds and $\pi'i=0$.
\item {\rm ($\sharp$)} holds and $p'=p$.
\item {\rm ($\sharp$)} holds and $p'j=0$.
\end{thlist}
Then {\rm (i) + ($\flat$)} $\iff$ {\rm (i) + ($\sharp$)} $\iff$ {\rm (ii)} $\iff$ {\rm
  (iii)} $\iff$ {\rm (iv)} $\iff$ {\rm (v)}.
\end{propose}

\begin{proof} Suppose that $\pi'$ exists. Let
  $A\in\cT$. Applying $\pi'$ to the exact triangle \eqref{eqgluebis}, we get
  an exact triangle
\[\pi'ipA\to \pi' A\by{\alpha_A}\pi A\by{+1}\]

This shows that $\alpha_A$ is an isomorphism if and only if $\pi'ipA=0$. Thus,
($\alpha_A$ injective for all $A$)$\iff$ ($\pi'ipA=0$ for all $A$) $\iff$ ($\pi'i=0$)
since $p$ is essentially surjective. Thus, (ii) $\iff$ (iii). On the
other hand, (i) + ($\flat$) $\iff$ (iii) is obvious by adjunction. The other
equivalences are obtained dually.
\end{proof}

\subsubsection{Filtered triangulated categories}

\begin{defn}\label{d16.3tri} Let $\cT$ be a triangulated category.\\
a) A \emph{filtration} on $\cT$ is a sequence of thick subcategories
\[\dots \longby{\iota_{n-1}}\cT_{\le n}\longby{\iota_n}\cT_{\le
n+1}\longby{\iota_{n+1}}\dots\to \cT.\] 
It is \emph{separated} if $\bigcap \cT_{\le n}=\{0\}$ and \emph{exhaustive} if $\cT=\bigcup
\cT_{\le n}$.\\  
b) A filtration on $\cT$ is \emph{split} if all $\iota_n$ have exact right adjoints.\\
c) A filtration is a \emph{weight filtration} if it is separated, exhaustive and split.
\end{defn}

Let $\cT$ be provided with a filtration. Define
\[\cT_n = \cT_{\le n}/\cT_{\le n-1},\qquad \cT_{n,n+1} = \cT_{\le n+1}/\cT_{\le n-1}\]
so that we have exact sequences
\[0\to \cT_n\longby{i_n} \cT_{n,n+1}\longby{\pi_n} \cT_{n+1}\to 0\]
where $i_n$ is induced by $\iota_n$.

\begin{propose} \label{proadjtri}
Suppose that $\cT$ is provided with an exhaustive split filtration as in Definition
\ref{d16.3tri} a), b). For all $n\in\Z$, let us write $j_n$ for the right adjoint of the
localisation functor $\cT_{\le n}\to \cT_{\le n}/\cT_{\le n-1}$ (Proposition \ref{psplitbis}).
By abuse of notation, we shall write $\cT_n$ for the thick subcategory $j_n(\cT_{\le n}/\cT_{\le
n-1})\subset \cT$ (\ibid). Then:
\begin{enumerate}
\item Every object $A\in \cT$ is provided with a unique filtration
\[\dots\to A_{\le n}\to A_{\le n+1}\to\dots \to A\]
with $A_{\le n}\in \cT_{\le n}$ and $A_n:=\cone(A_{\le n-1}\to A_{\le n})\in \cT_n$. One has 
$A_{\le n}=A$ for $n$ large enough.
\item The functors $A\mapsto A_{\le n}$ and $A\mapsto A_n$ are well-defined, as well as
$A\mapsto A_{\ge n}:=A/A_{\le n-1}$. 
\item If moreover the filtration is separated, the $A\mapsto A_n$ form a set of conservative
functors.  
\item We have $\Hom_\cT(\cT_m,\cT_n)=0$ if $m<n$.
\end{enumerate}
\end{propose}

\begin{proof} (1) Since $\cT=\bigcup \cT_{\le n}$, there exists $n_0$ such that $A\in \cT_{\le
n_0}$. For $n<n_0$, write $I_n$ for the inclusion $\cT_{\le n}\to \cT_{\le n_0}$ and $P_n$ for
its right adjoint, which exists and is exact as a composition of the right adjoints of the
$\iota_r$ for $n\le r<n_0$. Define
\[A_{\le n} = I_nP_n A.\]

Since $A=\iota_{n_0}\varpi_{n_0} A$, $A_{\le n}$ does not depend on the choice of $n_0$, and
we have a filtration of $A$ as requested. Clearly, $A_{\le n}\in \cT_{\le
n}$ and $A_{\le n_0} =A$. The fact that $A_n\in \cT_n$ follows from Proposition
\ref{psplitbis}, which also shows the uniqueness of the filtration. 

(2) This still follows from Proposition \ref{psplitbis}.

(3) Let $f:A\to B$ be such that $f_n:A_n\to B_n$ is an isomorphism for all $n\in\Z$. Let
 $C=\cone(f)$. Then $C_n = 0$ for all $n$. Thus $C\in \bigcap
\cT_{\le n} = 0$ and $f$ is an isomorphism.

(4) It suffices to show  that \allowbreak $\Hom_\cT(\cT_{\le n-1},\cT_n)=0$, which follows
again from Proposition \ref{psplitbis}. 
\end{proof}

\subsubsection{Glueing equivalences of triangulated categories}

\begin{thm}\label{tgluebis} Let $0\to \cS'\by{i} \cS\by{\pi} \cS''\to 0$ and
  $0\to \cT'\by{i'} \cT\by{\pi'} \cT''\to 0$ be as in \eqref{eqabsetupbis}, and
  consider a naturally commutative diagram of exact functors 
\[\begin{CD}
0@>>> \cS'@>{i}>> \cS@>{\pi}>> \cS''@>>> 0\\
&&@V{R'}VV @V{R}VV @V{R''}VV\\
0@>>> \cT'@>{i'}>> \cT@>{\pi'}>> \cT''@>>> 0.
\end{CD}\]
\begin{enumerate}
\item The following are equivalent: 
\begin{thlist}
\item The natural ``base change" transformation $R'p\Rightarrow p'R$ of
  Section \ref{sbase} is a natural isomorphism. 
\item   $Rj\Rightarrow j'R''$ is a natural isomorphism. 
\item  $R(j\cS'')\subseteq j'\cT''$.
\end{thlist}
\item Suppose that
\begin{thlist}
\item $R'$ and $R''$ are fully faihtful;
\item the conditions of {\rm (1)} are verified;
\item for any two objects $A'\in \cS'$, $A''\in \cS''$, the map
\[
\Hom_\cS(j A'',i A')\to \Hom_\cT(R(jA''),R(iA'))
\]
is an isomorphism.
\end{thlist}
Then $R$ is fully faithful.
\item Suppose that Condition {\rm (iii)} of {\rm (2)} holds and that  $R'$ and $R''$ are
essentially surjective. Then so is $R$.
\item Suppose that $R'$ and $R''$ are equivalences of categories and that the conditions of
{\rm (1)} and Condition {\rm (iii)} of {\rm (2)} are satisfied. Then $R$ is an equivalence of
categories.
\end{enumerate}
\end{thm}

\begin{proof} 
(1) Let $A\in \cS$. We have a commutative diagram of exact triangles
(using \eqref{eqgluebis} and the base change morphisms): 
\[\begin{CD}
R(ipA)@>>> R(A)@>>> R(j\pi A)@>+1>> \\
@VVV @V{||}VV @VVV \\
i'p'R(A)@>>> R(A)@>>> j'\pi'R(A)@>+1>>.
\end{CD}\]

Thus the left vertical map is an isomorphism if and only if the right
vertical map is one, which shows that (i) $\iff$ (ii). If this is the
case, then $R(j\pi A)\in j'(\cT')$, hence (iii). Conversely, if (iii)
holds, then all vertical maps must be isomorphisms by the uniqueness
of \eqref{eqgluebis}. 

(2)  A. We first show the fullness of $R$. Let $A_1,A_2\in\cS$ and let $g\in\cT(R(A_1),R(A_2))$.
By the functoriality of \eqref{eqgluebis}, we get a commutative diagram 
\[\begin{CD}
i'p'R(A_1)@>>> R(A_1)@>>> j'\pi'R( A_1)@>+1>>\\
@V{i'p'(g)}VV @V{g}VV @V{j'\pi'(g)}VV \\
i'p'R(A_2)@>>> R(A_2)@>>> j'\pi'R(A_2)@>+1>>.
\end{CD}\]

Using the equivalent conditions of (1), $i'p'(g)$ and $j'\pi'(g)$
respectively give maps 
\[g':i'R'(p A_1)\to i'R'(p A_2),\qquad g'':j' R''(\pi A_1)\to j' R''(\pi A_2).\]

By the fullness of $R'$ and $R''$, $g'$ and $g''$ are induced by maps
\[f':p A_1\to p A_2,\qquad f'':\pi A_1\to \pi A_2.\]

Now consider the diagram
\[\begin{CD}
A_1@>>> j\pi A_1@>>>ipA_1[1]@>+1>> \\
&& @V{j(f'')}VV @V{i(f'[1])}VV\\
A_2@>>> j\pi A_2@>>> ipA_2[1]@>+1>> 
\end{CD}\]

By the injectivity in Condition (iii), the square commutes, hence
there exists a map $f:A_1\to A_2$ filling in this diagram. Now the map
$R(f)-g$ induces a map $h:Rj\pi(A_1)\to Rip(A_2)$. By the surjectivity in
Condition (iii), $h$ is of
the form $R(h')$, and then $R(f-h')=g$. Thus $R$ is full.

B. To see that $R$ is faithful, it now suffices by A.1.1 to show that it is conervative. Let
$A\in\cS$ be such that $RA=0$. Then $R''\pi A=0$, hence $\pi A=0$ since $R''$, being fully
fiahtful, is conservative. Hence $A\simeq iA'$ for some $A'\in \cS'$. Then, $RiA'\simeq
iR'A'=0$, hence $A'=0$ by the conservativity of $i$ and $R'$ and $A=0$.

(3) Let $B\in \cT$. Then $p' B$ is in the essential image of $R''$ and
$\pi' B$ is in the essential image of $R'$. Thus $B$ fits in an exact
triangle of the form
\[B\to R(jA'')\by{\delta'} R(iA'[1])\by{+1}.\]

By the surjectivity in (2) (iii), $\delta'$ is of the form $R(\delta)$ for
$\delta:jA''\to iA'[1]$. Let $A$ denote a fibre (= shifted cone) of $\delta$. We
can then fill in the diagram
\[\begin{CD}
R(iA')@>>> R(A)@>>> R(jA'')@>R(\delta)>>\\
@V{\|}VV && @V{\|}VV\\
R(iA')@>>> B@>>> R(jA'')@>R(\delta)>>
\end{CD}\]
with a map which is automatically an isomorphism.

(4) This is just collecting the previous results.
\end{proof}

Here is a converse to Theorem \ref{tgluebis}:

\begin{thm}\label{tglueoppbis}
Keep the notation of Theorem \ref{tgluebis}. Then:
\begin{enumerate}
\item If $R$ is faithful (\resp full), so is $R'$.
\item If $R$ is faithful (\resp full) and the equivalent conditions of Theorem
\ref{tgluebis} (1) hold, so is $R''$.
\item If $R$ is essentially surjective, $R''$ is essentially surjective, and so is $R'$
provided $R''$ is conservative.
\end{enumerate}
\end{thm}

\begin{proof} Same as for Theorem \ref{tglueopp}.\end{proof}

\subsection{The case of $t$-categories}

Let $\cT$ be a triangulated category provided with a $t$-structure with heart $\cA$
\cite[\S 1]{BBD}.

\subsubsection{$t$-exact functors}

\begin{lemma}\label{lJ.1} Let $\cS$ be another $t$-category, with heart $\cB$, and let $T:\cS\to
\cT$ be a triangulated functor. Assume that $T(\cB)\subseteq \cA$ and that the $t$-structure on
$\cS$ is bounded. Then:
\begin{enumerate}
\item $T$ is $t$-exact;
\item the induced functor $T:\cB\to \cA$ is exact.
\end{enumerate}
\end{lemma}

\begin{proof} (1) Let us show right exactness. Let $X\in \cS^{\ge 0}$. We must show that
$T(X)\in
\cT^{\ge 0}$. By assumption, there exists $n\ge 0$ such that $X\in \cS^{[0,n]}$. For $n=0$, we
have
$T(X)\in
\cA$ by hypothesis. For $n>0$, we may argue by induction on $n$, using the exact triangle
\[H_n(X)[n]\to X\to Y\by{+1}\]
with $Y\in \cS^{[0,n-1]}$. Left exactness is proven similarly.

(2) This follows from \cite[Prop. 1.3.17 (i)]{BBD}.
\end{proof}

\subsubsection{The case of a $2$-step filtration}

Here we assume that $\cA$ sits in a short exact sequence \eqref{eqabsetup}. 

\begin{propose}\label{psplitt1} Let $\cT'$ be the full subcategory of $\cT$
consisting of objects $T$ such that $H^i(T)\in \cA'$ for all $i\in\Z$. Then:
\begin{enumerate}
\item  $\cT'$ is thick in $\cT$.
\item The $t$-structure of $\cT$ induces a $t$-structure on $\cT'$.
\item Suppose that \eqref{eqabsetup} is split in the sense of Definition \ref{dsplit}, and that
the $t$-structure of $\cT$ is bounded. Then
\begin{thlist}
\item The exact sequence
\[0\to\cT'\by{i} \cT\by{\pi} \cT/\cT'\to 0\]
is split in the sense of Definition \ref{dsplitbis}, and the right adjoint $p$ to $i$ is
$t$-exact.
\item Via the right adjoint $j$ to $\pi$ (see Proposition \ref{psplitbis}), the $t$-structure
of $\cT$ induces a $t$-structure on $\cT/\cT'$, for which $\pi$ is $t$-exact. Moreover,
$j(\cT/\cT')=\{C\in \cT\mid H^*(C)\in j(\cA'')\}$.
\end{thlist}
\end{enumerate}
\end{propose}

\begin{proof} (1) The thickness of $\cT'$ in $\cT$ follows from the thickness of
$\cA'$ in $\cA$. (2) is clear.

It remains to prove (3). 

(i) To prove that $i$ has a
right adjoint $p$, it suffices to prove Condition (iii) of Proposition
\ref{psplitbis} (compare \cite[lemmes 3.2.1 and 3.2.3]{morel}).

Let $\cS$ be the full subcategory of $\cT$ consisting of those objects
verifying Condition (iii) of Proposition \ref{psplitbis}. We have to show that
$\cS=\cT$. Since the $t$-structure is supposed to be bounded, it suffices to
check that $\cS$ is triangulated and contains $\cA$.

It is clear that $\cS$ is stable under shifts. Let $A_1,A_2\in \cS$,
$f:A_1\to A_2$ and
$A'_1,A'_2\in \cT'$ satisfying Condition (iii) of Proposition \ref{psplitbis}
respectively for $A_1$ and $A_2$. This condition implies that the composition $A'_1\to A_1\to
A_2$ factors through a (unique) map
$f':A'_1\to A'_2$. We also get a map $f'':\cone(f_1)\to \cone(f_2)$ defining a map of exact
triangles. Let
$A'_3=\cone(f')$: we may find a map $f_3:A'_3\to A_3$ defining a morphism of exact triangles.
Then the cone of $f_3$ is isomorphic to the cone of $f''$, hence has the property of
Proposition \ref{psplitbis}. This shows that $\cS$ is triangulated.

Let $A\in \cA$. By definition of ``split" and Proposition \ref{psplit} (2), $A$ sits
in a short exact sequence $0\to A'\to A\to A''\to 0$ with $A'\in \cA'$ and
$\Ext^r_\cA(B',A'')=0$ for all $B'\in \cA'$ and all $r\ge 0$. Since the $t$-structure is
bounded, this implies that $\Hom_\cT(\cT',A'')=0$, and $A\in \cS$.

This shows the existence of the right adjoint $p$. Moreover, the proof of Proposition
\ref{psplitbis} shows that, for $A,A',A''$ as in the last paragraph, $p(A)=A'$. Thus $p$
respects the hearts of the $t$-structures on $\cT$ and $\cT'$, and since the one on $\cT$ is
bounded, $p$ is $t$-exact by Lemma \ref{lJ.1}.

(ii) To show that the $t$-structure of $\cT$ induces a $t$-structure on $j(\cT/\cT')$, we have
to show that, for $X\in j(\cT/\cT')$, $\tau_{\le 0} X\in j(\cT/\cT')$, or equivalently that if
$p(X)=0$, then $p(\tau_{\le 0} X)=0$: this is clear since $p$ is $t$-exact. The proof that
$\pi$ is $t$-exact is then the same as the proof of the $t$-exactness of $p$. 

For the last
assertion, let $C\in j(\cT/\cT')$. Then, for any $i\in\Z$, $H^i(j(C))=j(H^i(C))\in j(\cA'')$.
Conversely, if $C\in \cT$ is such that $H^i(C)\in j(\cA'')$ for all $i\in\Z$, then $p
H^i(C)=H^i(pC)=0$ for all $i\in\Z$, hence $pC=0$ since the $t$-structure is bounded, and $C\in
\im j$ by Proposition \ref{psplitbis}.
\end{proof}

\subsubsection{Filtered $t$-categories}

We now assume that  $\cA$ is filtered in the sense of Definition \ref{d16.3} a).

\begin{defn}\label{dJ.1} Let $n\in\Z$. We denote by $\cT_{\le n}$ the full subcategory of $\cT$
consisting of objects $T$ such that $H^i(T)\in \cA_{\le n}$ for all $i\in\Z$.
\end{defn}

This definition is similar to that of Sophie Morel in \cite[Prop. 3.1.1]{morel}.

\begin{propose}\label{psplitt} \
\begin{enumerate}
\item The $\cT_{\le n}$ define a filtration of $\cT$ in the sense of Definition
\ref{d16.3tri} a).
\item For any $n$, the $t$-structure of $\cT$ induces a $t$-structure on $\cT_{\le n}$.
\item It is split if the filtration on $\cA$ is and the $t$-structure is bounded.
\item It is separated if the filtration on $\cA$ is, and if the $t$-structure is
nondegenerate. 
\item It is exhaustive if the filtration on $\cA$ is and if the $t$-structure is bounded.
\end{enumerate}
\end{propose}

\begin{proof} (1), (2) and (3) follow respectively from Proposition \ref{psplitt1} (1), (2) and
(3). (4) and (5) are easy. 
\end{proof} 

\begin{remark}\label{rE.21}
As in \cite[Def. 1.2.1]{HuberLN}, let $\cT$ be a $t$-category with compatible triangulated
endofunctors $W_n$ with commute with the $t$-structure and such that the transformations
$W_n\to  W_{n+1}$ and $W_n\to Id$ are inclusions on the heart $\cA$ of $\cT$. In particular,  
the $W_n$ induce an increasing sequence of \emph{exact} subfunctors of the 
identity of $\cA$. Thus we are in the situation of Remark \ref{lF.3.6} if they are separated
and exhaustive on all objects of $\cA$. By this remark, we then have a weight filtration on
$\cA$ in the sense of Definition \ref{dE.6}. 

Define $\cT_{\le n}$ as in  Definition \ref{dJ.1}. By Proposition \ref{psplitt},
this is a weight filtration on
$\cT$ in the sense of Definition \ref{d16.3tri} c). By Proposition \ref{proadjtri},   we get
endofunctors of
$\cT$:
\[W'_n C = C_{\le n}.\]

We claim that $W'_n C=W_n C$. First,
$W_n C \in \cT_{\le n}$ because $W_n$ commutes with $t$ and this is true on $\cA$
by definition of $\cA_{\le n}$. By the universal property of $W'_n C$, the map $W_n C\to C$
factors uniquely into
\[W_n C\to W'_n C.\]

That this is an isomorphism is checked again after truncation, and it is true by
Remark \ref{lF.3.6}.

Summarising: given a $t$-category $(\cT,t)$, a weight filtration $(W_n)$ in the sense of
\cite[Def. 1.2.1]{HuberLN}  is equivalent to a weight filtration on $\cT$ in the sense of
Definition \ref{d16.3tri} c) which is compatible with the $t$-structure, provided that the
weight filtration is separated and exhaustive on objects of the heart.
\end{remark}

\newpage
\printindex

\vfill

\end{document}